\documentclass[12pt]{book}

\usepackage{makeidx}
\usepackage[latin1]{inputenc}                  
\usepackage[T1]{fontenc}
\usepackage[francais]{babel}      
\usepackage{amsmath,amsfonts,amssymb}
\usepackage{latexsym} 
\usepackage{fullpage}

\newcommand{\Ac}{\mathcal A}
\newcommand{\Bc}{\mathcal B}
\newcommand{\Cc}{\mathcal C}
\newcommand{\Vc}{\mathcal V}
\newcommand{\Dc}{\mathcal D}

\newcommand{\Mc}{\mathcal M}
\newcommand{\Sc}{\mathcal S}
\newcommand{\Zc}{\mathcal Z}
\newcommand{\Nc}{\mathcal N}
\newcommand{\Pc}{\mathcal P}
\newcommand{\Hc}{\mathcal H}
\newcommand{\Qc}{\mathcal Q}
\newcommand{\Fc}{\mathcal F}

\newcommand{\Lc}{\mathcal L}
\newcommand{\Jc}{\mathcal J}

\newcommand{\Rb}{\mathbb{R}}
\newcommand{\Cb}{\mathbb{C}}
\newcommand{\Hb}{\mathbb{H}} 
\newcommand{\Nb}{\mathbb{N}}
\newcommand{\Zb}{\mathbb{Z}}
\newcommand{\Tb}{\mathbb{T}}
 
\newcommand{\Qb}{\mathbb{Q}}

\newcommand{\Ad }{\mbox{\rm Ad}}
\newcommand{\ad }{\mbox{\rm ad}}
\newcommand{\Coad }{\mbox{\rm Coad}}
\newcommand{\Ind}{\mbox{\rm Ind}}
\newcommand{\End}{\mbox{\rm End}}

\newcommand{\Id}{\mbox{\rm Id}}
\newcommand{\card}{\mbox{\rm Card}}

\newcommand{\nn}[1]{| #1 |}
\newcommand{\nd}[1]{\left\|#1 \right\|}

\newcommand{\pr}[2]{\mbox{pr}_{#1}(#2)}

\newcommand{\tr}[1]{\mbox{\rm trace } (#1) }
\newcommand{\supp}{\mbox{\rm supp } }

\newcommand{\fbs}[1]{\Jc_{#1}}
\newcommand{\fb}[2]{\fbs{#1}(#2)}

\newcommand{\fls}[2]{\Lc_{#1,#2}}
\newcommand{\fl}[3]{\fls{#1}{#2}(#3)}

\newcommand{\flbs}[2]{\bar{\Lc}_{#1,#2}}
\newcommand{\flb}[3]{\flbs{#1}{#2}(#3)}

\newcommand{\flsz}[1]{\Lc_{#1}}
\newcommand{\flz}[2]{\flsz {#1} ({#2})}

\usepackage{theorem}
\newtheorem{rem}{Remarque}
\theoremstyle{break}
\newtheorem{thm_intro}{Th\'eor\`eme}
\newtheorem{thm}{Th\'eor\`eme}[chapter]
\newtheorem{thm_princ}[thm]{Th\'eor\`eme principal}
\newtheorem{defi}[thm]{D\'efinition}
\newtheorem{lem}[thm]{Lemme}
\newtheorem{prop}[thm]{Proposition}
\newtheorem{cor}[thm]{Corollaire}

\newenvironment{proof}[1][]
{\begin{trivlist}
  \item[]\hspace{\parindent}{\em
      D\'emonstration {#1}:}}
  {\hfill $\square$
  \end{trivlist}}

\makeindex

\begin{document}

\thispagestyle{empty}
\sloppy\hbadness=5000\vbadness=100
{\hspace{-1cm}Num\'ero d'ordre : 7552 }\\
{\center 

\fbox{%
\raisebox{.8cm}{\parbox{11cm}{%
\large \sc \center Universit\'e Paris-Sud
\\UFR scientifique d'Orsay}
}}\\
\vspace{1cm}
{\huge \sc Th\`ese\\} 
\vspace{1cm}
Pr\'esent\'ee par\\
\vspace{1cm}
{\large \sc V\'eronique Fischer} \\
\vspace{1cm}
Pour obtenir le grade de \\
\vspace{1cm}
{\large \sc Docteur en sciences de l'universit\'e Paris XI Orsay}\\
\vspace{1cm}
Sp\'ecialit\'e : {\sc Math\'ematiques}\\
\vspace{1cm}
Sujet :
\vspace{.5cm}
{\large \sc 
\'Etude de deux classes de groupes nilpotents de pas deux }\\
\vspace{1cm}
Soutenue le 5 Juillet 2004 devant la commission d'examen compos\'ee de\\

\vspace{1cm}
\begin{center}
\begin{tabular}[c]{ll}
{Mr \sc Jean-Philippe Anker} & \\
{Mr \sc Pascal Auscher} & \\
{Mr \sc Laurent Clozel  } & \\
{Mr \sc Jean-Louis Clerc}&rapporteur\\
{Mr \sc G\'erard Lion } & \\
{Mr \sc Noel Lohou\'e} & 
\end{tabular}
\end{center}
}
\newpage
~\newpage


\tableofcontents

\chapter*{Introduction}

\addcontentsline{toc}{chapter}{\numberline{}Introduction}

Le but de ce travail est l'\'etude de la continuit\'e $L^p$ 
de certains op\'erateurs
sur deux classes de groupes nilpotents de pas deux,
avec comme outils, des formules de Plancherel
et des fonctions sph\'eriques.

La premi\`ere classe de groupes est form\'ee 
des groupes appel\'es de type H (ou de type Heisenberg),
la seconde des groupes nilpotents libres \`a deux pas.
On notera $N_{v,2}$ le groupe nilpotent libre de pas deux 
\`a $v$ g\'en\'erateurs tout au long de ce travail.
Pour ce dernier, 
nous avons d\'evelopp\'e un calcul de Fourier radial.

Les op\'erateurs principalement \'etudi\'es sont
les fonctions maximales associ\'ees aux sph\`eres de
Kor\'anyi et leurs fonctions d'aires,
ainsi que les op\'erateurs de convolution
d\'efinis gr\^ace au calcul de Fourier radial sur $N_{v,2}$
(probl\`eme des multiplicateurs).

\subsection*{Fonctions maximales}

Les fonctions maximales associ\'ees \`a  des boules, 
des sph\`eres\ldots
sont des outils naturels pour d\'emontrer 
la convergence presque partout des moyennes 
de fonctions sur ces supports
(th\'eor\`eme de diff\'erentiation de Lebesgue);
de plus, des fonctions maximales
associ\'ees \`a des semi-groupes d'op\'erateurs
interviennent dans des probl\`emes ergodiques
(th\'eor\`eme de Wiener).
Dans une autre direction,
certaines fonctions maximales permettent 
le contr\^ole d'op\'erateurs
(th\'eor\`eme des int\'egrales singuli\`eres \cite{CoifW} etc\ldots).

Sur les groupes de Lie,
plusieurs fonctions maximales ont  d\'ej\`a \'et\'e \'etudi\'ees.
Elles proviennent d'une part 
de la recherche de r\'esultats analogues au th\'eor\`eme de Wiener :
des in\'egalit\'es maximales  $L^p$
pour des familles de mesures formant un semi-groupe,
sont connues
sur des groupes de Lie semi-simples
\cite{nevo_simplegrI,nevo_simplegrII,margulis_nevo_stein_semisimplegr,nevo_stein_semisimplegr}
et sur le groupe de  Heisenberg
\cite{Nevo}.
Leurs d\'emonstrations reposent 
sur la m\'ethode classique d'\'etude des fonctions maximales 
de semi-groupes d'op\'erateurs \cite{stein_topics_in_an},
ainsi que sur l'\'evaluation de fonctions d'aire
gr\^ace \`a des propri\'et\'es spectrales.\\
D'autre part, sur l'espace euclidien 
ou plus g\'en\'eralement sur les groupes homog\`enes munis d'une norme homog\`ene \cite{Folst},
on s'int\'eresse \`a la  fonction maximale associ\'ee aux dilat\'ees d'une surface donn\'ee;
par exemple, la continuit\'e~$L^p$ des fonctions maximales associ\'ees
\`a la sph\`ere euclidienne   \cite{Maxfunc,stein_wainger},
ou \`a la sph\`ere de Kor\'anyi sur le groupe de Heisenberg
\cite{cowling}
a d\'ej\`a \'et\'e d\'emontr\'ee.
D'autres r\'esultats ont \'et\'e obtenus 
pour des fonctions maximales associ\'ees \`a un support dont la
courbure rotationelle ne s'annule pas, 
en utilisant les int\'egrales oscillantes
\cite{sogge_stein,schmidt,muller_seeger}.

\paragraph{Nos r\'esultats.}
Un chapitre de cette th\`ese est consacr\'e
\`a la d\'emonstration d'in\'egalit\'es~$L^p$ pour la fonction maximale sph\'erique 
sur les groupes de type H et $N_{v,2}$ et leur fonctions d'aires, en utilisant le m\^eme point
de d\'epart que \cite{Maxfunc}.
Ce r\'esultat est d\'ej\`a connu
pour le groupe de Heisenberg \cite{cowling} 
et plus g\'en\'eralement  pour les groupes de type H \cite{schmidt}
(avec des indices optimaux pour $p$ dans les deux cas),
mais pas pour les groupes~$N_{v,2}$.

Notre m\'ethode repose 
sur l'interpolation de la m\^eme famille analytique d'op\'erateurs que~\cite{Maxfunc},
sur la continuit\'e~$L^p$ 
de la fonction maximale standard 
et le contr\^ole~$L^2$ de fonctions d'aires.
Dans notre cas, 
ce contr\^ole s'effectuera par des techniques spectrales
proches de \cite{Nevo}.

\subsection*{Transform\'ee de Fourier radiale}

Les propri\'et\'es spectrales 
utilis\'ees
sont celles donn\'ees par  
les fonctions sph\'eriques born\'ees.
Leurs expressions sont bien connues
sur le groupe de Heisenberg \cite{bjr},
tout comme leurs g\'en\'eralisations sur les groupes de type~H
au sens de Damek et Ricci \cite{bigth}.

\paragraph{Nos r\'esultats.}
Nous avons construit les fonctions sph\'eriques born\'ees du groupe~$N_{v,2}$,
\`a l'aide de la th\'eorie des repr\'esentations du groupe 
produit semi-direct de $N_{v,2}$ par le groupe orthogonal ou sp\'ecial orthogonal.
Elles  n'\'etaient pas jusqu'alors explicites.
Nous avons aussi
confront\'e les expressions trouv\'ees \`a d'autres
caract\'erisations, surtout
celles donn\'ees par les repr\'esentations sur $N_{v,2}$.
Nous avons ensuite 
explicit\'e les mesures de Plancherel radiale et non radiale.
Les r\'esultats sont concordants entre eux,
et avec l'\'etude de Strichartz 
\cite[section 6]{stric}.
Nous avons ainsi obtenu la formule de Plancherel sph\'erique 
et la formule d'inversion radiale.

\subsection*{Multiplicateurs de Fourier sur \mathversion{bold}{$N_{v,2}$}}

Nous nous sommes int\'eress\'es ensuite
au probl\`eme des multiplicateurs
d\'efinis par passage en Fourier sph\'erique  
sur le groupe $N_{v,2}$.
Gr\^ace au calcul de Fourier sph\'erique
et \`a de longs calculs
proches du cas des groupes de type~H \cite{steinmullericci},
nous avons obtenu
certaines estimations $L^2$ \`a poids sur le groupe~$N_{v,2}$.
En appliquant le th\'eor\`eme des int\'egrales singuli\`eres,
nous avons alors abouti \`a des conditions suffisantes
pour le probl\`eme des multiplicateurs en Fourier.
Cependant, ces conditions ne sont pas satisfaites m\^eme pour les fonctions constantes.
L'objectif de cette premi\`ere \'etude est d'exposer nos solutions \`a quelques points techniques.

\newpage

\section*{Pr\'esentation de ce document et de nos r\'esultats}

\paragraph{Organisation de ce document.}
Dans le premier chapitre,  
nous  donnons  les  d\'efinitions, les notations 
sur les fonctions maximales des deux classes de groupes 
que nous \'etudions :
les groupes de type~H et
les groupes~$N_{v,2}$.  
Nous rappelons \'egalement la notion de fonctions sph\'eriques
qui existent sur les groupes  de type~H, 
et des \'el\'ements de la th\'eorie des repr\'esentations  
qui nous permettront de construire les fonctions sph\'eriques
sur  les   groupes~$N_{v,2}$.

Dans le chapitre~\ref{chapitre_fonction_maximale_spherique},
nous  montrons des in\'egalit\'es  $L^p$ 
pour la fonction  maximale  sph\'erique
sur  les groupes  de type~H
et sur les groupes~$N_{v,2}$.
Cela repose sur le contr\^ole $L^2$ de fonctions d'aire.

Dans le chapitre~\ref{chapitrecalculfourier},
nous donnons les expressions
des fonctions sph\'eriques
et la formule de Plancherel radiale.

Dans le chapitre~\ref{chapitre_utilisation},
gr\^ace au calcul de Fourier pr\'ec\'edemment d\'evelopp\'e,
nous montrons  le contr\^ole $L^2$ de fonction d'aire pour~$N_{v,2}$,
et nous \'etudions le probl\`eme des multiplicateurs sur~$N_{v,2}$.

\paragraph{\'Enonc\'es des  r\'esultats.}
Jusqu'\`a la fin de ce chapitre d'introduction,
nous pr\'esentons techniquement 
les r\'esultats suivants : 
\begin{itemize}
\item  les in\'egalit\'es maximales sph\'eriques sur les groupes de
  type~H 
  et sur les groupes~$N_{v,2}$,
\item  le calcul de Fourier du groupe $N_{v,2}$.
\end{itemize}

\subsection*{Fonction maximale sph\'erique}

Nous consid\'erons un groupe de Lie $N$ de type~H ou $N_{v,2}$.
C'est un groupe de Lie nilpotent connexe simplement connexe
que l'on peut identifier via l'exponentielle \`a son alg\`ebre de Lie~$\Nc$.

L'alg\`ebre de Lie~$\Nc$ est stratifi\'ee :
$\Nc\,=\,\Vc\oplus\Zc$
\index{Notation!Espace!$\Vc,\Zc$}
avec
$[\Vc,\Vc]=\Zc$
et
$[\Nc,\Zc]=\{0\}$
($\Zc$ est le centre de cette alg\`ebre).
On d\'efinit les dilatations naturellement associ\'ees \`a une
alg\`ebre stratifi\'ee de pas deux:
$$
\forall\, r>0\;,\quad
\forall\, X\in\Vc\; , Z\in\Zc
\quad
\delta_r.(X+Z)=rX+r^2Z
\quad.
$$
\index{Dilatation}
Elles induisent des dilatations $n\mapsto r.n$, $r>0$
sur le groupe~$N$.
De plus, les deux sous-espaces~$\Vc$ et~$\Zc$ 
sont naturellement munis de normes euclidiennes.
On peut ainsi d\'efinir la norme de Kor\'anyi
sur $N$ :
$$
\nn{n}
\,=\,
{\left(\nn{X}^4+\nn{Z}^2\right)}^\frac14
\quad\mbox{o\`u}\quad
n=\exp (X+Z)\;, \quad X\in\Vc\; ,\quad Z\in\Zc
\quad.
$$

On consid\`ere une base orthonorm\'ee pour $\Vc$ et $\Zc$
dans le cas d'un groupe de type~H,
et la base canonique
form\'ee par les g\'en\'erateurs et leurs crochets
dans le cas d'un groupe nilpotent libre \`a deux pas.
On peut ainsi fixer une mesure de Lebesgue sur $\Nc$ 
puis une mesure de Haar $dn$ sur $N$.

Cette norme et cette mesure \'etant fix\'ees,
on note $\mu$ la mesure support\'ee par la sph\`ere unit\'e 
\index{Notation!Mesure!$\mu$}
$S_1:=\{n,\;\nn{n}=1\}$,
pour laquelle on a le passage en coordonn\'ees polaires :
\index{Coordonn\'ee polaire!sur les groupes homog\`enes}
$$
\forall \,   f\; ,\qquad    
\int_N f(n)dn    \,=\,
\int_0^\infty\int_{S_1}
f(r.n)d\mu(n)r^{Q-1}dr 
\quad  .
$$
o\`u $Q=\dim \Vc+2\dim \Zc$ est la dimension homog\`ene du groupe $N$.
\index{Notation!Dimension!$Q$}
\index{Fonction maximale!sph\'erique $\Ac$}
La fonction maximale sph\'erique est l'op\'erateur $\Ac$ 
donn\'e pour une fonction $f$ localement int\'egrable sur $N$ par :
$$
\Ac.f(n)
\,:=\,
\sup_{r>0}
\nn{\int_{S_1} f(n\; r.{n'}^{-1}) d\mu(n')}\quad,\quad n\in N
\quad.
$$
Nous avons \'etudi\'e les propri\'et\'es  de cette fonction maximale $\Ac$ pour la norme $L^p$;
les espaces $L^p$ que nous consid\'erons sont relatifs \`a la mesure de Haar~$dn$ :
$$
\nd{f}_{L^p}
\,=\,
{\left(\int_N \nn{f(n)}^p dn\right)}^\frac1p
\quad.
$$
Nous  montrons des in\'egalit\'es maximales sph\'eriques $L^p$
sur $N$ dans le chapitre~\ref{chapitre_fonction_maximale_spherique} :

\begin{thm_intro}[In\'egalit\'e $L^p$ pour \mathversion{bold}{$\Ac$}]
  \label{thm_intro_in_max_sph_Lp}
\index{Fonction maximale!Th\'eor\`eme}
\index{Notation!Dimension!$v$}
\index{Notation!Dimension!$z$}
Notons $v=\dim \Vc$ et $z=\dim \Zc$.
  La fonction maximale sph\'erique v\'erifie des in\'egalit\'es $L^p$,
  \begin{itemize}
  \item[$\bullet$] $2\leq p\leq\infty$ si $v=4$,
  \item[$\bullet$] si $v>4$, dans le cas d'un groupe de type H, $(v-2)/(v-3)<p\leq\infty$,  
  \item[$\bullet$] si $v> 4$, dans le cas $N=N_{v,2}$,
    $2h_0/(2h_0-1)<p\leq\infty$
    o\`u $h_0$ est le minimum de 
    $v+1$ et de la partie enti\`ere de $(z-1)/4$, 
  \end{itemize}
  c'est-\`a-dire qu'il existe 
  une constante $C$ 
  qui d\'epend seulement de $v,z,p$ 
telle que :

  $$
  \forall\, f\in L^p\; ,\qquad
  \nd{\Ac.f}_{L^p}
  \leq C\nd{f}_{L^p}
  \quad.
  $$ 
\end{thm_intro}

Pour d\'emontrer ce th\'eor\`eme, 
nous avons besoin des expressions explicites 
des fonctions sph\'eriques born\'ees; 
elles sont connues sur les groupes de type~H \cite{bigth}; 
nous les avons explicit\'ees sur $N_{v,2}$ :

\subsection*{Calcul de Fourier radial sur \mathversion{bold}{$N_{v,2}$}}

On note $X_1,\ldots,X_v$ 
\index{Notation!Dimension!$v$}
les g\'en\'erateurs de l'alg\`ebre de Lie du groupe $N_{v,2}$, et $v=2v'$ ou $v=2v'+1$.
Les vecteurs $X_1,\ldots,X_v$ engendre une base d'un sous-espace $\Vc$,
que l'on munit du produit scalaire pour laquelle la base $(X_1,\ldots,X_v)$ est orthogonale.
 
On convient aussi de noter
le simplexe :
  \index{Notation!Ensemble de param\`etres!$\Lc,\bar{\Lc}$}
$\Lc:=\{ \lambda^*_1>\ldots>\lambda^*_{v'}>0 \}$,
et son adh\'erence 
$\bar{\Lc}:=\{ \lambda^*_1\geq\ldots\geq\lambda^*_{v'}\geq0 \}$.
\`A un \'el\'ement non nul $\Lambda^*\in\bar{\Lc}$,
on associe
les entiers $v_0$, $v_1$, 
le multi-indice d'entier $m\in \Nb^{v_1}$ 
et le $v_1$-uplet $(\lambda_1,\ldots,\lambda_{v_1})\in \Rb^{v_1}$ 
de la mani\`ere suivante :
\begin{itemize}
\item l'entier $v_0$ est tel que 
  $\lambda^*_{v_0}> 0$ et $\lambda^*_{v_0+1}=0$ :
  c'est le nombre de $\lambda^*_i$ non nuls;
\item $v_1$ est le nombre de $\lambda^*_i$ non nuls distincts,
  et les $\lambda_j$ sont les $\lambda^*_i$ non nuls et distincts, ordonn\'es de
  fa\c con strictement d\'ecroissante :
  $$
  \{
  \lambda^*_1\,\geq\, \ldots \,\geq\, \lambda^*_{v_0}\,>\,0
  \}
  \,=\,
  \{
  \lambda_1\, > \, \ldots \, > \, \lambda_{v_1}\,>\,0
  \}
  \quad;
  $$
\item on note $m_j$ le nombre de param\`etres $\lambda^*_i$ \'egaux \`a $\lambda_j$,
  on d\'efinit \'egalement :
  $m_0:=m'_0:=0$ et
  $m'_j:=\sum_{i=1}^j m_i$ pour
  $j=1,\ldots v_1$;
  on a $m'_{v_1}:=m_1+\ldots+m_{v_1-1}+m_{v_1}=v_0$.
\end{itemize}

Nous avons obtenu les expressions explicites
des fonctions sph\'eriques born\'ees 
de la paire de Guelfand 
$(N_{v,2} , SO(v))$ et $(N_{v,2} , O(v))$.
Nous pr\'esentons ces derni\`eres :

\begin{thm_intro}[Fonctions sph\'eriques sur \mathversion{bold}{$N_{v,2}, O(v)$}]
  \label{thm_intro_fcnsph}
  \index{Fonction sph\'erique!sur $N_{v,2}$}
  Les param\`etres des fonctions sph\'eriques born\'ees sont 
  $r^*,\Lambda^*,l$, d\'ecrits dans ce qui suit :
  \index{Notation!Param\`etres!$r^*,\Lambda^*,l,\epsilon$}
  \begin{enumerate}
  \item  $\Lambda^*\in \bar{\Lc}$;\\
    lorsque  $\Lambda^*$ est non nul,
    on lui associe comme d\'ecrit ci-dessus 
    les entiers $v_0$, $v_1$, 
    le multi-indice d'entier $m\in \Nb^{v_1}$ 
    et le $v_1$-uplet $(\lambda_1,\ldots,\lambda_{v_1})\in \Rb^{v_1}$ ;
  \item le multi-indice $l\in \Nb ^{v_1}$ 
    si $\Lambda^*\not=0$, 
    rien sinon;
  \item $r^*\geq 0$, avec $r^*=0$ si $2v_0=v$.
  \end{enumerate}

  Avec ces param\`etres, 
  les fonctions sph\'eriques born\'ees pour $N_{v,2}, O(v)$ 
  sont donn\'ees par :\\
  \textbf{Si} {\mathversion{bold}{$\Lambda^*\not=0$}}
  $$
  \phi^{r^*,\Lambda^*,l}(n)
  \,=\,
  \int_{k\in K}
  \Theta^{r^*,\Lambda^*,l}(k.n)
  dk,
  \quad n\in N_{v,2}
  \quad,
  $$
  \index{Notation!Fonction sph\'erique!$\phi^{r^*,\Lambda^*,l}$}
  \index{Notation!$\Theta^{r^*,\Lambda^*,l}$}
  o\`u $\Theta^{r^*,\Lambda^*,l}$ est la fonction sur $N_{v,2}$ donn\'ee 
  pour $n=\exp(X+A)\in N_{v,2}$ par :
  $$
  \Theta^{r^*,\Lambda^*,l}(n)
  \,=\,
  e^{i <r^*X_v^*,X>}
  e^{i\sum_{j=1}^{v'} \lambda^*_j a_{2j-1,2j}}
  \Pi_{j=1}^{v_1}
  \flb {l_j} {m_j-1} {\frac{\lambda_j}2 \nn{\pr {j} X }^2}
  \quad,
  $$
  o\`u on a d\'ecompos\'e
  $X=\sum_{j=1}^v x_j X_j$
  et
  $A=\sum_{i,j} a_{i,j}[X_i,X_j]$
  et not\'e :
  \begin{itemize}
  \item $\flbs n \alpha $ la fonction de Laguerre normalis\'ee 
    de degr\'e $n$ et de param\`etre $\alpha$ 
    (voir sous section~\ref{app_lag}),
  \item $\mbox{pr}_j$ la projection orthogonale
    sur l'espace vectoriel engendr\'e 
    par les  $2m_j$ vecteurs :
    $$
    X_{2i-1}\, ,\; X_{2i}\,,\quad
    m'_{j-1}\,<\,i\,\leq\, m'_j
    \quad.
    $$
    \index{Notation!$\mbox{pr}_j$}
  \end{itemize}
  \textbf{Si} {\mathversion{bold}{$\Lambda^*=0$}}
  \index{Notation!Fonction sph\'erique!$\phi^{r^*,0}$}
  $$
  \phi^{r^*,0}(n)
  \,=\,
  \fb {\frac{v-2}2}{r^*\nn{X}}
  \quad ,
  \quad n=\exp(X+A)\in N_{v,2}
  \quad,
  $$
  o\`u $\fbs \alpha$ est la fonction de Bessel r\'eduite
  (voir sous section~\ref{app_bessel}).
\end{thm_intro}

Ce th\'eor\`eme sera d\'emontr\'e
dans le chapitre~\ref{chapitrecalculfourier},
ainsi que le th\'eor\`eme donnant les fonctions sph\'eriques born\'ees pour $SO(v)$.

Les fonctions sph\'eriques sont fonctions propres des op\'erateurs diff\'erentiels sur~$N_{v,2}$ invariant \`a gauche et sous $SO(v)$, en particulier du laplacien de Kohn 
$L=-\sum_i X_i^2$. Dans le m\^eme chapitre, nous donnerons les valeurs propres associ\'ees \`a $L$ pour chaque fonction sph\'erique born\'ee :
\index{Sous-laplacien!valeur propre} 
avec les notations du th\'eor\`eme~\ref{thm_intro_fcnsph},
on a dans les cas $\Lambda^*\not=0$ et $\Lambda^*=0$ respectivement :
$$
L. \phi^{r^*,\Lambda^*,l}
\,=\,
\left( \sum_{j=1}^{v_1}
  \lambda_j (2 l_j+m_j) 
  +{r^*}^{2}\right)
\phi^{r^*,\Lambda^*,l}
\quad\mbox{et}\quad
L. \phi^{r^*,0}
\,=\,
{r^*}^{2}
\phi^{r^*,0}
\quad .
$$

Dans le  chapitre~\ref{chapitre_utilisation},
nous donnons l'expression de la mesure de Plancherel radiale pour $N_{v,2},O(v)$:

\begin{thm_intro}[Mesure de Plancherel radiale]\label{thm_intro_mesure_plancherel}
  \index{Mesure de Plancherel!radiale}
  La mesure de Plancherel radiale
  est support\'ee 
  par les fonctions sph\'eriques $\phi^{r^*,\Lambda^*,l}$, 
  dont les param\`etres $X^*,\Lambda^*,l$ sont dans l'ensemble 
  $\Pc_v=\Pc$ donn\'e par :
  \begin{eqnarray*}
    \Pc_{2v'}
    &=&
    \{ 
    (0,\Lambda^*,l)\; , \quad \Lambda^*\in \Lc\, ,\; l\in \Nb^{v'}
    \} \quad,
    \\
    \Pc_{2v'+1}
    &=&
    \{ 
    (r^*,\Lambda^*,l)\; , \quad 
    r^*\in \Rb^+\, ,\; \Lambda^*\in \Lc\, ,\; l\in \Nb^{v'}
    \}  \quad.
  \end{eqnarray*}
  En identifiant les fonctions $\phi^{r^*,\Lambda^*,l}$, 
  et leurs param\`etres $r^*,\Lambda^*,l$,
  la mesure de Plancherel 
  est la mesure sur l'ensemble $\Pc$ produit  tensoriel : 
  \begin{itemize}
  \item de la mesure sur le simplexe $\Lc$ donn\'ee par :
    $$
    \Pi_j \lambda_j 
    d \eta(\Lambda)
    \quad, \quad \Lambda=(\lambda_1,\ldots,\lambda_{v'})\quad,
    $$
    o\`u $\eta$ est la mesure sur le simplexe $\Lc$ 
    dont l'expression est donn\'ee dans le lemme~\ref{lem_pass_coord_pol},
  \item de la mesure $\sum$ de comptage sur $\Nb^{v'}$,
  \item et si $v=2v'+1$ de la mesure de Lebesgue $dr^*$ sur $\Rb^+$,
  \end{itemize}
  $$
  c(v)
  \,=\,
  \left\{
    \begin{array}{ll}
      {(2\pi)}^{-\frac{v(v-1)}2-v'}
      &\quad
      \mbox{si}\; v=2v' \; ,\\
      2 {(2\pi)}^{-\frac{v(v-1)}2-1-v'}
      &\quad
      \mbox{si}\; v=2v'+1 \; .
    \end{array}\right.
  $$
\end{thm_intro}

Il d\'ecoule des th\'eor\`emes~\ref{thm_intro_fcnsph} et~\ref{thm_intro_mesure_plancherel} 
un calcul de Fourier radial sur $N_{v,2},O(v)$;
on \'etudie alors le probl\`eme des multiplicateurs de Fourier dans la section~\ref{sec_multiplicateur}.


\chapter{G\'en\'eralit\'es}
\label{chapitregen2}

Dans la premi\`ere section de ce chapitre, 
nous rappelons les d\'efinitions 
des fonctions maximales sph\'eriques,
puis nous pr\'esentons les groupes de type~H 
et les groupes libres nilpotents \`a deux pas 
ainsi que leurs structures homog\`enes.
Nous donnons ensuite les propri\'et\'es caract\'eristiques
des fonctions sph\'eriques 
dans la seconde section, 
puis dans une troisi\`eme section, 
des \'el\'ements sur la th\'eorie des repr\'esentations. 

\section{Groupes et fonctions maximales \'etudi\'es}
\label{sec_groupe_fcnmax}

Dans ce travail, 
tous les groupes de Lie nilpotents sont suppos\'es 
CONNEXES SIMPLEMENT CONNEXES. 
Lorsqu'une mesure de Haar $dn$ est fix\'ee 
sur un groupe~$N$ nilpotent,
pour une fonction $f$ localement int\'egrable sur~$N$,
on d\'efinit sa norme $L^p$ :
$$
\nd{f}_{L^p}
\,:=\,
\nd{f}_p
\,:=\,
{\left( \int \nn{f(n)}^p dn \right)}^\frac1p
\quad;
$$
on notera aussi parfois $\nd{f}_2=\nd{f}$.

\subsection{Fonctions maximales sph\'eriques}
\label{subsec_fcnmax_sph}

Dans cette section, on consid\`ere un groupe $N$
homog\`ene muni d'une norme homog\`ene
\cite{Folst};
on note  $n\mapsto \nn{n}$ la norme homog\`ene,
$n\mapsto r.n$, $r>0$, la famille de dilatations,
\index{Dilatation}
et $Q$ la dimension homog\`ene.
\index{Notation!Dimension!$Q$}
Lorsqu'il n'y aura pas de confusion sur la structure choisie,
on omettra  le   qualificatif  ``homog\`ene''.

Le groupe $N$ est donc nilpotent.

On suppose qu'une mesure de Haar $dn$ est fix\'ee.
Pour un ensemble $E\subset N$,
on note 
$r.E=\{ r.n,\, n\in E\}$ l'ensemble dilat\'e,
et $\nn{E}$ sa mesure de Haar lorsque $E$ est mesurable;
toujours dans ce cas, on a $\nn{r.E}= r^Q\nn{E}$.

On d\'efinit comme dans le cas euclidien,
la boule homog\`ene
centr\'ee en~$n$ de rayon~$r$ :
$$
B(n,r)
\,:=\,
\{n'\in N\quad:\quad   
\nn{n{n'}^{-1}}\, < \,  r\}  
\quad,
$$
et la sph\`ere homog\`ene
centr\'ee en~$n$ de rayon~$r$ :
$$
S(n,r)
\,:=\,
\{n'\in N\quad:\quad   
\nn{n{n'}^{-1}}\, =\,  r\}  
\quad.
$$
En  particulier,  
la  boule unit\'e  $B(0,1)$  est  not\'ee
$B_1$,
et la   sph\`ere  unit\'e  $S(0,1)$, $S_1$.

Il  existe une unique mesure de  Radon $\mu$
\index{Notation!Mesure!$\mu$} 
sur la sph\`ere  unit\'e $S_1$ 
telle que l'on ait l'\'egalit\'e
\cite{Folst},  proposition~1.15 
\index{Coordonn\'ee polaire!sur les groupes homog\`enes}:
\begin{equation}
  \label{formule_changement_polaire}
  \forall    f\in L^1(N),\qquad    
  \int_{N}
  f(n)dn    
  \,=\,
  \int_{r=0}^\infty\int_{S_1}
  f(rn)d\mu(n)r^{Q-1}dr 
  \quad.
\end{equation}
On note $\mu_s$ la mesure dilat\'ee de la mesure $\mu$ 
dans le sens suivant:
$$
\int_{S_1} f(s.n)d\mu(n) \,=\, \int_{s.S_1} f(n) d\mu_s(n) \quad .
$$

Nous d\'efinissons alors 
la \textbf{fonction maximale sph\'erique}
\index{Fonction maximale!sph\'erique $\Ac$}
$$
\Ac.f\,:=\, \sup_{s>0}\nn{\mu_s*f} \quad .
$$
Pour une fonction $f$ localement int\'egrable sur $N$, 
la fonction maximale $\Ac.f$ est mesurable
(il suffit de consid\'erer le supremum 
sur tous les rationels positifs).

Nous nous sommes int\'eress\'es  
aux propri\'et\'es $L^p$ de la fonction maximale $\Ac$ :
$$
\forall\, f\;, \qquad 
\nd{\Ac.f}_{L^p}\,\leq\, C_p\, \nd{f}_{L^p} \quad .
$$

Rappelons la d\'efinition de la fonction maximale standard
que nous noterons $\Mc$
\index{Fonction maximale!standard $\Mc$}
pour une fonction $f$ localement int\'egrable sur $N$ :
$$
\Mc .f(n)
\,:=\,
\frac 1{\nn{B(n,r)}}
\int_{B(n,r)} f(n') dn'
\quad n\in N \quad.
$$
Il est bien connu que
cette fonction  maximale 
v\'erifie des in\'egalit\'es 
$L^p, 1<p\leq\infty$ et $L^1$-faible,
\cite[th\'eor\`eme  fondamental des  int\'egrales singuli\`eres]{CoifW} :
\begin{equation}
  \label{maxst}
  \forall\, f\;, \qquad 
  \nd{\Mc.f}_{L^p}\,\leq\, C_p\, \nd{f}_{L^p} \quad .
\end{equation}
On en d\'eduit le corollaire:

\begin{cor}
  \label{cor_fcndec}
  Soit $F:N\mapsto\Rb$ une  fonction int\'egrable 
  positive telle que
  $F(n)=m(\nn{n})$ o\`u  $m$ est  une fonction d\'efinie  sur $\Rb^+$,
  d\'ecroissante.    
  On  d\'efinit  les   fonctions  $F_t,   t>0$  
  par  $F_t(n)=t^{-Q}F(t^{-1}.n)$,  
  et  leurs  op\'erateurs de  convolution
  $T_t:f\mapsto F_t*f$.
  
  Alors la famille d'op\'erateur  
  $\{  T_t, t>0\}$  v\'erifie  une in\'egalit\'e  maximale :
  $L^p,  1<p\leq \infty$:
  $$
  \forall\,   f\in L^p\;,
  \qquad   \nd{\sup_{t>0}\nn{T_t.f}}_{L^p}   \leq   C
  \nd{F}_{L^1} \;\nd{f}_{L^p} \quad ,
  $$
  o\`u $C$ est  une  constante  qui ne d\'epend que de la strucure homog\`ene du groupe $N$.
\end{cor}

Mais pour  la fonction maximale sph\'erique, 
le  cas g\'en\'eral n'est pas  connu.   
Il l'est  sur l'espace euclidien $\Rb^n$ \cite{Maxfunc},
et sur le groupe de Heisenberg $\Hb^n$ pour la norme de Kor\'anyi 
\cite{cowling,schmidt} avec des indices pour $p$ optimaux.

\subsection{Groupes de type Heisenberg}
\label{subsec_def_typeH}

Introduits par Kaplan \cite{kaplan},
les groupes  de  type Heisenberg  (ou  de  type~H) 
g\'en\'eralisent  les groupes de Heisenberg 
dans le sens o\`u 
les solutions \'el\'ementaires du sous-laplacien 
sont formellement identiques.

\begin{defi}
  \label{defi_alg_typeH}
  Soit $\Nc$ une alg\`ebre de Lie. 
  $\Nc$ est une \textbf{alg\`ebre de  type~H} 
  (ou de type Heisenberg) 
  \index{Alg\`ebre!de type H}
  lorsqu'elle v\'erifie les trois conditions suivantes:
  \begin{itemize}
  \item[(1)] En  tant qu'espace vectoriel, 
    $\Nc$ est muni  d'un produit scalaire not\'e $<\, ,>$ 
    et se d\'ecompose en somme directe orthogonale
    de deux sous-espaces non nuls $\Vc$ et $\Zc$ :  
    $\Nc=\Vc\oplus\Zc$.
  \item[(2)]  En tant  qu'alg\`ebre  de Lie,  
    $[\Vc,\Vc]\subset \Zc$  et  $[\Nc,\Zc]=\{0\}$.
    En particulier, cette alg\`ebre de Lie est nilpotente.
  \end{itemize}    
  Lorsque  les  conditions  (1)  et  (2) sont satisfaites,  
  on  d\'efinit l'application 
  $J:\Zc\rightarrow \End\, \Vc$ 
  par :
  $$
  \forall    X,X'\in     \Vc    ,\;\forall    
  Z\in    \Zc,\qquad
  <J(Z)X,X'>\,=\,<Z,[X,X']> \quad .
  $$
  \begin{itemize}
  \item[(3)] Pour tout $Z\in\Zc$, 
    on a: $ J^{2}(Z)=-{|Z|}^2 \Id_\Vc$.
  \end{itemize}
  Un \textbf{groupe de type H}
  \index{Groupe!de type H} 
  est un groupe  de lie $N$ 
  dont l'alg\`ebre de Lie $\Nc$ est de type~H
\end{defi}
\index{Notation!Espace!$\Vc,\Zc$}

Lorsque $N$ est un groupe de type~H
avec les notations de la d\'efinition ci-dessus, 
on a $[\Vc,\Vc]=\Zc$ 
et $\Zc$ est le centre de l'alg\`ebre de Lie.
On identifie 
$N\sim \Nc \sim \Vc\oplus \Zc$ pour noter $n=(X,A)\in N$.

\begin{rem}\label{V_dim_paire}
  L'application $J$ est lin\'eaire, inversible 
  et \`a valeurs dans l'ensemble des  endomorphismes  antisym\'etriques de  l'espace  vectoriel $\Vc$.  
  En cons\'equence, 
  l'espace vectoriel $\Vc$ est de dimension paire.
\end{rem}

\begin{rem}\label{dim_centre<}
  La condition~(3) \'etablit le lien 
  entre la structure euclidienne et
  celle  de  Lie.  
  Elle  est  \'equivalente \`a  la  condition~(3bis) suivante :
  \begin{itemize}
  \item[(3bis)] Pour tout $X\in\Vc,\nn{X}=1$,  
    $\ad(X)$ est une isom\'etrie de
    $(\ker\, \ad(X))^\perp$ sur $\Zc$.
  \end{itemize}
  En cons\'equence, on a : $\dim \Zc < \dim \Vc$.
\end{rem}

\paragraph{Exemple :  le  groupe  de Heisenberg.}
\index{Groupe!de Heisenberg $\Hb^n$}
Dans ce travail, 
on choisit la loi suivante sur le groupe de Heisenberg $\Hb^n=\Cb^n\times\Rb$ :
$$
(z_1,\ldots,    z_n,   t)   .(z'_1,\ldots,    z'_n,   t')   \,=\,
(z_1+z'_1,\ldots,z_n+z'_n  ,   t+t'+\frac12\sum_i  \Im  z_i\bar{z}'_i)
\quad.
$$
Son alg\`ebre de Lie s'identifie comme espace vectoriel \`a
$\Rb^{2n}\oplus \Rb$.
Le centre de cette alg\`ebre est $\Zc=\Rb(0,1)\sim \Rb$.
Sa structure d'alg\`ebre de type~H
correspond \`a des matrices~$J_t, t\in \Zc\sim \Rb$
de taille $2n$,
diagonalis\'ees par bloc 2-2, 
dont les $n$ blocs 2-2 sont tous  de la forme :
$$
t
\left[\begin{array}{cc}
    0&-1\\1&0
  \end{array}\right]
\quad.
$$
Un produit direct de groupes de Heisenberg 
(ou de type~H) \'etant encore un groupe de type~H;
on obtient ainsi une grande classe d'exemples 
de groupes de type~H.

\subsection{Groupes nilpotents libres \`a deux pas}
\label{subsec_def_grlib}

On d\'efinit d'abord 
l'alg\`ebre de Lie nilpotente libre 
\`a~deux~pas et $v$~g\'en\'erateurs.
Heuristiquement, 
c'est l'alg\`ebre de Lie nilpotente de~pas~deux 
engendr\'ee par $v$ vecteurs libres,
tels que les seules relations entre leurs crochets
soient celles n\'ecessaires \`a~l'anticommutativit\'e.
Pour des pas quelconques,
les alg\`ebres de Lie libres nilpotentes se d\'efinissent par propri\'et\'e universelle 
\cite[Chap.V \S 4]{jacobson}.

\begin{defi}
  \label{defi_alg_libre}
  Une \textbf{alg\`ebre   de  Lie  libre nilpotente \`a deux pas} et $v$ g\'en\'erateurs
  \index{Alg\`ebre!libre nilpotente \`a 2 pas}
  est une alg\`ebre de Lie qui admet en tant qu'espace vectoriel 
  une base :
  $$
  X_i\, , \; 1\leq i\leq v
  \; ; \; X_{i,j}\, , \; 1\leq i<j\leq v \; ,
  \quad\mbox{telle que} \quad
  X_{i,j}=[X_i,X_j],i<j
  \quad.
  $$
  Elle est unique \`a  isomorphisme  pr\`es  
  et on la note $\Nc_{v,2}$.
\end{defi}

Les bases $X_1,\ldots,X_v$ et $X_{i,j}, i<j$ sont appel\'ees canoniques dans la suite.

\index{Notation!Espace!$\Vc,\Zc$}
On notera  $\Vc$ 
l'espace vectoriel 
engendr\'e par  les vecteurs $X_i, i=1,  \ldots,  v$, 
et  $\Zc$  l'espace  vectoriel  engendr\'e par  les
vecteurs $X_{i,j}, i<j$.
Le sous-espace  $\Zc$ 
est le centre  de l'alg\`ebre de  Lie; 
il est de  dimension $z=v(v-1)/2$.  
\index{Notation!Dimension!$z$}

On convient tout au long de ce travail que
lorsque l'on \'ecrit $X+A\in \Nc$, 
on sous-entend $X\in \Vc$ et $A\in \Zc$.

\begin{defi}
  Le \textbf{groupe libre nilpotent \`a  deux pas} 
  \index{Groupe! libre nilpotent \`a 2 pas ($N_{v,2}$)}
  est  le groupe nilpotent
  dont l'alg\`ebre de  Lie est $\Nc_{v,2}$.  
  On le  note $N_{v,2}$.
\end{defi}

\subsubsection*{Une r\'ealisation de \mathversion{bold}{$\Nc_{v,2}$}.}

Outre la d\'efinition par g\'en\'erateurs,
on peut d\'efinir l'alg\`ebre de Lie nilpotente libre \`a deux pas comme suit.
Soit~$(\Vc, <,>)$ un espace vectoriel euclidien de dimension $v$.
On note $O(\Vc)$
le groupe des transformations orthogonales de $(\Vc, <,>)$
et $SO(\Vc)$ son sous-groupe des transformations sp\'eciales orthogonales.
Leur alg\`ebre de Lie commune s'identifie
\`a l'espace vectoriel des transformations antisym\'etriques de~$(\Vc, <,>)$,
que l'on note $\Zc$.
On d\'efinit la somme ext\'erieure d'espaces vectoriels :
$\Nc=\Vc \oplus \Zc$.

\paragraph{D\'efinissons le crochet de Lie.}
On d\'efinit une application bilin\'eaire
$[.,.]$
sur~$\Vc\times \Vc$ \`a valeurs dans~$\Zc$
par 
$$
[X,Y].(V)\,=\, <X,V>Y-<Y,V>X
\quad X,Y,V\in \Vc,
\quad.
$$
On \'etend alors cette application bilin\'eaire antisym\'etrique
sur~$\Nc\times\Nc$ par :
$$
[.,.]_{\Nc\times \Zc}\,=\,
[.,.]_{\Zc\times \Nc}\,=\,
0\quad.
$$
Ce crochet v\'erifie trivialement l'identit\'e de Jacobi;
il munit l'espace vectoriel $\Nc$ d'une structure d'alg\`ebre de Lie nilpotente de pas deux.

\paragraph{Action de \mathversion{bold}{$O(\Vc)$}.}
De plus, les groupes $O(\Vc)$ et son sous-groupe $SO(\Vc)$ agissent 
d'une part, par automorphisme sur $\Vc$,
et d'autre part, 
par la repr\'esentation adjointe 
$Ad_\Zc$ sur leur alg\`ebre de Lie commune $\Zc$.
On peut donc d\'efinir une action de~$O(\Vc)$ et de~$SO(\Vc)$
sur l'alg\`ebre de Lie~$\Nc$.
Montrons qu'elle respecte le crochet; 
il suffit de voir pour $X,Y,V\in\Vc$ et $k\in O(\Vc)$ :
\begin{eqnarray*}
  [k.X,k.Y](V)
  &=&
  <k.X,V>k.Y-<k.Y,V>k.X\\
  &=&
  k.\left( <X, ^tk.V>Y-<Y, ^tk.V>X\right)\\
  &=&
  k.[X,Y](k^{-1}.V)
  \,=\, Ad_\Zc k.[X,Y]
  \quad.
\end{eqnarray*}
On a donc une action par automorphisme du groupe $O(\Vc)$ 
(et de son sous-groupe $SO(\Vc)$)
sur l'alg\`ebre $\Nc$.

\paragraph{Produit scalaire sur \mathversion{bold}{$\Zc$}.}
L'application donn\'ee par
$X\wedge Y\mapsto [X,Y]$.
s'\'etend en un isomorphisme (d'espace vectoriel) de $\Lambda^2(\Vc)$ sur $\Zc$;
et les \'el\'ements $[X,Y]$ sont des g\'en\'erateurs de $\Zc$.
Comme sur $\Lambda(V)^2$, 
on d\'efinit le produit scalaire de $\Zc$ 
en \'etendant par bilin\'earit\'e la forme suivante :
$$
<[X,Y],[X',Y']>
\,=\, 
<X,X'><Y,Y'>-<X,Y'><X',Y>
\quad.
$$

On remarque 
$<[X,Y],[X',Y']>= <[X,Y]X',Y'>$.
Et donc gr\^ace \`a l'identification entre $\Zc$ et $\Zc^*$ par le produit scalaire,
on a pour $A^*\in\Zc^*$ et $X,Y\in \Vc$:
\begin{equation}
  \label{egalite_crochet_pdtsc}
  <A^*,[X,Y]>
  \,=\,
  <A^*.X,Y>
  \quad.  
\end{equation}

\subsubsection*{Lien avec la d\'efinition par g\'en\'erateurs.}
Fixons une base  orthonormale $X_1,\ldots,X_v$ de~$\Vc$;
alors les vecteurs $X_{i,j}:=[X_i,X_j], i<j$ forment une base de $\Zc$.
Ainsi, l'alg\`ebre de Lie~$\Nc$ s'identifie
\`a l'alg\`ebre de Lie nilpotente libre \`a deux pas 
avec pour g\'en\'erateurs $X_1,\ldots,X_v$.

La base $X_{i,j}, i<j$ est orthonormale pour le produit scalaire de $\Zc$;
elle permet d'identifier l'espace vectoriel~$\Zc$ \`a l'ensemble $\Ac_v$
des matrices antisym\'etriques de taille $v$.

L'\'egalit\'e~(\ref{egalite_crochet_pdtsc})
se red\'emontrent alors en d\'eveloppant $A^*$ et $ X,Y$ sur les bases canoniques.

\paragraph{Action de \mathversion{bold}{$K=O(v)$} ou \mathversion{bold}{$K=SO(v)$}.}
Par leurs bases canoniques,
on identifie $\Vc$ \`a $\Rb^v$ 
et $\Zc$ \`a $\Ac_v$;
on a donc les deux actions par automorphismes :
$$
\left\{\begin{array}{rcl}
    K\times \Vc
    &\longrightarrow& \Vc\\
    k,X&\longmapsto&k.X
  \end{array}\right.
\quad\mbox{et}\quad
\left\{\begin{array}{rcl}
    K\times \Zc
    &\longrightarrow& \Zc\\
    k,A&\longmapsto&k.A=kAk^{-1} 
  \end{array}\right. \quad.
$$
On v\'erifie facilement :
$k.[X,X']=
[k.X,k.X']$ 
en d\'eveloppant $X$ et $X'$ sur la base canonique.
On retrouve l'action (par automorphismes) du groupe $K$ 
sur l'alg\`ebre de Lie $\Nc_{v,2}$,
puis sur le groupe $N_{v,2}$.

\subsection{Choix de la structure de groupe homog\`ene.}
\label{subsec_structure_homogene}

Soit $N$ un groupe de type~H, 
ou un groupe libre nilpotent \`a deux pas, 
et $\Nc$ son alg\`ebre de Lie.
On \'ecrit $\Nc=\Vc\oplus\Zc$
en reprenant les notations 
des d\'efinitions~\ref{defi_alg_libre} et~\ref{defi_alg_typeH}.
L'alg\`ebre $\Nc$ est stratifi\'ee \`a deux pas
\cite{Folst}.

On \'equipe alors le groupe de dilatations adapt\'ees 
\`a la stratification :
$$
\forall X\in\Vc,Z\in\Zc\quad
\forall r>0
\qquad
r.\exp(X+Z)\,=\,\exp(rX+r^2Z)
\quad,
$$
\index{Dilatation}
et d'une norme homog\`ene ``de Kor\'anyi'' :
$$
\forall X\in\Vc,Z\in\Zc
\qquad
\nn{\exp(X+Z})\, =\, {( \nn{X}^4+\nn{Z}^2 )}^{\frac14}
\quad.
$$
o\`u les normes euclidiennes sur $\Vc$ et $\Zc$ 
sont choisies de la mani\`ere suivante :
\begin{itemize}
\item si $N$ est un groupe de type~H, 
  ces normes euclidiennes sont issues de la norme euclidienne pour $\Nc$;
\item si $N=N_{v,2}$ 
  est un groupe libre nilpotent \`a deux pas,
  ces normes euclidiennes sont celles pour lesquelles les bases 
  $X_i, 1\leq i\leq v$
  et $X_{i,j}, i<j$ sont orthogonales respectivement.
\end{itemize}
Le groupe $N$ est ainsi muni 
de la structure naturelle 
de groupe homog\`ene avec une norme homog\`ene 
pour un groupe stratifi\'e.
La dimension homog\`ene est
$Q=\dim \Vc+2\dim\Zc$
\index{Notation!Dimension!$Q$}.

Nous fixons la mesure de Haar suivante sur $N$ : 
$$
\int_N f(n)dn
\,=\,
\int_{\Vc}\int_{\Zc}f(\exp(X+Z)) dXdZ
\quad,
$$
o\`u  les mesures de Lebesgues $dX,dZ$ 
sur $\Vc$ et $\Zc$ 
sont fix\'ees telles que :
\begin{itemize}
\item si $N$ est un groupe de type~H, 
  une base orthonormale 
  sur chacun des sous espaces $\Vc$ et $\Zc$
  donne lieu \`a deux mesures de Lebesgues 
  $dX$ et $dZ$ sur $\Vc$ et $\Zc$ respectivement;
\item si $N=N_{v,2}$ 
  est un groupe libre nilpotent \`a deux pas,
  les mesures de Lebesgues sur $\Vc$ et $\Zc$ 
  sont celles donn\'ees par les bases canoniques 
  $X_i, 1\leq i\leq v$
  et $X_{i,j}, i<j$
  respectivement.
\end{itemize}

\section{Fonctions sph\'eriques}
\label{sec_genfcnsph}

Dans cette section,
on rappelle d'abord
les diff\'erentes d\'efinitions \'equivalentes 
des fonctions sph\'eriques.
Puis, 
on donne les d\'efinitions des paires de Guelfand 
et leurs propri\'et\'es.
Nous illustrons alors ces  notions 
par l'exemple du groupe de Heisenberg.
On donne \'egalement
leurs g\'en\'eralisations au sens de \cite{bigth} sur les groupes de type~H.
Enfin, on pr\'ecise  l'action du groupe sp\'ecial orthogonal 
sur le groupe libre nilpotent \`a deux pas qui en fait une paire de Guelfand.

Pour ce qui suit, 
on renvoit aux cours~\cite{faro2} 
et au livre~\cite{helgason}.

Dans ces rappels et la premi\`ere sous-section,
on consid\`ere un groupe~$G$, localement compact, 
et un de ses sous-groupes compacts~$K$,
sur lesquels sont fix\'ees des mesures de Haar :
$dg$ sur $G$, et $dk$ (de masse 1) sur $K$.

\begin{defi}[Fonctions sph\'eriques et caract\`eres]
  \label{defi_fcnsph}
  \index{Fonction sph\'erique}
  Une fonction  sph\'erique est une fonction $\Phi$ 
  continue sur~$G$,  biinvariante par~$K$ 
  (i.e. invariante sous les actions 
  \`a gauche et \`a droite de $K$), 
  telle que l'application :
  $$
  \chi:f \,\longmapsto\, \int_G f(g)\Phi(g^{-1})dg
  $$
  soit  un  caract\`ere  non  nul 
  de  l'alg\`ebre  de  convolution des fonctions continues, 
  \`a support compact, biinvariantes par~$K$; c'est-\`a-dire qu'elle doit v\'erifier $\chi(f_1*f_2)=\chi(f_1)\chi(f_2)$ pour toutes fonctions $f_1,f_2:G\rightarrow \Cb$  
  continues, \`a support compact, biinvariantes par~$K$.

  La fonction obtenue par passage au quotient
  sur $G/K$ s'appelle encore fonction sph\'erique.
\end{defi}

On peut aussi les d\'efinir
par une \'equation fonctionnelle
(th\'eor\`eme~\ref{thm_eqfonc}),
et dans le cas de groupe de Lie, 
comme fonction propre 
d'op\'erateurs diff\'erentiels 
(th\'eor\`eme~\ref{thm_fcnsph_opdiff}),
\cite[sec.3 ch.X]{helgason}.

\begin{thm}[Fonctions sph\'eriques et \'equation fonctionnelle]
  \label{thm_eqfonc}
  \index{Fonction sph\'erique!et \'equation fonctionnelle}
  Soit $\Phi\not  =0$ une fonction continue sur  $G$, biinvariante par
  $K$. La fonction $\Phi$ est sph\'erique si et seulement si
  elle v\'erifie :
  \begin{equation}\label{eqfonc}
    \forall g_1,g_2\in G,
    \qquad
    \int_K \Phi(g_1kg_2)dk
    \,=\,
    \Phi(g_1)\Phi(g_2)
    \quad.
  \end{equation}
  En particulier on a $\Phi(e)=1$.
  De plus, 
  si $\Phi$ est born\'ee, 
  alors elle est born\'ee par~1.
\end{thm}

\begin{thm}[Fonctions sph\'eriques et op\'erateurs diff\'erentiels]
  \label{thm_fcnsph_opdiff}
  \index{Fonction sph\'erique!et op\'erateur diff\'erentiel}
  On suppose que le groupe~$G$ est un groupe de Lie connexe.
  Soit $\phi :G/K\rightarrow \Cb$ 
  une fonction $C^\infty$, invariante \`a gauche sous $K$.

  La fonction $\phi$ est sph\'erique 
  si et seulement si 
  $\phi(e)=1$,
  et $\phi$
  est fonction propre de tous les
  op\'erateurs diff\'erentiels sur $G/K$ 
  invariants sous l'action \`a gauche de $G$.
\end{thm}

\subsection{Sur une paire de Guelfand}
\label{subsec_paireGuelf}

On   choisit   une  mesure   de   Haar   $dg$   sur  $G$.    
On   note
$L^1(K\backslash G/K)$ 
le sous-espace  des fonctions int\'egrables sur
$G$, biinvariantes par $K$. 
C'est une alg\`ebre de convolution,
qui s'identifie \`a l'alg\`ebre de convolution 
des fonctions int\'egrables sur $G/K$, 
invariantes sous l'action \`a droite de $K$.

\begin{defi}
  La  paire $(G,K)$  est  dite  \textbf{paire de  Guelfand}
  \index{Paire de Guelfand}  
  si l'alg\`ebre  de
  convolution $L^1(K\backslash G/K)$ est commutative.
\end{defi}

\begin{thm}[Fonctions sph\'eriques et spectre de  
  \mathversion{bold}{$L^1(K\backslash G/K)$}]
  \index{Fonction sph\'erique!et spectre}
  \label{thm_fcnsph_sp}
  On suppose que $(G,K)$ est une paire de Guelfand.

  Soit $\Phi$ une  fonction sph\'erique born\'ee. 
  L'application $\chi$  de la d\'efinition~\ref{defi_fcnsph}  
  s'\'etend en un caract\`ere non nul 
  de l'alg\`ebre commutative $L^1(K\backslash G/K)$.
  
  R\'eciproquement,  
  tout caract\`ere  non nul  de $L^1(K\backslash G/K)$
  est de cette forme.

  Notons $Sp(L^1(K\backslash G/K))$
  le spectre de l'alg\`ebre $L^1(K\backslash G/K)$   
  et $\Omega$,
  \index{Notation!Ensemble de fonctions sph\'eriques!$\Omega$}
  l'ensemble   des   fonctions sph\'eriques born\'ees. 
  L'application :
  $$
  \Xi:
  \left\{\begin{array}{rcl}
      \Omega
      &\longrightarrow&
      Sp(L^1(K\backslash G/K))\\
      \Phi
      &\longmapsto&
      [\chi_\Phi:f \,\longmapsto\, \int_G f(g)\Phi(g^{-1})dg]
    \end{array}\right.
  \quad .
  $$
  est une bijection.
  C'est  un hom\'eomorphisme
  lorsque l'on munit 
  le spectre de sa topologie usuelle faible-*,
  et $\Omega$ de la  convergence uniforme 
  sur  tout compact de $G$ (ou $G/K$).
\end{thm}

\paragraph{Produit semi-direct \mathversion{bold}{$\triangleleft$}.}
Soient $H$ un groupe, 
et $K$ un sous-groupe du groupe d'automorphismes de $H$.
Le produit semi-direct de $H$ par $K$ est l'ensemble $G=K\times H$,
muni de la loi :
$$
(k_1,h_1),(k_2,h_2)\in G,
\quad
(k_1,h_1).(k_2,h_2)
\,=\,
(k_1k_2,h_1\, k_1.h_2)
\quad ,
$$
qui donne \`a $G$ un structure de groupe.
\index{Groupe!produit semi-direct $\triangleleft$}
On note le groupe~$G$ produit semi-direct du groupe~$H$ par~$K$ : $G=K\triangleleft H$.
On identifie souvent le groupe~$K$ avec le sous-groupe~$K\times \{e\}\subset G$, 
ainsi que le groupe~$N$ avec le sous-groupe~$\{\Id\}\times N \subset G$. 
On fait de m\^eme pour leurs \'el\'ements.

Dans ce cas, si le groupe  $G=K\triangleleft H $ 
produit semi-direct de $H$ par $K$
est tel  que 
$(G,K)$  est une  paire de Guelfand,  
on dit  aussi que
$(H,K)$ est une paire de Guelfand.
Les fonctions sur $G$ biinvariantes par $K$ 
sont en bijection avec
les fonctions sur $H$ invariantes sous $K$. 
On note ${L^1}^\natural$ 
l'ensemble des fonctions int\'egrables sur $H$ 
et invariantes sous $K$. 
La paire $(H,K)$ est alors de Guelfand 
si et seulement si 
l'alg\`ebre de convolution ${L^1}^\natural$ 
est commutative.

\vspace{1em}

Nous allons illustrer la notion de paire de Guelfand 
en donnant un exemple dans la sous-section qui suit,
exemple qui nous sera utile
lors de la construction des fonctions sph\'eriques 
sur le groupe $N_{v,2}$ dans le chapitre~\ref{chapitrecalculfourier}. 
Il s'agit du groupe de Heisenberg~$\Hb^{v_0}$ 
(dont la loi a \'et\'e rappel\'ee dans la sous-section~\ref{subsec_def_typeH}),
pour l'action de certains sous-groupes $K$
du groupe unitaire $U_{v_0}$ 
donc du groupe d'automorphismes de $\Hb^{v_0}$ :
$$
\forall (z,t)\in \Hb^{v_0}=\Cb^{v_0}\times  \Rb,\;
u \in k,
\qquad
u.(z,t)=(u.z,t)
\quad.
$$

\subsection{Sur le groupe de Heisenberg}
\label{subsec_fcnsph_heis}

Pour tout $n\in \Nb, n\geq 1$, 
{\mathversion{bold}{$U_n$}}
\index{Notation!Groupe!$U_n$} 
d\'esigne le groupe des matrices unitaires de taille $n$.
On identifie 
une matrice de $U_n$ avec 
l'endomorphisme sur $\Cb^n$,
qu'elle repr\'esente dans la base canonique.

On se donne deux entiers~$v_0,v_1\in \Nb$,
puis un $v_1$-uplet d'entiers 
$m=(m_1,\ldots,m_{v_1})\in \Nb^{v_1}$ 
tels que $\sum_{j=1}^{v_1} m_j =v_0$.
Nous nous int\'eressons \`a tous les groupes 
$K$ de la forme :
$$
K(m;v_1;v_0)
\,=\,
U_{m_1}\times \ldots \times U_{m_{v_1}}
\quad.
$$
Le groupe $U_{m_j}$ agit sur les $m_j$ variables
$z_i, m'_{j-1}<i\leq m'_j$, o\`u l'on a not\'e :
\begin{eqnarray*}
  m_0
  \,:=\,
  m'_0
  \,:=\,
  0
  \qquad\mbox{et}\qquad
  m'_j
  \,:=\,
  \sum_{i=1}^j m_i\; ,
  \quad
  j=1,\ldots {v_1}
  \quad.
\end{eqnarray*}

On dit qu'une fonction sur $\Hb^{v_0}$
est radiale si elle est invariante sous $K(m;v_0;v_1)$;
une fonction $f:\Hb^{v_0}\mapsto \Cb$ est donc radiale
si et seulement si elle peut s'\'ecrire : 
\begin{equation}
  \label{formule_fcn_red_heis}
  f(z,t)
  \,=\,
  f^\natural(r_1,\ldots,r_{v_1},t)
  \qquad\mbox{o\`u}\qquad
  r_j=\sum_{m'_{j-1}<i\leq m'_j} z_i^2
  \quad.  
\end{equation}

\begin{thm}[\mathversion{bold}{$(\Hb^{v_0},K(m;v_1;v_0))$} ]
  \label{thm_paire_G_heis}
  \begin{description}
  \item[a)] 
    $(\Hb^{v_0},K(m;v_1;v_0))$ est une paire de Guelfand.
    \index{Paire de Guelfand!sur le groupe de Heisenberg}  
  \item[b)]
    Les fonctions sph\'eriques born\'ees de cette paire 
    sont :
    \index{Fonction sph\'erique!sur le groupe de Heisenberg}
    \index{Notation!Fonction sph\'erique!$\omega_{{\lambda},l},\omega_\mu$}
    \begin{enumerate}
    \item les fonctions 
      $\omega=\omega_{{\lambda},l}$ 
      param\'etr\'ees par 
      ${\lambda}\in \Rb^* $ 
      et $l=(l_1,\ldots,l_{v_1})\in\Nb^{v_1}$ : 
      $$
      \omega(z,t)
      \,=\,
      e^{-i\lambda t}
      \overset{v_1}{\underset{j=1}{\Pi}}
      \flb {l_j} {m_j-1} {\frac {\nn{\lambda}}2 \nn{\pr j z}^2} 
      \quad,
      $$
      o\`u on a not\'e 
      $\mbox{pr}_j$ la projection sur les $m_j$ variables 
      $z_i, m'_{j-1}<i\leq m'_j$,
      et $\flbs l m$ la fonction de Laguerre normalis\'ee
      (voir sous-section~\ref{app_lag}).
    \item les fonctions 
      $\omega=\omega_\mu $
      param\`etr\'ees par 
      $\mu=(\mu_1,\ldots,\mu_{v_1})\in {\Rb^{+}}^{v_1}$ :
      $$
      \omega(z,t)
      \,=\,
      \overset{v_1}{\underset{j=1}{\Pi}}
      \fb {m_j-1} {{{\mu}}_j \nn{ \pr j z}}
      \quad,
      $$
      o\`u $\fbs \alpha$ est la fonction de Bessel r\'eduite
      (voir sous-section~\ref{app_bessel}).
    \end{enumerate}
  \end{description}
\end{thm}

La d\'emonstration de la partie a)
est analogue \`a la preuve du cas $(\Hb^n,U_n)$,
\cite{Far} th\'eor\`eme V.6.

\begin{proof}[du th\'eor\`eme~\ref{thm_paire_G_heis}.a)]
  On pose $K=K(m;v_1,v_0)$.
  Il suffit de montrer que 
  l'ensemble ${L^1}^\natural$
  des fonctions radiales int\'egrables sur $\Hb^{v_0}$ 
  est une alg\`ebre de convolution commutative.

  On d\'efinit le morphisme~$\theta$ du groupe~$\Hb^{v_0}$ 
  par :  $\theta (z,t)=(\bar z , -t)$.
  Pour une fonction $f:\Hb^{v_0}\rightarrow \Cb$, 
  on adopte les notations (pour $h\in\Hb^{v_0},k\in K$):
  $$
  \check{f}(h)
  \,=\,
  f(h^{-1})
  \qquad
  f^k(h)
  \,=\,
  f(k.h),
  \qquad
  f^\theta(h)
  \,=\,
  f(\theta(h))
  \quad ;
  $$
  On a alors les propri\'et\'es suivantes :
  \begin{itemize}
  \item  une fonction $f$ est radiale 
    si et seulement si on a : $f^k=f$ pour tout $k\in K$;
  \item si une fonction $f$ est radiale, alors la fonction $\check{f}$ est aussi radiale.
  \end{itemize}
  Comme le groupe $K$ est un sous-groupe 
  du groupe d'automorphismes de $\Hb^{v_0}$, 
  on voit que pour deux fonctions $f,g$ sur $\Hb^{v_0}$, 
  on a :
  $  {(f*g)}^k=  f^k*g^k$ pour tout $k\in K$.
  En cons\'equence, le produit de convolution de deux fonctions radiales est encore une fonction radiale et $L_K^1(\Hb^{v_0})$ est une alg\`ebre de convolution.

  On constate que
  $\theta$ est  un automorphisme involutif,
  et que par contre, 
  l'inverse est un antiautomorphisme involutif 
  de $\Hb^{v_0}$; ainsi, pour deux fonctions $f,g$, on a :
  $$
  f*g
  \,=\,
  {(f^\theta*g^\theta)}^\theta
  \qquad\mbox{et}\qquad
  f*g
  \,=\,
  {(\check{g}*\check{f})}\check{}
  \quad.
  $$
  De plus,
  on la propri\'et\'e suivante :
  $$
  \forall h\in \Hb^{v_0},\quad
  \exists k\in K:
  \qquad
  \theta(h)=k.h^{-1}
  \quad ;
  $$
  donc pour une fonction~$f$ radiale,
  on a : $f^\theta=\check{f}$. 
  Puis pour deux fonctions radiales $f,g$, on a :
  $$
  f*g
  \,=\,
  {(f^\theta*g^\theta)}^\theta
  \,=\,
  {(\check{f}*\check{g})}\check{}
  \,=\,
  g*f
  $$ 
  L'alg\`ebre de convolution 
  ${L^1}^\natural$ est donc commutative.
\end{proof}

Pour la d\'emonstration du th\'eor\`eme~\ref{thm_paire_G_heis}.b), on fait les remarques suivantes :
d'une part, 
les groupes $K(m;v_1;v_0)$ et $\Hb^{v_0}$ 
sont des groupes de Lie;
d'autre part, on connait
l'alg\`ebre $D^\natural$
des op\'erateurs diff\'erentiels sur $\Hb^{v_0}$ 
invariants \`a gauche et sous $K(m;v_1;v_0)$,
\cite{helgason} chapter X section 2 :
\begin{prop}[\mathversion{bold}{$D^\natural$}]
  \label{prop_gen_alg_op}
  L'alg\`ebre  $D^\natural$,
  est engendr\'ee par  les sous-laplaciens :
  $$
  \Delta_j
  \,=\,
  -\sum_{m'_{j-1}<i\leq m'_j} X_i^2+Y_i^2\; ,
  \quad j=1,\ldots,{v_1}\quad ,
  $$
  et la d\'erivation en la variable du centre $T$ 
\end{prop}

La preuve du th\'eor\`eme~\ref{thm_paire_G_heis}.b)
est analogue \`a $(\Hb^n,U_n)$,
\cite{Far} th\'eor\`eme V.12.
Elle utilise les propri\'et\'es
des fonctions hyperg\'eom\'etriques confluentes 
(lemme~\ref{lem_fcn_hypergeom}),
et des fonctions de Bessel 
(lemme~\ref{lem_fcnBessel_equadiff}),
ainsi que l'expression des sous-laplaciens 
$\Delta_j$ en coordonn\'ees radiales :
\begin{lem}
  \label{lem_souslap_coord_rad}
  Soit une fonction $f$ sur $\Hb^{v_0}$, 
  radiale et assez r\'eguli\`ere.
  Si $f^\natural$ d\'esigne la fonction d\'efinie
  en~(\ref{formule_fcn_red_heis}), 
  on a :
  $$
  \Delta_j.f
  \,=\,
  \Delta_j^\natural .f^\natural
  \quad,\quad
  j=1,\ldots,v_1\quad,
  $$
  o\`u $\Delta_j^\natural$ est l'op\'erateur diff\'erentiel sur les
  fonctions $g(r_j,t)$ donn\'e par :
  $$
  -\Delta_j^\natural .g(r_j,t)
  \,=\,
  4r_j \frac{\partial^2 g }{\partial r_j^2}
  +4 m_j \frac{\partial g }{\partial r_j}
  +\frac14 r_j \frac{\partial^2 g}{\partial t^2}
  \quad,\quad
  j=1,\ldots,v_1
  \quad.
  $$  
\end{lem}

\begin{proof}[du th\'eor\`eme~\ref{thm_paire_G_heis}.b)]
  D'apr\`es les th\'eor\`emes~\ref{thm_eqfonc}
  et~\ref{thm_fcnsph_opdiff},
  une fonction 
  $\omega:\Hb^{v_0}\rightarrow \Cb$
  est sph\'erique born\'ee
  si et seulement si elle est radiale, 
  born\'ee par sa valeur 1 en 0,
  et fonction propre des g\'en\'erateurs
  de $D^\natural$ :
  \begin{eqnarray}
    \Delta_j.\omega
    &=& \alpha_j \omega\; ,
    \quad j=1,\ldots, v_1 \; ,
    \label{eq_souslapj_fcnpropre}\\
    T.\omega 
    &=& \beta \omega \; ,
    \label{eq_T_fcnpropre}
  \end{eqnarray}
  o\`u $\alpha_j,j=1,\ldots, v_1$ et $\beta$ 
  sont des nombres complexes,
  d'apr\`es la proposition~\ref{prop_gen_alg_op}.

  Supposons que $\omega$ soit une telle fonction.
  L'op\'erateur~$T$ \'etant 
  la d\'eriv\'ee selon la variable du centre, 
  la solution de l'\'equation~(\ref{eq_T_fcnpropre}) 
  est de la forme :
  $$
  \omega(z,t)
  \,=\,
  e^{\beta t}
  \Omega(z)
  \quad.
  $$ 
  Comme $\omega$ est born\'ee et radiale, 
  $\beta=i\lambda, \lambda\in \Rb$,
  $$
  \Omega(z)
  \,=\,
  \Omega (r_1,\ldots, r_{v_1})
  \qquad\mbox{avec}\qquad
  r_j=\sum_{m'_{j-1}<i\leq m'_j} z_i^2
  \quad.  
  $$
  $\omega$ \'etant solution des 
  \'equations~(\ref{eq_souslapj_fcnpropre}), 
  l'expression radiale des sous-laplaciens
  (lemme~\ref{lem_souslap_coord_rad}) implique :
  $$
  4r_j \frac{\partial^2 \Omega }{\partial r_j^2}
  +4 m_j \frac{\partial \Omega}{\partial r_j}
  +\frac14 r_j \beta^2 \Omega
  \,=\,
  -\alpha_j \Omega
  \; , \quad  j=1,\ldots,v_1
  \quad;
  $$
  et $\Omega$ est n\'ecessairement de la forme :
  $$
  \Omega (r_1,\ldots, r_{v_1})
  \,=\,
  \Omega_1 (r_1)\ldots \Omega_{v_!}(r_{v_1})
  \quad,
  $$
  o\`u chaque fonction $\Omega_j(r_j)$ est $C^\infty$, 
  born\'ee et v\'erifie l'\'equation :
  $$
  4r_j \frac{\partial^2 \Omega_j }{\partial r_j^2}
  +4 m_j \frac{\partial \Omega_j}{\partial r_j}
  +\frac14 r_j \beta^2 \Omega_j
  \,=\,
  -\alpha_j \Omega_j
  \quad.
  $$
  Puisque la fonction $\omega$  est born\'ee 
  par sa valeur~1 en~0,
  il en va de m\^eme pour $\Re \omega$ sur $\Hb^{v_0}$. 
  Par cons\'equent,  
  la fonction $\Re\Omega$, 
  puis les fonctions $\Re\Omega_j$ 
  sont born\'ee par leurs valeurs~1 en~0. 
  On en d\'eduit :
  \begin{equation}
    \label{maj_realpha_pos}
    \Re \Omega_j'(0) \leq 0
    \qquad\mbox{et donc}\qquad
    \Re\alpha_j\geq 0
    \quad. 
  \end{equation}

  Nous devons distinguer deux cas 
  $\lambda=0,\lambda\not =0$.
  \paragraph{\mathversion{bold}{$\lambda$} est nul :}
  chaque fonction $\Omega_j$ est une fonction $C^\infty$ born\'ee, prenant la valeur 1 en 0 et qui v\'erifie l'\'equation :
  $$
  4r_j \frac{\partial^2 \Omega_j }{\partial r_j^2}
  +4 m_j \frac{\partial \Omega_j}{\partial r_j}
  +\alpha_j \Omega_j
  \,=\,
  0
  \quad.
  $$
  D'apr\`es le lemme~\ref{lem_fcnBessel_equadiff}, 
  la fonction $\Omega_j$ est donn\'ee par
  $\Omega_j(r_j)=\fb {m_j-1} {\mu_j\sqrt{ r_j}}$
  avec $\mu_j^2=\alpha_j\in \Rb$.
  Gr\^ace \`a (\ref{maj_realpha_pos}), on a aussi $\alpha_j \geq 0$;
  on peut ainsi supposer $\mu_j=\sqrt{\alpha_j}\in\Rb^{+}$.
  On en d\'eduit que la fonction $\omega=\omega_{\mu}, \mu=(\mu_1,\ldots,\mu_{v_1})$ est de la forme :
  $$
  \omega_{\mu}(z,t)
  \,=\,
  \underset{j=1}{\overset{v_1}{\Pi}}
  \fb {m_j-1} {\mu_j \sqrt{r_j}}
  \qquad\mbox{o\`u}\qquad
  r_j=\sum_{m'_{j-1}<i\leq m'_j} z_i^2
  \quad.  
  $$
  \paragraph{\mathversion{bold}{$\lambda$} n'est pas nul :}
  on effectue les changements de variables 
  et de fonctions suivants :
  $$
  v_j=\frac {\nn{\lambda}}2 r_j
  \qquad\mbox{et}\qquad
  \Omega_j(r_j)=e^{-\frac{v_j}2}F_j(v_j)\; ,
  \quad j=1,\dots, v_1
  \quad.
  $$
  La fonction $F_j(v_j)$ v\'erifie alors 
  l'\'equation hyperg\'eom\'etrique confluente~(\ref{eq_hypergeom_confl})
  de param\`etres $m_j>0$ et 
  $(\nn{\lambda}m_j-\alpha_j)/(2\nn{\lambda})$.
  D'apr\`es le lemme~\ref{lem_fcn_hypergeom},
  la fonction $F_j$ est donc 
  la fonction hyperg\'eom\'etrique confluente :
  $$
  F_j
  \,=\,
  F({\nn{\lambda}m_j-\alpha_j}/{2\nn{\lambda}},m_j,.)
  \quad.
  $$
  De plus on voit 
  $\Re (\nn{\lambda}m_j-\alpha_j)/(2\nn{\lambda})<m_j$. 
  Et le premier param\`etre peut se mettre 
  sous la forme d'un entier n\'egatif :
  $$
  \frac{\nn{\lambda}m_j-\alpha_j}{2\nn{\lambda}}
  \,=\,
  -l_j
  \qquad\mbox{avec}\quad l_j\in \Nb
  \quad,
  $$
  car sinon  la fonction
  $$
  \Omega_j:\;
  r_j\in \Rb^+ 
  \,\longmapsto \,
  e^{-\frac {\nn{\lambda}}4 r_j}F_j(\frac {\nn{\lambda}}2 r_j)
  $$ 
  ne serait pas born\'ee. 

  Ainsi, on trouve : 
  $F_j=F(-l_j,m_j,.)=
  {C_{l_j+m_j-1}^{l_j}}^{-1} L_{l_j}^{m_j-1}$, puis :
  $$
  \Omega(r_j)
  \,=\,
  \frac1{C_{l_j+m_j-1}^{l_j}}
  e^{-\frac{\nn{\lambda}}4r_j}
  L_{l_j}^{m_j-1}(\frac{\nn{\lambda}}2 r_j)
  \,=\,
  \flb {l_j}{m_j-1} {\frac{\nn{\lambda}}2 r_j}
  \quad .
  $$
  On en d\'eduit la forme suivante de
  $\omega=\omega_{\lambda, l}$ :
  $$
  \omega_{\lambda, l}(z,t)
  \,=\,
  e^{-i\lambda t}
  \underset{j=1}{\overset{v_1}{\Pi}}
  \flb {l_j}{m_j-1} {\frac{\nn{\lambda}}2 r_j}
  \qquad\mbox{o\`u}\qquad
  r_j=\sum_{m'_{j-1}<i\leq m'_j} z_i^2
  \quad.  
  $$

  R\'eciproquement, 
  on peut ``remonter'' les calculs 
  pour voir que les solutions trouv\'ees sont bien 
  radiales, $C^\infty$, born\'ees par leur valeur 1 en 0, 
  et fonctions propres communes aux op\'erateurs $T$ et $\Delta_j, j=1,\ldots,v_1$. 
\end{proof}
Il d\'ecoule de cette preuve
que la valeur propre associ\'ee \`a la
fonction sph\'erique 
$\omega=\omega_{{\lambda},l}$ 
pour le sous-laplacien $\Delta_j$
vaut 
$ \alpha_j=
\nn{\lambda}(2l_j+m_j)$.

\subsection{Sur les groupes de type~H}
\label{subsec_fcnsph_typeH}

On convient dans tout ce texte que pour tout $n\in \Nb, n\geq 1$,
{\mathversion{bold}{$O(n)$}}
\index{Notation!Groupe!$O(n)$}
d\'esigne le groupe
des matrices orthogonales 
de taille $n$.
On confondra souvent une matrice de $O(n)$ 
et l'endomorphisme qu'elle repr\'esente dans la base canonique sur $\Rb^n$. 

On reprend \'egalement les notations pour un groupe $N$ de type~H 
de la sous-section~\ref{subsec_def_typeH},
et on note $v=\dim \Vc$ et $z=\dim \Zc$.
D'apr\`es la remarque~\ref{V_dim_paire}, 
$v$ est paire : $v=2v'$.

Le groupe compact $O(v)$ agit sur l'espace $\Vc\sim \Rb^v$.
On dit qu'une fonction est radiale
lorsqu'elle est $O(v)$-invariante.
Par dualit\'e, on d\'efinit les distributions radiales.
Pour une classe de distributions $\Cc$, 
on note $\Cc^\natural$ 
l'ensemble de ses distributions radiales.

Le groupe $O(v)$ agit,
mais g\'en\'eralement pas par automorphismes sur le groupe $N$.
Il ne peut donc pas \^etre question 
de fonctions sph\'eriques de $N$ sous $K$, 
ni de paire de Guelfand $(N, O(v))$. 
Toutefois,  dans l'article~\cite{bigth}, 
les auteurs
donnent un sens aux fonctions sph\'eriques born\'ees,
en consid\'erant l'op\'erateur de radialisation sous~$O(v)$ :
$$
f\,\longrightarrow\,
f^\natural:
\left\{\begin{array}{rcl}
    \Vc\times\Zc&\mapsto&\Cb\\
    (X,Z)&\rightarrow& 
    \int_{O(v)} f(k.X,Z) dk
  \end{array}\right.
\quad,
$$
o\`u $dk$ est la mesure de Haar de masse 1 du groupe compact $O(v)$.
Cet op\'erateur v\'erifie les m\^emes propri\'et\'es 
que l'op\'erateur analogue sur une paire de Guelfand;
les auteurs de  l'article~\cite{bigth}
d\'efinissent alors les fonctions sph\'eriques 
comme les fonctions radiales au sens ci-dessus,
et propres pour les op\'erateurs diff\'erentiels invariants sous radialisation.
Comme sur le groupe de Heisenberg,
ces op\'erateurs diff\'erentiels forment une alg\`ebre
not\'ee $D^\natural$
qui admet pour g\'en\'erateurs :
\index{Notation!Op\'erateur!$L$}
$$
L\,=\,
-\sum_{i=1}^v X_i^2
\qquad \mbox{et}\qquad
U_i\; , \quad i=1,\ldots, z\quad,
$$
o\`u $X_1, \ldots, X_v$ est une base orthonormale de~$\Vc$, et $U_1,
\ldots, U_z$ une base orthonormale de~$\Zc$.
Les fonctions sph\'eriques sont explicites 
et ont les m\^emes propri\'et\'es spectrales que dans le cas d'une paire de Guelfand.
Nous redonnons ici une partie du contenu de cet article, 
essentiellement le th\'eor\`eme (3.3) page~227, 
avec nos notations.

\begin{thm}[Fonctions sph\'eriques sur un groupe de type~H]
  \label{bigthm}
  \index{Fonction sph\'erique!sur les groupes de type H}
  L'ensemble ${L^1}^\natural$
  des fonctions int\'egrables sur $N$ et radiales est une
  sous-alg\`ebre commutative de $L^1(N)$  pour la convolution.  

  On  note $\Omega$  l'union  (disjointe) des  ensembles de fonctions sur $N$ :
  \index{Notation!Ensemble de fonctions sph\'eriques!$\Omega_L,\Omega_B$}
  \begin{eqnarray*}
    \Omega_L
    &=&
    \{    \Phi_{\zeta,l}
    \quad:\quad
    \zeta\in\Zc-\{0\}\quad l\in \Nb \}\quad,\\
    \Omega_B
    &=&
    \{    \Phi_{r}
    \quad:\quad
    r\geq 0 \}\quad,
  \end{eqnarray*}
  o\`u les fonctions sont donn\'ees par :
  \index{Notation!Fonction sph\'erique!$\Phi_{\zeta,l}$, $\Phi_r$}
  \begin{eqnarray*}
    \Phi_{\zeta,l}(X,Z)
    &=&
    e^{i<\zeta,Z>}  \flb {l}{v'-1}{\frac  1 2
      \nn{\zeta}\nn{X}^2}
    \quad ,\\
    \Phi_r(X,Z)
    &=&
    \fb {v'-1} {r\nn{X}} \quad;
  \end{eqnarray*}
  $\flbs{n}{\alpha}$ d\'esigne  la  fonction  de
  Laguerre normalis\'ee, et
  $\fbs \alpha$ la  fonction  de  Bessel r\'eduite
  (voir section~\ref{sec_fonctionspeciale}).
  
  Les fonctions de $\Omega$ sont $C^\infty$,
  born\'ees par leur valeur~1 en~0.  
  De plus, elles  sont sph\'eriques dans le sens  o\`u ce sont des
  fonctions continues  radiales  qui v\'erifient les deux propri\'et\'es suivantes
  :
  \begin{enumerate}
  \item elles  v\'erifient  l'\'equation fonctionelle~(\ref{eqfonc});
    et donc, 
    pour $f_1,f_2$ deux distributions radiales 
    telles que $f_1*f_2$, $<f_i,\phi>, i=1,2, \phi \in \Omega$ ont un sens, on a :
    \begin{equation}
      \forall \phi\in\Omega
      \quad : \quad
      <f_1*f_2,\phi>
      \,=\,
      <f_1,\phi><f_2,\phi>
      \quad;
      \label{formule_caractere}
    \end{equation}
  \item elles sont les fonctions $C^\infty$, born\'ees par leur valeur~1 en~0, 
    fonctions propres communes 
    aux op\'erateurs diff\'erentiels de $D^\natural$.
  \end{enumerate}  
\end{thm}

On retrouve le cas de Heisenberg $N=\Hb^n$,
non pas pour $U_n$ 
mais pour $O(2n)$, 
ce qui ne change rien aux fonctions sph\'eriques.

D'apr\`es~(\ref{formule_caractere}),
les fonctions sph\'eriques sont les caract\`eres 
de  l'alg\`ebre ${L^1}^\natural$. 
Le th\'eor\`eme~\ref{thm_fcnsph_sp} s'\'etend aussi au cas 
d'un groupe de type~H.

Pour le lecteur int\'eress\'e 
par la preuve du th\'eor\`eme~\ref{thm_intro_in_max_sph_Lp}
sur les groupes de type~H
(in\'egalit\'e $L^p$ pour la fonction maximale pour les sph\`eres homog\`enes),
il est inutile d'aller plus loin dans cette section.

\subsection{Sur le groupe $N_{v,2}$}
\label{subsec_fcnsph_grlib}

D'apr\`es \cite[Theorem 5.12]{pgelf}, 
$(N_{v,2},SO(v))$ est une paire de Guelfand;
donc les fonctions int\'egrables invariantes sous $SO(v)$
forment une alg\`ebre de convolution radiale;
a fortiori, c'est \'egalement le cas
des fonctions int\'egrables invariantes sous $O(v)$. 
Les distributions invariantes sous $O(v)$ seront appel\'ees radiales.

\begin{thm}
  \label{thm_paire_G_grlib}
  \index{Paire de Guelfand!sur $N_{v,2}$}  
  $(N_{v,2},O(v))$ et $(N_{v,2},SO(v))$ 
  sont deux paires de Guelfand.  
\end{thm}

Dans le chapitre~\ref{chapitrecalculfourier},  
nous   d\'eterminons  les  fonctions sph\'eriques born\'ees 
associ\'ees \`a cette paire avec la m\'ethode suivante :
comme les  fonctions sph\'eriques born\'ees d'une  paire de
Guelfand $(G,K)$ sont les fonctions de type positif associ\'ees aux repr\'esentations du groupe  $G$ qui
ont  les propri\'et\'es  d'\^etre irr\'eductible  et d'avoir  un vecteur
$K$-fixe non nul, 
nous allons  construire ces repr\'esentations 
en suivant la th\'eorie de Mackey.

Nous rappelons ces notions dans la section qui suit. 

\section{Repr\'esentations}
\label{sec_rep_paireG}

Dans cette section,
nous rappelons 
d'abord les propri\'et\'es 
des fonctions de type positif,
puis  le lien entre 
les fonctions sph\'eriques born\'ees d'une paire de Guelfand
et les repr\'esentations,
ainsi que 
des \'el\'ements de th\'eorie des repr\'esentations 
(th\'eor\`emes de Kirillov et de Mackey)
que nous utiliserons dans le chapitre~\ref{chapitrecalculfourier}.

\vspace{1em}

Tout au  long de ce travail, nous ne consid\'erons que
des groupes localement compacts, et sur ces groupes
des repr\'esentations unitaires et continues
sur des espaces de Hilbert s\'eparables.
Nous renvoyons par exemple \`a \cite{mackey_chicago} pour les
d\'efinitions de repr\'esentation, 
repr\'esentation irr\'eductible et \'equivalente.
Pour un groupe $G$ (localement compact),
on note $\hat{G}$ l'ensemble des repr\'esentations irr\'eductibles de
$G$ quotient\'e par cette relation d'\'equivalence (not\'ee $\sim$).
\index{Notation!Ensemble de classe de repr\'esentations!$\hat{G}$}
On confond souvent une classe et la donn\'ee d'un de ses \'el\'ements.

\subsection{Cas d'une paire de Guelfand}
\label{subsec_rep_pguelf}

Voici d'abord les propri\'et\'es 
du sous-espace des vecteurs $K$-invariants 
\cite{faro2} :
\begin{thm}[Sous-espace invariant]
  \label{thm_rep_ssesp}
  Soient  $G$ un  groupe 
  et  $K$  un de ses sous-groupes compacts.
  
  Soit $(\Pi,\Hc)$ une repr\'esentation de $G$.
  
  On note $\Hc_K$ le sous-espace des vecteurs 
  $K$-invariants de $\Hc$.

  \begin{itemize}
  \item  Si  $\dim  \Hc_K  =1$,  
    alors  la  repr\'esentation  $\Pi$  est irr\'eductible.
  \item  Si  $(G,K)$  est  une   paire  de  Guelfand  
    et  si  $\Pi$  est irr\'eductible, 
    alors $\dim \Hc_K\leq 1$.
  \end{itemize}

\end{thm}

On  peut  caract\'eriser  
les  fonctions sph\'eriques born\'ees de type positif d'une paire de Guelfand 
\`a l'aide des repr\'esentations \cite{faro2,helgason}. 
Dan le cas qui va nous int\'eresser, 
les fonctions sph\'eriques born\'ees sont de type positif
\cite[Corollary 8.4]{pgelf} :

\begin{thm}[Fonctions sph\'eriques et repr\'esentations]
  \label{thm_fcnsph_rep}
  \index{Fonction sph\'erique!et repr\'esentation}
  Soit $(G,K)$ une paire de Guelfand.
  
  \begin{itemize}
  \item[a)] Les fonctions  sph\'eriques born\'ees  de  type positif sont 
    les fonctions de  type positif $\Phi$
    associ\'ees  \`a 
    une classe de repr\'esentations irr\'eductibles
    qui poss\`edent au moins un vecteur $K$-fixe non nul.

    Dans ce cas, l'espace  des vecteurs $K$-invariants
    est la droite $\Cb \Phi$.
  \item[b)]   Si de plus $G$ est le produit semi-direct $K\triangleleft N$,
    o\`u $N$ est un groupe de Lie nilpotent connexe simplement connexe 
    et $K$ un groupe de Lie compact,
    alors les fonctions sph\'eriques born\'ees sont de type positif.
  \end{itemize}

\end{thm}

\paragraph{Exemple : la paire de Guelfand
  \mathversion{bold}{$(K(m;v_0;v_1),\Hb^{v_0})$}.}
Reprenons les notations et
les r\'esultats d\'evelopp\'es dans la sous-section~\ref{subsec_fcnsph_heis}.
D'apr\`es le th\'eor\`eme~\ref{thm_paire_G_heis},
les fonctions sph\'eriques born\'ees sont les fonctions 
$\omega=\omega_{{\lambda},l}$ ou $\omega=\omega_\mu $
d\'efinies sur $\Hb^{v_0}$, 
ou encore leurs extensions 
$\Omega^\omega$ au groupe entier 
\index{Notation!Fonction sph\'erique!$\Omega^\omega$}
$H_{heis}=K(m;v_0;v_1)\triangleleft \Hb^{v_0}$ : 
$$
\forall\, (k,h)\in H_{heis}=K(m;v_0;v_1)\triangleleft \Hb^{v_0}
\qquad
\Omega^\omega(k,h)
\,=\,
\omega(h)
\quad.
$$
Pour une fonction sph\'erique born\'ee $\omega$,
on note  $(\Hc_\omega,\Pi_\omega)$
la repr\'esentation irr\'eductible de $H_{heis}$
associ\'ee \`a $\Omega^\omega$
comme fonction de type positif,
\cite{faro2,helgason}.

\begin{lem}
  \label{lem_rep_Pi_omega}
  \begin{itemize}
  \item   Les repr\'esentations irr\'eductibles sur 
    $K(m;v_0;v_1)\triangleleft \Hb^{v_0}$
    ayant un vecteur $K(m;v_0;v_1)$-fixe (non nul)
    sont les repr\'esentations $(\Hc_\omega,\Pi_\omega)$
    associ\'ees aux fonctions de type positif $\Omega^\omega$,
    extension des fonctions 
    $\omega=\omega_{{\lambda},l}$, $\lambda\in\Rb^*,l\in\Nb^{v_1}$ et
    $\omega=\omega_\mu$, $\mu\in {\Rb^+}^{v_1}$.

    L'espace des vecteurs $K(m;v_0;v_1)$-fixes 
    pour la repr\'esentation~$\Pi_\omega$  
    est   la droite $\Cb \Omega^\omega$.  
    \index{Notation!Repr\'esentation!$\Pi_\omega$}
  \item   Sur le centre~$Z:=\{ (0,t)\, :\, t\in \Rb\}$ 
    du groupe de Heisenberg~$\Hb^{v_0}=\Cb^{v_0}\times \Rb$,
    la repr\'esentation~$\Pi_\omega$ co\"\i ncide avec le caract\`ere :
    $$
    \left\{
      \begin{array}{ll}
        (0,t)\mapsto e^{i\lambda t}&
        \mbox{si}\; \omega=\omega_{\lambda,l}\; ,\\
        (0,t)\mapsto 1 &
        \mbox{si}\; \omega=\omega_\mu\; .
      \end{array}\right.
    $$

  \end{itemize}
\end{lem}

\begin{proof}[du lemme~\ref{lem_rep_Pi_omega}]
  La premi\`ere partie du lemme 
  est une cons\'equence du th\'eor\`eme~\ref{thm_fcnsph_rep}.

  Pour la seconde,
  comme la repr\'esentation $\Pi_\omega$ est irr\'eductible,
  sa restriction au centre $Z$  est de dimension 1.
  Gr\^ace \`a l'expression de 
  $\omega=\omega_{\lambda,l},\omega_\mu$ 
  th\'eor\`eme~\ref{thm_paire_G_heis},
  on calcule :
  $$
  \forall\,(0,t)\in Z\; ,
  \forall\,(z',t')\in \Hb^{v_0}
  \quad:\quad
  \omega(z',t+t')
  \,=\,
  \omega(0,t)\omega(z',t')
  \quad;
  $$ 
  on en d\'eduit
  gr\^ace \`a la d\'efinition de $\Omega^\omega$,
  pour $(0,t)\in Z$ et $g'=(k';z',t')\in H_{heis}$ :
  \begin{eqnarray*}
    \Omega^\omega\left( (\Id; 0,-t)g'\right)
    &=&
    \Omega^\omega(k';z',t'-t)
    \,=\,
    \omega(z',t'-t)\\
    &=&
    \omega(0,-t)\omega(z',t')
    \,=\,
    \omega(0,-t)\Omega^\omega(g')
    \quad.
  \end{eqnarray*}
  Ainsi sur $Z$, 
  la repr\'esentation~$\Pi_\omega$ 
  co\"\i ncide avec la repr\'esentation de dimension 1, 
  donn\'ee par le caract\`ere $(0,t)\mapsto \omega (0,-t)$
  sur la droite $\Cb \Omega$ 
  donc sur l'espace de Hilbert tout entier.
\end{proof}

Rappelons maintenant quelques \'el\'ements de la th\'eorie des repr\'esentations dont nous aurons besoin pour utiliser le th\'eor\`eme~\ref{thm_fcnsph_rep}.

\subsection{M\'ethode des orbites}
\label{subsec_kirillov}

Redonnons la m\'ethode des orbites pour les groupes
nilpotents 
(voir par exemple 
\cite{pukanszky}, partie II, chapitre III \S  3 
et \cite{kirillov}, \S 15),
et appliquons-la au groupe $N_{v,2}$.
La description de $\hat{N}$ est d\'eja connue
(voir \cite{gaveau} 
avec une autre m\'ethode que celle des orbites).

\begin{thm}[Kirillov]
  \label{thm_kirillov}
  \index{Repr\'esentation!th. de Kirillov}
  Soient $N$ un groupe de Lie nilpotent 
  et $\Nc$ son alg\`ebre de Lie.

  Pour une forme lin\'eaire $f\in \Nc^*$ 
  et un polarisation $\Nc_0$ en $f$,
  on d\'efinit
  $\chi$   l'homomorphisme 
  dont la diff\'erentielle est $if$
  sur le sous-groupe $N_0=\exp \Nc_0$,
  et     $\Ind_{N_0}^N \chi$ 
  la repr\'esentation induite 
  par $\chi$ de $N_0$ sur $N$.
  Cette repr\'esentation  est irr\'eductible, 
  et sa classe d'\'equivalence not\'ee $T_f$ ne d\'epend pas 
  de la polarisation $\Nc_0$ en $f$. 

  On a la bijection de Kirillov :
  $$
  \left\{
    \begin{array}{rcl}
      \Nc^*/N&\longrightarrow&\hat{N}\\
      N.f&\longmapsto&N.T_f
    \end{array}
  \right.
  \quad.
  $$
\end{thm}

Nous l'appliquons au groupe libre nilpotent \`a deux pas
not\'e dans cette sous section $N=N_{v,2}$.
On note aussi ici $\Nc=\Nc_{v,2}$ son alg\`ebre de Lie, et $\Nc^*$ son dual;
$\Vc^*$ et $\Zc^*$
d\'esignent les espaces duaux 
de $\Vc$ et $\Zc$ respectivement;
lorsque l'on \'ecrit 
$X^*+A^*\in\Nc^*$,
on sous-entend $X^*\in\Vc^*$ et $A^*\in \Zc^*$.

Les expressions des repr\'esentations $T_f$, $f\in \Nc^*/N$,
peuvent s'obtenir et s'\'ecrire (longuement)
en utilisant les bases canoniques, 
et l'image de leurs vecteurs 
par des transformations orthogonales;
mais cette d\'emarche rend les choix effectu\'es peu lisibles.
Ici, nous allons utiliser la seconde r\'ealisation 
de l'alg\`ebre de Lie $\Nc=\Nc_{v,2}$,
que nous avons d\'ecrite dans la sous-section~\ref{subsec_def_grlib}.

\paragraph{Conventions concernant les \'el\'ements de $\Zc^*$.}
Rappelons que l'espace vectoriel $\Zc^*$ est identifi\'e par le produit scalaire naturel \`a $\Zc$,
l'ensemble des endomorphismes antisym\'etriques.
Supposons $A^*\in \Zc^*$ fix\'e.
On lui associe alors la forme bilin\'eaire antisym\'etrique $\omega_{A^*}$ 
sur $\Vc$  donn\'ee pour $X,Y\in\Vc$ par :
$$
\omega_{A^*}(X,Y)
\,=\,
<A^*X,Y>
\quad.
$$
D'apr\`es l'\'egalit\'e~(\ref{egalite_crochet_pdtsc}),
on a aussi $\omega_{A^*}(X,Y)=<A^*,[X,Y]>$.

Le radical de la forme $\omega_{A^*}$ est \'egal
au noyau $\ker A^*$ de l'endomorphisme antisym\'etrique~$A^*$;
son suppl\'ementaire orthogonale dans $(\Vc,<,>)$ est l'image de~$A^*$,
not\'ee~$\Im A^*$.
Ainsi, $\omega_{A^*}$ induit sur $\Im A^*$
la forme sympl\'ectique not\'ee $\omega_{A^*,r}$; 
en particulier, la dimension de l'espace $\Im A^*$ est paire et sera not\'ee $2v_0$.

\paragraph{Choix d'un sous espace isotrope.}
Fixons $E_1$, un espace vectoriel
maximal totalement isotrope pour $\omega_{A^*,r}$.
Sa dimension est $v_0$.
On pose $E_2=A^*E_1$.
Comme $E_1$ est inclus dans le suppl\'ementaire $\Im A^*$ de $\ker A^*$,
les sous-espaces vectoriels $E_1$ et $E_2$ sont isomorphes donc
la dimension de $E_2$ est aussi $v_0$.

Par d\'efinition de $\omega_{A^*,r}$ et comme  
$E_1$ est totalement isotrope pour $\omega_{A^*,r}$,
on voit que l'espace vectoriel $E_2$
est aussi le suppl\'ementaire orthogonal de $E_1$ 
dans $(\Im A^*,<,>)$. 
Comme l'endomorphisme $A^*$ est un isomorphisme normal en restriction \`a $\Im A^*$,
on en d\'eduit dans $(\Im A^*,<,>)$ :
$$
A^*E_2
\,=\,
A^* {(E_1)}^\perp
\,=\,{(A^* E_1)}^\perp
\,=\,
E_2^\perp=E_1
\quad.
$$
Finalement,
$E_2$ est aussi un
sous-espace totalement isotrope de $\omega_{A^*,r}$,
qui est maximal \`a cause des dimensions.

On note 
$p_0:\Vc\rightarrow \ker A^*$,
$p_1:\Vc\rightarrow E_1$
et $p_2:\Vc\rightarrow E_2$
les projections orthogonales.
On a $\Id_\Vc=p_0+p_1+p_2$.

\paragraph{D\'efinition des repr\'esentations associ\'ees.}
\index{Notation!Repr\'esentation!$U_{X^*,A^*}$}
Soit $X^*\in \ker A^*$.
On d\'efinit les repr\'esentations 
$(\Hc_{X^*,A^*},U_{X^*,A^*})$ par ce qui suit~:
\begin{itemize}
\item si $A^*=0$, c'est la repr\'esentation de dimension 1 
  (i.e. $\Hc_{X^*,A^*}=\Cb$),
  donn\'ee par  le caract\`ere :
  $\exp(X+A) \mapsto\, \exp(i<X^*,X>)$.
\item si $A^*\not=0$,
  $\Hc_{X^*,A^*}=L^2(E_1)$,
  $F\in \Hc_{X^*,A^*}$, 
  $n=\exp(X+A)$,
  $X'\in  E_1$ :
  \begin{eqnarray*}
    U_{X^*,A^*} (n).F (X')
    &=&
    \exp\left(i<A^*, \frac12[p_1(X+2X'),p_2(X)]+A>\right)\\
    && \quad e^{ i<X^*,X>}
    F(p_1(X)+X')
  \end{eqnarray*}
  On montrera dans ce qui suit, 
  que le choix de $E_1$ pour la construction de cette repr\'sentation 
  $(\Hc_{X^*,A^*},U_{X^*,A^*})$
  avec $A^*\not=0$,
  n'en change pas la classe d'\'equivalence.

\end{itemize}
\begin{rem}  \label{rem_expression_noyau}
  L'alg\`ebre de Lie du noyau $\ker U_{X^*,A^*}$
  de $U_{X^*,A^*}$ est
  $$
  \left( \ker A^*\cap {(X^*)}^\perp\right) \oplus {(A^*)}^\perp
  \quad,  
  $$
  o\`u ${(X^*)}^\perp$ est l'espace vectoriel orthogonal \`a $X^*$ dans $(\Vc,<,>)$,
  et  ${(A^*)}^\perp$ est l'espace vectoriel orthogonal \`a $A^*$ dans $(\Zc,<,>)$.
  En effet, avec les notations ci-dessus, on a l'\'equivalence :
  $$
  \forall X'\in E_1 \quad         
  U_{X^*,A^*} (n).F (X')=F(X')
  \Longleftrightarrow
  \left\{
    \begin{array}{l}
      p_1(X)=p_2(X)=0\\
      <X^*,X>=0\\
      A=0
    \end{array}\right.
  \quad.
  $$
\end{rem}
\begin{rem} \label{rem_expression_centre}
  La repr\'esentation $U_{X^*,A^*}$ s'identifie sur le centre $\Zc$ avec le caract\`ere :
  $$
  \exp A\,\longrightarrow\, \exp (i<A^*,A>)
  \quad.
  $$
\end{rem}

Ces repr\'esentations sont celles donn\'ees par la m\'ethodes des orbites :
\begin{prop}
  \label{prop_hatN}
  Le repr\'esentant privil\'egie de chaque orbite $\Nc^*/N$ 
  est $f=A^*+X^*$, 
  o\`u $A^*\in \Zc^*$ et  $X^*\in \Vc^*$
  tel que $X^*\in \ker A^*$.
  On a $(\Hc_{X^*,A^*},U_{X^*,A^*})\in T_f$.
\end{prop}

On obtient donc $\hat{N}$ comme l'ensemble des classes 
$T_{X^*+A^*}$
des repr\'esentations $U_{X^*,\Ac^*}$,
avec $A^*\in \Zc^*$ et $X^*\in \ker A^*$.

Le reste de cette sous section est consacr\'e 
\`a la d\'emonstration de cette proposition.

\paragraph{Repr\'esentant de $\Nc^*/N$.}
Donnons l'expression des repr\'esentations adjointe
et coadjointe pour $n=\exp(X+A)\in N$:
\begin{eqnarray*}
  \forall X'+A'\in \Nc&\qquad
  \Ad.n (X'+A') 
  &=\,
  X'+A'
  +[X, X']
  \quad ,\\    
  \forall X^*+A^*\in \Nc^*&\qquad
  \Coad.n(X^*+A^*) 
  &=\,
  X^*+A^*
  -A^*.X
  \quad .
\end{eqnarray*}
Ainsi,
l'orbite $N.f$ de $f=X^*+A^*\in \Nc^*$ pour l'action coadjointe de $N$
est l'espace affine $X^*+\Im A^*+A^*\subset\Nc^*$.
D\'ecomposons 
$X^*=(X^*-X^*_0)+X^*_0$ 
o\`u en identifiant $\Vc\sim \Vc^*$ par le produit scalaire,
$p_0(X^*)=X^*_0\in\ker A^*$, $X^*-X^*_0\in {\ker A^*}^\perp=\Im A^*$;
on a $\Coad.\exp( -(X^*-X^*_0).(f)=p_0(X^*)+A^*)$.
Ainsi, on peut choisir comme repr\'esentant privil\`egi\'e 
d'une orbite $X^*_0+A^*$ avec $X^*_0\in\ker A^*$.

\paragraph{Construction d'une repr\'esentation   associ\'ee.}
Fixons une forme lin\'eaire $f=X^*+A^*$ avec $X^*\in\ker A^*$.
On d\'efinit la forme $B_f$ bilin\'eaire antisym\'etrique  sur $\Nc$ associ\'ee \`a $f$ :
$$
\forall \; V,V'\in\Nc
\quad :\quad 
B_f(V,V')\,=\,f([V,V'])
\quad.
$$
Or on voit facilement gr\^ace \`a  l'\'egalit\'e~(\ref{egalite_crochet_pdtsc}) :
$$
B_f(X+A,X'+A')
\,=\,
f([X,X'])
\,=\, 
<A^*,[X,X']>
\,=\, 
w_{A^*}(X,X')
\quad .  
$$
puis que :
\begin{itemize}
\item si $A^*=0$, alors $B_f$ est nulle, et $\Lc_f=\Nc$ est une polarisation en $f$,
\item si $A^*$ est non nul, et si 
  $E_1$ est un sous-espace totalement isotrope pour $\omega_{A^*,r}$,
  alors en posant $E_2:=A^* E_1\subset \Im A^*$,
  le sous espace  $\Lc_f:=E_2\oplus \ker A^*\oplus\Zc$
  est une polarisation en $f$.
\end{itemize}

On note 
\begin{itemize}
\item $L=L_f=\exp \Lc_f$ le sous-groupe de $N$ d'alg\`ebre de Lie $\Lc_f$.
\item $\chi$ le caract\`ere sur $L$ dont la diff\'erentielle est $if$.
\item $U$ la repr\'esentation de $N$ induite par $\chi$.
\end{itemize}
Lorsque $A^*=0$, alors on a $L=N$, et $U$ est la repr\'esentation
sur $N$ qui s'identifie au caract\`ere $\chi$:
$$
\forall n=\exp (X+A) \in N,
\qquad 
\chi(n)=\exp( i<X^*,X>)
\quad.
$$

Pla\c cons-nous maintenant dans le cas $A^*\not=0$,
et explicitons $U$.

L'espace de la repr\'esentation est $\Hc_U$ est
l'ensemble des fonctions 
$G: N\rightarrow \Cb$ telles que :
\begin{eqnarray*}
  &1.&  \forall l\in L, n\in N 
  \quad G(ln)=\chi(l) F(n)\quad ,\\
  &2.&  \dot{n}\rightarrow \nn{G(n)}\in L^2(N/L) \quad . 
\end{eqnarray*}
La repr\'esentation est donn\'ee par :
$$
\forall\, G\in\Hc_U\; , \quad
\forall\, n,n'\in N\; ,
\qquad U (n).G(n')
\,=\,
G(n'n)
\quad .
$$

Explicitons 
l'expression de la repr\'esentation $(\Hc_U,U)$.

\begin{lem}
  Soit $G\in \Hc_U$ et $n=\exp (X+A)$ et $n'=\exp (X'+A')$.
  On a :
  \begin{eqnarray*}
    U(n).G(n')
    &=&
    \exp \left(i<A^*, \frac12\left([p_1(X+2X'),p_2(X)]+[p_1(X'),p_2(X')]\right)+A'+A>\right)\\
    && \quad e^{i<X^*,X+X'>}
    G\circ\exp\circ p_1(X+X')
    \quad.
  \end{eqnarray*}  
\end{lem}

On utilise les notations $E_2,p_0,p_1,p_2$ 
d\'evelopp\'ees plus haut
($E_1$ \'etant fix\'e). 
Rappelons $\Id_\Vc=p_0+p_1+p_2$.

\begin{proof} On garde les notations du lemme.
  On a $n'n=abc$ o\`u :
  \begin{eqnarray*}
    a
    &=&
    \exp(A'+A+\frac 12 [X',X])
    \quad,\quad
    c
    \,=\,
    \exp( p_1(X+X'))
    \quad,\\   
    b
    &=&
    \exp \left(X+X'\right)\;
    \exp \left(-p_1(X+X')  \right)\\
    &=&
    \exp (X+X' -p_1(X+X')
    +\frac12 [X+X',
    -p_1(X+X') ] )\\
    &=&
    \exp( (p_2+ p_0)(X+X')+\frac12 [X+X',
    -p_1(X+X')])
    \quad.
  \end{eqnarray*}
  On remarque $a,b\in L$, donc $G(n'n)= \chi(a) \chi(b) G(c)$. 
  Or par d\'efinition de $\chi$, on a d'une part :
  \begin{eqnarray*}
    \chi(a)
    &=&
    \exp(if( A'+A+\frac 12 [X',X]))\\
    &=&
    \exp(i<A^*,A'+A> + \frac i2<A^*, [X',X]>)
    \quad,
  \end{eqnarray*}
  d'autre part,
  \begin{eqnarray*}
    \chi(b)
    &=&
    \exp\left(if\left( (p_2+ p_0)(X+X')+ \frac12 [X+X', -p_1(X+X')]\right)\right)\\
    &=&
    \exp\left( i<X^*,(p_2+p_0)(X+X')> + \frac i2 <A^*,[X+X',-p_1(X+X') ]>\right)
    \quad.
  \end{eqnarray*}

  Comme on a choisi $X^*\in \ker A^*$, on a $<X^*, p_1(X+X')>=0$ et :
  $$
  <X^*,(p_2+ p_0)(X+X')>
  \,=\,
  <X^*,X+X'>\quad;
  $$

  On a 
  comme   le sous espace $E_1$ est isotrope pour
  $\omega_{A^*,r}$,
  \begin{eqnarray*}
    &&<A^*,[X+X',-p_1(X+X') ]>
    \,=\,
    \omega_{A^*}(X+X',-p_1(X+X'))\\
    &&\qquad=\,
    \omega_{A^*,r}((p_1+p_2)(X+X'),-p_1(X+X'))
    \,=\,
    \omega_{A^*,r}(p_2(X+X'),-p_1(X+X'))\\
    &&\qquad=\,
    <A^*,[p_1(X+X'),p_2(X+X') ]>
    \quad ,
  \end{eqnarray*}
  et comme   le sous espace $E_2$ l'est aussi :
  \begin{eqnarray*}
    <A^*, [X',X]>
    &=&
    \omega_{A^*}(X',X)
    \,=\,
    \omega_{A^*,r}((p_1+p_2)X',(p_1+p_2)X)\\
    &=&
    \omega_{A^*,r}(p_1(X'),p_2(X))
    +
    \omega_{A^*,r}(p_2(X'),p_1(X))\\
    &=&
    <A^*, [p_1(X'),p_2(X)]+[p_2(X'),p_1(X)]>
    \quad.
  \end{eqnarray*}
  En rassemblant les termes, on obtient donc :
  \begin{eqnarray*}
    G(n'n)
    &=&
    \exp(i<A^*,A'+A> + \frac i2<A^*, [p_1(X'),p_2(X)]+[p_2(X'),p_1(X)]>)\\
    &&
    \exp( i<X^*,X+X'>+\frac i2 <A^*,[p_1(X+X'),p_2(X+X') ]>)\\
    &&
    \quad  G(\exp(p_1(X+X')))
    \quad ,
  \end{eqnarray*}
  puis l'expression de $(\Hc_U,U)$.
\end{proof}

Maintenant 
la transformation unitaire:
$$
\begin{array}{rcl}
  \Hc_U &\longrightarrow &\Hc_{X^*,A^*}=L^2(E_1)\\
  G&\longmapsto& F=\{X_1\mapsto G\circ\exp\circ p_1(X_1))\}
\end{array}
\quad.
$$
entrelace les repr\'esentations $U$ et  $U_{X^*,A^*}$ de $N$.

Deux choix diff\'erents de sous espace $E_1$ totalement isotrope pour $\omega_{A^*,r}$
conduisent \`a deux polarisations pour la m\^eme forme lin\'eaire $f=X^*+A^*$,
et donc \`a deux repr\'esentations $U$ \'equivalentes,
puis \`a deux repr\'esentations $U_{X^*,A^*}$ \'equivalentes.

Ceci ach\`eve la d\'emonstration 
de  la proposition~\ref{prop_hatN}.

\subsection{Description de $\hat{N}/G$}
\label{subsec_cor_kirillov}

Nous d\'emontrons ici le corollaire du th\'eor\`eme de Kirillov
suivant :
\begin{cor}[\mathversion{bold}{$\hat{N}/G$}]
  \label{cor_thm_kirillov}
  Soient $G,N$ deux groupes.
  On suppose que $N$ est nilpotent; 
  on note $\Nc$ son alg\`ebre de Lie et 
  $\Nc^*$ le dual de cette alg\`ebre.
  On suppose \'egalement que 
  $G$ agit contin\^ument par automorphismes sur $N$;
  le groupe $G$ agit alors sur l'ensemble $\hat{N}$ 
  et sur le dual $\Nc^*$ (par automorphismes):
  $$
  \left\{ \begin{array}{lcl}
      G\times \hat {N}
      &\longrightarrow &
      \hat{N}\\
      (g,\rho)
      &\longmapsto&
      g.\rho=
      \{n\mapsto \rho(g^{-1}.n)\}
    \end{array}\right. 
  \;\mbox{et}\;
  \left\{ \begin{array}{lcl}
      G\times \Nc^*
      &\longrightarrow &
      \Nc^*\\
      (g,f)
      &\longmapsto&
      g.f=
      \{n\mapsto f(g^{-1}.n)\}
    \end{array}\right.
  \;.
  $$
  Pour un \'el\'ement $f\in \Nc^*$,
  on note $G.f$ l'orbite pour cette derni\`ere action,
  et $T_f$ la repr\'esentation de $N$ associ\'ee par le th\'eor\`eme de Kirillov.

  On a pour $f\in\Nc^*$ et $g\in G$ : 
  \begin{equation}
    \label{eg_rep_action}
    g.T_f
    \,=\,
    T_{g.f}
    \quad .
  \end{equation}
  On en d\'eduit la bijection :
  $$
  \left\{
    \begin{array}{rcl}
      N\backslash \Nc^*/G&\longrightarrow&\hat{N}/G\\
      G.f&\longmapsto&G.T_f
    \end{array}
  \right.
  \quad.
  $$
\end{cor}

\begin{proof}[du corollaire~\ref{cor_thm_kirillov}]
  Montrons l'\'egalit\'e~(\ref{eg_rep_action}).
  Fixons  $f\in\Nc^*$ et $g\in G$. 
  On choisit une polarisation $\Nc_0$ en $f$.
  On note alors $g.\Nc_0=\Nc_0'$. 
  Comme $G$ agit par automorphismes sur $N$,
  on v\'erifie ais\'ement que 
  $\Nc_0'$ est une polarisation en $f':=g.f$.
  On note $\chi$ et $\chi'$ 
  les homomorphismes sur $N_0=\exp \Nc_0$ et $N_0'=\exp \Nc_0'$
  ayant pour diff\'erentielles $if$ et $if'$ respectivement. 
  On a :
  $$
  \Ind_{N_0}^N \chi = (\Hc,\Pi) \in T_f
  \qquad\mbox{et}\qquad
  \Ind_{N_0'}^N \chi'=(\Hc',\Pi')  \in T_{f'}
  \quad .
  $$
  \paragraph{Mesures choisies.}
  Dans ces inductions, 
  on suppose que les mesures ont \'et\'e choisies de la mani\`ere suivante :
  \begin{itemize}
  \item Les mesures $dn$ et $dn_0$ sur $N$ et $N_0$ sont choisies
    telle qu'il existe une mesure $d\dot{n}$ sur
    $N/N_0$, $N$-invariante  et
    qu'elles v\'erifient
    pour toute fonction continue $h$ \`a support compact
    sur~$N$ :
    $$
    \int_{N} h(n) dn
    \,=\,
    \int_{N/N_0} \left( \int_{N_0} h(n\, n_0) dn_0 \right) d\dot{n}
    \quad.
    $$
  \item Le groupe $G$ agit sur l'alg\`ebre de Lie $\Nc$ par automorphismes.
    Le jacobien du changement de variable $n\mapsto g.n$ 
    est donc la valeur absolue du d\'eterminant 
    de l'application $X\mapsto g.X$ sur $\Nc$;
    en particulier, il  est constant sur $N$;
    on le note $\nn{\det g}$.

    On d\'efinit la mesure $dn'_0$ sur $N'_0$
    comme la mesure-image par  $n\mapsto g.n$  de $\nn{\det g} dn_0$ :
    c'est-\`a-dire pour toute fonction continue $h$ \`a support compact
    sur $N'_0$ :
    $$
    \int_{N_0}h(g.n_0)\nn{\det g} dn_0
    \,=\,
    \int_{N'_0}h(n'_0)  dn'_0
    \quad;
    $$
  \item On d\'efinit l'ensemble $C_0(N;N_0)$
    des fonctions continues sur $N$,
    invariantes par $N_0$, 
    qui pass\'ees au quotient sur
    $N/N_0$, sont \`a support compact;
    on fait de m\^eme pour $C_0(N;N_0')$.
    Les mesures (positives de Radon) sur $N/N_0$ 
    s'identifient aux
    formes lin\'eaires positives sur $C_0(N;N_0)$; 
    et de m\^eme sur $N/N'_0$.

    Comme $g$ est un automorphisme continu 
    qui envoie $N_0$ sur $N'_0$,
    on a la bijection :
    $$
    \left\{
      \begin{array}{rcl}
        C_0(N;N_0)&\longrightarrow&C_0(N;N_0')\\
        h&\longmapsto& h(g.):n\mapsto h(g.n)
      \end{array}\right.\quad;
    $$
    cela permet de d\'efinir $d\dot{n}'$ 
    la mesure sur $N/N'_0$ 
    comme la forme lin\'eaire positive sur $C_0(N;N_0')$ 
    donn\'ee par :
    $$
    h \longmapsto \int_{N/N_0} h(g.n) d\dot{n}
    \qquad \mbox{i.e.}\qquad
    \int_{N/N_0}h(g.n) d\dot{n}
    \,=\,
    \int_{N/N'_0}h(n) d\dot{n}'
    \quad.
    $$
  \end{itemize}
  La mesure $d\dot{n}'$ est $N$-invariante car $d\dot{n}$ l'est;
  en effet, on a pour $ h\in C_0(N;N'_0), n_1\in N$ :
  $$
  \int_{N/N_0'}h(n_1 n) d\dot{n}'
  \,=\,
  \int_{N/N_0}h(n_1\,g.n) d\dot{n}
  \,=\,
  \int_{N/N_0}h(\,g.n) d\dot{n}
  \quad.
  $$
  De plus, la mesure $d\dot{n}'$ v\'erifie 
  pour toute fonction $h$ continue \`a support compact dans $N$ :
  \begin{eqnarray*}
    \int_{N} h(n) dn
    &=&
    \int_{N} h(g.n) \nn{\det g} dn
    \,=\,
    \int_{N/N_0} \left( \int_{N_0} h(g.(n\, n_0)) \nn{\det g} dn_0\right) d\dot{n}\\
    &=&
    \int_{N/N_0} \left( \int_{N_0} h(g.n\, g.n_0)) \nn{\det g} dn_0 \right) d\dot{n}
    \,=\,
    \int_{N/N_0'} \left( \int_{N_0'} h(n'\, n_0') dn_0' \right) d\dot{n}'
    \quad.
  \end{eqnarray*}
  par d\'efinition de $dn'_0$ et $d\dot{n}'$. 
  \paragraph{Op\'erateur d'entrelacement entre
    \mathversion{bold}{$g.\Pi$} et \mathversion{bold}{$\Pi'$}.}
  On pose  pour une fonction $u\in \Hc'$ et pour $n\in N$ :
  $\{A.u\} (n) \,=\, u(g.n)$.

  Montrons $A : \Hc' \rightarrow \Hc$.
  Soit $u\in \Hc'$. On a pour $n_0\in N_0$ et $ n\in N$ :
  $$
  A.u(n_0\,n)
  \,=\,
  u(g.(n_0\,n))
  \,=\,
  u(g.n_0\,g.n)
  \,=\,
  \exp(if'(g.n_0))
  u(g.n)
  \quad,
  $$
  par d\'efinition de $\Hc'$, car $g.n_0\in N'_0$.
  Comme $f'=g.f$ et $u(g.n)=A.u(n)$,
  on a bien : 
  $$
  A.u(n'_0n)=\exp(i f(n'_0)A.u(n))
  \quad.
  $$
  De plus,
  la fonction $\nn{A.u}$ passe au quotient en une fonction sur $N/N_0$
  qui est  localement int\'egrable
  car $n\mapsto g.n$ est continue;
  d'apr\`es le choix des mesures $d\dot{n}$ et $d\dot{n}'$,
  on a :
  \begin{eqnarray*}
    \int_{N/N_0} \nn{A.u(n)}^2 d\dot{n}
    &=&
    \int_{N/N_0} \nn{u(g.n)}^2 d\dot{n}\\
    &=&
    \int_{N/N'_0} \nn{u(n')}^2 d\dot{n'}
    \,=\,
    \nd{u}_{\Hc'}^2 \;<\infty
    \quad;
  \end{eqnarray*}
  et donc $A.u\in \Hc$ et $\nd{A.u}_{\Hc}=\nd{u}_{\Hc'}$.

  Montrons $ g.\Pi\circ A =A\circ \Pi' $.
  En effet pour $g\in G,u\in \Hc', n,n'\in N$, on a:
  \begin{eqnarray*}
    \left\{g.\Pi(n)\circ A . u \right\}(n')
    &=&
    \left\{\Pi(g^{-1}.n)\circ A .u\right\}(n')
    \,=\,
    A.u (n'\,g^{-1}.n)\\
    &=&
    u(g. (n'\,g^{-1}.n))
    \,=\,
    u(g.n'\,n)
    \quad,\\
    \left\{ A\circ \Pi'(n).u\right\}\,(n')
    &=&
    \Pi'(n).u(g.n')
    \,=\,
    u(g.n'\, n)
    \quad.
  \end{eqnarray*}

  Nous venons de montrer que $A$ est un op\'erateur unitaire 
  qui entrelace $g.\Pi$ et $\Pi'$.
  L'\'egalit\'e~(\ref{eg_rep_action}) entre classe de repr\'esentations est donc d\'emontr\'ee.

  D'apr\`es cette \'egalit\'e, 
  les classes de repr\'esentations associ\'ees 
  aux formes lin\'eaires de $G.f$ 
  sont les \'el\'ements de $G.T_f$.
\end{proof}

Sous les hypoth\`eses du corollaire pr\'ec\'edent,
pour un \'el\'ement $\rho\in\hat{N}$,
on note $G_\rho$ son \textbf{groupe stabilisateur}
\index{Groupe!stabilisateur d'une repr\'esentation} :
\index{Notation!Groupe!$G_\rho$} 
$$
G_\rho
\,=\,
\{ g\in G;\quad  g.\rho\,=\, \rho \}
\quad.
$$

\begin{prop}  [Stabilisateur]
  \label{prop_stab}
  Soient $N$ un groupe nilpotent, 
  et $K$ un groupe (localement compact) 
  qui agit contin\^ument sur~$N$ comme groupe d'automorphismes.
  On note $G=K\triangleleft N$ leur produit semi-direct;
  le groupe $G$ agit contin\^ument par automorphisme sur~$N$.

  Fixons $\rho\in \hat{N}$. 
  Alors le groupe $G_\rho$ stabilisateur de $\rho$ 
  peut s'\'ecrire 
  comme $K_\rho\triangleleft N$,
  o\`u $K_\rho$ est le sous-groupe :
  $$
  K_\rho
  \,:=\,
  \{ k\in K:\;
  k.\rho = \rho
  \} 
  \,\subset\, K \quad .
  $$
  \index{Notation!Groupe!$K_\rho$} 

  On peut toujours supposer $\rho=T_f, f\in  \Nc^*$,
  et dans ce cas :
  $$
  K_\rho
  \,=\,
  \{ k\in K\subset G:\;
  k.f \in N.f
  \} \quad .
  $$
\end{prop}

\begin{proof}[de la proposition~\ref{prop_stab}]
  D'apr\`es le th\'eor\`eme~\ref{thm_kirillov} de Kirillov, 
  il existe une forme lin\'eaire $f\in \Nc^*$ tel que 
  $ T_f\, = \, \rho$.
  D'apr\`es le th\'eor\`eme~\ref{thm_kirillov} de Kirillov
  et son corollaire~\ref{cor_thm_kirillov}, on a :
  $$
  g \in G_\rho
  \,\Longleftrightarrow\,
  g.\rho \,=\, \rho 
  \,\Longleftrightarrow\,  
  g.T_f\,=\,T_{g.f}\,=\, T_f
  \,\Longleftrightarrow\,
  g.f\in N.f
  \quad,
  $$
  d'o\`u :
  $$
  \forall g=(k,n)\in G
  \quad :  \quad
  g \in G_\rho
  \,\Longleftrightarrow\,
  k.f \in N.f
  \quad .
  $$
  On en d\'eduit que le groupe $G_\rho$ peut s'\'ecrire 
  comme $K_\rho\triangleleft N$,
  o\`u $K_\rho$ est le sous-groupe :
  $$
  K_\rho
  \,=\,
  \{ k\in K :\;
  k.f \in N.f
  \}\,\subset\, K
  \quad.
  $$

  Or ``en remontant les \'equivalences'' pr\'ec\'edentes,
  on voit :
  $$
  k.f \in N.f
  \,\Longleftrightarrow\,
  T_{k.f}\,=\, k.T_f \,=\, T_f 
  \,\Longleftrightarrow\,
  k.\rho \,=\,\rho
  \quad .
  $$
  On obtient la premi\`ere caract\'erisation de $K_\rho$ donn\'ee dans la proposition~\ref{prop_stab}.
\end{proof}

Nous ne consid\'ererons ici que 
le cas d'un sous groupe normal~$N$ d'un group~ $G$; 
le groupe~$G$ agit alors sur~$N$ par conjugaison, 
donc sur le dual~$\Nc^*$ par la repr\'esentation coadjointe.

\subsection{Th\'eor\`eme de Mackey}

Une partie de la th\'eorie de Mackey d\'ecrit 
$\hat{G}$ en fonction de $\hat{N}$
lorsque $N$ est un sous groupe distingu\'e ferm\'e de type~I 
de~$G$;
le probl\`eme est d'\'etendre les repr\'esentations $\rho$ 
de $N$ 
\`a leurs stabilisateurs $G_\rho$ 
lorsque le quotient $\hat{N}/G$ a une structure mesurable ``raisonnable''.
Ce probl\`eme fut \'etudi\'e plus avant 
en consid\'erant les multiplicateurs de $G_\rho/N$;
nous n'irons pas dans cette direction.

Pour l'\'enonc\'e du th\'eor\`eme ci-dessous, 
nous renvoyons par exemple 
\`a \cite[ch.III sec.B theorem~2]{lipsman}.

\begin{thm}[Mackey]
  \label{thm_mackey}
  \index{Repr\'esentation!th. de Mackey}
  Soient $N$ un groupe nilpotent,
  et $K$ est un groupe compact qui
  agit contin\^ument sur $N$ par automorphismes.
  On note
  $G=K\triangleleft N$ 
  le produit semi-direct.

  Le groupe $G$ agit sur $N$ par conjugaison, 
  donc sur $\hat{N}$.
  Pour $\rho\in \hat{N}$, 
  on note $G_\rho$ 
  le stabilisateur de $\rho$,
  et on pose 
  $$
  \check{G}_\rho
  \,=\,
  \{ \nu \in \hat{G}_\rho;
  \quad \nu_{|N}\;\mbox{est un multiple de }\; \rho \}
  \quad .
  $$

  Alors
  pour $\rho \in \hat{N}$ et $\nu\in \check{G}_\rho$, 
  la repr\'esentation
  $\Pi_\nu =\Ind_{G_\rho}^G \nu $ est irr\'eductible;
  $\hat{G}$ est l'union disjointe :
  $$
  \hat{G}
  \,=\,
  \bigcup_{\rho \in \hat{N}/G}
  \{ \Pi_\nu;\quad \nu \in \check{G}_\rho \}
  \quad.  
  $$
\end{thm}

Ce th\'eor\`eme est vrai 
lorsque
$G$ est un groupe localement compact et 
$N$ un sous-groupe distingu\'e ferm\'e 
r\'eguli\`erement plong\'e de type I de~$G$.

\subsubsection{Propri\'et\'es utilis\'ees avec le 
  th\'eor\`eme~\ref{thm_mackey}}

Nous utiliserons le corollaire suivant 
du th\'eor\`eme des sous-groupes de Mackey
(voir par exemple \cite[ch.II sec.A subsec.1 theorem 1]{lipsman}) : 

\begin{cor}[Th\'eor\`eme des sous-groupes]
  \label{cor_thmssgr}
  \index{Repr\'esentation!th. des sous-groupes}

  Soient $G_1,G_2$ deux sous-groupes ferm\'es d'un groupe $G$.

  Si l'ensemble des doubles classes de $G$ sous $G_1$ et $G_2$ 
  est r\'eduit \`a celle de l'\'el\'ement neutre,
  alors pour toute repr\'esentation $\nu$ de $G_1$, on a :
  $$
  {\left[\Ind_{G_1}^G \nu \right]}_{|G_2}
  \, \sim \,
  \Ind_{G_1\cap G_2 }^{G_2} (\nu_{|G_1\cap G_2})
  \quad.
  $$
\end{cor}

De plus, nous utiliserons
une version faible du th\'eor\`eme du nombre d'entrelacement 
(voir par exemple \cite[ch.II sec.A lemma~5]{lipsman}) :

\begin{lem}[Th\'eor\`eme du nombre d'entrelacement]
  \label{lem_thmnbentr}
  \index{Repr\'esentation!th. du nombre d'entrelacement}

  Soient $H$ un sous-groupe ferm\'e d'un groupe $G$ 
  et $\gamma$ une repr\'esentation de $H$,
  tels que la vari\'et\'e homog\`ene 
  $G/H$ admette une mesure $G$-invariante finie.
  Le nombre de fois 
  que la repr\'esentation $\Ind_H^G \gamma$ contient $1_G$
  (comme un facteur direct discret)
  est \'egale au nombre de fois que $\gamma$ contient $1_H$
  (comme un facteur direct discret).

\end{lem}


\chapter{Fonction maximale sph\'erique}
\label{chapitre_fonction_maximale_spherique} 

Dans ce chapitre, 
nous d\'emontrons 
le th\'eor\`eme~\ref{thm_intro_in_max_sph_Lp},
c'est-\`a-dire des in\'egalit\'es $L^p$ pour la fonction maximale
sph\'erique associ\'ee \`a la norme de Kor\'anyi sur les groupes de
type~H ou~$N_{v,2}$. 
Le r\'esultat est d\'eja connu sur les groupes de Heisenberg
\cite{cowling}, et sur les groupes de type~H  \cite{schmidt},
mais pas pour les groupes $N_{v,2}$.

Tout au long des deux sections qui suivent, 
$N$ d\'esigne un groupe de type H, 
ou un groupe libre nilpotent de pas deux.
On reprend les notations de la section~\ref{sec_groupe_fcnmax},
en particulier $\Nc=\Vc\oplus\Zc$.
On pose $\dim \Vc=v$ et $\dim \Zc=z$,
$Q=v+2z$ et $v=2v'$ ou $2v'+1$.
Le groupe $N$ est muni de sa structure de groupe homog\`ene 
et d'une norme homog\`ene, 
sur lequel on a fix\'e un mesure de Haar $dn$
(sous-section~\ref{subsec_structure_homogene}).
On note $\mu$  la mesure
pour laquelle on a le passage 
en coordonn\'ees polaires~(\ref{formule_changement_polaire});
\index{Notation!Mesure!$\mu$}
on en connait facilement l'expression (voir proposition~\ref{prop_expression_mu}).
On note $\Ac$ la fonction maximale sph\'erique
(sous-section~\ref{subsec_fcnmax_sph}).

Dans les deux section qui suivent, 
nous montrons le th\'eor\`eme~\ref{thm_intro_in_max_sph_Lp} :
\begin{thm_princ}[\mathversion{bold}{$\nd{\Ac}_{p\rightarrow p}$}]
  \label{thm_princ_in_max_sph_Lp}
\index{Fonction maximale!Th\'eor\`eme}
  Pour $v'\geq 2$ et pour $p\in [1,\infty]$ tel que :
  \begin{itemize}
  \item[a)] $2\leq p\leq \infty$ si $v'=2$,
  \item[b)] si $v'>2$, dans le cas d'un groupe de type H, $(v'-1)/(v'-3/2)<p\leq\infty$,  
  \item[c)] si $v'\geq 2$, dans le cas $N=N_{v,2}$,
    $2h_0/(2h_0-1)<p\leq\infty$
    o\`u $h_0$ est le minimum de 
    $v+1$ et de la partie enti\`ere de $(z-1)/4$, 
  \end{itemize}
  la fonction maximale sph\'erique $\Ac$ v\'erifie 
  des in\'egalit\'es $L^p$ :
  $$
  \forall\, f\in L^p\; ,\qquad
  \nd{\Ac.f}_{L^p}
  \leq C\nd{f}_{L^p}
  \quad,
  $$ 
  o\`u $C$ est une constante de
  $v,z,p$.
\end{thm_princ}

Ce r\'esultat est d\'ej\`a connu sur les groupes de Heisenberg
\cite{cowling}, et on peut le d\'eduire de \cite{schmidt} pour les
groupes de type~H.
Les indices  pour $p$ alors obtenus sont optimaux : $p>n/(n-1)$,
o\`u $n=v+z$ est la dimension topologique du groupe.
Nous n'obtiendrons pas l'optimalit\'e dans le cas des groupes de type~H,
mais nous couvrirons le cas du
groupe libre nilpotent de pas  deux.
En effet, la courbure rotationnelle de la sph\`ere de Korn\'ayi
du groupe $N_{v,2}$ s'annule sur le centre.

Notre d\'emonstration va suivre le m\^eme point de d\'epart que
\cite{Maxfunc}; nous serons amen\'es \`a \'etudier les fonctions d'aires
d\'efinies pour des fonctions $f\in \Sc(N)$ 
de la classe de Schwartz sur $N$ 
par :
\index{Fonction d'aire $S^j$}
\index{Notation!$S^j(f)$}
$$
S^j(f) 
\,:=\,
\sqrt{ \int_{0}^\infty  \nn{\partial_s^j(f*\mu_s) }^2
  s^{2j-1} ds }
\quad;
$$
elle repose sur les deux th\'eor\`emes suivants :
\begin{thm}[\mathversion{bold}{$\nd{\Ac}_{p\rightarrow p}$}
  et \mathversion{bold}{$\nd{S^j}_{2\rightarrow 2}$}]
  \label{thm_fcnd'aire_inegmax}
\index{Fonction maximale!Th\'eor\`eme}
  Soit $h\in\Nb$.
  On suppose $v'\geq 2$ et $1\leq h<(Q-2)/2$.
  Si on a un contr\^ole $L^2$ pour les fonctions d'aires 
  $S^j, j=1,\ldots h$ :
  $$
  \forall j=1,\ldots h\quad
  \exists C>0\quad
  \forall f\in \Sc(N)\qquad
  \nd{S^j(f)}
  \,\leq\, C\,\nd{f}
  $$
  alors  pour $p\in [1,\infty]$ tel que :
  \begin{itemize}
  \item[a)] $2\leq p\leq \infty$ si $h=1$,
  \item[b)] $(2h)/(2h-1)<p\leq\infty$ si $h>1$, 
  \end{itemize}
  la fonction maximale sph\'erique $\Ac$ v\'erifie 
  des in\'egalit\'es $L^p$ :
  $$
  \forall\, f\in L^p\; ,\qquad
  \nd{\Ac.f}_{L^p}
  \leq C\nd{f}_{L^p}
  \quad,
  $$ 
  o\`u $C$ est une constante 
  de $v,z,p,h$.
\end{thm}

Le th\'eor\`eme ci-dessus se g\'en\'eralise au groupe stratifi\'e de
rang 2.

\begin{thm}[\mathversion{bold}{$\nd{S^j}_{2\rightarrow 2}$}]
  \label{thm_L2_fcnd'aire}
On suppose $v'\geq 2$. On a :
  \begin{itemize}
  \item[a)] dans le cas d'un groupe de type H :
  $$
  \forall h=1,\ldots,v'-1\quad
  \forall f\in L^2(N)\qquad
  \nd{S^h(f)}\,\leq\,C\,\nd{f}
  \quad,
  $$
  \item[b)] dans le cas d'un groupe $N_{v,2}$,
  $$
  \forall h \in \Nb,\quad 1\leq h\leq \frac{z-1}4,v+1 \quad
  \forall f\in L^2(N)\qquad
  \nd{S^h(f)}\,\leq\,C\,\nd{f}
  \quad,
  $$
  \end{itemize}
  o\`u $C$ est une constante de
  $v,z$. 
\end{thm}

Ces deux derniers th\'eor\`emes impliquent 
 le th\'eor\`eme~\ref{thm_princ_in_max_sph_Lp}.

\section{In\'egalit\'e $L^p, \, p\geq 2$ pour $\Ac$}

Le but de cette section est de d\'emontrer
le th\'eor\`eme \ref{thm_fcnd'aire_inegmax}.a).
Il repose sur
le contr\^ole $L^2$ de la fonction maximale pour les boules,
l'interpolation de Marcinkiewicz, 
et la proposition suivante :
\begin{prop}
  \label{prop_Ac_Mc_s}
  Soit $f\in \Sc(N)$. 
  On a :
  $$
  \Ac.f 
  \,\leq\, 
  Q\nn{B_1}
  \Mc.f +\frac1{\sqrt{2Q}}S^1.f 
  \quad .
  $$
\end{prop}

\subsection{D\'emonstration  du  th\'eor\`eme~\ref{thm_fcnd'aire_inegmax}.a)}

Admettons la proposition~\ref{prop_Ac_Mc_s},
et supposons que les hypoth\`eses du  th\'eor\`eme~\ref{thm_fcnd'aire_inegmax}.a)
sont v\'erifi\'ees; c'est-\`a-dire que la fonction d'aire $S^1$
v\'erifie une in\'egalit\'e $L^2$.

\paragraph{Montrons que \mathversion{bold}{$\Ac$ }
  v\'erifie des in\'egalit\'es \mathversion{bold}{$L^2$}
  et \mathversion{bold}{$L^\infty$}.}
La fonction maximale $\Mc$ v\'erifie des in\'egalit\'es $L^p$ 
(voir~(\ref{maxst})),
en particulier $L^2$.
Comme $S^1$ v\'erifie aussi une in\'egalit\'e $L^2$,
on en d\'eduit que c'est aussi le cas pour $\Ac$;
il existe une constante $C_2>0$ telle que : 
\begin{equation}
  \label{ineg_Ac_L2}
  \forall f
  \qquad 
  \nd{\Ac.f}_{L^2}
  \,\leq\,
  C_2 \nd{f}_{L^2}
  \quad.  
\end{equation}
Trivialement $\Ac$ v\'erifie une in\'egalit\'e $L^\infty$;
il existe une constante $C_\infty>0$ telle que : 
\begin{equation}
  \label{ineg_Ac_Linfty}
  \forall f
  \qquad 
  \nd{\Ac.f}_{L^\infty}
  \,\leq\,
  C_\infty \nd{f}_{L^\infty}
  \quad.
\end{equation}

\paragraph{Interpolons.}
Fixons    momentan\'ement
$r(.):(N,dn)\mapsto  \Rb^+$  une  fonction  mesurable,
et consid\'erons l'op\'erateur $T$ d\'efini sur les fonctions simples de
$N$ par :
$$
T.f(n)
\,=\,
\left( \mu_{r(n)}*f\right)\,(n)
\quad.
$$
D'apr\`es~(\ref{ineg_Ac_L2}) et~(\ref{ineg_Ac_Linfty}),
$T$ est born\'e sur~$L^2$ et sur~$L^\infty$;
d'apr\`es l'interpolation de Marcinkiewicz \cite{StW},
$T$ s'\'etend en un op\'erateur born\'e sur chaque $L^p$ 
pour $2\leq p\leq\infty$, avec des constantes qui ne d\'ependent que
de $p, C_2, C_\infty$.
Ceci est vrai pour toute fonction 
$r(.):(N,dn)\mapsto  \Rb^+$ mesurable.
On peut donc repasser au supremum :
la fonction maximale $\Ac$ v\'erifie une in\'egalit\'e $L^p, \, 2\leq
p\leq\infty$.

Le th\'eor\`eme \ref{thm_fcnd'aire_inegmax}.a) sera ainsi d\'emontr\'e 
lorsque nous aurons prouv\'e la proposition~\ref{prop_Ac_Mc_s}.

\subsection{Contr\^ole de $\Ac$ par $\Mc$ et $S^1$}

Le but de cette sous-section est de d\'emontrer
la proposition~\ref{prop_Ac_Mc_s}.
On proc\`ede comme dans le cas euclidien \cite{stein_Han}.

Soit  $f\in \Sc(N)$.   
La fonction $(s,n)\mapsto f*\mu_s(n)$ 
est alors d\'erivable et de d\'eriv\'ee continue selon 
$(s,n)\in\Rb^{*+}\times N$,
et  on a :
\begin{eqnarray*}
  f*\mu_t(n)
  &=&
  \frac1{t^Q}\int_{0}^t\partial_s(s^Qf*\mu_s(n)) ds\\  
  &=&
  \frac1{t^Q}\int_{0}^tQs^{Q-1}f*\mu_s(n)ds
  +\frac1{t^Q}\int_{0}^ts^Q\partial_s(f*\mu_s(n)) ds
  \quad , \\  
\end{eqnarray*}

Pour le premier terme du membre de droite, on voit 
d'une part d'apr\`es la 
formule~(\ref{formule_changement_polaire})
du passage en coordonn\'ees polaires:
$$
\int_0^ts^{Q-1}f*\mu_s(n)ds
\,=\,
\int_0^t
\int_{S_1} f(n\,s.{n'}^{-1})d\mu(n') s^{Q-1} ds
\,=\,
\int_{B(0,t)}
f(n{n'}^{-1}) dn'
\quad,
$$
d'autre part $\nn{B(n,t)}=\nn{B(0,t)}=t^Q\nn{B(0,1)}$ d'o\`u :
$$
\frac1{t^Q}
\,=\,
\frac{\nn{B_1}}{\nn{B(n,t)}}
\quad.
$$
On a donc :
$$
\frac1{t^Q}\int_{0}^tQs^{Q-1}f*\mu_s(n)ds
\,=\,
Q\nn{B_1}
\frac1{\nn{B(n,t)}}
\int_{B(n,t)}
f(n'') dn''
\quad .
$$

Pour le second terme,
on utilise H\"older :
\begin{eqnarray*}
  \frac1{t^Q}\nn{\int_0^ts^Q\partial_s(f*\mu_s(n)) ds}
  &\leq&
  \frac1{t^Q}{\left(\int_0^t\nn{s^{Q-\frac12}}^2ds\right)}^{\frac 12}
  {\left(\int_0^t\nn{s^{\frac12}\partial_s(f*\mu_s(n))}^2
      ds\right)}^{\frac12} \\
  &\leq&
  \frac1{\sqrt{2Q}}
  {\left(\int_0^{\infty}s\nn{\partial_s(f*\mu_s(n))}^2
      ds\right)}^{\frac12}
  \,=\,
  \frac1{\sqrt{2Q}}
  S^1.f(n)
  \quad.  
\end{eqnarray*}

On a donc pour tout $t>0$:
$$
\nn{f*\mu_t(n)} 
\,\leq\, 
Q\nn{B_1}
\frac1{\nn{B(n,t)}}
\int_{B(n,t)}
f(n'') dn''
+\frac1{\sqrt{2Q}}S^1.f(n)
\quad .
$$
Prenons le supremum en $t$ dans ce qui pr\'ec\`ede :
$$
\Ac.f(n) 
\,\leq\, 
Q\nn{B_1}
\Mc.f(n) +\frac1{\sqrt{2Q}}S^1.f(n)
\quad .
$$

Ceci ach\`eve la d\'emonstration de la proposition~\ref{prop_Ac_Mc_s},
et donc celle du th\'eor\`eme \ref{thm_fcnd'aire_inegmax}.a).

\section{In\'egalit\'e $L^p, \, p<2$ pour $\Ac$}

Le but de cette section est de d\'emontrer
le th\'eor\`eme \ref{thm_fcnd'aire_inegmax}.b).
Pour cela, 
nous allons placer 
l'op\'erateur de convolution par $\mu$ 
dans une famille analytique d'op\'erateurs $A^\alpha$,
pour laquelle
nous montrerons des in\'egalit\'es maximales \`a l'aide des fonctions d'aires.
Nous finirons par un argument d'interpolation. 

Nous choisissons la m\^eme famille analytique d'op\'erateurs
$A^\alpha$ que dans  \cite{Maxfunc}.

\index{Notation!Op\'erateur!$A^\alpha$}
\begin{prop}[\mathversion{bold}{$A^\alpha$}]
  \label{prop_famille_Aalpha}
  Nous d\'efinissons pour
  $\alpha\in\Cb\backslash(-\Nb)$ 
  sur $\Rb^+$ la fonction $m^\alpha$
  et sur $N$ la fonction $F^{\alpha}$ par :
  $$
  m^\alpha(x) 
  \,:=\, 
  2\frac{{(1-x^2)}_+^{\alpha-1}} {\Gamma(\alpha)}
  \;,\; x\in \Rb^+\;
  \quad\mbox{et}\quad
  F^\alpha(n) 
  \, :=\,
  m^\alpha(\nn{n})
  \; ,\; n\in N
  \quad .
  $$
\index{Notation!$m^\alpha,F^\alpha$}

  Pour~$\Re\alpha>0$, 
  nous d\'efinissons ainsi l'op\'erateur $A^\alpha$ de convolution avec
  $F^\alpha$. 
  Il est continu sur $L^p, 1<p\leq \infty$.

  Cette  famille  d'op\'erateurs $\{A^\alpha\}$  
  est  analytique dans le demi-plan  $\Re\alpha>0$,
  et se prolonge analytiquement 
  dans le demi-plan $\Re\alpha >1-h$ 
  lorsque $h<(Q-2)/2$ en une famille d'op\'erateurs
  sur la classe de fonctions  $\Sc(N)$ de Schwartz.
  Pour $\alpha=0$, $A^0$ est l'op\'erateur de convolution avec $\mu$.
\end{prop}

\subsection{Famille d'op\'erateurs $\{A^\alpha\}$}

Le but de cette sous-section
est de d\'emontrer la proposition~\ref{prop_famille_Aalpha}
On remarque :
\begin{lem}
  \label{lem_Falpha}
  Pour $\Re\alpha>0$, 
  $F^\alpha$ est int\'egrable et radiale sur $N$.
  Pour $\alpha >1$, $m^\alpha$ est d\'ecroissante.
\end{lem}

\begin{proof}
  En effet, on a :
  $$
  \int_{N}    \nn{F^\alpha(n)}   dn   
  \,\leq\,
  \int_{0}^1
  \frac{{(1-r^2)}^{\Re\alpha-1}} {\nn{\Gamma(\alpha)}} r^{Q-1}dr \quad,
  $$  
\end{proof}

Ainsi pour~$\Re\alpha>0$, 
l'op\'erateur $A^\alpha$ est l'op\'erateur de convolution avec la fonction $F^\alpha\in L^1$; 
c'est donc un op\'erateur born\'e sur $L^p, 1<p\leq\infty$.
De plus, la famille  d'op\'erateurs $\{A^\alpha\}$  
est  analytique sur  le demi-plan  $\Re\alpha>0$.   
Nous avons \`a montrer qu'elle se prolonge analytiquement, 
et que $A^0$ est l'op\'erateur de convolution avec $\mu$

Nous aurons besoin de l'expression de $\mu$ :
\begin{prop}
  \label{prop_expression_mu}
  On note $\tilde{\sigma}_n$ la mesure de la sp\`ere unit\'e
  euclidienne $S_1^{(n)}$ sur $\Rb^n$.

  Avec les notations ci-dessus,
  on a pour une fonction $g$ localement int\'egrable sur $N$ :
  \begin{eqnarray*}
    \int_{S_1} g(n)d\mu(n)  
    \,=\,
    2\int_{0}^1 \int_{S_1^{(v)}}\int_{S_1^{(z)}}
    g(\exp(rX +\sqrt{(1-r^4)}Z))\\
    d\tilde{\sigma}_z(Z)
    d\tilde{\sigma}_v(X)
    r^{v-1}{(1-r^4)}^\frac{z-2}2dr
    \quad .
  \end{eqnarray*}
\end{prop}
On remarque que
la  mesure $\mu$  est sym\'etrique  et radiale
c'est-\`a-dire  invariante sous  $n\mapsto   n^{-1}$
et sous $O(v)$ respectivement.

\begin{proof}[de la proposition~\ref{prop_expression_mu}]
  Posons :
  $$
  I(g)
  \,:=\,
  \int_{ N} g(n)dn 
  \,=\, 
  \int_{\Vc } \int_{\Zc}g(\exp(X+Z))dXdZ
  \quad .
  $$
  En effectuant un passage en coordonn\'ee  polaire 
  en $X\in\Vc\sim\Rb^v$ et en $Z\in\Zc\sim\Rb^z$, 
  $$
  I(g)=
  \int_0^\infty \int_{S_1^{(v)}}\int_0^\infty
  \int_{S_1^{(z)}}
  g(\exp(r_1X + r_2Z))
  d\tilde{\sigma}_z(Z) r_2^{z-1}
  d\tilde{\sigma}_v(X)r_1^{v-1}dr_1
  \quad .
  $$
  Consid\'erons le changement de variables suivant :
  $$
  (r_1,r_2)\,\longrightarrow\, (r'_1,r')
  \quad\mbox{o\`u}\quad
  r_1=r'r'_1
  \quad\mbox{et}\quad
  {r'}^4=r_1^4+r_2^2
  \quad,
  $$
  dont le jacobien est  
  $r_2dr_2\wedge dr_1=  2{r'}^4 dr'\wedge dr'_1$.
  Avec ce changement de variable, on a :
  \begin{eqnarray*}
    I(g)
    \,=\,
    2\int_0^\infty\int_0^1 
    \int_{S_1^{(v)}}\int_{S_1^{(z)}}
    g(r'\exp(r'_1X +{(1-{r'_1}^4)}^\frac12Z ))
    d\tilde{\sigma}_z(Z)d\tilde{\sigma}_v(X)\\
    {r'_1}^{v-1}{(1-{r'_1}^4)}^\frac{z-2}2{r'}^{Q-1}
    dr'_1dr'    \quad.
  \end{eqnarray*}

  De l'unicit\'e de la mesure $\mu$ 
  v\'erifiant la  formule~(\ref{formule_changement_polaire}) 
  du passage en coordonn\'ee polaire, 
  on trouve l'expression de $\mu$ donn\'ee dans la proposition.
\end{proof}

Pour montrer 
la proposition \ref{prop_famille_Aalpha}, 
nous aurons \'egalement besoin du lemme technique suivant :

\begin{lem}
  \label{lem1_prop_famille_Aalpha}
  On note $D_r$ l'op\'erateur diff\'erentiel 
  sur les fonctions de la variable $r\in \Rb$ : 
  $$
  D_r
  \,:=\,
  \partial_r\,\frac1{2r}
  \,=\,
  \frac1{2r} \,\partial_r
  -\frac1{2r^2}
  \quad.
  $$

  On peut \'ecrire l'op\'erateur
  $f\mapsto D_r^h.\left[f*\mu_r r^{Q-1}\right]$
  comme une combinaison lin\'eaire sur $j=0,\ldots,h$ de :
  $f\mapsto r^{Q-1-2h+j}  \partial_r^j( f*\mu_r )$.
\end{lem}

\begin{proof}[du lemme~\ref{lem1_prop_famille_Aalpha}]
  On a :
  $$
  \begin{array}{l}
    \displaystyle{   D_r.   
      \left[r^{Q-1-2h+j}   \partial_r^j  f*\mu_r  \right]   
      \,=\,  
      \frac12
      \partial_r
      \left[r^{Q-2-2h+j}  
        \partial_r^j f*\mu_r \right] }\\
    \displaystyle{ 
      \qquad =\, 
      \frac12(Q-2-2h+j)r^{Q-3-2h+j} \partial_r^j
      f*\mu_r
      +  \frac12 r^{Q-2-2h+j}  \partial_r^{j+1} f*\mu_r }\\
    \\
    \displaystyle{ 
      \qquad=\,  
      K_1 r^{Q-1-2(h+1)+j} \partial_r^j f*\mu_r
      + K_2 r^{Q-1-2(h+1)+(j+1)} \partial_r^{j+1} f*\mu_r
      \quad,}
  \end{array}
  $$
  o\`u les $K_i,i=1,2$ sont des constantes de $Q,h,j$.
  On en d\'eduit le lemme par r\'ecurrence.
\end{proof}

La proposition~\ref{prop_famille_Aalpha} sera d\'emontr\'ee
lorsque nous aurons montr\'e le lemme suivant :
\index{Notation!Op\'erateur!$A^\alpha_{h,j}$}
\begin{lem}
  \label{lem2_prop_famille_Aalpha}
  Pour $\Re\alpha>1-h$, 
  l'op\'erateur  $A^\alpha$ 
  co\"\i ncide sur $\Sc(N)$  avec  l'op\'erateur :
  $$
  A^{\alpha,h}
  \,:=\,
  f   
  \,\longmapsto\, 
  \int_{0}^1   m^{\alpha+h+1}(r)    \;
  D_r^h.\left[f*\mu_r r^{Q-1}\right]
  dr
  \quad .
  $$ 
\end{lem}

En effet, 
d'une part,
en utilisant le lemme~\ref{lem1_prop_famille_Aalpha},
on voit facilement que 
la famille d'op\'erateur $A^{\alpha,h}, \Re\alpha>1-h$
sur $\Sc(N)$
est analytique.
D'autre part,
l'op\'erateur $A^{0,1}$ co\"incide avec l'op\'erateur de convolution
avec $\mu$ :
$$
A^{0,1}.f
\,=\,
-\int_{0}^1 2\partial_r  \left( \frac{1}{-2r}
  f*\mu_r r^{Q-1} \right)  dr 
\,=\,
\left[ \frac{1}{r} f*\mu_r r^{Q-1}
\right]_{0}^1 
\,=\,
f*\mu
\quad .
$$

\begin{proof}[du lemme~\ref{lem2_prop_famille_Aalpha}]
  Montrons d'abord $A^\alpha = A^{\alpha,0}$:
  par passage en coordonn\'ees polaires 
  (\ref{formule_changement_polaire}),
  pour $f\in \Sc(N)$, on a :  
  $$
  A^\alpha.f(n)
  \,=\,
  \int_N
  f(n{n'}^{-1})  F^\alpha(n')  dn'
  \,=\,
  \int_{0}^\infty
  f*\mu_r  m^\alpha(r)  r^{Q-1}dr
  \quad.
  $$
  En int\'egrant par partie,
  on a donc :
  \begin{eqnarray*}
    &&  A^\alpha.f
    \,=\,
    \int_{0}^1 
    \frac{2{(1-r^2)}^{\alpha-1}}  {\Gamma(\alpha)}
    f*\mu_r r^{Q-1} dr\\
    &&  \quad=\,
    \left[ \frac{2{(1-r^2)}^{\alpha}}  {\alpha\Gamma(\alpha)}
      \, 
      \frac{1}{-2r}f*\mu_r r^{Q-1} \right]_{0}^1
    -  \int_{0}^1 
    \frac{2{(1-r^2)}^{\alpha}}  {\alpha\Gamma(\alpha)}
    \,
    \partial_r \left( \frac{1}{-2r}  f*\mu_r r^{Q-1} \right)
    dr 
    \; .
  \end{eqnarray*}
  Le crochet  est  nul (on  a
  suppos\'e  $Q\geq  3$). 
  En utilisant l'\'equation fonctionnelle 
  (\ref{eq_fonc_Gamma}) pour $\Gamma$,
  l'op\'erateur  $A^\alpha$ co\"\i ncide  avec $A^{\alpha,1}$
  qui a un sens pour $\Re\alpha>-1$.

  Si \`a l'ordre $h$, 
  l'op\'erateur  $A^\alpha$ co\"\i ncide  avec $A^{\alpha,h}$,
  alors effectuons une int\'egration par  partie :
  \begin{eqnarray*}
    &&
    \int_{0}^1  m^{\alpha+h}(r) 
    D_r^h.\left[f*\mu_r r^{Q-1}\right]
    dr
    \,=\, 
    \int_{0}^1
    \frac{2{(1-r^2)}^{\alpha+h-1}} {\Gamma(\alpha+h)} 
    D_r^h.\left[f*\mu_r r^{Q-1}\right]
    dr\\
    &&\quad=\, 
    \left[
      \frac{2{(1-r^2)}^{\alpha+h}} {(\alpha+h)\Gamma(\alpha+h)} 
      \;\frac1{-2r}  D_r^h.\left[f*\mu_r r^{Q-1}\right]
    \right]_0^1\\
    &&\qquad\qquad
    -\int_{0}^1   
    \frac{2{(1-r^2)}^{\alpha+h}}  {(\alpha+h)\Gamma(\alpha+h)}
    \;\partial_r.\left[\frac1{-2r}  
      D_r^h.\left[f*\mu_r r^{Q-1}\right]\right] dr 
  \end{eqnarray*}
  Le  crochet en  $r=1$ est nul.
  D'apr\`es le  lemme~\ref{lem1_prop_famille_Aalpha},  
  on peut  mettre $r^{Q-1-2h-1}$  en facteur  
  et donc le  crochet en  $r=0$ est  nul.   
  En utilisant l'\'equation fonctionnelle 
  (\ref{eq_fonc_Gamma}) pour $\Gamma$ 
  et la d\'efinition de $D_r$, 
  l'op\'erateur  $A^\alpha$ co\"\i ncide avec $A^{\alpha,h+1}$.

  On a donc montr\'e par r\'ecurence $A^{\alpha}=A^{\alpha,h}$
  lorsque $\Re \alpha>1-h$,
  pour tout $0\leq h< (Q-2)/2$.
\end{proof}

Gr\^ace aux lemmes~\ref{lem1_prop_famille_Aalpha}
et~\ref{lem2_prop_famille_Aalpha},
en d\'eveloppant
$D_r^h. \left[f*\mu_r r^{Q-1}\right]$
dans l'\'ecriture $A^\alpha=A^{\alpha,h}$,
on obtient : 
\index{Notation!Op\'erateur!$B^\alpha_{h,j}$}
\begin{cor}[\mathversion{bold}{$B^\alpha_{h,j}$}]
  \label{cor_Aalpha_sum_Balpha}
  Pour $\Re\alpha>1-h$ et $1\leq h< (Q-2)/2$,  
  l'op\'erateur  $A^\alpha$ co\"\i ncide 
  avec une combinaison lin\'eaire 
  sur $0 \leq j \leq h$ des op\'erateurs suivants :
  $$
  B^\alpha_{h,j} 
  \,:=\, 
  f \,\longmapsto\, \int_{0}^1
  m^{\alpha+h}(r) r^{Q-1-2h+j} \partial_r^j f*\mu_r dr
  \quad .
  $$
\end{cor}

\subsection{In\'egalit\'es maximales pour $A^\alpha$}

Nous donnons le sens suivant \`a la notion de dilatation :
\index{Dilatation} 
\begin{itemize}
\item pour une fonction $f$ sur $N$ :
  $$
  {(f)}_r(n) 
  \,=\, 
  r^{-Q}f(\frac1r.n) 
  \quad ,
  $$
\item pour une mesure $\nu$ sur $N$ :
  $$
  \nu_r(n)  
  \,=\, 
  \nu(r.n)
  \quad\mbox{et donc}\quad
  \int_N f d\nu_r  
  \,=  \,
  \int_N  f(r.n) d\nu(n) 
  \quad,
  $$
\item pour un  op\'erateur $T$ qui op\`ere sur  un espace de fonctions
  stables par dilatations au sens ci-dessus :
  $$
  T_r.f(n) 
  \,=\, 
  \left(T\,f(r.)\right)(r^{-1}.n)
  \quad\mbox{o\`u}\quad
  f(r.)=n\mapsto f(r.n) 
  \quad.
  $$
\end{itemize}

Ces notations sont coh\'erentes dans le sens o\`u :
si $T$ est un op\'erateur de convolution 
avec une fonction ou une mesure $\psi$,
alors $T_r$ est un op\'erateur de convolution 
avec la fonction ou la mesure $\psi_r$ respectivement.

Pour un op\'erateur $T$  qui op\`ere sur  un espace de fonctions stables par dilatation,
on peut ainsi consid\'erer 
la \textbf{fonction maximale associ\'ee 
  \`a la famille de ses dilat\'ees} $T_r, r>0$ :
$\sup_{r>0} \nn{T_r}$.

\vspace{1em}

On d\'efinit
$\Ac^\alpha$ 
la fonction maximale associ\'ee aux dilat\'es de $A^\alpha$ :
$\Ac^\alpha:=\sup_r\nn{A_r^\alpha}$,
et $\Bc^\alpha_{h,j}$ 
la fonction maximale associ\'ee aux dilat\'es de $B^\alpha$ :
$\Bc_{h,j}^\alpha:=\sup_r\nn{({B^\alpha_{h,j})}_r}$.
\index{Notation!Op\'erateur!$\Ac^\alpha$}
\index{Notation!Op\'erateur!$\Bc^\alpha_{h,j}$}

Nous souhaitons montrer 
des in\'egalit\'es maximales $L^2$ et $L^p$ :
\begin{prop}
  [\mathversion{bold}{$\nd{\Ac}_{p\rightarrow p}, 
    \Re\alpha\geq1$}]
  \label{prop_inegalitemax_Lp}
  On a  un contr\^ole maximal $L^p$ pour $A^\alpha$ :
  \begin{eqnarray*}
    &
    \forall [a,b]\subset[1,\infty[\quad
    \forall p\in]1,\infty]\quad
    \exists C>0\quad
    \forall x\in [a,b]\quad
    \forall y\in\Rb\quad
    \forall\, f\in \Sc(N)&\\
    &\nd{\Ac^{x+iy}.f}_{L^p(N)} 
    \,\leq \,C\, 
    e^{2y} \nd{f}_{L^p(N)}
    \quad.&
  \end{eqnarray*}
\end{prop}

\begin{prop}
  [\mathversion{bold}{$\nd{\Ac}_{2\rightarrow 2},
    \Re\alpha<0$}]
  \label{prop_inegalitemax_L2}
  Soit $1\leq h<(Q-2)/2$.
  Si on a un contr\^ole $L^2$ des fonctions d'aires
  $S^j, j=1,\ldots,h$ :
  $$
  \forall j=1,\ldots h\quad
  \exists C>0\quad
  \forall f\in \Sc(N)\qquad
  \nd{S^j(f)}
  \,\leq\, C\,\nd{f}
  $$
  alors on a un contr\^ole maximal $L^2$ pour $A^\alpha$ :
  \begin{eqnarray*}
    &
    \forall [a,b]\subset[1-h,\infty[\quad
    \exists C>0\quad
    \forall x\in [a,b]\quad
    \forall y\in\Rb\quad
    \forall\, f\in \Sc(N)&\\
    &\nd{\Ac^{x+iy}.f}_{L^2(N)} 
    \,\leq \,C\, 
    e^{2y} \nd{f}_{L^2(N)}
    \quad.&
  \end{eqnarray*}
\end{prop}

La preuve de la proposition~\ref{prop_inegalitemax_Lp}
repose sur le corollaire~\ref{cor_fcndec}.

\begin{proof}[de la proposition~\ref{prop_inegalitemax_Lp}]
  On a :
  $$
  \nn{A^\alpha.f}
  \, = \,
  \nn{f*F^\alpha} 
  \,\leq\,
  \nn{f}*\nn{F^\alpha}
  \quad .
  $$
  et :
  $$
  \nn{F^\alpha(n)}     
  \,=\,
  \nn{     \frac{{(1-\nn{n}^2)}_+^{\alpha-1}}
    {\Gamma(\alpha)}} 
  \,\leq\,
  \frac{{(1-\nn{n}^2)}_+^{\Re\alpha-1}}
  {\nn{\Gamma(\alpha)}}
  \,=\,
  \nn{\frac{\Gamma(\Re\alpha)}{\Gamma(\alpha)}}
  F^{\Re\alpha}(n)
  \quad .
  $$
  On en d\'eduit
  localement en $x=\Re\alpha$, uniform\'ement en $y=\Im\alpha$,
  gr\^ace \`a la majoration~(\ref{cqtit}) :
  $$
  \nn{A^\alpha.f}             
  \,\leq\,
  \nn{\frac{\Gamma(x)}{\Gamma(\alpha)}}
  \nn{f}*F^{x}       
  \,\leq\, 
  \nn{\frac{\Gamma(x)}{\Gamma(\alpha)}}
  \,A^{x}.\nn{f} 
  \,  \leq \;C\;  
  e^{ 2 y} \,A^x.\nn{f}
  \quad ,
  $$
  $C$ \'etant une constante issue de la majoration~(\ref{cqtit}).
  D'apr\`es le lemme~\ref{lem_Falpha},
  la  fonction  $F^x$   
  v\'erifie les hypoth\`eses
  du  corollaire~\ref{cor_fcndec} :
  l'op\'erateur  $\Ac^x$   v\'erifie  une
  in\'egalit\'e  $L^p$.
  On obtient la  majoration voulue  de 
  $\nd{\Ac^\alpha.f}_{L^p}$, pour tout $1<p\leq\infty$.
\end{proof}

La proposition~\ref{prop_inegalitemax_L2}
repose sur
le corollaire~\ref{cor_Aalpha_sum_Balpha}
et les deux lemmes suivants :

\begin{lem}
  \label{lem_Balpha_fcndaire}
  Pour $0<j\leq h<(Q-2)/2$, on a 
  un contr\^ole ponctuel des $\Bc^\alpha_{h,j}$
  par les fonctions d'aires $S^j$ :
  \begin{eqnarray*}
    &
    \forall [a,b]\subset[\frac12-h,\infty[\quad
    \exists C>0\quad
    \forall x\in [a,b]\quad
    \forall y\in\Rb\quad
    \forall\, f\in \Sc(N)&\\
    &
    \nn{\Bc^{x+iy}_{h,j}.f}
    \,\leq\, 
    C\,  e^{2y}  S^j(f)
    \quad.&
  \end{eqnarray*} 
\end{lem}

\begin{lem}
  \label{lem_Balpha_h0}
  Pour $1\leq h <(Q-2)/2$, on a un contr\^ole $L^2$ de $\Bc^\alpha_{h,0}$ :
  \begin{eqnarray*}
    &
    \forall [a,b]\subset[1 -h,\infty[\quad
    \exists C>0\quad
    \forall x\in [a,b]\quad
    \forall y\in\Rb\quad
    \forall\, f\in \Sc(N)&\\
    &
    \nd{\Bc^{x+iy}_{h,0}.f}
    \,\leq\, 
    C\,  e^{2 y}\nd{f}  
    \quad.&
  \end{eqnarray*} 
\end{lem}

La proposition~\ref{prop_inegalitemax_L2}
sera donc d\'emontr\'ee 
lorsque l'on aura prouv\'e ces deux lemmes.

\begin{proof}[du lemme~\ref{lem_Balpha_fcndaire}]
  Pour $f\in\Sc(N)$ et $t>0$, on a:
  $$
  \begin{array}{l}
    \displaystyle{  
      \partial_r\left[  f(t.)*\mu_r(t^{-1}n)\right]
      \,=\,
      \partial_r  \int_{r.S_1}
      f(t(t^{-1}n\,{n'}^{-1})d\mu_r(n')} \\
    \displaystyle{  
      \quad =\,
      \partial_r  \int_{r.S_1}
      f(n\,t{n'}^{-1})d\mu_r(n')
      = \partial_r  \int_{S_1}
      f(n\;rt{n'}^{-1})d\mu(n')}\\
    \displaystyle{
      \quad =\,
      t\partial_{r'}\left[\int_{S_1}
        f(n\;r'{n'}^{-1})d\mu(n')\right]_{|r'=rt}  
      \,=\,
      t\left[\partial_{r'}  
        \left[f*\mu_{r'}(n)\right]
      \right]_{|r'=rt}    }
    \quad,
  \end{array}
  $$
  puis par r\'ecurence, pour $h\geq 1$ :
  $$
  \partial_r^h \left[f(t.)*\mu_r(t^{-1}n) \right]
  \,=\,
  t^h\left[\partial_{r'}^h
    \left[f*\mu_{r'}(n)\right]\right]_{|r'=rt}
  \quad ,
  $$
  d'o\`u :
  \begin{eqnarray}
    {(B^\alpha_{h,j})}_t.f 
    &=& 
    \int_{0}^1 m^{\alpha+h}(r)
    r^{Q-1-2h+j} t^j [\partial_{r'}^j f*\mu_{r'}]_{|r'=rt} dr 
    \quad ,\nonumber\\
    \nn{{(B^\alpha_{h,j})}_t.f}^2
    &\leq&
    \int_{0}^1 \nn{ r^{j-\frac12}  t^j  
      [\partial_{r'}^j f*\mu_{r'}]_{|r'=rt} }^2dr 
    \int_{0}^1   
    \nn{ m^{\alpha+h}(r)  r^{Q-1-2h+\frac12}}^2  dr 
    \quad, \label{bhj2}
  \end{eqnarray}
  par H\"older.
  On pose $\alpha=x+iy$.
  La seconde  int\'egrale du membre  de droite de (\ref{bhj2}) 
  est  major\'ee par :
  \begin{eqnarray*}
    &&
    \int_{0}^1  \nn{ m^{\alpha+h}(r) r^{Q-1-2h+\frac12}}^2  dr 
    \,\leq\,
    \int_{0}^1  \frac{4{(1-r^2)}^{2(x+h-1)}}  
    {\nn{\Gamma(\alpha+h)}^2}
    r^{2Q-1-4h} dr \\
    &&\quad=\,
    \frac{4} {\nn{\Gamma(\alpha+h)}^2}
    \int_0^1    {(1-r')}^{2(x+h-1)}
    {r'}^{Q-1-2h}  \frac{dr'}2 
    \,=\,
    2   \frac{\Gamma(2x+2h-1)\Gamma(Q-2h)}
    {\Gamma(2x+Q-1)\nn{\Gamma(\alpha+h)}^2}
  \end{eqnarray*}
  gr\^ace \`a (\ref{eg_int_Gamma}) lorsque 
  $2x+2h-1,Q-2h>0$ .
  En utilisant l'estimation~(\ref{cqtit}),
  lorsque $h<Q/2$, 
  le terme pr\'ec\'edent
  est major\'e \`a une constante pr\`es 
  localement en $x>-h +\frac12$ et $x>-(Q-1)/2$
  uniform\'ement en $y\in\Rb$ par :
  $ e^{4 y}$.
  
  Pour la premi\`ere int\'egrale du membre de droite de
  l'in\'egalit\'e~(\ref{bhj2}),  
  effectuons le changement  de variable
  $r'=rt$ :
  $$
  \begin{array}{c}    
    \displaystyle{  
      \int_{0}^1    
      \nn{ r^{j-\frac12} t^j  [\partial_{r'}^j f*\mu_{r'}]_{|r'=rt} }^2dr
      \,=\,
      \int_0^t \nn{ {(\frac{r'}t)}^{j-\frac12} t^j
        \partial_{r'}^j f*\mu_{r'}}^2 \frac{dr'}t}\\
    \displaystyle{      
      \leq\,
      \int_0^\infty      
      \nn{ \partial_{r'}^j f*\mu_{r'}}^2 {r'}^{2j-1} dr'
      \,=\,
      S^j(f)^2}
    \quad .
  \end{array}
  $$
\end{proof}

Passons \`a la preuve du  lemme~\ref{lem_Balpha_h0};
elle repose sur le corollaire~\ref{cor_fcndec}.

\begin{proof}[du lemme~\ref{lem_Balpha_h0}]
  $B^\alpha_{h,0}$  est un op\'erateur
  de convolution  avec la fonction $G^\alpha_h$ donn\'ee par :
  $G^\alpha_h(n):= F^{\alpha+h}(n)\nn{n}^{-2h}$.
  En effet,  gr\^ace \`a la
  formule~(\ref{formule_changement_polaire}),
  le passage en coordonn\'ees  polaires donne :
  $$
  B^\alpha_{h,0} .f(n)
  \,=\,
  \int_{0}^1 m^{\alpha+h}(r) r^{Q-1-2h}  f*\mu_r(n) dr\\
  \,=\,
  f*G^\alpha_h(n) 
  \quad .
  $$
  On note $\alpha=x+iy$ et on a :
  $$
  \nn{G^\alpha_h(n)}          
  \,\leq\,
  \frac{\Gamma(x+h)}{\nn{\Gamma(\alpha+h)}}F^{x+h}(n)\nn{n}^{-2h}    
  \,  =\, 
  \frac{\Gamma(x+h)}{\nn{\Gamma(\alpha+h)}}G^{x+h}(n) 
  \,\leq \,C\,  
  e^{2y} G^x_h(n) 
  \quad ,
  $$
  le derni\`ere majoration \'etant due \`a~(\ref{cqtit}),
  d'o\`u  localement en $x$ :
  \begin{equation}
    \label{majbh0}
    \nn{ B^\alpha_{h,0}.f} \leq  \nn{f}*\nn{G^\alpha_h}
    \,\leq\,
    \;C\; e^{2y} B^x_{h,0}.\nn{f}
    \quad  .
  \end{equation}
  Pour $Q-1-2h>-1$,
  elle est int\'egrable:
  $$
  \int_N        G^x_h(n')        dn'
  \,=\,
  \int_{0}^1
  \frac{{(1-r^2)}^{\alpha+h-1}} {\Gamma(x+h)} r^{Q-1-2h} dr 
  \quad .
  $$
  Ainsi, localement en  $x$ tel  que $x+h-1>0$ et pour $h<Q/2$, 
  la  fonction $G^x_h$ v\'erifie
  les hypoth\`eses du corollaire~\ref{cor_fcndec};
  on  en d\'eduit que l'op\'erateur
  $B^x_{h,0}$  v\'erifie une  in\'egalit\'e maximale~$L^p$  pour tout
  $1<p\leq\infty$      en     particulier     $L^2$.       Avec     la
  majoration~(\ref{majbh0}),       on       en       d\'eduit    le  lemme~\ref{lem_Balpha_h0}.
\end{proof}

\subsection{Interpolation}
\label{subsec_interpolation}

On ach\`eve ici  la d\'emonstration du
th\'eor\`eme~\ref{thm_fcnd'aire_inegmax}.b).
On suppose comme dans les hypoth\`eses du th\'eor\`eme,
$h>1$ et que chaque fonction d'aire $S^j, j=1,\ldots,h$ 
v\'erifie une in\'egalit\'e~$L^2$.

\paragraph{Lin\'earisation.}
Fixons    momentan\'ement
$r(.):(N,dn)\mapsto  \Rb^+$  une  fonction  mesurable.   
Pour  $z\in \Cb$
dans la bande $0\leq \Re z \leq 1$,  
on d\'efinit l'op\'erateur
$T^z$ sur les fonctions simples de $N$ :
$$
T^z.f(n)
\,=\,
e^{z^2}A_{r(n)}^{\alpha(z)}.f(n) 
\quad ,
$$
o\`u on a not\'e $\alpha(z)=z+(1-z)x_1$ avec $0>x_1>1-h$ fix\'e.

On a la majoration ponctuelle de $T^z.f$ 
gr\^ace \`a la fonction maximale
sph\'erique :
$$
\forall \, f
\qquad\forall\, z\in\Cb\, ,
0\leq \Re z \leq 1
\qquad
\nn{T^z.f}
\,\leq\, 
e^{-{(\Im z)}^2}
\Ac^{\alpha(z)}.f
\quad.
$$
D'apr\`es les propositions~\ref{prop_inegalitemax_Lp}
et~\ref{prop_inegalitemax_L2},
l'op\'erateur~$T^z$ v\'erifie :
\begin{itemize}
\item lorsque $\Re z=1$, une in\'egalit\'e $L^p$ avec $1<p\leq \infty$,
\item lorsque $\Re z = 0$, une in\'egalit\'e $L^2$.
\end{itemize}

\paragraph{Interpolation.}
Gr\^ace \`a la  nouvelle   famille  d'op\'erateurs
$\{T^z\}_{0\leq \Re z \leq 1}$
est  une famille  analytique  d'op\'erateurs 
admissible  au sens  de~\cite{StW}.
On interpole en $z_0$ tel que $\alpha(z_0)=0$.
L'op\'erateur~$T^{z_0}$  v\'erifie donc 
une in\'egalit\'e~$L^q$ 
o\`u le param\`etre~$q$ est tel que :
\begin{equation}
  \label{egalite_1/q}
  \frac1q
  \,=\,
  \frac{1-z_0}{2}+\frac{z_0}p
  \quad \mbox{et}\quad
  1<p\leq\infty \quad.
\end{equation}
La constante de cette in\'egalit\'e $L^q$ 
ne d\'epend que de $N$ et
des  constantes des in\'egalit\'es  
obtenues pour~$\Ac^\alpha$ 
(propositions~\ref{prop_inegalitemax_Lp}
et~\ref{prop_inegalitemax_L2}). 
Elle est en particulier ind\'ependante du choix de~$r(.)$.

\paragraph{Fin de la d\'emonstration du th\'eor\`eme~\ref{thm_fcnd'aire_inegmax}.b)}
On peut donc ``repasser au supremum'':      
l'op\'erateur
$e^{z_0^2}\Ac^{\alpha(z_0)}$ 
v\'erifie  la m\^eme in\'egalit\'e~$L^q$,
et ce pour tout $q$ tel que~(\ref{egalite_1/q})
avec $z_0=\frac{x_1}{x_1-1}\in [0,(h-1)/h]$ (car $1-h<x_1<0$).
Par cons\'equent, 
la fonction maximale sph\'erique $\Ac$ v\'erifie
\`a une constante pr\`es une in\'egalit\'e $L^q$
pour $(2h)/(2h-1)<q\leq 2$.

D'apr\`es la partie a) d\'ej\`a d\'emontr\'ee,
la fonction maximale sph\'erique $\Ac$ satisfait \'egalement
une in\'egalit\'e $L^q$
pour $2\leq q\leq \infty$.
Le th\'eor\`eme~\ref{thm_fcnd'aire_inegmax} est ainsi compl\`etement
d\'emontr\'e.

\section{Fonctions d'aire pour un groupe de type H}

Le but de cette section est de d\'emontrer
le th\'eor\`eme \ref{thm_L2_fcnd'aire}.a).
Nous montrerons la partie b) correspondant \`a $N_{v,2}$
dans la section~\ref{sec_hatS_grlib}. 

Comme dans le chapitre~\ref{chapitregen2},
on note $\Omega$ l'ensemble
des fonctions sph\'eriques born\'ees de $N$ pour $O(v)$.
On pose pour $\omega\in\Omega$ :
$$
\hat{S}^j(\omega)
\,:=\,
\sqrt{\int_0^\infty 
  \nn{\partial_s^j< \mu_s,\omega>}^2
  s^{2j-1} ds}
\quad.
$$
\index{Notation!$\hat{S}^j(\omega)$}

Le th\'eor\`eme \ref{thm_L2_fcnd'aire}.a) sera d\'emontr\'e
une fois que l'on aura d\'emontr\'e 
les deux propositions suivantes :

\begin{prop}[\mathversion{bold}{$S^j$} et
  \mathversion{bold}{$\hat{S}^j$}]
  \label{prop_fcnd'aire_hat{S}}
  Pour $j\in\Nb-\{0\}$,
  s'il existe une constante $C>0$ telle que :
  $$
  \forall \omega\in \Omega\qquad
  \hat{S}^j(\omega)\,\leq\, C\quad,
  $$
  alors
  $$
  \forall f\in L^2(N)\qquad
  \nd{S(f)}\,\leq\,C\,\nd{f}
  \quad,
  $$
\end{prop}

Au cours de la preuve,
on utilisera une mesure spectrale $E$ sur $\Omega$,
qui nous am\`enera \`a estimer $\hat{S}^j$
sur $\Omega$ tout entier.
On pourrait donner une expression explicite de $E$,
gr\^ace \`a une formule de Plancherel non radiale,
``adapt\'ee aux fonctions sph\'eriques'' dans le sens de \cite[theorem~G]{pgelf}
dans le cas d'une ``vraie'' paire de Gelfand.
On peut par exemple choisir
la formule donn\'ee par les repr\'esentations de Bargmann ou de Schr\"odinger.
Nous utiliserons cette autre m\'ethode sur le groupe nilpotent libre \`a deux pas.
 
\begin{prop}[\mathversion{bold}{$\nd{\hat{S}^j}_\infty$}]
  \label{prop_maj_hat{S}}
  Si $v'\geq2$,
  il existe une constante $C>0$ telle que :
  $$
  \forall j=1,\ldots, v'-1\quad
  \forall \omega\in \Omega\qquad
  \hat{S}^j(\omega)\,\leq\, C\quad.
  $$
\end{prop}

\subsection{Fonction d'aire et mesure spectrale}

Cette sous-section est consacr\'ee \`a la preuve
de la proposition~\ref{prop_fcnd'aire_hat{S}}.
Nous aurons besoin de la proposition suivante :

\begin{prop}
  \label{prop_mes_sp_moy}
  On note End $L^2(N)$ 
  les endomorphismes continus 
  de l'espace de Hilbert $L^2(N)$. 
  
  Il existe une mesure spectrale $E$
  pour l'alg\`ebre commutative  ${L^1}^\natural$ 
  \`a valeur dans End~$L^2(N)$
  telle que:
  \begin{equation}
    \label{egalite_spectrale_E}
    \forall F\in{L^1}^\natural\quad
    \forall f,g\in L^2(N)\qquad
    <f*F,g >
    \,=\, 
    \int_{\Omega} <F,\omega> dE_{f,g}(\omega)
    \quad .  
  \end{equation}
  On a :
  \begin{equation}
    \label{inegalite_spectrale_E}
    \nd{\partial_s^j(f*\mu_s)}^2
    \,\leq\,
    \int_\Omega
    \nn{\partial_s^j<\mu_s,\omega>}^2
    dE_f(\omega)
    \quad.    
  \end{equation}
\end{prop}

Lorsqu'on admet la proposition ci-dessus,
la d\'emonstration de
la proposition~\ref{prop_fcnd'aire_hat{S}}
est ais\'ee.
\begin{proof}[de la proposition~\ref{prop_fcnd'aire_hat{S}}]
  Gr\^ace \`a Fubini, 
  puis \`a~(\ref{inegalite_spectrale_E}), 
  et enfin de nouveau par Fubini,
  on a :
  \begin{eqnarray*}
    \nd{S^j(f)}^2
    &=&
    \int_0^\infty \nd{\partial_s^jf*\mu_s}^2 s^{2j-1}ds\\  
    &\leq&
    \int_0^\infty 
    \int_\Omega
    \nn{\partial_s^j<\mu_s,\omega>}^2
    dE_f(\omega)
    s^{2j-1}ds
    \,=\,
    \int_\Omega \nn{\hat{S^j}(\omega)}^2 
    dE_f(\omega)
    \quad.
  \end{eqnarray*}
  Donc si $\hat{S^j}(\omega)$ est born\'e par $C$ 
  ind\'ependemment de $\omega\in\Omega$,
  alors, $\nd{S^j(f)}^2$ est born\'e par $C^2$ 
  multipli\'e par 
  $$
  \int_\Omega dE_f(\omega)
  \,\leq\, \nd{f}^2
  \quad.
  $$
\end{proof}

\begin{rem}
  Comme c'\'etait d\'ej\`a le cas dans
  \cite{nevo_simplegrI,nevo_simplegrII,margulis_nevo_stein_semisimplegr,nevo_stein_semisimplegr},
  nous avons besoin d'une estimation $L^\infty$ de  $\hat{S}^j$ sur tout le spectre $\Omega$.
Nous pourrions nous passer 
de l'estimations de $\hat{S}^j$ sur la partie $\Omega_B$ 
en utilisant une formule de Plancherel non-radiale adapt\'ee.
Dans le cas d'un groupe de type~H, ces deux m\'ethodes conduisent au m\^eme r\'esultat.
\end{rem}

\subsubsection{D\'emonstration de la
  proposition~\ref{prop_mes_sp_moy}}

On utilisera le lemme suivant,
qui construit une approximation de l'unit\'e radiale :

\begin{lem}
  \label{lem_fcn_supp_int}
  Il existe une fonction $\phi:N\mapsto [0,1]$,
  $C^\infty$,
  radiale,
  \`a support dans la boule unit\'e $B_1$,
  v\'erifiant $\int_N \phi(n) dn =1$.

  Pour une telle fonction $\phi$, 
  on d\'efinit alors pour $\eta>0$ 
  les fonctions $\phi_\eta$ par :
  $$
  \phi_\eta(n)
  \,:=\,
  \eta^{-Q}\phi(\eta^{-1}.n)
  \;,\quad n\in N
  \quad.
  $$
  Les fonctions $\phi_\eta,\eta>0$ 
  forment une approximation radiale de l'unit\'e
  sur~$N$.
\end{lem}
\begin{proof}[du lemme~\ref{lem_fcn_supp_int}]
  Il existe $\phi_1:N\mapsto [0,1]$
  une  fonction $C^\infty$,
  \`a support dans la boule unit\'e $B_1$,
  et telle que $\phi_1(0)=1$.
  On d\'efinit :
  $$
  \phi_2(n)
  \,=\,
  \int_{O(v)}
  \phi_1(k.n)
  dk
  \quad\mbox{puis}\quad
  \phi(n)
  \,=\,
  \frac{\phi_2(n)}
  {\int_N \phi_2(n)dn}
  \quad.
  $$
  La fonction $\phi$ ainsi d\'efinie  convient.
\end{proof}

\paragraph{D\'emontrons la premi\`ere partie de
  la proposition~\ref{prop_mes_sp_moy},}
c'est-\`a-dire
l'existence d'une mesure spectrale $E$
satisfaisant~(\ref{egalite_spectrale_E}).
Consid\'erons le morphisme d'alg\`ebre:
$$
\Pi: \left\{\begin{array}{lcl}
    L^1(N)&\longrightarrow& \mbox{End } L^2(N)\\
    F &\longmapsto&\{ \Pi(F):f\mapsto f*F\}
  \end{array}\right.
\quad.
$$
Notons  $C$  l'adh\'erence  de $\Pi {L^1}^\natural$ 
dans l'alg\`ebre  norm\'ee  End $L^2(N),\nd{.}$ par la norme des op\'erateurs.
Comme d'une part $\Pi$ est un morphisme continu d'alg\`ebres norm\'ees
et que d'autre part
d'apr\`es le
th\'eor\`eme~\ref{bigthm},
l'alg\`ebre de convolution ${L^1}^\natural$ est commutative,
l'ensemble $C$  est  une sous-alg\`ebre
commutative   de  End $ L^2(N),\nd{.}$.  
De   plus  l'alg\`ebre~$C$ est aussi normale :
\begin{equation}
  \label{egalite_adj_F*}
  {(\Pi(F))}^*
  \,=\,\Pi(F^*)
  \quad\mbox{o\`u}\quad
  F\in L^1(N) 
  \quad\mbox{en notant}\quad
  F^*(n)=\bar{F}(n^{-1})
  \quad.
\end{equation}
D'apr\`es les propri\'et\'es spectrales 
des $C^*$-alg\`ebre commutatives
(voir par  exemple  \cite{rudin}),
l'alg\`ebre $C$ admet un mesure spectrale 
$E^\Pi:Sp(C)\rightarrow \Pc$,
o\`u on a not\'e $Sp(C)$
le spectre de l'alg\`ebre $C$
et $\Pc$
l'ensemble  des projections continues de $L^2(N)$  :
$$
\forall\,T\in C\,,\qquad
T\,=\, \int_{Sp(C)}\widehat{T}dE^\Pi 
\quad.
$$
On d\'efinit l'application :
$$
\Pi': \; \left\{
  \begin{array}{r c l}
    Sp(C)
    &\longrightarrow&
    Sp({L^1}^\natural)\\
    \Lambda
    &\longmapsto&
    \Lambda\circ\Pi
  \end{array}
\right.  \quad .
$$
On  en  d\'eduit  une  mesure  spectrale $E$  du  spectre  de
${L^1}^\natural$,     
identifi\'e \`a $\Omega$ 
par  le th\'eor\`eme~\ref{thm_fcnsph_sp}
en posant :
$$
E\,=\,  
\Pi^*E^\Pi:   \;   
\left\{\begin{array}{rcl}  
    \Bc(\Omega)
    &\longrightarrow&
    \Pc\\
    B &\longmapsto& E^\Pi({\Pi'}^{-1}(B))
  \end{array}\right.
\quad,
$$
o\`u $\Bc(\Omega)$  d\'esigne  les bor\'eliens  sur $\Omega$.
Cette mesure spectrale v\'erifie bien
la propri\'et\'e~(\ref{egalite_spectrale_E}).

\paragraph{D\'emontrons la seconde partie de
  la proposition~\ref{prop_mes_sp_moy},}
c'est-\`a-dire l'in\'egalit\'e~(\ref{inegalite_spectrale_E}).
Fixons  une approximation radiale de l'unit\'e
$\phi_\eta,\eta>0$ 
comme dans le lemme~\ref{lem_fcn_supp_int}.

Nous allons appliquer 
la  formule~(\ref{egalite_spectrale_E}) 
\`a la fonction
$F_{s,\eta,j}^**F_{s,\eta,j}$
o\`u $F_{s,\eta,j}=\partial_s^j (\phi_\eta*\mu_s)$
et la fonction $F^*_{s,\eta,j}$ 
est donn\'e par (\ref{egalite_adj_F*}).
On v\'erifie facilement 
que la fonction $F_{s,\eta,j}$ est radiale 
car $\mu_s$ et $\phi_\eta$ le sont, 
puis qu'elle est int\'egrable :
\begin{eqnarray*}
  \phi_\eta*\mu_s(n)
  &=&
  \int_{s.S_1} \phi_\eta(ng^{-1})d\mu_s(g)
  \,=\,
  \int_{S_1} \phi_\eta(n\,s.g^{-1})d\mu(g)
  \quad,\\
  \partial_s^j(\phi_\eta*\mu_s)(n)
  &=&
  \int_{S_1} \partial_s^j\phi_\eta(n\,s.g^{-1})d\mu(g)
  \quad , \\
  \nn{\partial_s^j(\phi_\eta*\mu_s)}(n)
  &\leq&
  \left\{    \begin{array}{ll}
      0
      &\mbox{si}\, sg^{-1}\not \in \eta.B_1\;,\\
      \int_{S_1} \sup\nn{D^j\phi_\eta}1_{\eta.B_1}(n\,s.g^{-1})d\mu(g)
      &\mbox{si}\, sg^{-1}\in \eta.B_1\;.
    \end{array}\right.
\end{eqnarray*} 
On a finalement :
$$
\int_N\nn{F_{s,\eta,j}(n)}dn
\,=\,
\int_N\nn{\partial_s^j(\phi_\eta*\mu_s)}(n)dn
\,\leq\,
\eta^{-Q+j}\sup\nn{D\phi}\nn{\eta.B_1}
\,\leq\, \eta^j \nn{B_1}\, \sup\nn{D\phi}
\quad.
$$
d'o\`u $F_{s,\eta,j}\in {L^1}^\natural$,
puis $F_{s,\eta,j}^**F_{s,\eta,j}\in {L^1}^\natural$.
Appliquons (\ref{egalite_spectrale_E}) 
\`a cette derni\`ere fonction :
\begin{equation}
  \label{egalite_spectrale_appliquee}
  \forall\,f\in L^2(N)
  \qquad
  {<f*F_{s,\eta,j}^**F_{s,\eta,j},f>}_{L^2(N)}
  \,=\,
  \int_{\omega\in\Omega}
  <F_{s,\eta,j}^**F_{s,\eta,j}),\omega>
  dE_f(\omega)
  \quad,    
\end{equation}
Or on voit 
pour le membre de gauche de cette \'egalit\'e, 
d'apr\`es~(\ref{egalite_adj_F*}) :
\begin{eqnarray*}
  {<f*F_{s,\eta,j}^**F_{s,\eta,j},f>}
  &=&
  {<\Pi(F_{s,\eta,j}^**F_{s,\eta,j}).f,f>}\\
  &=&
  {<\Pi(F_{s,\eta,j}).f,\Pi(F_{s,\eta,j}).f>}
  \,=\,
  \nd{\Pi(F_{s,\eta,j}).f}^2\\
  &=&
  \nd{\partial_s(f*\phi_\eta*\mu_s)}^2
  \,=\,
  \nd{\phi_\eta*\partial_s(f*\mu_s)}^2
  \quad,
\end{eqnarray*}
et pour le membre de droite 
comme $<.,\omega>$ est un caract\`ere de ${L^1}^\natural$ :
$$
<F_{s,\eta,j}^**F_{s,\eta,j},\omega>
\,=\,
\nn{<F_{s,\eta,j},\omega>}^2
\,=\,
\nn{\partial_s^j<\mu_s,\omega>}^2\nn{<\phi_\eta,\omega>}^2
\quad.
$$
L'\'egalit\'e~(\ref{egalite_spectrale_appliquee}) 
devient donc :
$$
\nd{\phi_\eta*\partial_s^j(f*\mu_s)}^2
\,=\,
\int_{\Omega}
\nn{\partial_s^j<\mu_s,\omega>}^2\nn{<\phi_\eta,\omega>}^2
dE_f(\omega)
\quad.
$$
Maintenant comme la fonction $\omega$ est born\'ee par 1, 
et que l'int\'egrale de $\phi_\eta$ vaut 1,
on a $\nn{<\phi_\eta,\omega>}\leq 1$,
et le membre de droite de l'\'egalit\'e pr\'ec\'edente
est major\'ee par 
$$
\int_{\Omega}
\nn{\partial_s^j<\mu_s,\omega>}^2
dE_f(\omega)
\quad;
$$
et comme $\phi_\eta,\eta>0$ est une approximation de l'unit\'e,
le membre de droite tend vers 
$\nd{\partial_s^j(f*\mu_s)}^2$
lorsque $\eta$ tend vers 0.
On en d\'eduit la majoration~(\ref{inegalite_spectrale_E}). 

Ceci ach\`eve la d\'emonstration de la proposition~\ref{prop_mes_sp_moy}.

\subsection{Contr\^ole $L^\infty$ de $\hat{S}^j$
  dans le cas d'un groupe de type~H}

Cette sous-section est consacr\'ee
\`a la preuve de la proposition~\ref{prop_maj_hat{S}}
dans le cas d'un groupe de type~H.

Nous aurons besoin du lemme technique :

\begin{lem}[D\'eriv\'ee d'une fonction de $s^2$]
  \label{lem_derivee_fcn_s2}
  Soit $f$ est une fonction r\'eguli\`ere et $h\in \Nb$.
  On pose $g(s)=f(s^2)$.

  $g^{(h)}$ s'\'ecrit comme combinaisons lin\'eaires 
  de $s^{d(j,h)}f^{(h'+j)}(s^2)$, 
  \begin{itemize}
  \item o\`u $d(j,h)=2j$, sur ${0\leq j\leq h'}$, 
    si $h=2h'$, 
  \item o\`u $d(j,h)=2j+1$, sur ${0\leq j\leq h'-1}$,
    si $h=2h'-1$.
  \end{itemize}
  On remarque :
  \begin{equation}
    \label{egalite_d(j,h)+h}
    d(j,h)+h=2h'+2j
    \quad.
  \end{equation}
\end{lem}

Nous supposons dans cette section $v'\geq 2$.
Dans le cas d'un groupe de type~H,
comme $\Omega=\Omega_L\cup\Omega_B$,
la proposition~\ref{prop_maj_hat{S}}
est \'equivalente aux deux propositions suivantes :

\begin{prop}
  \label{prop_maj_hat{S}_Omega_B}
  Il existe une constante $C>0$ telle que :
  $$
  \forall h=1,\ldots, v'-1\quad
  \forall \omega\in \Omega_B\quad
  \hat{S}^h(\omega)\,\leq\, C\quad.
  $$
\end{prop}

\begin{prop}
  \label{prop_maj_hat{S}_Omega_L}
  Il existe une constante $C>0$ telle que :
  $$
  \forall h=1,\ldots, v'-1\quad
  \forall \omega\in \Omega_L\quad
  \hat{S}^h(\omega)\,\leq\, C\quad.
  $$
\end{prop}

\begin{proof}[de la proposition~\ref{prop_maj_hat{S}_Omega_B}]
  Soit $\omega=\Phi_r$. On a pour $n=(X,Z)\in N$ :
  \begin{eqnarray*}
    \omega(s.n)
    &=&
    \Phi_r(sX,s^2Z)
    \,=\,
    \fb {v'-1} {rs\nn{X}}
    \quad,\\
    \partial_s^h    \omega(s.n)
    &=&
    {(r\nn{X})}^h
    \fbs {v'-1} ^{(h)} (rs\nn{X})
    \quad,
  \end{eqnarray*}
  d'o\`u :
  $$
  \nn{\hat{S}^h(\omega)}^2
  \,\leq\,
  \int_{0}^\infty s^{2h-1}\int_{S_1}
  \nn{{(r\nn{X})}^h\fbs{v'-1}^{(h)}(rs\nn{X})}^2 d\mu(X,Z)ds 
  \quad;
  $$
  par le changement de variable 
  $s'= rs\nn{X}$,
  cette derni\`ere int\'egrale vaut :
  $$
  \int_{S_1}       \int_{0}^\infty       {s'}^{2h-1}
  \nn{\fbs {v'-1}^{(h)}(s')}^2  ds'  d\mu(X,Z) 
  \,=\,
  \nn{\mu}
  \int_{0}^\infty
  {s'}^{2h-1} \nn{\fbs{v'-1}^{(h)}(s')}^2 ds' 
  \quad ,
  $$
  qui est finie 
  d'apr\`es le lemme~\ref{lem_maj_int_fb},
  lorsque $-1<2h-1< 2(v'-1)$.
\end{proof}

\begin{proof}[de la proposition~\ref{prop_maj_hat{S}_Omega_L}]
  Soit $\omega=\Phi_{\zeta,l}$.
  Ici, on a $\omega(s.n)=\Phi_{s^2\zeta,l}(n)$;
  avec les notations du lemme~\ref{lem_derivee_fcn_s2},
  $\partial_s^h\omega(s.n)$ s'\'ecrit 
  comme une combinaison lin\'eaire de :
  $s^{d(j,h)}f^{(h'+j)}(s^2)$
  o\`u
  $$
  f(s)
  \,:=\,
  \Phi_{s\zeta,l}(n)
  \,=\,
  e^{is<\zeta,Z>}
  \flb l {v'-1}{\frac12 s\nn{\zeta}\nn{X}^2}
  \quad\mbox{et}\quad n=(X,Z)\in N
  \quad.
  $$
  Calculons les d\'eriv\'ees de cette derni\`ere fonction :
  $$
  \partial_s^{j}f(s)
  \,=\,
  \sum_{m=0}^{j} C_j^m   
  {<\zeta,Z>}^{j-m}
  {(\frac12 \nn{\zeta}\nn{X}^2)}^m
  \flbs l {v'-1}^{(m)}(\frac12 s\nn{\zeta}\nn{X}^2)\quad
  e^{is<\zeta,Z>}
  $$
  et donc le terme $\nn{\partial_s^{j}f(s)}$ est major\'e 
  \`a une constante (de $j$) pr\`es par :
  $$
  \nn{\zeta}^{j}
  \sum_{m=0}^{j}   
  \nn{Z}^{j-m}  \nn{X}^{2m}  
  \nn{\flbs{l}{v'-1}^{(m)}(\frac{s}2 \nn{\zeta} \nn{X}^2)}
  \quad .
  $$
  Gr\^ace \`a cette majoration,   
  l'expression $\nn{\hat{S}^h(\omega)}^2$ 
  est ainsi major\'ee 
  \`a une  constante (de $h$) pr\`es 
  par le maximum sur 
  $0\leq j\leq h/2$ et $0\leq m \leq h'+j$ de :
  \begin{eqnarray*}
    I(j,h,l,m)   
    :=
    \int_0^\infty\!\!
    {\left(\int_{S_1}\!\!
        s^{d(j,h)}
        \nn{\zeta}^{h'+j}
        \nn{Z}^{h'+j-m}  \nn{X}^{2m}  
        \nn{\flbs{l}{v'-1}^{(m)}(\frac{s^2}2 \nn{\zeta} \nn{X}^2)}
        d\mu(X,Z)    \right)}^2 s^{2h-1}ds\\
    =
    \int_0^\infty
    \nn{\zeta}^{2(h'+j)}
    {\left(\int_{S_1} \!\!\!\!
        \nn{Z}^{h'+j-m}  \nn{X}^{2m}  
        \nn{\flbs{l}{v'-1}^{(m)}(\frac{s^2}2 \nn{\zeta} \nn{X}^2)}
        d\mu(X,Z)    \right)}^2 s^{4(h'+j)-1}ds
    \quad ,
  \end{eqnarray*}
  gr\^ace \`a~(\ref{egalite_d(j,h)+h}).

  Utilisons l'expression de $\mu$ 
  donn\'ee dans la proposition~\ref{prop_expression_mu} 
  dans ce qui suit :
  \begin{eqnarray*}
    \int_{S_1} 
    \nn{Z}^{h'+j-m}  \nn{X}^{2m}  
    \nn{\flbs{l}{v'-1}^{(l)}(s')}
    d\mu(X,Z)
    \,\leq\,
    \int_{0}^1    r^{2m}
    \nn{\flbs{l}{v'-1}^{(m)}(s')}
    {(1-r^4)}^{\frac{w-2}2}r^{v-1}dr
    \quad,
  \end{eqnarray*}
  et donc par H\"older :
  \begin{eqnarray*}
    &&  {\left(\int_{S_1} \!\!\!\!
        \nn{Z}^{h'+j-m}  \nn{X}^{2m}  
        \nn{\flbs{l}{v'-1}^{(m)}(s')}
        d\mu(X,Z)    \right)}^2\\
    &&\quad\leq\,
    \int_{0}^1      {(1-r^4)}^{-\frac12}dr      \int_{0}^1
    \nn{\flbs{l}{v'-1}^{(m)}(s')}^2
    r^{2(2m+v-1)}{(1-r^4)}^{w-\frac32}dr
    \quad,
  \end{eqnarray*}

  L'int\'egrale $I(j,h,l,m)$  est donc major\'ee  \`a une
  constante pr\`es par :
  $$
  \int_{0}^\infty     \nn{\zeta}^{2(h'+j)}     \int_{0}^1
  \nn{\flbs{l}{v'-1}^{(m)}(\frac{s^2}2 \nn{\zeta} r^2)}^2
  r^{2(2m+v-1)}{(1-r^4)}^{w-\frac32}dr s^{4(j+h')-1} ds
  \quad;
  $$
  cette  derni\`ere int\'egrale est \'egale 
  par Fubini et le changement de variable
  $s'=\frac{s^2}2  \nn{\zeta}  r^2$
  \`a :
  \begin{eqnarray*}
    \int_{0}^1\int_{0}^\infty 
    \nn{\zeta}^{2(h'+j)-1} 
    \nn{\flbs{l}{v'-1}^{(m)}(s')}^2
    {\left(\frac{2s'}{\nn{\zeta}r^2}\right)}^{2(j+h')-1}
    \frac{ds'}{\nn{\zeta}r^2}r^{2(2m+v-1)}{(1-r^4)}^{w-\frac32}dr\\
    \; =\,
    \int_{0}^1r^{2(-2(j+h')+2m+v-1)}{(1-r^4)}^{w-\frac32}dr
    \int_{0}^\infty
    \nn{\flbs{l}{v'-1}^{(m)}(s')}^2
    {(2s')}^{2(j+h')-1}
    ds'\quad.
  \end{eqnarray*}
  La premi\`ere int\'egrale  ci-dessus est bien finie car 
  $-2(j+h')+2l+v-1\geq
  -2h+v-1>-\frac12$ 
  lorsque $h\leq v'-1$.
  D'apr\`es le lemme~\ref{lem_cq_Mark}
  la seconde est finie tant que $j+h' \leq v'-1$,
  donc tant que $h\leq v'-1$.
\end{proof}


\chapter{Transform\'ee de Fourier radiale pour \mathversion{bold}{$N_{v,2}$}}
\label{chapitrecalculfourier}

Dans ce chapitre, 
nous explicitons 
les fonctions sph\'eriques born\'ees
du groupe nilpotent libre \`a 2 pas
et \`a $v$ g\'en\'erateurs not\'e $N_{v,2}$,
et la mesure de Plancherel associ\'ee.

\paragraph{Notations.}
On convient dans cette section d'identifier par la base canonique
les \'el\'ements de $\Zc$ et $\Zc^*$ \`a des matrices antisym\'etriques.
On associe \`a un \'el\'ement
$\Lambda^*\in \overline{\Lc}$,
$v_0\in \Nb$,
et si $\Lambda^*\not=0$, 
$v_1\in\Nb$, 
le multi-indice d'entier $m\in \Nb^{v_1}$ 
et le $v_1$-uplet $(\lambda_1,\ldots,\lambda_{v_1})\in \Rb^{v_1}$ 
de la mani\`ere suivante :
\index{Notation!Param\`etres!$r^*,\Lambda^*,l,\epsilon$}
\begin{itemize}
\item l'entier $v_0$ est tel que 
  $\lambda^*_{v_0}> 0$ et $\lambda^*_{v_0+1}=0$ :
  c'est le nombre de $\lambda^*_i$ non nuls;
\item $v_1$ est le nombre de $\lambda^*_i$ non nuls distincts,
  et les $\lambda_j$ sont les $\lambda^*_i$ non nuls et distincts, ordonn\'es de
  fa\c con strictement d\'ecroissante :
  $$
  \{
  \lambda^*_1\,\geq\, \ldots \,\geq\, \lambda^*_{v_0}\,>\,0
  \}
  \,=\,
  \{
  \lambda_1\, > \, \ldots \, > \, \lambda_{v_1}\,>\,0
  \}
  \quad;
  $$
\item pour $j=1,\ldots,v_1$, 
  $m_j$ est le nombre de param\`etres $\lambda^*_i$ \'egaux \`a $\lambda_j$;
  on d\'efinit \'egalement :
  \begin{eqnarray*}
    m_0
    \,:=\,
    m'_0
    \,:=\,
    0
    \qquad\mbox{et}\qquad
    m'_j
    \,:=\,
    \sum_{i=1}^j m_i\; ,
    \quad
    j=1,\ldots v_1
    \quad;
  \end{eqnarray*}

  on a $m'_{v_1}:=m_1+\ldots+m_{v_1-1}+m_{v_1}=v_0$.
\end{itemize}
Si $\Lambda^*\not=0$, on peut donc mettre
la matrice antisym\'etrique  $D_2(\Lambda^*)$ 
sous la forme (pour les notations voir les sous-sections~\ref{subsec_reduction}
et~\ref{subsec_SpO}):
\index{Notation!Matrice antisym\'etrique!$D_2(\Lambda)$}
$$
D_2(\Lambda^*)
\,=\,
\left[
  \begin{array}{cccc}
    \lambda_1^* J &&&\\
    0 & \ddots & 0&\\
    &&\lambda_{v'}^* J&\\
    &&&(0)
  \end{array}\right]
\,=\,
\left[
  \begin{array}{ccc|c}
    \lambda_1 J_{m_1} &&&\\
    0 & \ddots & 0&\\
    &&\lambda_{v_1} J_{m_{v_1}}&\\
    \hline
    &&&0
  \end{array}\right]
\quad.
$$

On convient de noter :
\begin{itemize}
\item $\Qc$
  l'ensemble des $(r^*,\Lambda^*)\in \Rb^+\times\bar{\Lc}$
  v\'erifiant $r^*=0$ 
  si $2v_0=v$,
  \index{Notation!Ensemble de param\`etres!$\Qc$}
\item $\flbs n \alpha $ la fonction de Laguerre normalis\'ee 
  (voir section~\ref{app_lag}),
\item $\mbox{pr}_j$ la projection 
  sur l'espace vectoriel engendr\'e 
  par les  $2m_j$ vecteurs 
  $$
  X_{2i-1}\, ,\; X_{2i}\,,\quad
  m'_{j-1}\,<\,i\,\leq\, m'_j
  \quad.
  $$
  \index{Notation!$\mbox{pr}_j$}
\end{itemize}

Avec ces notations,
nous redonnons maintenant l'\'enonc\'e 
du th\'eor\`eme~\ref{thm_intro_fcnsph} :

\begin{thm_princ}[Fonctions sph\'eriques born\'ees pour 
  \mathversion{bold}{$N_{v,2}, O(v)$}]
  \label{thm_fsphO}
  \index{Fonction sph\'erique!sur $N_{v,2}$}
  Les param\`etres des fonctions sph\'eriques born\'ees 
  sont 
  $(r^*,\Lambda^*)\in\Qc$,
  puis le multi-indice $l\in \Nb ^{v_1}$ 
  si $\Lambda^*\not=0$, 
  $\emptyset$ sinon.

  Avec ces param\`etres, 
  les fonctions sph\'eriques born\'ees pour $N_{v,2}, O(v)$ 
  sont donn\'ees par :\\
  \textbf{Si} {\mathversion{bold}{$\Lambda^*\not=0$}}
  $$
  \phi^{r^*,\Lambda^*,l}(n)
  \,=\,
  \int_{O(v)}
  \Theta^{r^*,\Lambda^*,l}(k.n)
  dk,
  \quad n\in N_{v,2}
  \quad,
  $$
  \index{Notation!Fonction sph\'erique!$\phi^{r^*,\Lambda^*,l}$}
  \index{Notation!$\Theta^{r^*,\Lambda^*,l}$}
  o\`u $\Theta^{r^*,\Lambda^*,l}$ est la fonction donn\'ee 
  pour $n=\exp(X+A)\in N_{v,2}$ par :
  $$
  \Theta^{r^*,\Lambda^*,l}(n)
  \,=\,
  e^{i r^*<X_v^*,X>}
  e^{i<D_2(\Lambda^*),A>}
  \overset{v_1}{\underset{j=1} \Pi}
  \flb {l_j} {m_j-1} {\frac{\lambda_j}2 \nn{\pr {j} X }^2}
  \quad,
  $$
  o\`u on a not\'e $dk$ la mesure de Haar de masse 1 du groupe $O(v)$.\\
  \textbf{Si} {\mathversion{bold}{$\Lambda^*=0$}}
  \index{Notation!Fonction sph\'erique!$\phi^{r^*,0}$}
  $$
  \phi^{r^*,0}(n)
  \,=\,
  \int_{O(v)}
  e^{i r^*<X_v^*,k.X>}
  dk,
  \quad n=\exp(X+A)\in N_{v,2}
  \quad.
  $$
\end{thm_princ}

On trouve aussi les fonctions sph\'eriques pour $N_{v,2},SO(v)$ :
\begin{thm_princ}[Fonctions sph\'eriques born\'ees pour 
  \mathversion{bold}{$N_{v,2}, SO(v)$}]
  \label{thm_fsphSO}
  \index{Fonction sph\'erique!sur $N_{v,2}$}
  Les param\`etres des fonctions sph\'eriques born\'ees sont 
  $(r^*,\Lambda^*)\in\Qc$,
  ainsi que si $\Lambda^*\not=0$,
  $\epsilon=\pm1$ et le multi-indice $l\in \Nb ^{v_1}$. 

  Avec ces param\`etres, 
  les fonctions sph\'eriques born\'ees pour $N_{v,2}, SO(v)$ 
  sont donn\'ees par :\\
  \textbf{Si} {\mathversion{bold}{$\Lambda^*\not=0$}}
  $$
  \phi^{r^*,\Lambda^*,l,\epsilon}(n)
  \,=\,
  \int_{SO(v)}
  \Theta^{r^*,\Lambda^*,l,\epsilon}(k.n)
  dk,
  \quad n\in N_{v,2}
  \quad,
  $$
  \index{Notation!Fonction sph\'erique!$\phi^{r^*,\Lambda^*,l,\epsilon}$}
  \index{Notation!$\Theta^{r^*,\Lambda^*,l,\epsilon}$}
  o\`u on a not\'e $dk$ la mesure de Haar de masse 1 du groupe $SO(v)$,
  et $\Theta^{r^*,\Lambda^*,l,\epsilon}$ la fonction donn\'ee 
  pour $n=\exp(X+A)\in N_{v,2}$ par :
  $$
  \Theta^{r^*,\Lambda^*,l,\epsilon}(n)
  \,=\,
  e^{i r^*<X_v^*,X>}
  e^{i<D_2^\epsilon(\Lambda^*),A>}
  \overset{v_1}{\underset{j=1} \Pi}
  \flb {l_j} {m_j-1} {\frac{\lambda_j}2 \nn{\pr {j} X }^2}
  \quad.
  $$
  \textbf{Si} {\mathversion{bold}{$\Lambda^*=0$}}
  $$
  \phi^{r^*,0}(n)
  \,=\,
  \int_{SO(v)}
  e^{i r^*<X_v^*,k.X>}
  dk,
  \quad n=\exp(X+A)\in N_{v,2}
  \quad.
  $$
\end{thm_princ}

Pour les deux th\'eor\`emes pr\'ec\'edents,
dans le cas $\Lambda^*=0$, 
on retrouve les fonctions de Bessel 
``comme dans le cas des groupes de Heisenberg 
et des groupes de type H'' (lemme~\ref{lem_fcn_bessel_intsph}):
$$
\phi^{r^*,0}(\exp(X+A))
\,=\,
\fb {\frac{v-2}2}{r^*\nn{X}}
\quad ,
$$
o\`u $\nn{.}$ d\'esigne la norme euclidienne sur $\Vc$ 
pour la base canonique des g\'en\'erateurs.
Si de plus, $r^*=0$, 
on trouve la fonction constante $1$.

\section{Expression des fonctions sph\'eriques born\'ees}

Le but de cette section est de d\'emontrer
les deux th\'eor\`emes  pr\'ec\'edents,
c'est-\`a-dire de donner les expressions
des fonctions sph\'eriques born\'ees
de la paire de Guelfand 
$$
N=N_{v,2}\; , \;K=O(v)\;\mbox{ou}\; SO(v)\;,
\qquad\mbox{ou encore}\qquad
G=K\triangleleft N\; ,\; K
\quad.
$$

Nous convenons dans cette section
que le terme ``fonction sph\'erique'' 
signifie ``fonction sph\'erique born\'ee''
ou encore ``fonction sph\'erique de type positif''
(th\'eor\`eme~\ref{thm_fcnsph_rep}b).

Pour $\rho\in \hat{N}$,
on note :
\begin{itemize}
\item  $G_\rho$ et $K_\rho$ les groupes de stabilit\'e de la
  classe de $\rho$ sous $G$ et $K$ respectivement,
  \index{Notation!Groupe!$G_\rho$} 
  \index{Notation!Groupe!$K_\rho$} 
\item $\check{G_\rho}$ 
  l'ensemble des classe de repr\'esentations 
  $\nu\in\hat{G_\rho}$ telles que
  $\nu_{|N}$ est un multiple de $\rho$,
\item $\tilde{G}_\rho$ l'ensemble des classes 
  $\nu\in\check{G}_\rho$ 
  telles que 
  l'espace des vecteurs $K$-invariants 
  de la repr\'esentation $\Ind_{G_\rho}^G \nu$
  est de dimension 1.
  \index{Notation!Ensemble de classe de repr\'esentations!$\tilde{G}_\rho,\check{G}_\rho$}
\end{itemize}
Gr\^ace aux th\'eor\`emes des sous groupes et du nombre d'entrelacement,
on verra 
que $\nu\in\check{G}_\rho$  est dans $\tilde{G}_\rho$
si et seulement si l'espace de ses vecteurs $K_\rho$-invariants est une droite
(voir plus loin le lemme~\ref{lem_tildeGrho}).

Les preuves des th\'eor\`emes~\ref{thm_fsphO} et~\ref{thm_fsphSO} reposent 
sur les deux th\'eor\`emes et la proposition
qui suivent :
\begin{thm}
  [\mathversion{bold}{$\tilde{G}_\rho$}]
  \label{thm_tildeGrho}
  \begin{itemize}
  \item[a)] Fixons $\rho\in \hat{N}$ 
    et $(\Hc^\nu,\nu)$ un repr\'esentant d'une classe 
    de $\tilde{G_\rho}$.
    Soit $\vec{u}_\nu$ un  vecteur unitaire
    $K_\rho$-invariant pour $\nu$.
    On lui associe un vecteur unitaire $K$-invariant $f_{\vec{u}_\nu}$
    pour $\Ind_{G_\rho}^G \, \nu $;
    la fonction de type positif alors associ\'ee \`a 
    $\Ind_{G_\rho}^G \, \nu \in \hat{G}$ pour ce vecteur  $f_{\vec{u}_\nu}$
    est la fonction not\'ee $\phi^\nu$ donn\'ee par :
    \begin{equation}
      \label{fcnsphnu}
      \phi^\nu(n)
      \,=\,
      \int_{k \in K}
      {\big<\nu(I,k.n).\vec{u}_\nu\,,\,
        \vec{u}_\nu \big>}_{\Hc^\nu} dk,
      \quad 
      n\in N
      \quad ,  
    \end{equation}
    o\`u $dk$ d\'esigne la mesure de Haar de masse 1 du groupe compact 1. 
  \item[b)]   
    On obtient toutes les fonctions sph\'eriques born\'ees 
    comme les fonctions sph\'eriques de type positifs $\phi^\nu$
    lorsque $\rho$ parcourt 
    un ensemble de repr\'esentants de $\hat{N}/G$,
    et que $\nu$ parcourt un ensemble de repr\'esentants 
    de~$\tilde{G}_\rho$.
  \end{itemize}
\end{thm}

\begin{prop}[\mathversion{bold}{$\Nc^*/G$}]
  \label{prop_Nc^*/G}
  \begin{itemize}
  \item[a)] Si $K=O(v)$,
    les orbites  de l'action coadjointe de $N$ sous $G$ 
    sont param\`etr\'ees par 
    $(r^*,\Lambda^*)\in\Qc$.
    Elles sont not\'ees $O(r^*,\Lambda^*)$.
    Le repr\'esentant privil\'egi\'e de $O(r^*,\Lambda^*)$
    est $r^*X_v^*+D_2(\Lambda^*)$.
  \item[b)] Si $K=SO(v)$,
    les orbites  de l'action coadjointe de $N$ sous $G$ 
    sont param\`etr\'ees par 
    $(r^*,\Lambda^*)\in\Qc$ et $\epsilon=\pm1$.
    Elles sont not\'ees $O(r^*,\Lambda^*,\epsilon)$.
    Le repr\'esentant privil\'egi\'e de $O(r^*,\Lambda^*)$
    est $r^*X_v^*+D_2^\epsilon (\Lambda^*)$.
  \end{itemize}
\end{prop}
Evidemment, les orbites $O(r^*,\Lambda^*,1)$ et $O(r^*,\Lambda^*,-1)$ 
se confondent lorsque
la derni\`ere coordonn\'ee de $\Lambda^*$ est 
$\lambda_{v'}=0$. 

\begin{thm}[\mathversion{bold}{$\phi^\nu$}]
  \label{thm_fcnsphnu}
  Soient $(r^*,\Lambda^*)\in\Qc$ et $\epsilon=\pm1$ si
  $\Lambda^*\not=0$, sinon $\emptyset$.
  Soit $\rho\in T_f$ une repr\'esentation 
  associ\'ee au repr\'esentant privil\'egi\'e 
  $f=r^*X_v^*+D_2^\epsilon (\Lambda^*)$.
  \begin{itemize}
  \item[a)] Si $\Lambda^*=0$,
    $\tilde{G}_\rho$ est l'ensemble des classes 
    des repr\'esentations $\nu^{r^*,0}$.

    La fonction $\phi^\nu$
    associ\'ee \`a $\nu=\nu^{r^*,0}$
    par (\ref{fcnsphnu}) est $\phi^{r^*,0}$
    (donn\'ee dans les th\'eor\`emes~\ref{thm_fsphO} et~\ref{thm_fsphSO}).
  \item[b)] Si $\Lambda^*\not=0$,
    $\tilde{G}_\rho$ 
    contient les classes des repr\'esentations irr\'eductibles
    $\nu^{r^*,\Lambda^*,l,\epsilon}$, $l\in \Nb^{v_1}$; 
    ces derni\`eres poss\`edent une droite invariante sous $K_\rho$.

    La fonction $\phi^\nu$
    associ\'ee \`a $\nu=\nu^{r^*,\Lambda^*,l,\epsilon}$
    par (\ref{fcnsphnu}) est $\phi^{r^*,\Lambda^*,l,\epsilon}$
    (donn\'ee dans le th\'eor\`eme~\ref{thm_fsphSO}).
  \end{itemize}
\end{thm}

Lors de la d\'emonstration de ce dernier th\'eor\`eme
(sous-sections~\ref{subsec_thm_fcnsphnu.a}
et~\ref{subsec_thm_fcnsphnu.b}),
nous donnerons les expressions
des repr\'esentations $\nu^{r^*,0}$
et $\nu^{r^*,\Lambda^*,l,\epsilon}$.
\index{Notation!Repr\'esentation!$\nu^{r^*,0},\nu^{r^*,\Lambda^*,l,\epsilon}$}

\subsection{D\'emarche de la preuve}

Admettons les deux th\'eor\`emes et la proposition qui pr\'ec\`edent,
et conservons leurs notations.
Du  corollaire~\ref{cor_thm_kirillov}
et de la proposition~\ref{prop_Nc^*/G},
on en d\'eduit que
$\hat{N}/G$
est l'ensemble des classes des repr\'esentations
associ\'ees \`a 
$f=r^*X_v^*+D_2^\epsilon (\Lambda^*)$,
lorsque  les param\`etres $(r^*,\Lambda^*)$ parcourt $\Qc$,
et le param\`etre $\epsilon$ parcourt $\pm1$ si $K=SO(v)$ et vaut $\emptyset$ si $K=O(v)$.
D'apr\`es le th\'eor\`eme~\ref{thm_tildeGrho}.b),
et les premi\`eres parties du th\'eor\`eme~\ref{thm_fcnsphnu}.a) et b),
on en d\'eduit que toutes les classes des repr\'esentations de $G$ 
qui ont une droite invariante par $K$
sont obtenues en consid\'erant 
les repr\'esentations induites par 
$\nu^{r^*,0}$, et $\nu^{r^*,\Lambda^*,l,\epsilon}, l\in \Nb^{v_1}$,
lorsque les param\`etres $(r^*,\Lambda^*)$ parcourt $\Qc$ et 
si $K=SO(v)$ et $\Lambda^*\not=0$, $\epsilon$ \'egale $\pm1$;
si $K=O(v)$, alors $\epsilon$ vaut toujours $\emptyset$.

Et donc d'apr\`es le th\'eor\`eme~\ref{thm_tildeGrho}.a),
les fonctions sph\'eriques 
sont les fonctions $\phi^\nu$ donn\'ees par (\ref{fcnsphnu}),
lorsque $\nu$ parcourt
$\nu^{r^*,0}$ et $\nu^{r^*,\Lambda^*,l,\epsilon}$.
D'apr\`es les secondes parties du th\'eor\`eme~\ref{thm_fcnsphnu}.a) et b),
les fonctions sph\'eriques sont donc :
\begin{itemize}
\item les fonctions 
  $\phi^{r^*,0}, r^*\in \Rb^+$,
\item et les fonctions 
  $\phi^{r^*,\Lambda^*,l,\epsilon}$ 
  o\`u 
  $(r^*,\Lambda^*)\in\Qc$
  et $l\in \Nb^{v_1}$
  ainsi que $\epsilon=\pm1$,
  si $K=SO(v)$ et $\epsilon=\emptyset$ si $K=O(v)$.
\end{itemize}

Les th\'eor\`emes~\ref{thm_fsphO} et~\ref{thm_fsphSO}
seront donc d\'emontr\'es lorsque nous aurons prouv\'e
les th\'eor\`emes~\ref{thm_tildeGrho} et \ref{thm_fcnsphnu},
ainsi que la proposition~\ref{prop_Nc^*/G}.
Le reste de cette section est consacr\'e
\`a leurs d\'emonstrations.
Nous commen\c cons par  d\'emontrer 
le th\'eor\`eme~\ref{thm_tildeGrho}
et la proposition~\ref{prop_Nc^*/G}.
Ensuite, 
nous d\'ecrivons le stabilisateur
et le groupe quotient $\overline{N}=N/ \ker\rho$
pour une repr\'esentation $\rho\in T_{r^*X_v^*+D_2^\epsilon(\Lambda^*)}$.
Nous pourrons alors
d\'emontrer le th\'eor\`eme~\ref{thm_fcnsphnu}
gr\^ace au lemme suivant :

\begin{lem}
  [Repr\'esentation quotient\'ee par son noyau]
  \label{lem_rep_noyau}
  \index{Repr\'esentation!quotient\'ee par son noyau}
  \begin{itemize}
  \item[a)] Une repr\'esentation d'un groupe (comme
    tout morphisme)  passe au quotient par le noyau ou un sous-groupe du noyau.

    De plus, si la repr\'esentation du groupe est irr\'eductible, alors la
    repr\'esentation quotient\'ee l'est aussi.

  \item[b)] Soient $\nu_1$ et $\nu_2$ deux repr\'esentations d'un groupe $G$.

    Si elles sont \'equivalentes alors 
    leurs noyaux co\"\i ncident $\ker \nu_1 = \ker \nu_2 $,
    et les repr\'esentations pass\'ees au quotient \`a tout sous-groupe de leurs noyaux communs, sont \'equivalentes.

    R\'eciproquement, si leurs noyaux co\"\i ncident et
    les repr\'esentations pass\'ees au quotient sous leur noyau commun, sont \'equivalentes, alors les repr\'esentation $\nu_1$ et $\nu_2$ sont \'equivalentes.
  \end{itemize}
\end{lem}

\subsection{Ensemble $\tilde{G}_\rho$}

Le but de cette sous-section est de d\'emontrer 
le th\'eor\`eme~\ref{thm_tildeGrho}.

Fixons $\rho\in\hat{N}$.
Nous reprenons les notations 
qui lui ont \'et\'e associ\'ees
au d\'ebut de la sous section pr\'ec\'edente.
D'apr\`es la proposition~\ref{prop_stab}, 
le groupe de stabilit\'e de la
classe de $\rho$ sous $G$
se met sous la forme 
$G_\rho=K_\rho\triangleleft N$,
o\`u $K_\rho$ est le sous-groupe compact :
$$
K_\rho
\,=\,
\{ k\in K:\;
k.\rho = \rho
\} 
\,\subset\, K
\quad.
$$

\begin{lem}[\mathversion{bold}{$\tilde{G}_\rho$}]
  \label{lem_tildeGrho}
  Soit $\nu\in \hat{G}_\rho$. 
  La classe de $\nu$ est dans 
  $\tilde{G}_\rho$
  si et seulement si
  la repr\'esentation $\nu$ restreinte \`a $K_\rho$
  contient exactement une fois $1_{K_\rho}$.
\end{lem}

\begin{proof}[du lemme~\ref{lem_tildeGrho}]
  Par d\'efinition de $\tilde{G}_\rho$,  
  une repr\'esentation $\nu\in \hat{G}_\rho$ est dans 
  $\tilde{G}_\rho$
  si et seulement si
  la repr\'esentation 
  ${[\Ind_{G_\rho}^G \, \nu \,]}_{|K}$ 
  contient exactement une fois $1_K$.

  Comme $G_\rho=K_\rho\triangleleft N$,
  o\`u $K_\rho$ est un sous-groupe (ferm\'e) de $K$,
  la double classe
  $G_\rho \backslash G / K$ est triviale. 
  Appliquons le corollaire~\ref{cor_thmssgr}
  du th\'eor\`eme des sous-groupes
  aux sous-groupes :
  $$
  G_1\,=\,G_\rho,
  \qquad
  G_2\,=\,K,
  \qquad
  G_1\cap G_2 =G_\rho \cap K= K_\rho,
  \quad .
  $$
  On obtient :
  $$
  \forall \nu\in\hat{G_\rho},
  \qquad
  {[\Ind_{G_\rho}^G \, \nu \,]}_{|K}
  \,= \, 
  \Ind_{K_\rho}^K (\nu_{|K_\rho}) 
  \quad.
  $$
  Nous appliquons alors le lemme~\ref{lem_thmnbentr} 
  aux groupes compacts 
  $K_\rho\subset K$ 
  et \`a la repr\'esentation $\gamma = \nu_{|K_\rho}$.
\end{proof}

Fixons $(\Hc^\nu,\nu) \in \tilde{G}_\rho$.
D'apr\`es le lemme~\ref{lem_tildeGrho},
le sous-espace vectoriel de $\Hc^\nu$ 
des vecteurs $K$-invariants est une droite 
$\Cb \vec{u}_\nu$,
$\vec{u}=\vec{u}_\nu$ \'etant l'un de ses deux vecteurs unitaires.  
On cherche \`a en d\'eduire un vecteur $K$-fixe
de la repr\'esentation~$  (\Hc^\Pi,\Pi)=  \Ind_{G_\rho}^G \, \nu$.

L'espace  $\Hc^\Pi$ est l'ensemble des
fonctions $f:G\rightarrow \Hc^\nu$ telles que :
\begin{eqnarray*}
  &1.&  \forall g_\rho \in G_\rho, g\in G \quad 
  f(g_\rho g)\,=\, \nu (g_\rho) f(g) 
  \quad ,\\
  &2.& \dot{g}  \rightarrow  \nd{f(g)}_{\Hc^\nu}\in L^2(G/G_\rho)
  \quad.  
\end{eqnarray*}
Faisons un petit apart\'e sur les mesures choisies :
\begin{itemize}
\item  sur $G=K\triangleleft N$ : la mesure de Haar $dg=dkdn$,
\item sur $G_\rho=K_\rho\triangleleft N$ :
  la mesure de Haar  $d\dot{g}=dk_\rho dn$, 
  o\`u $dk_\rho$ est la mesure de Haar normalis\'ee sur le groupe compact $K_\rho$;
\item sur $G/G_\rho\sim K/K_\rho$ :
  la mesure $d\dot{k}$
  identifi\'ee \`a la mesure de masse 1  
  sur $K/K_\rho$ invariante par translation sous $K$. 
\end{itemize}
La repr\'esentation $\Pi$ est donn\'ee par :
$$
\forall g,g'\in G,\;
f\in \Hc^\Pi
\qquad
\Pi(g).f(g')\,=\, f(g'g)
\quad .
$$

\begin{lem}
  Soit la fonction  $f$ sur $G$ donn\'ee par :
  $f(k,n)=\nu(I,n).\vec{u}$ pour $(k,n)\in G$.\\
  Le vecteur $f_{\vec{u}_\nu}=f\in \Hc^\Pi$ est $K$-invariant et unitaire.
\end{lem}

\begin{proof}
  Montrons que la fonction $f$ est dans l'espace $\Hc^\Pi$.\\
  Pour $g=(k,n)\in G$ 
  et $g_\rho =(k_\rho ,n_\rho)\in G_\rho$, 
  on a 
  $g_\rho \, g= (k_\rho \, k,n_\rho \, k_\rho .n)$
  et donc par d\'efinition de $f$ :
  $f(g_\rho \, g)=\nu(I,n_\rho \, k_\rho .n).\vec{u}$;
  or $(I,n_\rho \, k_\rho .n)=(k_\rho ,n_\rho) \,(I,n)\, (k_\rho^{-1},0)$.
  Comme $\nu$ est un morphisme, on a :
  \begin{eqnarray*}
    f(g_\rho \, g)
    &=& 
    \nu\left((k_\rho ,n_\rho) \,(I,n)\, (k_\rho^{-1},0) \right).\vec{u}
    \,=\,
    \nu (k_\rho ,n_\rho) \,\nu (I,n) \,\nu (k_\rho^{-1},0).\vec{u}\\
    &=&
    \nu (k_\rho ,n_\rho) \,\nu (I,n).\vec{u}
  \end{eqnarray*}
  car $\vec{u}\in\Hc^\nu$ est invariant sous $K_\rho$.
  On obtient donc $f(g_\rho \, g)=\nu (g_\rho). f(g)$,
  d'o\`u $f\in\Hc^\Pi$.

  Montrons que $f$ est un vecteur $K$-invariant.\\
  Pour $g=(k',n')\in G$ et $k\in K$, on a:
  $$
  \Pi(k,0).f(g')
  \,=\,
  f\left( g'\, (k,0) \right)
  \,=\,
  f(k'k,n')
  =
  \nu(I,n').\vec{u}
  \,=\,
  f(g')\quad ;
  $$
  il est aussi unitaire :    
  \begin{eqnarray*}
    \nd{f}^2
    &=&
    \int_{G/G_\rho}
    \nd{f(g)}^2_{\Hc^\nu} d\dot{g}
    \,=\,
    \int_{K/K_\rho}
    \nd{f(k,0)}^2_{\Hc^\nu} d\dot{k}\\
    &=&\int_{K/K_\rho}
    \nd{\nu(I,0).\vec{u}}^2_{\Hc^\nu}d\dot{k}=
    \int_{K/K_\rho}
    1 d\dot{k}
    \,=\,
    1 \quad ;
  \end{eqnarray*}
  car $d\dot{k}$ est de masse 1.
\end{proof}

On obtient donc la fonction sph\'erique, 
comme fonction de type positif
associ\'ee \`a la repr\'esentation $\Pi$, 
que l'on note
$\phi^\nu$ : 
$$
\phi^\nu(g)
\,=\,
{\big<\Pi(g).f,f\big>}_{\Hc^\Pi}
\,=\,
\int_{G/G_\rho}
{\big< \Pi(g).f(g')\,,\,f(g') \big>}_{\Hc^\nu} d\dot{g'}
\quad .
$$
Or d'apr\`es les expressions de $\Pi$ et $f$, 
pour $g=(k,n),g'=(k',n')\in G$,
on a :
\begin{eqnarray*}
  \Pi(g).f(g')
  &=&
  f(g'\,g)
  \,=\,
  f(k'\,k,n'\, k'.n)\\
  &=&
  \nu(I,n'\, k'.n).\vec{u}
  \,=\, \nu(I,n')\nu(I,k'.n).\vec{u}
  \quad,  
\end{eqnarray*}
puis :
\begin{eqnarray*}
  {\big< \Pi(g).f(g'),f(g') \big>}_{\Hc^\nu}
  &=&
  {\big<\nu(I,n')\nu(I,k'.n).\vec{u}\,,\,
    \nu(I,n')\vec{u} \big>}_{\Hc^\nu}\\
  &=&
  {\big<\nu(I,k'.n).\vec{u}\,,\,
    \vec{u} \big>}_{\Hc^\nu}
  \quad . \\
\end{eqnarray*}
On obtient donc  l'expression~(\ref{fcnsphnu}) :
$$
\phi^\nu(g)
\,=\,
\int_{K/K_\rho}
{\big<\nu(I,k'.n).\vec{u}\,,\,
  \vec{u} \big>}_{\Hc^\nu}
d\dot{k'}
\,=\,
\int_{K}
{\big<\nu(I,k.n).\vec{u}\,,\,
  \vec{u} \big>}_{\Hc^\nu} dk \quad ,
$$
d'apr\`es le choix de la mesure sur l'espace homog\`ene 
$K/K_\rho$ et la $K$-invariance de $\vec{u}$.

Ceci ach\`eve la d\'emonstration 
du th\'eor\`eme \ref{thm_tildeGrho}.a).

Maintenant, 
d\'emontrons le th\'eor\`eme \ref{thm_tildeGrho}.b).
Lorsque $\rho$ parcourt 
un ensemble de repr\'esentants de $\hat{N}/G$ 
et lorsque $\nu$ parcourt 
un ensemble de repr\'esentants de $\tilde{G}_\rho$, 
on obtient toutes les repr\'esentations irr\'eductibles
$\Pi=\Ind_{G_\rho}^G \, \nu$ qui ont une droite $K$-fixe 
d'apr\`es  le th\'eor\`eme~\ref{thm_mackey};
d'apr\`es les th\'eor\`emes~\ref{thm_rep_ssesp} 
et~\ref{thm_fcnsph_rep},
on obtient donc toutes les fonctions sph\'eriques born\'ees
en consid\'erant les fonctions sph\'eriques de type positif
associ\'ees \`a $\Pi$.

\subsection{Description de $\Nc^*/G$}
\label{subsecrep}

Nous d\'emontrons ici la proposition~\ref{prop_Nc^*/G}.
On garde les notations de la sous-section~\ref{subsec_kirillov}.
\begin{prop}
  [Coad de \mathversion{bold}{$G$}]
  \label{prop_coad}
  Soit $g=(k,n)\in G$ avec $n=\exp(X+A)\in N$.
  On a :
  $$
  \forall f=X^*+A^*\in \Nc^*,
  \qquad
  \Coad.g(f)
  \,=\,
  k.X^*+k.A^*-(k.A^*).X
  \quad .
  $$
\end{prop}
\begin{proof}[de la proposition~\ref{prop_coad}]
  Gardons les notations de la proposition.
  Pour tout $X'+A'\in\Nc$ on a :
  \begin{eqnarray*}
    &&g\exp( t(X'+A') g^{-1})
    \,=\,
    (k,n k.\exp(t (X'+A')))(k^{-1},k^{-1}n^{-1})\\
    &&\quad=\,
    (I, n k.\exp (t (X'+A'))\,  k k^{-1}n^{-1})
    \,=\,
    (I, n \exp(t (k.X'+k.A')) n^{-1})\quad.
  \end{eqnarray*}
  On en d\'eduit pour $X'+A'\in\Nc$ :
  $$
  \Ad.g(X'+A') 
  \,=\,
  \Ad.n (k.X'+k.A')
  \,=\,
  k.X'+k.A'
  +[X, k.X']
  \quad .
  $$
  puis :
  \begin{eqnarray*}
    &&\Coad.g(f)(X'+A') 
    \,=\,
    f\left(\Ad.g^{-1}(X'+A') \right)\\
    &&\quad =\,
    <X^*,k^{-1}.X'>
    +
    <A^*,k^{-1}.A'>
    -
    <A^*,[k^{-1}.X,k^{-1}.X']>
    \quad .
  \end{eqnarray*}
  Or on a par d\'efinition de l'action de $K$ sur $\Nc^*$ :
  $$
  <X^*,k^{-1}.X'> 
  \,=\,     
  <k.X^*,X'>
  \quad\mbox{et}\quad 
  <A^*,k^{-1}.A'>
  \,=\,     
  <k.A^*,A'>
  \quad;
  $$
  et aussi :
  \begin{eqnarray*}
    <A^*,[k^{-1}.X,k^{-1}.X']>
    &=&
    <A^*,k^{-1}.[X,X']>
    \,=\,
    <k.A^*,[X,X']>\\
    &=&
    <k.A^*.X,X'>\quad.
  \end{eqnarray*}
\end{proof}

Cette proposition nous permet 
de d\'eterminer toutes les orbites de $\Nc^*/G$
dans les deux cas $K=O(v)$ et $SO(v)$.

Fixons une orbite $O$.
Toutes les formes lin\'eaires
$f=X^*_f+A^*_f\in O$ sont telles que 
les matrices antisym\'etriques $A^*_f$ sont
toutes orthogonalement semblables
et nous leur associons
$\Lambda^*=(\lambda^*_1,\ldots,\lambda^*_{v'})\in\overline{\Lc}$
(proposition~\ref{prop_Ac_v/O(v)}).

On note $v_0$ la dimension des sous espaces isotropes maximals 
pour $\omega_{A^*_f,r}$ o\`u $A^*_f\in O$;
c'est aussi le nombre de $\lambda^*_i$ non nuls.

Pour chaque $f=X^*_f+A^*_f\in O$, nous choisissons alors :
\begin{enumerate}
\item Distingons les cas :
  \begin{itemize}
  \item si $K=O(v)$,
    $k\in K$  tel que la matrice antism\'etrique  
    $k.A_f^*=D_2(\Lambda^*)$ soit diagonalis\'ee par bloc 2-2; 
  \item si $K=SO(v)$, $k\in K$ et $\epsilon=\pm1$ tels que la matrice antism\'etrique  
    $k.A_f^*=D_2^\epsilon(\Lambda^*)$ soit diagonalis\'ee par bloc 2-2
    avec $\Lambda^*\in\overline{\Lc}$
    (proposition~\ref{prop_Ac_v/SO(v)}); 
  \end{itemize}
\item $X\in \Vc$ est
  tel que $(k.A_f^*).X\in \Vc^*$ soit
  \'egale \`a la projection orthogonale $X^*_0$ de $k.X_f^*\in \Vc^*$ 
  sur $\Im k.A^*_f=\Im D_2(\Lambda^*)$;
  en particulier, $X^*_0=0$ si $\Im  D_2(\Lambda^*)$, 
  c'est-\`a-dire 
  si $v=2v'$ et $\Lambda^*=(\lambda^*_1, \ldots, \lambda^*_{v'})$,
  avec aucun $\lambda_i$ nul;
\item $k'\in K$ qui laisse stable $\Im D_2(\Lambda^*)$,
  donc \'egalement ${(\Im D_2(\Lambda^*))}^\perp$,
  tel que $k'.X^*_0=x^*X^*_v, x^*\in \Rb$.\\
  Si $K=O(v)$ ou $\dim {(\Im D_2(\Lambda^*))}^\perp>1$, 
  on peut supposer $x^*=r^*\geq0$.\\
  Si $x^*<0$, $K=SO(v)$ et $\dim {(\Im D_2(\Lambda^*))}^\perp=1$ c'est-\`a-dire $v=2v_0+1$,
  on pose $r^*=-x^*$ et :
  $$
  k''
  \,=\,
  \left[
    \begin{array}{c|c}
      \Id_{2v'-2}&0\\
      \hline
      0&-\Id_2
    \end{array}
  \right]\quad.
  $$
\end{enumerate}
On obtient :
\begin{itemize}
\item si $K=O(v)$, $(k'k, \exp X).f= r^*X^*_v+D_2(\Lambda^*)$,
\item  si $K=SO(v)$ et $2v_0+1<v$, $(k'k, \exp X).f= r^*X^*_v+D_2^\epsilon(\Lambda^*)$,
\item  si $K=SO(v)$ et $2v_0+1=v$, $(k'k, \exp X).f= r^*X^*_v+D_2^{-\epsilon}(\Lambda^*)$,
\end{itemize}

La proposition~\ref{prop_Nc^*/G}
est donc d\'emontr\'ee.

\subsection{Stablisateur de $\rho$}
\label{subsec_stab}

Le but de cette sous-section 
est de d\'ecrire le stabilisateur 
$K_\rho$ d'une repr\'esentation  $\rho$ 
associ\'ee \`a $f=r^*X_v+D^\epsilon_2(\Lambda^*)\in \Nc^*$.
Rappelons (proposition~\ref{prop_stab})
que c'est le stabilisateur l'orbite coadjointe dans $K$ de $f$.

\begin{prop}[\mathversion{bold}{$K_\rho$}]
  \label{prop_Krho}
  \index{Notation!Groupe!$K_\rho$} 
  \index{Notation!Groupe!$K_1,K_2$} 

  Soit $(r^*,\Lambda^*)\in\Qc$. 
  $v_0$ d\'esigne le nombre de $\lambda^*_i$ non nuls, 
  o\`u $\Lambda^*=(\lambda^*_1,\ldots,\lambda^*_{v'})$; 
  On note 
  $\tilde{\Lambda}^*=(\lambda^*_1,\ldots,\lambda^*_{v_0})\in \Rb^{v_0}$.
  Soit  $\rho$ une repr\'esentation 
  associ\'ee \`a $f=r^*X_v+D^\epsilon_2(\Lambda^*)\in \Nc^*$,
  o\`u $\epsilon=\emptyset$ si $K=O(v)$,
  et $\epsilon=\pm1$ si $K=SO(v)$.

  Si $\Lambda^*=0$, alors le groupe $K_\rho$ est le sous-groupe de
  $O(v)$ qui stabilise $r^*X^*_v$.

  Si $\Lambda^*\not=0$, 
  alors le groupe $K_\rho$ est le produit direct $K_1\times K_2$, 
  o\`u 
  \begin{enumerate}
  \item le groupe $K_1$ est le groupe form\'e des \'el\'ements $k_1\in SO(v)$ de la forme :
    $$
    k_1 \,=\,
    \left[ \begin{array}{cc}
        \tilde{k}_1&0\\
        0& \Id
      \end{array}\right]
    \quad;
    $$
    tels que $\tilde{k}_1\in SO(2v_0)$ commute (matriciellement)
    avec $D_2(\tilde{\Lambda}^*)\in \Ac_{2v_0}$; 
  \item le groupe $K_2$ est le groupe form\'e des \'el\'ements $k_2\in K$ de la forme :
    $$
    k_2 \,=\,
    \left[ \begin{array}{cc}
        \Id&0\\
        0& \tilde{k}_2
      \end{array}\right]
    \quad,
    $$
    tels que $\tilde{k}_2\in O(v-2v_0-1)$ laisse stable $r^*X^*_v$.
  \end{enumerate}
\end{prop}

Dans la d\'emonstration de cette proposition,
on reprend les notations de son \'enonc\'e 
et on convient de noter $A^*=D_2^\epsilon (\Lambda^*)$
et $X^*=r^*X^*_v$.

On aura besoin du lemme suivant :

\begin{lem}
  \label{lem_Krho}
  \index{Notation!Groupe!$K_\rho$} 
  $$
  K_\rho
  \,=\, \{ k\in K\; :\; 
  kA^*=A^*k
  \quad\mbox{et}\quad
  kX^*=X^* \}
  \quad.
  $$  
\end{lem}

\begin{proof}[du lemme~\ref{lem_Krho}]
  D'apr\`es la proposition~\ref{prop_coad},
  pour $ f=X^*+A^*\in \Nc^*$, on~a~:
  $$
  \forall k\in K \quad
  \Coad.k(f)=k.X^*+k.A^*
  \quad\mbox{et}\quad
  N.f=X^*+\Im A^*
  \quad.      
  $$

  Fixons momentan\'ement $k\in K_\rho$.
  On a $\Coad.k(f)\in N.f$.
  Comme la d\'ecomposition $\Nc^*=\Vc^*\oplus \Zc^*$ 
  est en somme directe, on a :
  $$
  \mbox{(1)}\;
  k.A^*=A^*
  \qquad\mbox{et}\qquad
  \mbox{(2)}\;
  k.X^*\in X^*+\Im A^*
  \quad.
  $$
  La condition (1) implique 
  que les matrices $k$ et $A^*$ commutent. 
  Donc en particulier, 
  comme $X^*\in {\Im A^*}^\perp$, 
  $kX^*$ puis $kX^* - X^*$ sont aussi dans ${\Im A^*}^\perp$. 
  Or la condition (2)
  implique $kX^* - X^* \in \Im A^*$ :
  ce vecteur est donc nul.
  Ainsi $k$ commute avec $A^*$ et stabilise $X^*$.

  R\'eciproquement, si $k$ commute avec $A^*$ et stabilise $X^*$,
  alors on~a $\Coad.k(f)\in N.f$ et $k\in K_\rho$.  
\end{proof}

\begin{proof}[de la proposition~\ref{prop_Krho}]
  Lorsque $\Lambda^*=0$, $K_\rho$ est le fixateur dans $K$ 
  du vecteur $X^*\in \Vc^*\sim \Rb^v$.
  La premi\`ere partie de la proposition~\ref{prop_Krho} 
  est donc d\'emontr\'ee.

  D\'emontrons la seconde partie.
  On adopte les notations de  la section~\ref{sec_app_matrice_antisym},
  ainsi que $\tilde{\epsilon}=\epsilon$ si $K=SO(v)$ et $v_0=v'$, 
  et 1 sinon.
  Nous sommes dans le cas $\Lambda^*\not =0$.
  $A^*$ se met sous la forme:
  $$
  A^*
  \,=\,
  \left[\begin{array}{c|c}
      D_2(\tilde{\Lambda}^*)&0\\
      \hline  0&0
    \end{array}\right]
  \quad\mbox{avec}\quad
  D_2(\tilde{\Lambda}^*)
  \,=\,
  \left[\begin{array}{ccc}
      \lambda_1 J_{m_1}&0&0\\
      &\ddots&\\
      0&0&\lambda_{v_1} J^{\tilde{\epsilon}}_{m_{v_1}}
    \end{array}\right]
  \quad,
  $$
  o\`u on a not\'e  $\lambda_1>\lambda_2>\ldots>\lambda_{v_1}>0$ les
  param\`etres $\lambda^*_i$ not\'es de mani\`ere distincte,
  et $m_i$ le nombre des $\lambda^*_i=\lambda_i$.
  On convient \'egalement :
  \begin{eqnarray*}
    m_0:=m'_0:=0\; ,\qquad
    \mbox{et}\qquad
    m'_j=\sum_{i=1}^j m_i\; ,
    \quad
    j=1,\ldots v_1
    \quad.
  \end{eqnarray*}

  Soit $k\in K_\rho$. 
  L'image de $A^*$ est l'espace vectoriel 
  engendr\'e par $X_1,\ldots,X_{2v_0}$;
  le noyau de $A^*$ est l'espace vectoriel 
  engendr\'e par $X_{2v_0+1},\ldots,X_v$. 
  Comme les matrices $k$ et $A^*$ commutent,
  l'image et le noyau de $A^*$ sont stables par $k$;
  la matrice~$k$ peut donc s'\'ecrire sous la forme :
  $$
  k \,=\,
  \left[ \begin{array}{cc}
      \tilde{k}_1&0\\
      0& \tilde{k}_2
    \end{array}\right]
  \quad\mbox{avec}\; \tilde{k}_1\in O(2v_0)
  \;\mbox{et}\;
  \tilde{k}_2\in O(v-2v_0-1)
  \quad;
  $$
  de plus 
  la matrice  $\tilde{k}_1$ commute avec $D_2(\tilde{\Lambda}^*)$, 
  et la matrice  $\tilde{k}_2$ stabilise le vecteur $X^*$ 
  car $k$ stabilise le vecteur $X^*\in \ker A^*$.
  Maintenant,
  les espaces propres pour $A^*$ 
  sont les espaces vectoriels engendr\'es par les vecteurs 
  $X_{2i-1},X_{2i},m'_{j-1}<i\leq m_j$ avec $j=1,\ldots v_1$.
  Comme les matrices $\tilde{k_1}$ et $D_2(\tilde{\Lambda}^*)$
  commutent,
  ces sous espace sont propres pour $\tilde{k_1}$;
  on peut donc \'ecrire $\tilde{k_1}$
  avec des blocs 
  ${[\tilde{k}_1]}_j \in O(m_j)$, 
  $i=1,\ldots ,v_1$
  sur la diagonale;
  de plus,
  vu la forme de $D_2(\tilde{\Lambda}^*)$,
  chaque bloc ${[\tilde{k}_1]}_j,$ commute avec $J_{m_j}$ pour
  $j<v_1$,
  ou avec $\tilde{\epsilon}J_{m_{v_1}}$ pour $j=v_1$.

  Appliquons la proposition~\ref{prop_matrice_ortho_commutant_J}
  \`a chaque bloc de $\tilde{k}_1$ :
  $\det {[\tilde{k}_1]}_j =1$.
  Avec la d\'ecomposition de 
  $\tilde{k}_1$ en bloc matriciel, 
  on en d\'eduit : 
  $\det\tilde{k}_1=
  \Pi_j
  \det {[\tilde{k}_1]}_j=1$.

  On peut donc \'ecrire $k$ 
  comme le produit $k=k_1k_2=k_2k_1$
  avec $k_i\in K_i,i=1,2$ 
  pour les groupes $K_1,K_2$ 
  donn\'es dans la proposition~\ref{prop_Krho}.

  R\'eciproquement, toute matrice $k\in K_\rho$ s'\'ecrit sous cette forme-l\`a.
\end{proof}

On peut davantage d\'ecrire le sous-groupe~$K_1$ :

\begin{cor}
  [Isomorphisme entre \mathversion{bold}{$K_1$} et 
  \mathversion{bold}{$K(m;v_0;v_1)$}]
  \label{cor_isom_K}
  Fixons le param\`etre $\Lambda^*\in\overline{\Lc}-\{0\}$ 
  et \'eventuellement si $K=SO(v)$, $\epsilon=\pm1$.
  On leur associe comme ci-dessus 
  les indices $v_0,v_1$ et $m=(m_1,\ldots, m_{v_1})$. 

  On d\'efinit l'application si $K=SO(v)$ et $v_0=v',\epsilon=-1$,
  \index{Notation!Isomorphisme!$\Psi_1$}
  $$ 
  \Psi_1:
  \left\{
    \begin{array}{rcl}
      K_1
      &\longrightarrow&
      K(m;v_0;v_1)\\
      k_1
      &\longmapsto&
      \left( \psi_1^{(m_1)}({[\tilde{k}_1]}_1),\ldots,
        \psi_1^{(m_{v_1-1})}({[\tilde{k}_1]}_{v_1-1}),
        \psi_1^{(m_{v_1},-1)}({[\tilde{k}_1]}_{v_1})\right)
    \end{array}\right.
  \quad ,
  $$
  sinon :
  $$ 
  \Psi_1:
  \left\{
    \begin{array}{rcl}
      K_1
      &\longrightarrow&
      K(m;v_0;v_1)\\
      k_1
      &\longmapsto&
      \left( \psi_1^{(m_1)}({[\tilde{k}_1]}_1),\ldots,
        \psi_1^{(m_{v_1-1})}({[\tilde{k}_1]}_{v_1-1}),
        \psi_1^{(m_{v_1})}({[\tilde{k}_1]}_{v_1})\right)
    \end{array}\right.
  \quad ;
  $$
  le groupe 
  $K(m;v_0;v_1)=
  U_{m_1}\times U_{m_2}\times \ldots \times U_{m_{v_1}}$ 
  a d\'ej\`a \'et\'e d\'ecrit 
  dans la sous-section~\ref{subsec_fcnsph_heis}.

  L'application $\Psi_1$ est un isomorphisme de groupe entre $K_1$ et $K(m;v_0;v_1)$.
\end{cor}

\begin{proof}[du corollaire~\ref{cor_isom_K}]
  On reprend les notations de la d\'emonstration de la proposition~\ref{prop_Krho}.
  Appliquons la proposition~\ref{prop_matrice_ortho_commutant_J}.
  Pour~$k_1\in K_1$,
  \begin{itemize}
  \item  chaque bloc ${[\tilde{k}_1]}_j, j=1,\ldots, v_1-1$ 
    commute avec $J_{m_j}$,
    et donc 
    est isomorphe par $\psi_1^{(m_j)}$ 
    \`a une matrice unitaire de taille~$m_j$;
  \item si $K=SO(v), v_0=v',\epsilon=-1$,
    le bloc ${[\tilde{k}_1]}_{v_1}$
    avec $-J_{m_{v_1}}$;
    et donc  est isomorphe par $\psi_1^{(m_{v_1},-1)}$ 
    \`a une matrice unitaire de taille~$m_j$.
  \item sinon,
    le bloc ${[\tilde{k}_1]}_{v_1}$
    avec $J_{m_{v_1}}$;
    et donc  est isomorphe par $\psi_1^{(m_{v_1})}$ 
    \`a une matrice unitaire de taille~$m_j$.
  \end{itemize}
  R\'eciproquement, on a pour $(u_1,\ldots,u_{m_{v_1}})\in
  K(m;v_0;v_1)$
  selon les cas :
  $$
  ({\psi_1^{(m_1)}}^{-1}u_1,\ldots,
  {\psi_1^{(m_{v_1},-1)}}^{-1}u_{m_{v_1}})
  \quad\mbox{ou}\quad
  ({\psi_1^{(m_1)}}^{-1}u_1,\ldots,
  {\psi_1^{(m_{v_1})}}^{-1}u_{m_{v_1}})
  \;\in\; K_1
  \quad.
  $$
  L'application $\Psi_1$ est donc un morphisme de groupe
  entre $K_1$ et $K(m;v_0;v_1)$.
\end{proof}

\begin{rem}\label{rem_compatible_complex}
  Toujours d'apr\`es la proposition~\ref{prop_matrice_ortho_commutant_J},
  l'application $\Psi_1$ est compatible avec la complexification 
  dans le sens o\`u on a selon les cas :
  $$
  \psi_c^{(v',-1)}\{ \tilde{k}_1.(x_1,y_1,\ldots,x_{v'},y_{v'})\}
  \,=\,
  \Psi_1(k_1)
  .\psi_c^{(v',-1)}(x_1,y_1,\ldots,x_{v'},y_{v'})
  \quad.,
  $$
  ou
  $$
  \psi_c^{(v_0)}\{ \tilde{k}_1.(x_1,y_1,\ldots,x_{v_0},y_{v_0})\}
  \,=\,
  \Psi_1(k_1)
  .\psi_c^{(v_0)}(x_1,y_1,\ldots,x_{v_0},y_{v_0})
  \quad.
  $$
\end{rem}

\subsection{Groupe quotient $\overline{N}=N/\ker\rho$ }
\label{subsec_N/C}

\paragraph{Expression des repr\'esentants $\rho$ consid\'er\'es.}
Nous allons maintenant donner l'expression de
$\rho_{r^*,\Lambda^*,\epsilon}$ 
associ\'ee \`a la forme lin\'eaire
$r^*X_v+D_2^\epsilon(\Lambda^*)$
que nous allons consid\'erer. 

Vu la construction effectu\'ee dans la sous-section~\ref{subsec_kirillov},
dans le cas $\Lambda^*=0$, on pose 
$\rho_{r^*,0}=U_{X^*,0}$ 
la repr\'esentation de dimension 1 
donn\'ee par  le caract\`ere :
$$ \exp(X+A)\in N \,\longmapsto\, \exp(i<X^*,X>)\quad;$$
dans le cas $\Lambda^*\not=0$,
il reste \`a choisir un sous espace $E_1$; 
nous allons le faire gr\^ace \`a la base canonique.

Supposons donc $\Lambda^*\not=0$.
On note $A^*\in \Zc^*$ identifi\'e par la base canonique $X_{i,j}$, $i<j$ 
\`a  la matrice antisym\'etrique $D_2(\Lambda^*)\not=0$ 
et $2v_0$ la dimension de l'image $\Im A^*$ de $A^*$;
l'expression de la repr\'esentation
$U_{r^*X_v,D_2^\epsilon(\Lambda^*)}$ (mais pas sa classe)
d\'epend du choix de $E_1$, sous espace maximal totalement isotrope pour 
la forme 
$(X,Y)\mapsto <A^*X,Y>$ restreinte \`a $\Im A^*\times\Im A^*$.
Ici, gr\^ace \`a la base canonique,
on pose :
$$
E_1
\,=\,
\Rb X_1\oplus \ldots\oplus\Rb X_{2j-1}\oplus \ldots \oplus \Rb
X_{2v_0-1}
\quad.
$$
\index{Notation!Repr\'esentation!$\rho_{r^*,\Lambda^*,\epsilon}$}
On note $\rho_{r^*,\Lambda^*,\epsilon}$ 
la repr\'esentation 
$U_{r^*X_v,D_2^\epsilon(\Lambda^*)}$
correspondant \`a ce choix.
On peut donner explicitement l'expression de
$\rho_{r^*,\Lambda^*,\epsilon}$, 
ainsi que des repr\'esentations $\overline{\rho}$,
$\overline{\rho}_1$, et $\overline{\rho}_2$ que nous allons d\'efinir dans la suite.

\paragraph{Notations.}
Fixons une des repr\'esentations 
$(\Hc,\rho)=(\Hc_{r^*,\Lambda^*,\epsilon},\rho_{r^*,\Lambda^*,\epsilon})$
avec $\Lambda^*$ nul ou non.
On convient de noter :
\begin{itemize}
\item $\ker\rho$ le noyau de $\rho$,
\item $\overline{N}=N/\ker\rho$ le groupe quotient 
  et $\overline{\Nc}$ son alg\`ebre de Lie,
  \index{Notation!Espace!$\overline{\Nc}$}
  \index{Notation!Groupe!$\overline{N}$}
\item  $\overline{\rho}$ le morphisme induit sur $\overline{N}$
  (qui est une repr\'esentation sur le m\^eme espace que $\rho$, 
  dont on peut donner une expression explicite),
\item $\overline{n}\in\overline{N}$ 
  l'image de $n\in N$ par la projection canonique 
  $N \rightarrow \overline{N}$,
\item $\overline{Y}\in\overline{\Nc}$ 
  l'image de $Y\in\Nc$ par la projection canonique 
  $\Nc \rightarrow \overline{\Nc}$,
\item  si $\Lambda^*\not=0$ :
  $  \overline{B}=\nn{\Lambda^*}^{-1} \overline{D_2^\epsilon(\Lambda^*)} $.
  o\`u pour $\Lambda^*=(\lambda^*_1,\ldots,\lambda^*_{v'})\in
  \overline{\Lc}$,
  $ \nn{\Lambda^*}^2=\sum_{j=1}^{v'} {\lambda^*_j}^2$;
  on remarque $\nn{\Lambda^*}=\nn{D_2(\Lambda^*)}$,
  o\`u la norme pr\'ec\'edente est la norme pour laquelle 
  la base $X_{i,j}, i<j$ est orthonormale. 
\end{itemize}

Le but de cette sous-section est de d\'emontrer les trois propositions
suivantes :
\begin{prop}[\mathversion{bold}{$\overline{N}$} et \mathversion{bold}{$\overline{\rho}$}]
  \label{prop_barN_barrho}
  \index{Notation!Repr\'esentation!$\overline{\rho},\overline{\rho}_1,\overline{\rho}_2$}
  Soit  $(\Hc,\rho)=(\Hc_{r^*,\Lambda^*,\epsilon},\rho_{r^*,\Lambda^*,\epsilon})$.\\
  a) Si $\Lambda^*\not=0$,
  le groupe $\overline{N}$ est isomorphe au produit (direct) des deux groupes
  $\overline{N_1}$ et $\overline{N_2}$,
  et la repr\'esentation $\overline{\rho}$ est \'equivalente au produit
  tensoriel des repr\'esentations 
  $\overline{\rho_1}$ sur $\overline{N_1}$
  et $\overline{\rho_2}$ sur $\overline{N_2}$,
  o\`u :
  \begin{itemize}
  \item  le groupe de Lie nilpotent $\overline{N_1}$ a pour
    alg\`ebre de Lie $\overline{\Nc_1}$  qui
    admet pour base comme espace vectoriel :
    $\overline{X_1},\ldots,\overline{X_{2v_0}},  \overline{B}$;
    son centre est $\Rb \overline{B}$;
    \index{Notation!Espace!$\overline{\Nc_1},\overline{\Nc_2}$}
    \index{Notation!Groupe!$\overline{N_1},\overline{N_2}$}
  \item la repr\'esentation $\overline{\rho}_1$ 
    induit sur le centre le caract\`ere :
    $$
    \exp( a\overline{B})\longmapsto \exp(i a\nn{\Lambda^*})
    \quad.
    $$
  \item $\overline{N_2}$ et $\overline{\rho}_2$ sont d\'ecrits par :
    \begin{itemize}
    \item ou bien $r^*=0$ 
      et $\overline{N_2}$ est le groupe trivial,
      et $\overline{\rho}_2$ la repr\'esentation triviale;
    \item ou bien $r^*\not =0$ 
      et $\overline{N_2}$ est le groupe d'alg\`ebre de Lie
      $\Rb \overline{X}_v$,
      et $\overline{\rho}_2$ est la repr\'esentation 
      associ\'ee au caract\`ere :
      $$
      \exp x \overline{X}_v 
      \, \longrightarrow \,
      \exp (ix)\quad.
      $$
    \end{itemize}
  \end{itemize}
  b) Si $\Lambda^*=0$,
  alors $\overline{N}$ et $\overline{\rho}$ 
  ont la m\^eme description que 
  $\overline{N_2}$ et $\overline{\rho}_2$ ci-dessus.
\end{prop}

Avec les conventions de notations du d\'ebut de cette sous-section,
pour $k\in K_\rho$, les repr\'esentations $\rho$ et
$k.\rho $ sont \'equivalentes; 
en particulier, elles ont m\^eme noyau. 
Donc l'action de $K_\rho$ laisse stable $\ker \rho$ et passe au
quotient sur $\overline{N}$. 
Nous d\'ecrivons cette derni\`ere action.

\begin{prop}[\mathversion{bold}{$K_\rho$} et
  \mathversion{bold}{$\overline{N}$}]
  \label{prop_Krho_barN}
  On garde les notations des propositions~\ref{prop_Krho}
  et~\ref{prop_barN_barrho}.
  \begin{itemize}
  \item[a)] Si $\Lambda^*\not=0$,
    le groupe $K_1$ agit sur $\overline{N_1}$, 
    et le groupe $K_2$ agit trivialement sur~$\overline{N_2}$ :
    le produit semi-direct 
    $K_\rho \triangleleft \overline{N}$
    peut s'\'ecrire :
    $K_\rho \triangleleft\overline{N}
    \sim
    (K_1\triangleleft \overline{N_1} )
    \,\times\, K_2\,\times\, \overline{N_2}$.
  \item[b)] Si $\Lambda^*=0$,
    l'action de $K_\rho$ est triviale sur $\overline{N}$ :
    le produit semi-direct 
    $K_\rho \triangleleft \overline{N}$
    est en fait direct :
    $K_\rho \triangleleft \overline{N}
    \sim
    K_\rho \times \overline{N}$.
  \end{itemize}
\end{prop}

L'utilisation de la base canonique dans l'expression de 
$\rho_{r^*,\Lambda^*,\epsilon}$
nous permet d'\'etablir simplement un isomorphisme entre $\overline{N_1}$ et le groupe de Heisenberg
$\Hb^{v_0}= \Cb^{v_0}\times \Rb$. 
Gardons les notations 
de la proposition~\ref{prop_barN_barrho}, 
dans le cas $\Lambda^*\not=0$,
et du corollaire~\ref{cor_isom_K}.
On d\'efinit l'application 
$\Psi_2 : \Hb^{v_0}\rightarrow\overline{N_1}$
donn\'ee par 
si $v_0=v',\, K=SO(v),\, \epsilon=-1$ :
\index{Notation!Isomorphisme!$\Psi_2$}
\begin{eqnarray*}
  &&\Psi_2(x_1+iy_1,\ldots, x_{v'}+iy_{v'},t)\\
  &&\quad=
  \exp\left(\left\{ \sum_{j=1}^{v'-1}  
      \sqrt{\frac{\nn{\Lambda^*}}{\lambda^*_j}}
      \left(x_j\overline{X_{2j-1}}+y_i\overline{X_{2j}}\right)
      - \sqrt{\frac{\nn{\Lambda^*}}{\lambda^*_j}}
      \left(y_{v'}\overline{X_{2v'-1}}+x_{v'}\overline{X_{2v'}}\right)
    \right\}
    \quad +t\overline{B}\right)\quad.
\end{eqnarray*}
sinon :
$$
\Psi_2(x_1+iy_1,\ldots, x_{v_0}+iy_{v_0},t)
\,=\,
\exp \left(\sum_{j=1}^{v_0}  
  \sqrt{\frac{\nn{\Lambda^*}}{\lambda^*_j}}
  \left(x_j\overline{X_{2j-1}}+y_i\overline{X_{2j}}\right)
  \quad
  +t\overline{B}\right)
\quad.
$$

\begin{prop}[\mathversion{bold}{$N_1$} et \mathversion{bold}{$K_1$}]
  \label{prop_N1_K1}
  L'application $\Psi_2$ est un ismorphisme entre les groupes $N_1$ et $\Hb^{v_0}$. 
  Les groupes 
  $H:= K_1\triangleleft \overline{N_1}$ et
  $H_{heis}:= K(m;v_0;v_1) \triangleleft \Hb^{v_0}$
  sont isomorphes par l'application :
  \index{Notation!Isomorphisme!$\Psi_0$}
  $$
  \Psi_0\;:\;\left\{
    \begin{array}{rcl}
      H_{heis}= K(m;v_0;v_1) \triangleleft \Hb^{v_0}
      &\longrightarrow& 
      H=K_1\triangleleft \overline{N_1}\\
      \Psi_1^{-1}(k_1)\, , \,
      h
      &\longmapsto&
      k_1\, ,\,
      \Psi_2 (h)
    \end{array}\right.
  \quad .
  $$
\end{prop}

Les d\'emonstrations du cas $\Lambda^*=0$ sont directes.

\begin{proof}[des propositions~\ref{prop_barN_barrho} 
  et~\ref{prop_Krho_barN} 
  si $\Lambda^*=0$]
  Dans ce cas, la repr\'esentation $\rho$ s'identifie au
  caract\`ere $\chi$ dont la diff\'erentielle est $ir^*X^*_v$.
  Son noyau $\ker\rho$ a donc pour alg\`ebre de Lie 
  l'ensemble des vecteurs $X\in \Vc$ qui
  v\'erifient $<r^*X^*_v,X>=0$.
  Le groupe $\overline{N}$ a alors une alg\`ebre de Lie
  isomorphe \`a~$\Rb$ si $r^*$ est non nul,
  et \`a~$\{0\}$ sinon.

  La repr\'esentation $\rho$ se factorise en une repr\'esentation
  unitaire irr\'eductible $\overline{\rho}$ sur $\overline{N}$, qui est
  associ\'ee au caract\`ere si $r^*$ est non nul :
  $$
  \exp x \overline{X}_v 
  \, \longrightarrow \,
  \exp (ix) \quad,
  $$
  et si $r^*$ est nul, le caract\`ere trivial 1.
\end{proof}

Jusqu'\`a la fin de cette sous-section, 
on convient de noter $A^*=D_2(\Lambda^*)$
et $X^*=r^*X_v$,
ainsi que d'omettre
$r^*$ et $\overline{X_v}$
si $r^*=0$.

On connait :
\begin{itemize}
\item gr\^ace \`a la remarque~\ref {rem_expression_noyau},
  le noyau de la repr\'esentation $\rho_{r^*,\Lambda^*,\epsilon}$,
  dont on d\'eduit la base voulue pour $\Nc$,
\item gr\^ace \`a la remarque~\ref {rem_expression_centre},
  son expression sur le centre de $N$,
  dont on d\'eduit l'expression de $\overline{\rho}_1$ sur $\Rb B$.
\end{itemize}
Pour d\'emontrer la proposition~\ref{prop_barN_barrho}.a),
il reste \`a exhiber le centre de $\Nc$.
Pour cela, nous calculons tous les crochets de la base consid\'er\'ee.

\begin{lem}
  \label{lem_barN}
  Dans ce qui suit,
  on suppose $\Lambda^*\not=0$; 
  on garde les notations 
  de la proposition~\ref{prop_barN_barrho}.

  L'alg\`ebre de Lie $\overline{\Nc}$
  admet pour base comme espace vectoriel 
  la famille de vecteurs :
  $$
  \overline{X_1},\ldots,\overline{X_{2v_0}},  \overline{B},
  \quad\mbox{\`a laquelle on ajoute}\quad\overline{X_v}
  \;\mbox{si}\; 2v_0<v\;\mbox{et}\;r^*\not=0
  \quad.
  $$

  Les crochets des vecteurs de cette  base valent 0, sauf :
  $$
  [\overline{X_{2i-1}},\overline{X_{2i}}] 
  \,=\,
  \left\{
    \begin{array}{ll}
      \epsilon\frac{\lambda^*_{v'}}{\nn{\Lambda^*}}\overline{B}
      &\;\mbox{si}\;i=v' \;\mbox{et}\;K=SO(v),\\
      \frac{\lambda^*_i }{\nn{\Lambda^*}}\overline{B}
      &\;\mbox{sinon}.
    \end{array}\right.
  \quad.
  $$
  En particulier, $X^*$ commute avec tous les vecteurs $\overline{X_i}, i=1,\ldots,2v_0$. 
\end{lem}

\begin{proof}[du lemme~\ref{lem_barN}]
  On d\'eduit la base de $\overline{\Nc}$  de l'expression du noyau 
  (voir remarque~\ref {rem_expression_noyau}).
  On voit directement d'apr\`es (\ref{egalite_crochet_pdtsc}):
  \begin{itemize}
  \item  si $(i,j)\not=(2i'-1,2i')$,
    $<A^*,[X_i,X_j]>=<A^*.X_i,X_j>=0$ 
    et donc $[\overline{X_i},\overline{X_j}]=0$,
  \item  $<A^*,[X_{2i'-1},X_{2i'}]>
    =<A^*.X_{2i'-1},X_{2i'}>$ vaut
    $\epsilon \lambda^*_{v'}$ si $i=v'$ et $K=SO(v)$, et $\lambda_{i'}$ sinon; 
    et donc $[\overline{X_{2i'-1}},\overline{X_{2i'}}]$ vaut
    $\epsilon \lambda^*_{v'}\nn{\Lambda^*}^{-1}\overline{B}$ si $i=v'$ et $K=SO(v)$, 
    et $\lambda_{i'}\nn{\Lambda^*}^{-1}\overline{B}$ sinon.
  \end{itemize}
\end{proof}

\begin{proof}[de la proposition~\ref{prop_Krho_barN} 
  si $\Lambda^*\not=0$]
  Soit $n=\exp(X+A)$ avec $X$ d\'ecompos\'e selon la base canonique :
  \begin{equation}
    X\,=\,\sum_{i=1}^v x_i X_i\quad.
    \label{notx}
  \end{equation}
  On a :
  $$
  \overline{n}
  \,=\,
  \exp(\overline{X}+<D_2(\Lambda^*),A>\overline{B})\;,
  \quad\mbox{avec }\quad
  \overline{X}
  \,=\,
  \sum_{j=1}^{2v_0} x_j\overline{X_j} 
  +x_v\overline{X_v}
  \quad.
  $$
  Donc pour $k\in K_\rho$,
  on a : $k.n=\exp(k.X+k.A)$, et
  $\overline{k.n}=\exp(\overline{k.X}+\overline{k.A})$,
  avec :
  \begin{eqnarray*}
    \overline{k.X}
    &=&
    \sum_{j=1}^{2v_0} \overline{x_jk.X_j} 
    +
    x_v\overline{X_v}\quad,\\
    \overline{k.A}
    &=&
    <A^*,k.A>\overline{B}
    \,=\,
    <k^{-1}.A^*,A>\overline{B}
    \,=\,
    <A^*,A>\overline{B}
    \,=\,
    \overline{A} \quad, 
  \end{eqnarray*}
  car $k^{-1}.A^*=A^*$ et $k.X^*=X^*$. 
  On peut donc directement d\'efinir l'action (par automorphisme)
  du groupe $K_\rho$ sur le groupe $\overline{N}$.
  De plus, 
  on remarque les propri\'et\'es suivantes :
  \begin{itemize}
  \item Si $k\in K_1$ et $\overline{n}\in\overline{N_1}$
    alors $k.\overline{n}\in \overline{N_1}$;
    et donc $K_1$ agit sur $\overline{N_1}$.
  \item Si $k\in K_2$ et $\overline{n}\in\overline{N_2}$
    alors $k.\overline{n}=\overline{n}$; 
    et donc $K_2$ agit trivialement sur $\overline{N_2}$.
  \end{itemize}
\end{proof}

\begin{proof}[de la proposition~\ref{prop_N1_K1}]
  Notons $(z,t),(z',t')\in \Hb^{v_0}$ avec :
  $$
  z=(x_1+iy_1,\ldots,x_{v_0}+iy_{v_0})
  \quad\mbox{et}\quad
  z'=(x'_1+iy'_1,\ldots,x'_{v_0}+iy'_{v_0})
  \quad.
  $$
  D'apr\`es l'expression de 
  $(z,t).(z',t')$ et $\Psi_2$,
  on a  si $v_0=v',\, K=SO(v),\, \epsilon=-1$ :
  \begin{eqnarray*}
    &&    \Psi_2 \left(  (z,t).(z',t')\right)
    \,=\,
    \exp \left(
      \left\{\sum_{j=1}^{v'-1} 
        \sqrt{\frac{\nn{\Lambda^*}}{\lambda^*_j}} (x_j+x'_j)\overline{X_{2j-1}}
        +(y_j+y'_j)\overline{X_{2j}}\right.\right.\\
    && \qquad  \left. \left.
        -\sqrt{\frac{\nn{\Lambda^*}}{\lambda^*_{v'}}}(y_{v'}+y'_{v'})\overline{X_{2v'-1}}
        +(x_{v'}+x'_{v'})\overline{X_{2v'}}\right\}\;
      +(t+t'+\frac12 \sum_{j=1}^{v_0} x_jy'_j-y_jx'_j)\overline{A}\right)
    \quad,
  \end{eqnarray*}
  sinon :
  \begin{eqnarray*}
    \Psi_2\left (  (z,t).(z',t')\right)
    \,=\,
    \exp \left(\sum_{j=1}^{v_0} 
      \sqrt{\frac{\nn{\Lambda^*}}{\lambda^*_j}}
      (x_j+x'_j)\overline{X_{2j-1}}+(y_j+y'_j)\overline{X_{2j}}\right.\\
    +\left.(t+t'+\frac12 \sum_{j=1}^{v_0} x_jy'_j-y_jx'_j)\overline{A}\right)
    \quad.
  \end{eqnarray*}

  Or d'apr\`es la valeur des crochets des $\overline{X_j}$ donn\'ee
  dans le lemme~\ref{lem_barN}, 
  on voit :
  \begin{eqnarray*}
    &&\left[ 
      \sum_{j=1}^{v_0} 
      \sqrt{\frac{\nn{\Lambda^*}}{\lambda^*_j}}
      (x_j\overline{X_{2j-1}}+y_j\overline{X_{2j}})
      \; ,\;
      \sum_{j=1}^{v_0} 
      \sqrt{\frac{\nn{\Lambda^*}}{\lambda^*_j}}
      (x'_j\overline{X_{2j-1}}+y'_j\overline{X_{2j}})
    \right]\\
    &&\quad=\,
    \sum_{j=1}^{v_0} \frac{\nn{\Lambda^*}}{\lambda^*_j} x_j y'_j
    \left[ \overline{X_{2j-1}},\overline{X_{2j}}\right]
    +\frac{\nn{\Lambda^*}}{\lambda^*_j} y_j  x'_j
    \left[\overline{X_{2j}}, \overline{X_{2j-1}}\right]
    \,=\,
    \sum_{j=1}^{v_0} (x_jy'_j-y_jx'_j)\overline{B}
    \quad,
  \end{eqnarray*}
  ainsi que :
  \begin{eqnarray*}
    &&\left[ 
      -\sqrt{\frac{\nn{\Lambda^*}}{\lambda^*_{v'}}}
      (y_{v'}\overline{X_{2v'-1}}+x_{v'}\overline{X_{2v'}})
      \; ,\;
      -\sqrt{\frac{\nn{\Lambda^*}}{\lambda^*_{v'}}}
      (y'_{v'}\overline{X_{2v'-1}}+x'_{v'}\overline{X_{2v'}})
    \right]\\
    &&\quad=\,
    \frac1{\lambda^*_{v'}} 
    \left(y_{v'}x'_{v'} \left[ \overline{X_{2v'-1}},\overline{X_{2v'}}\right]
      +  x_{v'}  y'_{v'}
      \left[\overline{X_{2v'}}, \overline{X_{2v'-1}}\right]\right)
    \,=\,
    (x_{v'}y'_{v'} - y_{v'} x'_{v'}  )\overline{B}
    \quad.
  \end{eqnarray*}
  On en d\'eduit dans les deux cas :
  \begin{eqnarray*}
    \Psi_2\left (  (z,t).(z',t')\right)
    &=&
    \exp \left(\sum_{j=1}^{v_0}  
      \sqrt{\frac{\nn{\Lambda^*}}{\lambda^*_j}}
      (x_j\overline{X_{2j-1}}+y_j\overline{X_{2j}})
      \;        +t\overline{A}\right)\\
    &&\quad
    \exp\left( \sum_{j=1}^{v_0}  
      \sqrt{\frac{\nn{\Lambda^*}}{\lambda^*_j}}
      (x_j'\overline{X_{2j-1}}+y_j'\overline{X_{2j}})
      \;        +t'\overline{A}\right)\\
    &=&
    \Psi_2(z,t).\Psi_2(z',t')
    \quad.
  \end{eqnarray*}
  L'application $\Psi_2$ est donc un morphisme de groupe et il est clairement bijectif.
  Par d\'efinition de $\Psi_1$ et $\Psi_2$, 
  et d'apr\`es la remarque~\ref{rem_compatible_complex},
  on a :
  $$
  \forall\, \tilde{k}_1\in K_1\; ,
  h\in \Hb^{v_0}
  \quad:\quad
  \Psi_1^{-1}(k_1).h
  \,=\,
  \Psi_2^{-1}(k_1.\Psi_2(h))
  \quad;
  $$
  ainsi l'application $\Psi_0$ est un isomorphisme de groupe 
  $K(m;v_0;v_1) \triangleleft \Hb^{v_0} \rightarrow K_1\triangleleft \overline{N_1}$.
\end{proof}

\subsection{Cas $O(r^*,0)$}
\label{subsec_thm_fcnsphnu.a}

Nous montrons ici le th\'eor\`eme \ref{thm_fcnsphnu}.a).
Fixons $\rho=\rho_{r^*,0}$,
et $\nu\in\tilde{G}_\rho$.
On note 
$\nu_{|N}=c.\rho$, $1\leq c \leq \infty$
et $\overline{\nu}$ la repr\'esentation 
donn\'ee par le passage au quotient du groupe $G_\rho$ par $\ker \rho$
de la repr\'esentation $\nu$.

\paragraph{Description de \mathversion{bold}{$\tilde{G}_\rho$}.}
D'apr\`es la proposition~\ref{prop_Krho_barN}
dans le cas $\Lambda^*=0$,
le produit  semi-direct
$ K_\rho\triangleleft \overline{N}$  
est en fait direct,
donc la repr\'esentation $\overline{\nu}$ 
s'\'ecrit comme le produit tensoriel de deux
repr\'esentations unitaires irr\'eductibles :
\begin{itemize}
\item l'une de $\overline{N}$, 
  qui co\"\i ncide avec $c.\overline{\rho}$
  (\'etant irr\'eductible $c=1$),
\item et l'autre de $K_\rho$
  ayant un vecteur $K_\rho$-fixe
  (\'etant irr\'eductible, elle est triviale).
\end{itemize}
La repr\'esentation $\overline{\nu}$ coincide donc avec
la repr\'esentation :
$(k,n)\in K_\rho\times \overline{N}\mapsto \overline{\rho}(n)$.
Or d'apr\`es la proposition~\ref{prop_barN_barrho}
dans le cas $\Lambda^*=0$,
si $r^*\not=0$,
la repr\'esentation $\overline{\rho}$ est associ\'ee au caract\`ere 
$\exp x\overline{X}_v \mapsto e^{ix}$.
Donc $\overline{\nu}$ est la repr\'esentation associ\'ee au caract\`ere:
$$
(k,\exp (x\overline{X})) 
\, \longmapsto \,
\exp (ix)
$$
On en d\'eduit que la repr\'esentation $\nu=\nu^{r^*,0}$ est donn\'ee par:
\index{Notation!Repr\'esentation!$\nu^{r^*,0},\nu^{r^*,\Lambda^*,l,\epsilon}$}
$$
k,\exp (X + A) 
\,\longmapsto\,
e^{i <r^*X^*_v,X>}
\quad ,
$$
si $r^*\not=0$.
C'est aussi le cas si $r^*=0$, car alors $\overline{\rho}$ est la
repr\'esentation triviale $1$.

Nous venons donc de trouver que 
$\tilde{G}_\rho$ est l'ensemble des classes des 
repr\'esentations~$\nu^{r^*,0}$ lorsque $r^*$ parcourt $\Rb^+$.
L'espace de la repr\'esentation  de $\nu^{r^*,0}$ est de dimension~1. 
On note $\vec{u}$ un de ses vecteurs unitaires.

\paragraph{Formule pour 
  \mathversion{bold}{$\phi^\nu,\nu\in\tilde{G}_\rho$}.}
Explicitons d'apr\`es la formule~(\ref{fcnsphnu}), 
la fonction sph\'erique $\phi^\nu$
associ\'ee \`a la repr\'esentation~$\nu=\nu^{r^*,0}$.
On a :
$$
{\big<\nu(I,n).\vec{u}\,,\,
  \vec{u} \big>}_{\Hc^\nu}
\,=\,
e^{i <r^* X^*_v,X>}
\quad\mbox{d'o\`u}\quad
\phi^\nu(\exp (X+A))
\,=\,
\int_{k\in K}
e^{i <r^* X^*_v,k.X>}
dk
\quad.
$$
$\phi^\nu=\phi^{r^*,0}$,
et le th\'eor\`eme \ref{thm_fcnsphnu}.a) 
est ainsi d\'emontr\'e.

\subsection{Cas $O(r^*,\Lambda^*,\epsilon)$}
\label{subsec_thm_fcnsphnu.b}

Nous montrons ici le th\'eor\`eme \ref{thm_fcnsphnu}.b).
Fixons $\rho=\rho_{r^*,\Lambda^*,\epsilon}$ avec $\Lambda^*\not=0$,
et $\nu\in\tilde{G}_\rho$.
On note :
$\nu_{|N}=c.\rho$, $1\leq c \leq \infty$
et $\overline{\nu}$ la repr\'esentation 
donn\'ee par le passage au quotient du groupe $G_\rho$ par $\ker \rho$
de la repr\'esentation $\nu$.

\paragraph{Description de 
  \mathversion{bold}{$\tilde{G}_\rho$}.}
D'apr\`es la proposition~\ref{prop_Krho_barN}
dans le cas $\Lambda^*=0$,
le groupe
$K_\rho\triangleleft \overline{N}$ 
est isomorphe au produit direct de
$H=K_1\triangleleft \overline{N_1}$
et de $K_2$ et $\overline{N_2}$.
La repr\'esentation~$\overline{\nu}$ 
s'\'ecrit donc comme le produit tensoriel de trois
repr\'esentations unitaires irr\'eductibles:
\begin{enumerate}
\item l'une de $H$ 
  dont les vecteurs $K_1$-invariants forment une droite,
\item l'autre de $K_2$
  dont les vecteurs $K_2$-invariants forment une droite,
\item la derni\`ere,  de $\overline{N_2}$.
\end{enumerate}
\`A cause de l'irr\'eductibilit\'e, 
la repr\'esentation sur $K_2$ est triviale :
$K_2 \subset \ker \overline{\nu}$;
d'apr\`es le lemme~\ref{lem_rep_noyau},
$\overline{\nu}$ passe au quotient 
en une repr\'esentation unitaire 
irr\'eductible~$\overline{\overline{\nu}}$ 
sur~$H\times  \overline{N_2}$
qui co\"\i ncide avec~$c.\overline{\rho}$ sur $\overline{N}$ et 
dont les vecteurs $K_1$-invariants forment une droite.
Nous reprenons les notations du lemme~\ref{lem_rep_Pi_omega},
ainsi que celle de $\Psi_0$ (proposition~\ref{prop_N1_K1}).
Pour une fonction sph\'erique $\omega$ 
de~$(\Hb^{v_0},K(m;v_0;v_1))$,
on d\'efinit 
les repr\'esentations  
$(\Hc^\omega,\Pi^\omega)$ 
de~$H$ par :
$$
\Hc^\omega
\,=\,
\{F\circ \Psi_0\; , \quad F\in \Hc_\omega\}
\quad,\quad
\Pi^\omega=\Pi_\omega(\Psi_0 .)
\quad.
$$

\begin{lem}
  \label{lem_bar_bar_nu}
  La repr\'esentation 
  $\overline{\overline{\nu}}$ 
  est de la forme :
  \begin{eqnarray*}
    \forall\quad
    (k_1,\overline{n_1})\in H=K_1\triangleleft \overline{N_1 }\; ,\quad
    \overline{n_2}=\exp (x \overline{X}) \in \overline{N_2 }\; , \quad
    k_2\in K_2
    \quad :\\
    \overline{\overline{\nu}}
    \left( (k_1,\overline{n_1}) .(k_2, \overline{n_2})\right)
    \,=\,
    \gamma_1(k_1,\overline{n_1})
    \exp (i x)
    \quad ,
  \end{eqnarray*}
  o\`u $\gamma_1$ est une repr\'esentation du $K_1\triangleleft N_1$
  \'equivalente \`a $\Pi^\omega$ avec 
  $\omega=\omega_{\nn{\Lambda^*},l}$, $l\in \Nb^{v_1}$;
  la droite $K_1$-fixe de $\Hc^\omega$
  est $\Cb \Omega^\omega\circ\Psi_0$.
\end{lem}

\begin{proof}
  La repr\'esentation
  $\overline{\overline{\nu}}$ 
  s'\'ecrit comme le produit tensoriel de $\gamma_1$ et $\gamma_2$ tel que
  \begin{itemize}
  \item[(a)] 
    la repr\'esentation $\gamma_1$ du groupe
    $H$
    est irr\'eductible; 
    ses vecteurs $K_1$-invariants forment une droite;
    elle co\"\i ncide avec
    $c.\overline{\rho}$ sur $\overline{N_1}$;
  \item[(b)]
    la repr\'esentation $\gamma_2$ du groupe  $\overline{N_2}$
    est irr\'eductible et co\"\i ncide avec
    $c.\overline{\rho}$ sur $\overline{N_2}$. 
  \end{itemize}

  \`A cause de l'irr\'eductibilit\'e  d'apr\`es la
  condition (a), on a $c=1$.
  Donc si $N_2\not=\{0\}$ c'est-\`a-dire $r^*\not=0$, on voit :
  $$
  \overline{\rho}(\exp (x\overline{X_v}))
  \,=\,
  \gamma_2(\exp (x\overline{X_v}))
  \,=\,
  \exp (i x)
  \quad .
  $$
  De plus, 
  d'apr\`es la proposition~\ref{prop_N1_K1},
  $H$ est isomorphe \`a $H_{heis}=K(m;v_0;v_1)\triangleleft \Hb^{v_0}$ par $\Psi_0$;
  et d'apr\`es la premi\`ere partie du lemme~\ref{lem_rep_Pi_omega},
  on connait les repr\'esentations irr\'eductibles~$(\Hc_\omega,\Pi_\omega)$
  sur~$H_{heis}$
  dont les vecteurs $K_1$-invariants forment une droite.
  On en d\'eduit que les repr\'esentations irr\'eductibles de~$H$ 
  dont les vecteurs $K_1$-invariants forment une droite,
  sont toutes les repr\'esentations~$(\Hc_\omega,\Pi_\omega)$ 
  donn\'ees dans l'\'enonc\'e,
  lorsque $\omega$ parcourt
  l'ensemble des fonctions sph\'eriques de
  $H_{heis}$.
  De plus,
  la droite $K_1$-fixe de $\Hc^\omega$
  est $\Cb \Omega^\omega\circ\Psi_0$.

  Ainsi les  repr\'esentations~$\gamma_1$ 
  v\'erifiant~(a)
  sont les repr\'esentations~$\gamma_1$ 
  \'equivalentes \`a une repr\'esentation~$\Pi^\omega$ 
  satisfaisant : 
  $\Pi^\omega_{|\overline{N_1}}\sim\overline{\rho}_1$.
  Supposons cette condition v\'erifi\'ee.
  D'apr\`es  la restriction
  des repr\'esentations $\Pi^\omega$ et $\overline{\rho}_1$ sur le centre 
  $\exp \Rb \overline{B}$ de $\overline{N_1}$
  (voir respectivement la seconde  partie 
  du lemme~\ref{lem_rep_Pi_omega},
  et la proposition~\ref{prop_barN_barrho}),
  le cas $\omega=\omega_\mu$ est impossible, et
  la fonction $\omega$ est de la forme
  $\omega=\omega_{\nn{\Lambda^*},l}$.
  Par cons\'equent,
  les repr\'esentations 
  $\gamma_1$ v\'erifiant~(a)
  sont parmi les repr\'esentations $\gamma_1$ \'equivalentes \`a une 
  repr\'esentation $\Pi^\omega$ avec 
  $\omega=\omega_{\nn{\Lambda^*},l}$.
\end{proof}

\paragraph{\mathversion{bold}{$\tilde{G}_\rho \subset
    \tilde{G}_\rho'$}.}
On note la projection canonique 
($X$ est d\'ecompos\'e selon~(\ref{notx})):
$$
q_1
\; :\;
\left\{\begin{array}{rcl}
    N&\longrightarrow& \overline{N_1}\\
    n=\exp (X +A)
    &\longmapsto&
    \exp \left(\sum_{j=1}^{2v_0}
      x_j\overline{X_j}+\frac{<A^*,A>}{\nn{\Lambda^*}}\overline{B}\right)
    \quad.
  \end{array}\right.
$$
Par la projection canonique :
$$
\begin{array}{rcl}
  K_\rho\triangleleft N
  =K_1\times K_2\triangleleft N
  &\longrightarrow&
  \left(K_1\triangleleft \overline{N_1}\right)
  \times \left(K_2\times \overline{N_2}\right)\\
  k=k_1k_2, n=\exp (X+A )
  &\longmapsto&
  \left((k_1,q_1(n))
    \, ,\,
    (k_2,\exp (r^*<X^*_v, X>\overline{X}))
  \right)
\end{array}\quad,
$$
la repr\'esentation $\overline{\overline{\nu}}$ 
se rel\`eve en la repr\'esentation $\nu$;
gr\^ace au lemme~\ref{lem_bar_bar_nu},
$\nu$
est  \'equivalente \`a une repr\'esentation 
$(\Hc^\omega,\nu^{r^*,\Lambda^*,l})$
sur $K_\rho\triangleleft N$ 
avec  $\omega=\omega_{\nn{\Lambda^*},l}$,
donn\'ee par :
$$
\begin{array}{c}
  \displaystyle{  \forall 
    \quad
    k=k_1k_2\in K_\rho=K_1\times K_2\; ,
    \quad
    n=\exp (X+A)\in N
    \quad :} \\
  \\
  \displaystyle{
    \nu^{r^*,\Lambda^*,l} (k,n) =
    e^{i r^*<X^*_v,X>}
    \, \Pi^\omega (k_1, q_1(n))
    \quad.}
\end{array}
$$
\index{Notation!Repr\'esentation!$\nu^{r^*,0},\nu^{r^*,\Lambda^*,l,\epsilon}$}
On note $\tilde{G}_\rho '$ l'ensemble des classes des
repr\'esentations~$\nu^{r^*,\Lambda^*,l,\epsilon}$, 
$l\in\Nb^{v_1}$.
C'est un ensemble de classe de repr\'esentations de $K_\rho\triangleleft N$,
qui contient $\tilde{G}_\rho$; 
les restrictions \`a~$N$ de ses repr\'esentants 
$(\Hc^\omega,\nu^{r^*,\Lambda^*,l,\epsilon})$ 
peuvent (a priori) ne pas \^etre 
\'equivalentes \`a $\rho$. 
Cependant, par construction, 
elles ont une droite $K_\rho$-fixe
dont un des vecteurs unitaires est 
$\vec{u}=\Omega^\omega\circ \Psi_0$.
Notons $\phi^\nu$ la fonction sph\'erique associ\'ee.

\begin{lem}[Calcul de \mathversion{bold}{$\phi^\nu,\nu\in
    \tilde{G}_\rho'$}]
  \label{lem_phi_nu_G'}
  La fonction $\phi^\nu$
  est donn\'ee par :
  \begin{eqnarray}
    \phi^{r^*,\Lambda^*,l,\epsilon}(n)
    &=&
    \int_{k\in K}
    e^{i r^*<X^*_v,k.X>}
    \omega\circ
    \tilde{\Psi}_2 (k.n)   
    dk
    \label{egalite_phi_omega_Psi_q}\\
    &=&
    \int_{k\in K}
    \Theta^{r^*,\Lambda^*,l,\epsilon}(k.n)
    dk
    \quad,\nonumber
  \end{eqnarray}
  o\`u $\omega=\omega_{\nn{\Lambda^*},l}$ et $\tilde{\Psi}_2=\Psi_2^{-1}\circ q_1$, 
  et la fonction $\Theta^{r^*,\Lambda^*,l,\epsilon}$ est donn\'ee par :
  $$
  \Theta^{r^*,\Lambda^*,l,\epsilon}(\exp (X+A))
  \,=\,
  e^{i r^*<X^*_v,X>}
  e^{i<D_2^\epsilon(\Lambda^*),A>}
  \overset{v_1}{\underset{j=1} \Pi}
  \flb {l_j} {m_j-1} {\frac{\lambda_j}2 \nn{\pr j X }^2} 
  \quad .
  $$
\end{lem}

\begin{proof}[du lemme~\ref{lem_phi_nu_G'}]
  D'apr\`es l'expression de la
  repr\'esentation~$\nu=\nu^{r^*,\Lambda^*,l,\epsilon}$, 
  on~a:
  $$
  \nu (I,n) \vec{u}=
  e^{i r^*<X^*_v,X>}
  \, \Pi^\omega (I,
  q_1(n))
  \Omega^\omega \circ \Psi_0
  \quad,
  $$
  puis :
  $$
  {\big<\nu(I,n).\vec{u}\,,\,
    \vec{u} \big>}_{\Hc^\nu}
  \,=\,
  e^{i r^*<X^*_v,X>}
  \, {\big< \Pi^\omega (I,q_1(n))
    \Omega^\omega \circ \Psi_0\, ,\, 
    \Omega^\omega \circ \Psi_0 \big>}_{H^\omega}
  \quad.
  $$
  L'isomorphisme $\Psi_0$ donne :
  $$
  {\big< \Pi^\omega (I, q_1(n))
    \Omega^\omega \circ \Psi_0\, ,\, 
    \Omega^\omega \circ \Psi_0 \big>}_{H^\omega}
  \,=\,
  {\big< \Pi_\omega (I,
    \Psi_2^{-1}\circ q_1(n) )
    \Omega^\omega ,\, 
    \Omega^\omega \big>}_{H_\omega}
  \quad.
  $$
  De l'expression de $\Psi_2$  (proposition~\ref{prop_N1_K1}), 
  on en d\'eduit
  une expression $\tilde{\Psi}_2=\Psi_2^{-1}\circ q_1$ 
  ($X$ \'etant d\'ecompos\'e selon~(\ref{notx})) 
  si $K=SO(v), v_0=v',\epsilon=-1$ :
  $$
  \tilde{\Psi}_2 (\exp (X+A))  
  \,=\,
  (\sqrt{\frac{\lambda_1}{\nn{\Lambda^*}}}(x_1+ix_2),\ldots, 
  -\sqrt{\frac{\lambda_{v'}}{\nn{\Lambda^*}}} (x_{2v'}+ix_{2v'-1}),
  \frac{<D_2^\epsilon(\Lambda^*),A>}{\nn{\Lambda^*}})
  \quad,  
  $$
  sinon :
  $$
  \tilde{\Psi}_2 (\exp (X+A))  
  \,=\,
  (\sqrt{\frac{\lambda_1}{\nn{\Lambda^*}}}(x_1+ix_2),\ldots, 
  \sqrt{\frac{\lambda_{v_0}}{\nn{\Lambda^*}}} (x_{2v_0-1}+ix_{2v_0}),
  \frac{<D_2(\Lambda^*),A>}{\nn{\Lambda^*}})
  \quad.  
  $$
  Par d\'efinition de la  repr\'esentation $\Pi_\omega$,
  la fonction $\Omega^\omega$ \'etant la fonction de type
  positif associ\'ee \`a cette repr\'esentation
  (ou par l'\'equation fonctionnelle~(\ref{eqfonc}) des
  fonctions sph\'eriques pour $\Omega^\omega$), on a : 
  $$
  {\big< \Pi_\omega (I,
    \tilde{\Psi}_2 (n)   )
    \Omega^\omega ,\, 
    \Omega^\omega \big>}_{H_\omega}
  \,=\,
  \Omega^\omega\left( I,
    \tilde{\Psi}_2 (n)  \right)
  \,=\,
  \omega\circ
  \tilde{\Psi}_2 (n)   
  \quad.
  $$
  Gr\^ace \`a l'expression de $\tilde{\Psi}_2$,
  en renommant les $\lambda^*_i$ en les $\lambda_j$ distincts,
  et gr\^ace \`a l'expression de 
  $\omega=  \omega_{\nn{\Lambda^*},l}$
  th\'eor\`eme~\ref{thm_paire_G_heis},
  on obtient dans les deux cas :
  $$
  \omega\circ
  \tilde{\Psi}_2 (\exp (X+A))   
  \,=\,
  e^{i<D_2^\epsilon(\Lambda^*),A>}
  \overset{v_1}{\underset{j=1} \Pi}
  \flb {l_j} {m_j-1} {\frac{\lambda_j}2 \nn{\pr j X }^2} 
  \quad,
  $$
  puis l'expression~(\ref{fcnsphnu}) de la fonction
  $\phi^{\nu}=\phi^{r^*,\Lambda^*,l,\epsilon}$
  donn\'ee dans le lemme.
\end{proof}

Ceci ach\`eve
la d\'emonstration du th\'eor\`eme \ref{thm_fcnsphnu}.

Les th\'eor\`emes~\ref{thm_fsphO} et~\ref{thm_fsphSO} sont ainsi d\'emontr\'es.

\section{Remarques}
\label{sec_remarques}

Nous confrontons ici les r\'esultats 
des  th\'eor\`emes~\ref{thm_fsphO} et~\ref{thm_fsphSO}
avec ceux d\'ej\`a connus, 
ou ceux issus d'autres propri\'et\'es des fonctions sph\'eriques.
Nous obtiendrons des propri\'et\'es du sous-laplacien de Kohn.

\subsection{Repr\'esentation sur $N_{v,2}$}
\label{subsec_rq_repN}

Dans la section pr\'ec\'edente, 
nous avons caract\'eris\'e les fonctions sph\'eriques born\'ees 
des paires $(N_{v,2},O(v))$ et $(N_{v,2},SO(v))$
gr\^ace aux repr\'esentations 
des groupes $O(v)\triangleleft N_{v,2}$ et $SO(v)\triangleleft N_{v,2}$.
On peut aussi le faire
gr\^ace aux repr\'esentations sur $N=N_{v,2}$ 
\cite[theorem~G]{pgelf}. 
\index{Fonction sph\'erique!et repr\'esentation}
Rappelons bri\`evement une partie de ce r\'esultat.
Pour une repr\'esentation $(\Hc,\Pi)$ irr\'eductible de $N$, 
on d\'efinit pour un \'el\'ement $k$ du stabilisateur $K_\Pi$ dans $K$, 
l'op\'erateur d'entrelacement $W_\Pi$,
donn\'e (\`a une constante complexe de module 1 pr\`es) par : 
$W_\Pi(k)\circ \Pi(n)=\Pi(k.n)\circ W_\Pi(k)$.
On obtient ainsi la repr\'esentation projective
$W_\Pi:K_\Pi\mapsto \End \Hc$.
On d\'ecompose $\Hc=\sum_l V_l$ 
en une somme orthogonale de sous espaces irr\'eductibles invariants sous l'action de $W_\Pi$.
Les fonctions sph\'eriques sont donn\'ees par 
$$
\phi_{\Pi,\zeta}(n)
\,=\,
\int_K <\Pi(k.n).\zeta,\zeta> dk\quad,
$$
\index{Notation!Fonction sph\'erique!$\phi_{\Pi,\zeta}$}
lorsque $\Pi\in\hat{N}$ et $\zeta\in V_l$, $\nn{\zeta}=1$
(deux repr\'esentations $\Pi$ \'equivalentes  
ou deux vecteurs du m\^eme espace $V_l$
donnent la m\^eme fonction sph\'erique).
Au cours de la preuve de \cite[theorem 8.7]{pgelf} (et des lemmes qui le pr\'ec\`edent), 
on utilise que pour une fonction $f\in {L^1(N)}^\natural$,
$\Pi(f)$ pr\'eserve chaque sous espace $V_l$
et vaut  en restriction \`a ce sous espace,
\`a une constante pr\`es $\Id_{V_l}$
(lemme de Schur).
Cette constante vaut $<\phi_{\Pi,\zeta},f>$, $\zeta \in V_l,\nn{\zeta}=1$.
Cette propri\'et\'e est aussi vrai pour les mesures radiales de masses finies.
On a donc aussi pour $\zeta\in E_l$, $\zeta'\in E_{l'}$:
$$
\int_K <\Pi(k.n).\zeta,\zeta'> dk
\,=\,
\phi_{\Pi,\zeta}(n)<\zeta,\zeta'>
\quad,
$$
qui vaut 0, si $l\not=l'$.

\paragraph{Cas $\Lambda^*=0$.}
Chaque fonction sph\'erique $\phi^{r^*,0}$ 
est trivialement associ\'ee par \cite[theorem~G]{pgelf}
\`a  la repr\'esentation de $N$ 
de dimension 1 donn\'ee par le caract\`ere :
$$
n=\exp(X+A)
\,\longmapsto\,
\exp(ir^*<X^*_v,X>)
\quad.
$$
\paragraph{Cas $\Lambda^*\not=0$.}
Soient $(r^*,\Lambda^*)\in \Qc$
avec $\Lambda^*\not =0$ et $l\in \Nb^{v_1}, \epsilon=\pm1,\emptyset$.
On consid\`ere la repr\'esentation
$\Pi=\Pi_{r^*,\Lambda^*,\epsilon}$  de $N$ 
\index{Notation!Repr\'esentation!$\Pi_{r^*,\Lambda^*,\epsilon}$}
sur l'espace $\Hc=L^2(\Rb^{v_0})$
donn\'ee pour  une fonction $f\in \Hc$
de la variable $(y_1,\ldots,y_{v_0})\in\Rb^{v_0}$,
et pour $n=\exp(X+A)\in N$ 
avec $X$ d\'ecompos\'e selon~(\ref{notx}) 
si $v_0=v'$ et $\epsilon=-1$ par :
\begin{eqnarray*}
  \Pi(n).f(y)
  &=&
  \exp \left(i r^*x_v+<D_2^\epsilon(\Lambda^*),A>\right) \\
  &&
  \exp \,i\left( 
    \sum_{j=1}^{v'-1}
    \frac {\lambda^*_j}2 x_{2j}x_{2j-1} +\sqrt{\lambda^*_j} x_{2j}y_j
    -\;\frac {\lambda^*_{v'} }2 x_{2v'}x_{2v'-1}
    -\sqrt{\lambda^*_{v'}} x_{2v'}y_{v'}\right)\\
  &&\quad
  f(y_1+\sqrt{\lambda^*_1}x_1,\ldots,y_{v'}+\sqrt{\lambda^*_{v'}}x_{2v'-1})
  \quad ,
\end{eqnarray*}
et sinon par :
\begin{eqnarray*}
  \Pi(n).f(y)
  \,=\,
  \exp \,i\left( r^*x_v+<D_2^\epsilon(\Lambda^*),A>+
    \sum_{j=1}^{v_0}
    \frac{\lambda^*_j}2 x_{2j}x_{2j-1}
    +\sqrt{\lambda^*_j} x_{2j}y_j
  \right)\\
  f(y_1+\sqrt{\lambda^*_1}x_1,\ldots,y_{v_0}+\sqrt{\lambda^*_{v_0}}x_{2v_0-1})
  \quad .
\end{eqnarray*}
Cette repr\'esentation $\Pi=\Pi_{r^*,\Lambda^*,\epsilon}$ est \'equivalente 
\`a $\rho_{r^*,\Lambda^*,\epsilon}$
gr\^ace \`a l'op\'erateur d'entrelacement :
$$
F\in \Hc_{r^*,\Lambda^*,\epsilon}
\,\longmapsto\,
f\in\Hc
\quad\mbox{avec}\quad
f(y_1,\ldots,y_{v'})
\,=\,
{(\overset{v_0}{\underset{j=1}{\Pi}}\lambda_j^*)}^{\frac14}
F(\sqrt{\lambda_1^*}y_1,\ldots, \sqrt{\lambda_{v_0}^*}y_{v_0}) 
\quad.
$$
La  repr\'esentation $\Pi$ est donc irr\'eductible;
on peut aussi le voir directement en s'inspirant du cas du groupe de Heisenberg
\cite{foll}.
Ainsi, c'est une repr\'esentation de $N$
associ\'ee par Kirillov, \`a l'orbite contenant $r^*X_v+D_2^\epsilon(\Lambda^*)$.
Nous avons  d\'ej\`a d\'ecrit
son noyau $\ker \Pi$ dans la remarque~\ref{rem_expression_noyau},
puis son stabilisateur $K_\Pi$ dans la sous-section~\ref{subsec_stab},
enfin le groupe quotient $N/\ker \Pi$ dans la section~\ref{subsec_N/C}.
Rappelons que le groupe $N/\ker \Pi$ est isomorphe par $\Psi_2$ au groupe de Heisenberg,
\`a une \'eventuel facteur euclidien pr\`es,
et que $K_\Pi$ est isomorphe par $\Psi_1$ au groupe $K(m,v_0,v_1)\subset U_{v_0}$.

La d\'ecomposition de $\Hc$ par $W_\Pi$ 
se ram\`ene au m\^eme probl\`eme sur le groupe de Heisenberg pour le groupe $K(m,v_0,v_1)$.
Ce dernier est connu: en effet, par exemple pour $K=U_{v_0}$,
et pour les repr\'esentations de Bargmann sur $\Hb^{v_0}$, 
ces espaces sont les espaces de polyn\^omes homog\`enes de degr\'e fix\'e \cite[ch.IV,III.2]{Far};
les repr\'esentations de Bargmann et de Schr\"odinger sont \'equivalentes
et leur op\'erateur d'entrelacement envoie le mon\^ome homog\`ene 
sur la fonction de Hermite tel que le degr\'e \'egale le param\`etre
(\`a une normalisation pr\`es) \cite[ch.IV,III.1]{Far}.

On obtient ainsi une autre construction des fonctions sph\'eriques born\'ees de $(N_{v,2}, SO(v))$,
dans la forme que nous donnons maintenant (mais que nous ne d\'emontrons pas).

Pour $l=(l_1,\ldots, l_{v_1})\in \Nb^{v_1}$, 
on note $E_l$ l'ensemble des $\alpha=(\alpha^1,\ldots,\alpha^{v_1})$
tels que $\alpha^j \in \Nb^{m_j}, \nn{\alpha^j}=l_j$ pour $j=1,\ldots, v_1$;
pour un tel $\alpha\in E_l$,
nous consid\'erons la fonction unitaire 
$\zeta_\alpha\in \Hc$ 
donn\'ee gr\^ace aux fonctions de Hermite (voir sous-section~\ref{subsec_fcn_hermiite}) 
par :
$$
\zeta_\alpha :
\left\{
  \begin{array}{rcl}
    \Rb^{v_0}
    &\longrightarrow&
    \Rb\\
    y_1,\ldots,y_{v_0}
    &\longmapsto&
    \Pi_{j=1}^{v_1}
    h_{\alpha^j}(y_{m'_{j-1}+1},\ldots,y_{m'_j}) 
  \end{array}
\right. 
\quad.
$$
Comme les fonctions de Hermite sur $\Rb$ 
forment une base hilbertienne de $L^2(\Rb)$, 
la famille $\zeta_\alpha,\alpha\in E_l,l\in\Nb^{v_1}$ 
est une base orthonorm\'ee de l'espace de Hilbert~$\Hc$. 
On peut montrer que les espaces $V_l$, engendr\'es par $\zeta_\alpha,\alpha\in E_l$ 
sont irr\'eductibles pour l'action de $K_\Pi$.

On sait donc d\'ej\`a que
pour $\alpha\in E_l$ et $\alpha'\in E_{l'}$,
la fonction
$$
n\,\longmapsto\, \int_{K} 
{<\Pi_{r^*,\Lambda^*,\epsilon}(k.n).\zeta_\alpha,\zeta_{\alpha'}>}_\Hc 
dk
\quad,
$$
est une fonction sph\'erique si $l=l'$ et $\alpha=\alpha'$, et nulle sinon.
On montre 
que cette fonction sph\'erique est 
$$
\phi_{\Pi,\zeta_\alpha}
\,=\,
\phi^{r^*,\Lambda^*,l,\epsilon}
\quad, \alpha\in E_l
\quad;
$$
pour cela, par exemple,
on peut consid\'erer le vecteur :
$$
\zeta_l
\,=\,
{(\card E_l)}^{-\frac12}
\sum_{\alpha\in E_l}
\zeta_{\alpha}
\,\in V_l
\quad,
$$
et utiliser les propri\'et\'es des fonctions de Hermite et de Laguerre.
En particulier, 
pour $\alpha\in E_l,\, \alpha'\in E_{l'}$ et $l,l'\in \Nb^{v_1}$,
on a :
\begin{equation}
  \label{formule_phi_Lambda}
  \int_{K} 
  {<\Pi_{r^*,\Lambda^*,\epsilon}(k.n).\zeta_\alpha,\zeta_{\alpha'}>}_\Hc 
  dk
  \,=\,
  \left\{\begin{array}{ll}
      \phi^{r^*,\Lambda^*,l,\epsilon}(n)&\quad\mbox{si}\; l=l',\, \alpha=\alpha'\quad,\\
      0&\quad\mbox{sinon}\quad.
    \end{array}\right.
\end{equation}
pour $K=SO(v)$ ou $O(v)$.

Soit une fonction  $f\in {L^1(N)}^{\natural}$.
Gr\^ace \`a ce qui a \'et\'e rappel\'e sur $\Pi(f)$ ou par calcul direct,
on a pour $\alpha \in E_l$ :
\begin{equation}
  \label{fcnrad_U_zeta_phi}
  {\Pi_{r^*,\Lambda^*,\epsilon}(f)}.\zeta_\alpha
  \,=\,
  <f,\phi^{r^*,\Lambda^*,l,\epsilon}> \zeta_\alpha
  \quad.
\end{equation}

\subsection{Sous-laplacien}
\label{subsec_prop_souslap_fourier}

Le \textbf{sous-laplacien ``de Kohn''}
\index{Sous-laplacien} 
est l'op\'erateur diff\'erentiel :
\index{Notation!Op\'erateur!$L$}
$$
L
\,:=\,
- \sum_{i=1}^v X_i^2
\quad.
$$
C'est un op\'erateur sous-elliptique 
(\`a coefficients analytiques) 
invariant par translation \`a gauche et par $O(v)$ et $SO(v)$;
donc les fonctions sph\'eriques (qui en sont des fonctions propres) sont analytiques 
(ce que l'on pouvait d\'ej\`a voir sur leurs expressions explicites).

Chaque repr\'esentation $\Pi=\Pi_{r^*,\Lambda^*,\epsilon}$ 
induit la repr\'esentation suivante, not\'ee $d\Pi$
de  l'alg\`ebre des op\'erateurs diff\'erentiels sur $N$ 
invariants \`a gauche 
sur l'espace des fonctions de Schwarz $\Sc(\Rb^{v_0})$ :
\begin{eqnarray*}
  j=1,\ldots,v_0
  \qquad
  d\Pi(X_{2j-1})
  &=&
  \sqrt{\lambda^*_j} \partial_{y_j}
  \qquad\mbox{et}\qquad
  d\Pi(X_{2j})
  \,=\,
  i\sqrt{\lambda^*_j} y_j
  \quad ,\\
  2v_0<j<v
  \qquad
  d\Pi(X_j)
  &=&
  0
  \qquad\mbox{et}\qquad
  d\Pi(X_v)
  \,=\,
  ir^* \Id
  \quad ,\\
  \forall i<j 
  \qquad
  d\Pi(X_{i,j})
  &=&
  \left\{
    \begin{array}{ll}
      i\sqrt{\lambda^*_{j'}} \Id &\mbox{si}\;(i,j)=(2j'-1,2j') , j'\not=v'\; , \\
      i\epsilon\sqrt{\lambda^*_{j'}} \Id &\mbox{si}\;(i,j)=(2v'-1,2v') \; , \\
      0&\mbox{sinon.} 
    \end{array}\right.  
\end{eqnarray*}

En particulier, pour le sous-laplacien $L$, on a :
$$
d\Pi(L)
\,=\,
{r^*}^{2} \Id
-\sum_{i=1}^{v_0}
\lambda^*_i\left(\partial_{y_i}^2
  -y_i^2\right)
\quad.
$$

Comme chaque fonction de Hermite-Weber 
$h_k,k\in \Nb$ sur $\Rb$
v\'erifie l'\'equation diff\'erentielle: 
$y''+(2k+1-x^2)y=0$, 
on calcule facilement pour $\alpha\in E_l$:
\begin{equation}
  \label{formule_dU_zeta}
  d\Pi(L).\zeta_\alpha
  \,=\,
  \left( \sum_{j=1}^{v_1}
    \lambda_j (2 l_j+m_j)  
    +{r^*}^{2}\right)\zeta_\alpha
  \quad.  
\end{equation}
Ainsi, les fonctions $\zeta_\alpha, \alpha\in E_l, l\in \Nb^{v_1}$ 
forment une base orthonorm\'ee 
de vecteurs propres de l'op\'erateur $d\Pi(L)$.

Gr\^ace aux \'egalit\'es~(\ref{formule_phi_Lambda}) 
et~(\ref{formule_dU_zeta}), 
on d\'eduit la valeur propre associ\'ee 
\`a une fonction sph\'erique born\'ee 
$\phi^{r^*,\Lambda^*,l,\epsilon}$ pour le sous-laplacien $L$ :
\index{Sous-laplacien!valeur propre} 
\begin{equation} \label{eg_valp_deltax}
  L. \phi^{r^*,\Lambda^*,l,\epsilon}
  \,=\,
  \left( \sum_{j=1}^{v_1}
    \lambda_j (2 l_j+m_j) 
    +{r^*}^{2}\right)
  \phi^{r^*,\Lambda^*,l,\epsilon}
  \quad .
\end{equation}

Cette \'egalit\'e a un sens pour
$\Lambda^*=0$,
en convenant toujours 
que si $\Lambda^*=0$,
$\phi^{r^*,\Lambda^*,l,\epsilon}$ signifie $\phi^{r^*,0}$ 
(et \'evidemment $\lambda^*_i=0$ et $\lambda_j=0$ aussi).
Sa preuve est directe lorsque l'on
consid\`ere les expressions des fonctions sph\'eriques born\'ees donn\'ees dans
le th\'eor\`eme~\ref{thm_fsphSO}.
En particulier pour $\epsilon=1$, 
on trouve les valeurs propres pour le sous laplacien
et les fonction sph\'eriques de la paire $N_{v,2},O(v)$

Dans le cas $\Lambda^*\not=0$, 
on a donn\'e une \'ebauche de preuve ci dessus.
On peut aussi faire
le calcul direct et utiliser des identit\'es
remarquables sur les fonctions de Laguerre,
ou encore  en utilisant 
notre construction des fonctions sph\'eriques born\'ees 
et ce que l'on a rappel\'e 
dans  la sous-section~\ref{subsec_fcnsph_heis} sur le groupe de Heisenberg.

\begin{cor}
  [Transform\'ee de Fourier du noyau de \mathversion{bold}{$m(L)$}]
  \label{cor_transf_noyau}
  \index{Sous-laplacien!noyau d'une fonction du-} 
  Soit $m\in\Sc(\Rb)$.
  Le noyau $M\in \Sc(N)$ de l'op\'erateur $m(L)$ est une fonction radiale
  dont on connait la transform\'ee de Fourier sph\'erique :
  $$
  \big< M,\phi^{r^*,\Lambda^*,l,\epsilon} \big>
  \,=\,
  m(\sum_{j=1}^{v_1}
  \lambda_j (2 l_j+m_j)  
  +{r^*}^{2})
  \quad,
  $$
  pour tout $(r^*,\Lambda)\in\Qc$, $l\in \Nb^{v_1}$ si
  $\Lambda^*\not=0$,
  et $\epsilon=\pm1,\emptyset$.
\end{cor}

\subsection{Autres op\'erateurs diff\'erentiels}

Une autre m\'ethode 
pour d\'eterminer les fonctions sph\'eriques, 
serait d'utiliser le th\'eor\`eme~\ref{thm_fcnsph_opdiff}, 
c'est \`a dire de consid\'erer les fonctions sph\'eriques 
comme fonctions propres communes des op\'erateurs
diff\'erentiels invariants \`a gauche et par $O(v)$ ou $SO(v)$. 
C'est ce que nous avions fait
pour les paires de Guelfand $(\Hb^{v_0},K(m;v_0;v_1))$.
Mais sur le groupe $N=N_{v,2}$, il n'est pas ais\'e de trouver 
des g\'en\'erateurs pour ces op\'erateurs.

Outre le sous-laplacien $L$,
on connait d'autres op\'erateurs diff\'erentiels 
invariants \`a gauche et par $K$.
Citons d'abord le laplacien du centre :
$$
\Delta_\Zc
\,=\,
-\sum_{i<j} X_{i,j}^2
\quad,
$$
dont on calcule facilement la valeur propre associ\'ee 
\`a une fonction sph\'erique :
\begin{equation} \label{eg_valp_deltaa}
  \Delta_\Zc . \phi^{r^*,\Lambda^*,l,\epsilon}
  \,=\,
  \left(\sum_{i=1}^{v'} {\lambda^*_i}^2 \right)
  \phi^{r^*,\Lambda^*,l,\epsilon}
  \quad.
\end{equation}
Il y a \'egalement les op\'erateurs diff\'erentiels 
associ\'es \`a polyn\^ome $K$-invariant 
en les coefficients de matrice et de vecteur;
par exemple, le polyn\^ome $P(A,X)=\nd{A.X}^2$,
et tous les polyn\^omes $P(A,X)=\nd{A^n.X}^2,n\in\Nb$.

On peut donner des exemples d'op\'erateurs~$D_P$ 
issus de polyn\^ome $P$ seulement en les coefficients de matrice
invariant sous l'action par conjugaison de $K$ sur les
matrices :
\begin{itemize}
\item avec le polyn\^ome $P(A)=\nd{A}^2=\tr {A{}^tA}$, 
  on retrouve le laplacien du centre $D_P=2\Delta_\Zc$;
\item si l'on consid\`ere  le polyn\^ome caract\'eristique d'une matrice $A$:
  $$
  X^p+c_{p-1}(A) X^{p-1}+ \ldots c_0(A)
  \quad .
  $$ 
  Le polyn\^ome $c_i$ est sym\'etrique en les coefficients de la matrice
  $A$ et $K$-invariant.
  On peut donc lui associer un op\'erateur diff\'erentiel sur $\Nc$
  invariant \`a gauche et sous $K$, 
  en identifiant une matrice antisym\'etrique $A$ de coefficient
  antisym\'etrique $A_{i<j}$ et l'\'el\'ement $\sum A_{i,j} X_{i,j}$ 
  du centre de $\Nc$.  

  Par exemple, en \'etendant les notations 
  $X_{i,j}=-X_{j,i}, i<j$ et $X_{i,i}=0$,
  on peut consid\'erer le polyn\^ome en les coefficients matriciels 
  $c_0(A)=\det A$; 
  L'op\'erateur diff\'erentiel associ\'e 
  est :
  $$
  D_{c_0}
  \,=\,
  \sum \epsilon(\sigma)\Pi_i X_{i,\sigma(j)}
  \quad,
  $$
  la somme \'etant sur toutes les bijections 
  $\sigma$ de $\{1,\ldots,p\}$
  et $\epsilon(\sigma)$
  d\'esignant la signature de $\sigma$.
  On connait pour cet op\'erateur les valeurs propres associ\'ees \`a chaque
  fonction sph\'erique :
  $$
  D_{c_0} \phi^{r^*,\Lambda^*,l,\epsilon}
  \,=\,
  \left\{  \begin{array}{ll}
      \overset{v'}{\underset{i=1}{\Pi}} {\lambda^*_i}^2 \phi^{r^*,\Lambda^*,l,\epsilon}
      &\mbox{si}\; v=2v'\; ,\\
      0&\mbox{si}\; v=2v'+1\; .
    \end{array}\right.
  $$
\end{itemize}

Avec les \'egalit\'es~(\ref{eg_valp_deltaa}) et~(\ref{eg_valp_deltax}),
on d\'emontre :

\begin{lem}[Une famille de ferm\'es de \mathversion{bold}{$\Omega$}]
  \label{lem_topo_grlib}
  L'ensemble des fonctions sph\'eriques born\'ees $\Omega$
  (pour $SO(v)$ ou $O(v)$), 
  muni  la  convergence uniforme sur  tout compact,
  s'identifie au spectre de l'alg\`ebre ${L^1(N)}^\natural$,
  muni de la topologie faible-*.
  Le sous-ensemble $\Omega^N, N>0$ 
  de $\Omega$ 
  form\'e :
  \begin{itemize}
  \item des fonctions sph\'eriques 
    $\phi^{r^*,0}$ avec $r^*\leq N$,
  \item des fonctions sph\'eriques 
    $\phi^{r^*,\Lambda^*,l,\epsilon}$
    avec $\Lambda^*\in\overline{\Lc}-\{0\}$
    et $\epsilon=\pm1,\emptyset$
    v\'erifiant :
    $$
    \sum_{i=1}^{v'}{\lambda^*_i}^2 \leq N^2
    \qquad\mbox{et}\qquad
    \sum_{j=1}^{v_1}
    \lambda_j (2 l_j+m_j)  
    +{r^*}^{2}\leq N^2
    \quad,
    $$
  \end{itemize}
  est ferm\'e dans l'ensemble $\Omega$. 
\end{lem}

De ce lemme, on d\'eduit facilement :
\begin{cor}[Identification des structures bor\'eliennes de
  \mathversion{bold}{$\Omega$}]
  \label{cor_identification_bor_omega}
  Lorsque l'on identifie les fonctions sph\'eriques et leurs param\`etres,
  les structures bor\'eliennes de l'ensemble (topologique)
  $\Omega$ 
  et des param\`etres 
  $(r^*,\Lambda^*)\in\Qc,l\in\Nb^{v_1}, \epsilon=\pm1,\emptyset$ 
  des fonctions sph\'eriques born\'ees pour $SO(v)$ ou $O(v)$,
  s'identifient \'egalement.
\end{cor}

\section{Mesure de Plancherel radiale}

Dans cette section,
nous explicitons la mesure de Plancherel radiale 
pour les fonctions sph\'eriques pour $O(v)$.
La mesure de Plancherel radiale
est la mesure pour laquelle on a la formule de Plancherel radiale:
\begin{equation}
  \label{formule_plancherel_radiale}
  \forall f\in {L^1_{loc}(N)}^\natural
  \quad :\quad
  \nd{f}_{L^2(N)}
  \,=\,
  \nd{<f, \phi >}_{L^2(dm(\phi))}
  \quad.
\end{equation}
\index{Mesure de Plancherel!radiale}
Nous avons d\'ej\`a donn\'e son expression dans 
le th\'eor\`eme~\ref{thm_intro_mesure_plancherel}.
Dans cette section,
apr\`es avoir redonn\'e l'\'enonc\'e de ce th\'eor\`eme,
nous le d\'emontrons.

\subsection{Expression de la mesure de Plancherel radiale}
\label{subsec_mes_pl}

On identifie les fonctions sph\'eriques et leurs param\`etres.
Le corollaire~\ref{cor_identification_bor_omega}
nous permet d'identifier aussi les structures bor\'eliennes.
Le th\'eor\`eme qui suit montre que la mesure de Plancherel radiale est support\'ee 
par les fonctions sph\'eriques $\phi^{r^*,\Lambda^*,l}$, 
dont les param\`etres $r^*,\Lambda^*,l$ sont dans l'ensemble 
$\Pc_v=\Pc$ donn\'e par :
\index{Notation!Ensemble de param\`etres!$\Pc$}
\index{Notation!Ensemble de fonctions sph\'eriques!$\Pc$}
\begin{eqnarray*}
  \Pc_{2v'}
  &=&
  \{ 
  (0,\Lambda^*,l)\; , \quad \Lambda^*\in \Lc\, ,\; l\in \Nb^{v'}
  \} \quad,
  \\
  \Pc_{2v'+1}
  &=&
  \{ 
  (r^*,\Lambda^*,l)\; , \quad 
  r^*\in \Rb^+\, ,\; \Lambda^*\in \Lc\, ,\; l\in \Nb^{v'}
  \}  \quad,
\end{eqnarray*}
avec les notations du th\'eor\`eme~\ref{thm_fsphO}.

On note $\eta'$ la mesure sur le simplexe $\Lc$ donn\'ee par :
$$
d \eta'(\Lambda)
\,=\,
\Pi_j \lambda_j 
d \eta(\Lambda)
\quad, \quad \Lambda=(\lambda_1,\ldots,\lambda_{v'})\quad,
$$
o\`u $\eta$ est la mesure sur le simplexe $\Lc$ 
dont l'expression est donn\'ee dans le lemme~\ref{lem_pass_coord_pol},

On note \'egalement $m$ sur $\Pc$,
produit  tensoriel : 
\index{Notation!Mesure!$m$}
\begin{itemize}
\item de la mesure $\eta'$,
  \index{Notation!Mesure!$\eta'$}
\item de la mesure $\sum$ de comptage sur $\Nb^{v'}$,
\item et si $v=2v'+1$ de la mesure de Lebesgue $dr^*$ sur $\Rb^+$,
\end{itemize}
\`a la constante de normalisation pr\`es :
$$
c(v)
\,=\,
\left\{
  \begin{array}{ll}
    {(2\pi)}^{-\frac{v(v-1)}2-v'}
    &\quad
    \mbox{si}\; v=2v' \; ,\\
    2 {(2\pi)}^{-\frac{v(v-1)}2-1-v'}
    &\quad
    \mbox{si}\; v=2v'+1 \; .
  \end{array}\right.
$$
Dans le cas $v=2v'$,
on note la mesure :
$$
dm(\Lambda^*,l)
\,=\,
dm(\phi^{0,\Lambda^*,l})
\,=\,
c(v)  d\eta' \otimes \sum
\quad ,
$$
et dans le cas $v=2v'+1$ :
$$
dm(r^*,\Lambda^*,l)
\,=\,
dm(\phi^{r^*,\Lambda^*,l})
\,=\,
c(v) d\eta' \otimes \sum \otimes dr^*
\quad .
$$ 

Rappelons l'\'enonc\'e 
du th\'eor\`eme~\ref{thm_intro_mesure_plancherel} :

\begin{thm_princ}[Mesure de Plancherel radiale]
  \label{thm_mplanch}
  \index{Mesure de Plancherel!radiale}
  La mesure de Plancherel est support\'ee par $\Pc$ et s'identifie \`a $m$.
\end{thm_princ}

Dor\'enavant,
on confond les fonctions sph\'eriques $\phi$ 
et leurs param\`etres :
on \'ecrira  pour une fonction $g:\Pc\rightarrow \Cb$,
$g(\phi)$ ou $g(r^*,\Lambda^*,l)$
pour $g(r^*,\Lambda^*,l)$ lorsque
$\phi=\phi^{r^*,\Lambda^*,l}$.

Avant de passer \`a la d\'emonstration de ce th\'eor\`eme,
nous donnons quelques notations et propri\'t\'es 
concernant la transform\'ee de Fourier euclidienne.
Nous utiliserons les notations suivantes 
pour la transform\'ee de Fourier d'une fonction $f:\Rb^d\rightarrow \Cb$ :
$$
\Fc_{\zeta/x}.f= \int_{\Rb^d} f(x)e^{-i x.\zeta} dx
\quad\mbox{et}\quad
\check{\Fc}_{\zeta/x}.f= \int_{\Rb^d} f(x)e^{i x.\zeta} dx
\quad.
$$
On identifie $\Vc,\Vc^*$ et $\Rb^v$,
ainsi que $\Zc,\Zc^*$ et $\Rb^z,z=v(v-1)/2$; 
cela nous permet de consid\'erer les transform\'ees de Fourier
euclidienne en ces variables.

On remarque :
\begin{lem}
  \label{lem_fourier_A_fcnradiale}
  Soit $f\in {\Sc(N)}^\natural$. 
  Soient $A^*\in \Ac_v$ et $k\in O(v)$ tels que $k.A^*=A^*$.
  On a :
  $$
  \forall\, X\in \Vc\sim\Rb\qquad
  \Fc_{A^*/A}.\left[f(\exp (X+A))\right]
  \,=\,
  \Fc_{A^*/A}.\left[f(\exp (X+k^{-1}.A))\right]
  \quad.
  $$  
  On a aussi pour $A^*\in \Ac_v$ et $k\in O(v)$ :
  $$
  \int_{\Rb^v}\Fc_{A^*/A}.\left[f(\exp (X+A))\right]dX
  \,=\,
  \int_{\Rb^v}\Fc_{k.A^*/A}.\left[f(\exp (X+A))\right]dX
  \quad.
  $$  
\end{lem}

\begin{proof}[du lemme~\ref{lem_fourier_A_fcnradiale}]
  Pour la premi\`ere \'egalit\'e,
  effectuons le changement de variable $A'=k^{-1}.A$ :
  \begin{eqnarray*}
    \Fc_{A^*/A}.\left[f(k.\exp (X+A))\right]
    &=&
    \int_{\Ac_v}
    f(\exp (X+k^{-1}.A))
    e^{-i<A^*,A>} dA\\
    &=&
    \int_{\Ac_v}
    f(\exp (X+A'))
    e^{-i<A^*,k.A'>} dA'
    \quad.
  \end{eqnarray*}
  On a d'une part, par dualit\'e :$<A^*,k.A>=<k^{-1}.A^*,A>$,
  d'autre part $k^{-1}.A^*=A^*$
  car $A^*$ et $k$ commute par hypoth\`ese.

  Pour la seconde \'egalit\'e, 
  commen\c cons par remarquer :
  \begin{eqnarray*}
    &&\Fc_{k.A^*/A}.\left[f(\exp (X+A))\right]
    \,=\,
    \int_{\Ac_v}
    f(\exp (X+A))
    e^{-i<k.A^*,A>} dA\\
    &&\quad=\,
    \int_{\Ac_v}
    f(\exp (X+A'))
    e^{-i<A^*,k^{-1}.A>} dA
    \,=\,
    \int_{\Ac_v}
    f(\exp (X+k.A'))
    e^{-i<A^*,A'>} dA'\quad,
  \end{eqnarray*}
  apr\`es le changement de variable $A'=k^{-1}.A$.
  Maintenant, comme $f$ est radiale, on voit :
  $f(\exp (X+k.A'))=
  f(\exp (k^{-1}.X+A'))$; 
  on en d\'eduit :
  $$
  \Fc_{k.A^*/A}.\left[f(\exp (X+A))\right]
  \,=\,
  \Fc_{A^*/A}.\left[f(\exp (X+k.A))\right]
  \,=\,
  \Fc_{A^*/A}.\left[f(\exp (k^{-1}.X+A))\right]
  \quad,
  $$
  puis :
  \begin{eqnarray*}
    \int_{\Rb^v}\Fc_{k.A^*/A}.\left[f(\exp (X+A))\right]dX
    &=&
    \int_{\Rb^v}\Fc_{A^*/A}.\left[f(\exp (k^{-1}.X+A))\right]dX\\
    &=&
    \int_{\Rb^v}\Fc_{A^*/A}.\left[f(\exp (X'+A))\right]dX'
  \end{eqnarray*}
  apr\`es le changement de variable $X'=k^{-1}.X$.
\end{proof}

\paragraph{D\'emontrons le th\'eor\`eme~\ref{thm_mplanch}.}
On d\'efinit les sous groupes $K_i, i=1,\ldots,v_0$ 
de $K$ de la mani\`ere suivante : 
$K_i$ est l'ensemble des \'el\'ements $k\in SO(v)$
tels que $k.X_j=X_j$ pour tout $j\not= 2i-1,2i$,
qui laissent stable le plan 
$P_i:=\Rb X_{2i-1}\oplus \Rb X_{2i}$.
On note $dk_i$ la mesure de Haar (de masse 1)
du groupe compact $K_i$, $i=1,\ldots,v_0$.
Par restriction \`a $P_i$,
les \'el\'ements de $K_i$ 
sont des transformations orthogonales directes 
du plan $P_i$, 
et on a le passage en coordonn\'ees polaires :
\begin{equation}
  \label{formule_changement_polaire_Pi}
  \int_{P_i} f(x_{2i-1}X_{2i-1}+x_{2i}X_{2i}) dx_{2i-1}dx_{2i}
  \,=\,
  2\pi 
  \int_{K_i} \int_0^\infty f(r_i k_i.X_{2i-1}) r_i dr_i dk_i
  \quad.  
\end{equation}
On note $\tilde{K}=K_1\times\ldots\times K_{v'}$;
c'est un groupe compact, 
dont on note $d\tilde{k}=dk_1\ldots dk_{v'}$ 
la mesure de Haar de masse 1.
On d\'efinit 
la mesure sur ${\Rb^+}^{v'}$ :
$\tilde{r}d\tilde{r}
= r_1dr_1\ldots r_{v'}dr_{v'}$.

Par approximation, 
pour montrer le th\'eor\`eme~\ref{thm_mplanch},
il suffit de montrer la 
formule~(\ref{formule_plancherel_radiale})
pour la classe de fonctions ${\Sc(N)}^\natural$.
Soit donc $f\in{\Sc(N)}^\natural$.

Soit $\phi=\phi^{r^*,\Lambda^*,l}$ 
une fonction sp\'erique telle que $\Lambda^*\in\Lc$
et $r^*=0$ si $v=2v'$.
D'apr\`es son expression
donn\'ee dans le th\'eor\`eme~\ref{thm_fsphO},
on a :
\begin{eqnarray*}
  <f,\phi>
  &=&
  \int_{\Rb^v}\int_{\Ac_v}
  f(\exp (X+A))
  e^{i r^* x_v}  e^{i<D_2(\Lambda^*),A>}
  \overset{v'}{\underset{j=1}{\Pi}}
  \flz {l_j} {\frac{\lambda_j}2 (x_{2j-1}^2 +x_{2j}^2)}
  dAdX\\
  &=&
  \int_{\Rb}\int_{\Ac_v}
  \int_{P_{v'}}\ldots\int_{P_{v_1}}
  f\circ\exp \left(\sum_{i=1}^{v'}
    x_{2i-1}X_{2i-1}+x_{2i}X_{2i}
    \; +x_{2v'+1}X_{2v'+1}+A\right)\\
  &&\quad  e^{i r^* x_v}  e^{i<D_2(\Lambda^*),A>}
  \overset{v'}{\underset{j=1}{\Pi}}
  \flz {l_j} {\frac{\lambda_j}2 (x_{2j-1}^2+x_{2j}^2)}
  dx_1\ldots dx_{2v'}
  \; dAdx_{2v'+1}
  \quad.
\end{eqnarray*}
Appliquons aux $v'$ plans  $P_i, i=1,\ldots, v'$
la formule~(\ref{formule_changement_polaire_Pi}).
On obtient :
\begin{eqnarray*}
  <f,\phi>
  \,=\,
  {(2\pi)}^{v'}
  \int_{\Rb}\int_{\Ac_v}
  \int_{{\Rb^+}^{v'}}\int_{\tilde{K}}
  f\circ\exp \left(\sum_{i=1}^{v'}
    r_ik_i.X_{2i-1}
    \;+ x_{2v'+1}X_{2v'+1}+A\right)\\
  \overset{v'}{\underset{j=1}{\Pi}}
  \flz {l_j} {\frac{\lambda_j}2 r_j^2}
  d\tilde{k}\,\tilde{r}d\tilde{r}\,
  e^{i r^* x_v}  e^{i<D_2(\Lambda^*),A>}
  \; dAdx_{2v'+1}\quad.   
\end{eqnarray*}
On a donc :
\begin{eqnarray*}
  <f,\phi>
  \,=\,
  {(2\pi)}^{v'}
  \int_{{\Rb^+}^{v'}}\int_{\tilde{K}}
  F(\tilde{r},\Lambda^*,r^*,\tilde{k}) 
  \overset{v'}{\underset{j=1}{\Pi}}
  \flz {l_j} {\frac{\lambda_j}2 r_j^2}
  d\tilde{k}\,\tilde{r}d\tilde{r}
  \quad,
\end{eqnarray*}
o\`u on a pos\'e pour 
$\tilde{r}=(r_1,\ldots,r_{v'})\in {\Rb^+}^{v'}, 
\Lambda^*\in \Lc,r^*\in \Rb^+$  et 
$\tilde{k}=(k_1,\ldots, k_{v'})\in\tilde{K}$ :
\begin{eqnarray*}
  &&F(\tilde{r},\Lambda^*,r^*,\tilde{k})\\
  &&\, :=\,
  \Fc_{r^*/x_{2v'+1}}\,.\,\Fc_{D_2(\Lambda^*)/A}\,.\,
  \left[f\circ\exp \left( \sum_{i=1}^{v'}
      r_ik_i.X_{2i-1}
      + x_{2v'+1}X_{2v'+1} +A\right)\right] \;.
\end{eqnarray*}

Montrons que l'expression 
$F(\tilde{r},\Lambda^*,r^*,\tilde{k})$ ne
d\'epend pas de $\tilde{k}=(k_1,\ldots,k_{v'})$ :
par d\'efintion des sous-groupes $K_i$,
on a :
\begin{eqnarray*}
  &&f\left(\exp \left( \sum_{i=1}^{v'}
      r_ik_i.X_{2i-1}
      \;+ x_{2v'+1}X_{2v'+1} +A\right)\right)\\
  &&\qquad=\,
  f\left(k_1\ldots k_{v'}.\exp \left( \sum_{i=1}^{v'}
      r_iX_{2i-1}
      \;+ x_{2v'+1}X_{2v'+1} +{(k_1\ldots k_{v'})}^{-1}.A\right)\right) \quad;\\
  &&\qquad=\,
  f\left(\exp \left( \sum_{i=1}^{v'}
      r_iX_{2i-1}
      \;+ x_{2v'+1}X_{2v'+1} +{\tilde{k}}^{-1}.A\right)\right) \quad,
\end{eqnarray*}
car la fonction $f$ est radiale.
Maintenant comme les matrices orthogonales directes du plan 
commutent avec $J$,
la matrice $D_2(\Lambda^*)$ 
commute avec tous les $k_i\in K_i, i=1,\ldots,v'$
donc avec~$\tilde{k}$.
Ainsi d'apr\`es la premi\`ere \'egalit\'e du
lemme~\ref{lem_fourier_A_fcnradiale} et celle qui pr\'ec\`ede,
on a :
\begin{eqnarray*}
  &&  \Fc_{D_2(\Lambda^*)/A}\,.\,
  \left[f\left(\exp\left(  \sum_{i=1}^{v'}
        r_ik_i.X_{2i-1}
        \;+ x_{2v'+1}X_{2v'+1} +A\right)\right)\right]\\
  &&\quad=\,
  \Fc_{D_2(\Lambda^*)/A}\,.\,
  \left[f\left(\exp\left(  \sum_{i=1}^{v'}
        r_iX_{2i-1}
        \;+ x_{2v'+1}X_{2v'+1} + A\right)\right)
  \right]\quad,
\end{eqnarray*}
puis :
\begin{eqnarray*}
  F(\tilde{r},\Lambda^*,r^*,\tilde{k})
  &=&
  \Fc_{r^*/x_{2v'+1}}\,.\,\Fc_{D_2(\Lambda^*)/A}\,.\,
  \left[f\left(\exp\left(  \sum_{i=1}^{v'}
        r_i X_{2i-1}
        \;+ x_{2v'+1}X_{2v'+1} +A\right)\right)\right]\\
  &:=&
  F(\tilde{r},\Lambda^*,r^*)
  \quad.
\end{eqnarray*}

Avec cette nouvelle notation, 
on a :
\begin{eqnarray*} 
  <f,\phi>
  &=&
  {(2\pi)}^{v'}
  \int_{{\Rb^+}^{v'}}\int_{\tilde{K}}
  F(\tilde{r},\Lambda^*,r^*)
  \overset{v'}{\underset{j=1}{\Pi}}
  \flz {l_j} {\frac{\lambda_j}2 r_j^2}
  d\tilde{k}\,\tilde{r}d\tilde{r}\\
  &=&
  {(2\pi)}^{v'}
  \int_{{\Rb^+}^{v'}}
  F(\tilde{r},\Lambda^*,r^*)
  \overset{v'}{\underset{j=1}{\Pi}}
  \flz {l_j} {\frac{\lambda_j}2 r_j^2}
  \tilde{r}d\tilde{r}
  \quad.
\end{eqnarray*}

Effectuons les changements de variables 
$r'_j=\lambda_j r_j^2/2, j=1,\ldots,v'$;
on a 
$$
dr' = dr'_1\ldots dr'_{v'}= 
\overset{v'}{\underset{j=1}{\Pi}} \lambda_j
\tilde{r}d\tilde{r}
\quad,
$$
d'o\`u :
\begin{equation}
  \label{eg_tildeF_fphi}
  <f,\phi>
  \,=\,
  {(2\pi)}^{v'}
  \overset{v'}{\underset{j=1}{\Pi}} \lambda_j^{-1}
  \int_{{\Rb^+}^{v'}} 
  \tilde{F}(r',\Lambda^*,r^*)
  \overset{v'}{\underset{j=1}{\Pi}}
  \flz {l_j}  {r'_j}\;
  dr'
  \quad,  
\end{equation}
o\`u on a pos\'e 
pour 
$r'=(r'_1,\ldots,r'_{v'})\in{\Rb^+}^{v'}$,
$\Lambda^*\in \Lc$, et $r^*\in \Rb^+$ :
$$
\tilde{F}(r',\Lambda^*,r^*)
\,:=\,
F(\sqrt{\frac{2r'_1}{\lambda_1}},\ldots, 
\sqrt{\frac{2r'_{v'}}{\lambda_{v'}}},\Lambda^*,r^*)
\quad.
$$

Or les fonctions 
$\Pi_{j=1}^{v'} \flsz {l_j}, l\in \Nb^{v'}$ 
forment une base orthonormale de l'espace $L^2(\Rb^{v'})$
(voir sous-section~\ref{subsec_prop_fcnlag_0}).
Donc on a :
$$
\int_{{\Rb^+}^{v'}} 
\nn{\tilde{F}(r',\Lambda^*,r^*)}^2 dr'
\,=\,
\sum_l
\nn{\int_{{\Rb^+}^{v'}} 
  \tilde{F}(r',\Lambda^*,r^*)
  \overset{v'}{\underset{j=1}{\Pi}}
  \flz {l_j}  {r'_j}\;
  dr'}^2 
\quad.
$$
Le membre de droite 
d'apr\`es l'\'egalit\'e~(\ref{eg_tildeF_fphi}),
vaut 
${({(2\pi)}^{-v'}\Pi_j \lambda_j)}^2 \sum_l \nn{<f,\phi>}^2$.\\
Le membre de gauche se calcule directement 
d'abord en effectuant les changements de variable 
$r'_j=\lambda_j r_j^2/2$,
$$
\int_{{\Rb^+}^{v'}} 
\nn{\tilde{F}(r',\Lambda^*,r^*)}^2 dr'
\,=\,
\overset{v'}{\underset{j=1}{\Pi}} \lambda_j
\int_{{\Rb^+}^{v'}} 
\nn{F(\tilde{r},\Lambda^*,r^*)}^2 
\tilde{r}d\tilde{r}
\quad,
$$
puis en remarquant que par d\'efinition de $F$, on a :
\begin{eqnarray*}
  &&  \int_{{\Rb^+}^{v'}} 
  \nn{F(r,\Lambda^*,r^*)}^2 r_1dr_1 \ldots r_{v'}dr_{v'}
  \,=\,
  \int_{{\Rb^+}^{v'}}\int_{\tilde{K}}
  \nn{F(\tilde{r},\Lambda^*,r^*,\tilde{k})}^2 
  d\tilde{k}\,\tilde{r}d\tilde{r} \\
  &&\,=\,
  \int_{\Rb^{2v'}}
  \nn{\Fc_{r^*/x_{2v'+1}}\Fc_{D_2(\Lambda^*)/A}.
    \left[f\circ\exp\left(  \sum_{i=1}^{v'}
        x_{2i-1}X_{2i-1} +x_{2i}X_{2i}
        + x_{2v'+1}X_{2v'+1} +A\right)\right]}^2\\
  &&\qquad\qquad\qquad   {(2\pi)}^{-v'} dx_1dx_2\ldots dx_{2v'-1}dx_{2v'}
  \quad.
\end{eqnarray*}
en ayant appliqu\'e aux $v'$ plans  $P_i, i=1,\ldots, v'$
la formule~(\ref{formule_changement_polaire_Pi}).

On en d\'eduit l'\'egalit\'e :
\begin{eqnarray*}
  &&{({(2\pi)}^{-v'}\Pi_j \lambda_j)}^2
  \sum_l \nn{<f,\phi>}^2\\
  &&\quad=\,
  {(2\pi)}^{-v'}\Pi_j \lambda_j
  \int_{\Rb^{2v'}}
  \nn{\Fc_{r^*/x_{2v'+1}}\Fc_{D_2(\Lambda^*)/A}.
    \left[f\circ\exp\left( \sum_{i=1}^v x_iX_i +A\right)\right]}^2 dx_1\ldots dx_{2v'}
  \;.
\end{eqnarray*}
Dans le cas $v=2v'$, on a :
$$
{(2\pi)}^{-v'}\Pi_j \lambda_j
\sum_l \nn{<f,\phi>}^2
\,=\,
\int_{\Rb^v}
\nn{\Fc_{D_2(\Lambda^*)/A}\,.\,
  \left[f(\exp (X +A))\right]}^2 dX
\quad.
$$
Dans le cas $v=2v'+1$,
par parit\'e pour la variable $r^*$,
apr\`es int\'egration contre $dr^*$,
d'apr\`es la formule  de Plancherel sur~$\Rb$,
on obtient :
$$
{(2\pi)}^{-v'}\Pi_j \lambda_j
2\int_{\Rb^+}
\sum_l \nn{<f,\phi>}^2 dr^*
\,=\,
2\pi
\int_{\Rb^v} \nn{\Fc_{D_2(\Lambda^*)/A}\,.\,
  \left[f(\exp(X +A))\right]}^2 dX
\quad.
$$

Appliquons la seconde \'egalit\'e 
du lemme~\ref{lem_fourier_A_fcnradiale} 
pour tout $k\in O(v)$:
\begin{eqnarray*}
  &&\int_{\Rb^v} \nn{\Fc_{D_2(\Lambda^*)/A}\,.\,
    \left[f(\exp(X +k.A))\right]}^2 dX
  \,=\,
  \int_{\Rb^v} \nn{\Fc_{k.D_2(\Lambda^*)/A}\,.\,
    \left[f(\exp(X +A))\right]}^2 dX  \\
  &&\quad=\,
  \int_{SO(v)}
  \int_{\Rb^v} 
  \nn{\Fc_{k.D_2(\Lambda^*)/A}\,.\,
    \left[f(\exp(X +A))\right]}^2 dX  dk 
  \quad;
\end{eqnarray*}
et donc en int\'egrant sur $\Lc$, il vient :
\begin{eqnarray*}
  &&\int_{\Lc}
  \int_{\Rb} 
  \sum_l \nn{<f,\phi>}^2 dr^*
  d\eta(\Lambda^*)\\
  &&\quad =\,
  2\pi
  \int_{ SO(v)} \int_{\Lc} \int_{\Rb^v} 
  \nn{\Fc_{k.D_2(\Lambda^*)/A}\,.\,
    \left[f(\exp( X +A))\right]}^2 
  dX \; d\eta(\Lambda^*)dk
  \quad,
\end{eqnarray*}
sachant que dans le cas paire, 
le terme $2\pi$ et l'int\'egrale contre $dr^*$ n'y sont pas.

Effectuons le passage en coordonn\'ees polaires 
sur $\Ac_v$ (voir section~\ref{sec_app_matrice_antisym}) :
le membre de droite de cette derni\`ere \'egalit\'e 
peut s'\'ecrire :
$$
\int_{\Ac_v}\int_{\Rb^v} 
\nn{\Fc_{A^*/A}\,.\,
  \left[f(\exp(X +A))\right]}^2 dXdA^*
\quad,
$$
ou encore
gr\^ace \`a la formule de Plancherel euclidienne :
$$ 
{(2\pi)}^{\frac{v(v-1)}2}
\int_{\Ac_v}\int_{\Rb^v} 
\nn{f(\exp(X+A))}^2 dX \; dA
\,=\,
{(2\pi)}^{\frac{v(v-1)}2 }
\nd{f}^2
\quad.  
$$

Finalement, on a obtenu :
$$
{(2\pi)}^{-v'}\int_{\Lc}
\int_{\Rb} 
\sum_l \nn{<f,\phi>}^2 dr^*
\Pi_j \lambda_jd\eta(\Lambda^*)
\, =\,
2.2\pi {(2\pi)}^{\frac{v(v-1)}2 }
\nd{f}^2
\quad,
$$
sachant que dans le cas paire, 
le terme $2.2\pi$ et l'int\'egrale contre $dr^*$ n'y sont pas.

Ceci ach\`eve la
d\'emonstration du th\'eor\`eme~\ref{thm_mplanch}.

\subsection{Inversion}
\label{subsec_inversion}

Pour toute fonction $f\in L^2(N)$ radiale, 
on d\'efinit au sens $L^2$ :
$$
f(n)
\,=\,
\int_{\Pc}
\bar{\phi}(n)
<f,\phi>
dm(\phi)
\quad,
$$
c'est-\`a-dire que
$f$ est la fonction de ${L^2(N)}^\natural$
qui correspond 
\`a la forme sesquilin\'eaire 
de l'espace de Hilbert ${L^2(N)}^\natural$ :
$$
h\in {L^2(N)}^\natural
\,\mapsto\,
\int_{\Pc} <h,\phi> g(\phi) dm(\phi)
\quad;
$$ 
les int\'egrales ci-dessus ont bien un sens, car on a :
\begin{itemize}
\item d'une part $g(\phi)\in L^2(dm(\phi))$ par hypoth\`ese, 
\item d'autre part $<h,\phi>\in L^2(dm(\phi))$ 
  d'apr\`es la formule~(\ref{formule_plancherel_radiale}) 
  de Plancherel,
\item d'o\`u $<h,\phi> g(\phi)\in L^1(dm(\phi))$.
\end{itemize}

R\'eciproquement,
soit $g$ une fonction sur $\Pc$.
Si $g\in L^2(m)$,
on d\'efinit au sens $L^2$ la fonction 
$f:N\rightarrow \Cb$ suivante :
$$
f(n)
\,=\,
\int_{\Pc}
\bar{\phi}(n)
g(\phi)
dm(\phi)
\quad.
$$
C'est une fonction radiale de carr\'ee int\'egrable, 
et sa transform\'ee de Fourier est donn\'ee par : 
$<f,\phi>=g(\phi)$.

On peut aussi d\'efinir les distributions radiales \`a partir des distributions sur $\Pc$.
En effet,
si $h$ est une fonction $C^\infty$ 
\`a support compact sur $N$ (on note $h\in C_0^\infty(N)$), 
alors la fonction $\phi\mapsto <h,\phi>$
est \`a d\'ecroissance rapide sur $\Pc$ 
(d'apr\`es la forme des fonctions $\phi$ et les propri\'et\'es des fonctions de Laguerre)
et on a : $<h,\phi>=<h^\natural,\phi>$
o\`u on a not\'e :
$$
h^\natural
:n\in N\,\longmapsto\, \int_{O(v)}h(k.n)dk
\quad.
$$
Pour $g$ une distribution de Schwarz sur $\Pc$,
en particulier pour $g\in L^\infty(\Pc)$,
on d\'efinit alors l'application lin\'eaire  
$$
f:
h\in C_0^\infty(N) 
\,\longmapsto \, 
\int_N <h^\natural,\phi> g(\phi) dm(\phi)
\quad;
$$
l'application $f$ est ainsi 
une distribution (radiale) sur $N$, 
dont la transform\'ee de Fourier est
$<f,\phi>=g$.

\subsection{Lien avec le cas non radial}
\label{subsec_mesure_planchrel_nonradiale}

Nous donnons ici  avec nos notations
la formule de Plancherel non radiale,
qui est d\'ej\`a connue \cite{stric} section~6.

\index{Notation!Mesure!$m'$}
On note $m'$ la mesure produit tensoriel des mesures suivantes :
\begin{itemize}
\item $dk^*$ la mesure de Haar sur $O(v)$,
\item $m$ sur $\Pc$.
\end{itemize}

\begin{thm}[Formule de Plancherel non radiale]
  \label{formule_Plancherel}
  \index{Mesure de Plancherel!non radiale}

  Soit $\psi\in\Sc(N)$.
  \begin{eqnarray*}
    \psi(0)
    &=&
    \int
    \tr {k^*.\Pi_{r^*,\Lambda^*}(\psi) }
    dm'(r^*,k^*,\Lambda^*)
    \quad,\\  
    \nd{\psi}_{L^2(N)}^2
    &=&
    \int
    \nd{k^*.\Pi_{r^*,\Lambda^*}(\psi)}_{HS}^2
    dm'(r^*,k^*,\Lambda^*)\quad,
  \end{eqnarray*}
  o\`u $\nd{.}_{HS}$ d\'esigne la norme de Hilbert-Schmitt 
  sur les op\'erateurs de l'espace de Hilbert $L^2(R^{v'})$,
  et o\`u on convient  dans le cas $v=2v'$, 
  $\Pi_{r^*,\Lambda^*}=\Pi_{0,\Lambda^*}$ 
  et d'omettre l'int\'egration contre $dr^*$.
\end{thm}

On pose pour une fonction $f$ sur $N$ :
$f^*(n)=\overline{f(n^{-1})}$.
On d\'eduit de ce th\'eor\`eme,
ainsi que de l'\'egalit\'e~(\ref{fcnrad_U_zeta_phi}) :

\begin{cor}
  \label{cor_mespl_nonrad}
  Soient $h_1,h_2$ $C^\infty$ 
  deux fonctions \`a support compact sur $N$,
  et $F\in {L^1}^\natural$
  On a :
  $$
  \nn{<F*h_1,h_2>}
  \,\leq\,
  \sup_{\phi\in\Pc}\nn{<F,\phi>}
  \; \nd{h_1}_{L^2(N)}
  \nd{h_2}_{L^2(N)}
  \quad,
  $$

\end{cor}

\begin{proof}
  On  d\'eduit  de l'\'egalit\'e~(\ref{fcnrad_U_zeta_phi}),
  que pour deux fonctions $F_1,h\in L^1(N)$
  avec  $F_1$ radiale, on a, 
  en notant $\Pi=\Pi_{r^*,\Lambda^*}$ et pour $\alpha\in E_l$ :
  \begin{eqnarray*}
    <k^*.\Pi(F_1*h).\zeta_\alpha,\zeta_\alpha>
    &=&
    <\Pi(F_1)k^*.\Pi(h).\zeta_\alpha,\zeta_\alpha>
    \,=\,
    <k^*.\Pi(h).\zeta_\alpha,{\Pi(F_1)}^*\zeta_\alpha>\\
    <k^*.\Pi(F_1*h).\zeta_\alpha,\zeta_\alpha>
    &=&
    <\bar{F},\phi>
    <k^*.\Pi(h).\zeta_\alpha,\zeta_\alpha>
    \quad.
  \end{eqnarray*}
  Gr\^ace \`a une approximation de l'unit\'e radiale,
  cette derni\`ere \'egalit\'e est encore valable 
  dans le cas $F_1=F$, 
  distribution radiale 
  dont la transform\'ee de Fourier radiale 
  est born\'ee sur $\Pc$. 

  Reprenons les notations du corollaire.
  D'apr\`es ce qui pr\'ec\`ede, on a,
  toujours en notant $\Pi=\Pi_{r^*,\Lambda^*}, \Lambda^*\in \Lc$ :
  $$
  <k^*.\Pi(F*h_1*h_2^*).\zeta_\alpha,\zeta_\alpha>
  \,=\,
  <\bar{F},\phi^{r^*,\Lambda^*,l,\epsilon}>
  <k^*.\Pi(h_1*h_2^*).\zeta_\alpha,\zeta_\alpha>
  \quad.
  $$
  Comme la famille $\zeta_\alpha,\alpha\in E_l,l\in \Nb^{v'}$ 
  est une base orthonorm\'ee de $L^2(\Rb^{v'})$,
  on a en utilisant l'\'egalit\'e ci-dessus :
  \begin{eqnarray*}
    &&\tr{ k^*.\Pi(F*h_1*h_2^*)}
    \,=\,
    \sum_{l\in \Nb^{v'}} 
    \sum_{\alpha_\in E_l}
    <k^*.\Pi(F*h_1*h_2^*)
    .\zeta_\alpha,\zeta_\alpha>\\
    &&\qquad=\,
    \sum_{l\in \Nb^{v'}} 
    <\bar{F},\phi^{r^*X_{2v'+1},\Lambda^*,l}>
    \sum_{\alpha_\in E_l}
    <k^*.\Pi_{r^*X_{2v'+1},\Lambda^*}(h_1*h_2^*)
    .\zeta_\alpha,\zeta_\alpha>
    \quad,
  \end{eqnarray*}
  d'o\`u :
  $$
  \nn{\tr{ k^*.\Pi(F*h_1*h_2^*)}}
  \,\leq\,
  \sup_{\phi\in\Pc}
  \nn{<F,\phi>} 
  \sum_{l\in \Nb^{v'}} 
  \sum_{\alpha_\in E_l}
  \nn{<k^*.\Pi(h_1*h_2^*).\zeta_\alpha,\zeta_\alpha>}
  \quad.
  $$
  Or on a par Cauchy-Schwarz dans $L^2(\Rb^{v'})$ : 
  \begin{eqnarray*}
    \nn{<k^*.\Pi(h_1*h_2^*).\zeta_\alpha,\zeta_\alpha>}
    &=&
    \nn{<k^*.\Pi(h_1).\zeta_\alpha,k^*.\Pi(h_2).\zeta_\alpha>}\\
    &\leq&
    \nd{k^*.\Pi(h_1).\zeta_\alpha}_{L^2(\Rb^{v'})}
    \nd{k^*.\Pi(h_2).\zeta_\alpha}_{L^2(\Rb^{v'})}
    \quad,
  \end{eqnarray*}
  d'o\`u :
  \begin{eqnarray*}
    &&\nn{\tr{ k^*.\Pi(F*h_1*h_2^*)}}
    \,\leq\,
    \sup_{\phi\in\Pc}\nn{<F,\phi>}
    \sum_{l\in \Nb^{v'}} \sum_{\alpha_\in E_l}
    \nn{<k^*.\Pi(h_1).\zeta_\alpha,k^*.\Pi(h_2).\zeta_\alpha>}\\
    &&\qquad\leq\,
    \sup_{\phi\in\Pc}\nn{<F,\phi>}
    \nd{k^*.\Pi(h_1)}_{HS}
    \nd{k^*.\Pi(h_2)}_{HS}
    \quad.
  \end{eqnarray*}
  Utilisons la formule de Plancherel non radiale :
  $$
  <F*h_1,h_2>
  \,=\,
  F*h_1*h_2^*(0)
  \,=\,
  \int
  \tr{ k^*.\Pi_{r^*,\Lambda^*}(F*h_1*h_2^*)}
  dm'(r^*,k^*,\Lambda^*)
  \;,
  $$
  puis l'in\'egalit\'e ci-dessus (en omettant les indices des
  repr\'esentations $\Pi$ et certains arguments de $m'$) :
  \begin{eqnarray*}
    &&\nn{<F*h_1,h_2>}
    \,\leq\,
    \int
    \sup_{\phi\in\Pc}\nn{<F,\phi>}
    \nd{k^*.\Pi(h_1)}_{HS}
    \nd{k^*.\Pi(h_2)}_{HS}
    dm'(r^*,k^*,\Lambda^*)\\
    &&\qquad\leq\,
    \sup_{\phi\in\Pc}\nn{<F,\phi>}
    {\left(\int \nd{k^*.\Pi(h_1)}_{HS}^2
        dm'\right)}^\frac12
    {\left(\int \nd{k^*.\Pi(h_2)}_{HS}^2 
        dm'\right)}^\frac12\\
    &&\qquad\qquad=\,
    \sup_{\phi\in\Pc}\nn{<F,\phi>}
    \nd{h_1}_{L^2(N)}
    \nd{h_2}_{L^2(N)}
    \quad,
  \end{eqnarray*}
  d'apr\`es Cauchy-Schwarz puis 
  le corollaire~\ref{formule_Plancherel}.  
\end{proof}


\chapter{Utilisation du calcul de Fourier radiale 
sur \mathversion{bold}{$N_{v,2}$}}
\label{chapitre_utilisation}

Gr\^ace au calcul de Fourier obtenu dans le chapitre pr\'ec\'edent,
nous controlons la norme $L^2$ des fonctions d'aires 
dans le cas du groupe nilpotent libre \`a deux pas,
et nous \'etudions le probl\`eme des multiplicateurs.

\section{Contr\^ole $L^2$ des fonctions d'aire pour  $N_{v,2}$}
\label{sec_hatS_grlib}

Le but de cette section est de d\'emontrer
le th\'eor\`eme \ref{thm_L2_fcnd'aire}.b).

Comme dans le cas d'un groupe de type H,
on pose pour une  fonction sph\'erique born\'ee $\omega\in\Omega$  de $N$ pour $O(v)$ :
$$
\hat{S}^j(\omega)
\,:=\,
\sqrt{\int_0^\infty 
  \nn{\partial_s^j< \mu_s,\omega>}^2
  s^{2j-1} ds}
\quad.
$$
\index{Notation!$\hat{S}^j(\omega)$}

Nous reprenons les m\^emes notations et conventions
pour le groupe $N=N_{v,2}$ que 
dans le chapitre~\ref{chapitre_fonction_maximale_spherique} 
et pour les fonctions sph\'eriques born\'ees 
$\phi=\phi^{r^*,\Lambda^*,l}\in \Pc$ 
que dans la sous-section~\ref{subsec_mes_pl}.
En particulier $z=v(v-1)/2$ et $v=2v'$ ou $2v'+1$.
On suppose $v'\geq 2$ dans toute cette section.
\index{Notation!Dimension!$v'$}
\index{Notation!Dimension!$z$}

Le th\'eor\`eme \ref{thm_L2_fcnd'aire}.b) sera d\'emontr\'e
une fois que l'on aura d\'emontr\'e 
les deux propositions suivantes :

\begin{prop}[\mathversion{bold}{$S^j$} et
  \mathversion{bold}{$\hat{S}^j$}]
  \label{prop_fcnd'aire_hat{S}N}
  Pour $j\in\Nb-\{0\}$, on a :
  $$
  \forall f\in L^2(N)\qquad
  \nd{S(f)}\,\leq\, \sup_{\phi\in \Pc}  \nn{\hat{S}^j(\phi)} \;\nd{f}
  \quad,
  $$
\end{prop}

\begin{prop}[\mathversion{bold}{$\nd{\hat{S}^j}_\infty$}]
  \label{prop_maj_hat{S}lambdaN}
  Si $v'\geq2$,
  il existe une constante $C>0$ telle que :
  $$
  \forall j \leq \frac{z-1}4,v+1\quad
  \forall \omega\in \Pc\qquad
  \hat{S}^j(\omega)\,\leq\, C\quad.
  $$
\end{prop}

Nous pourrions utiliser la m\^eme d\'emarche que sur le groupe de type~H;
c'est-\`a-dire utiliser une mesure spectrale $E$ abstraite 
sur le spectre de l'alg\`ebre commutative ${L^1}^\natural$,
en identifiant ce spectre \`a l'ensemble des fonctions sph\'eriques $\Omega$.
Nous serions alors amen\'es \`a faire des estimations 
sur $\Omega$ tout entier.
Nous allons plutot utiliser la mesure de Plancherel non radiale 
du th\'eor\`eme~\ref{formule_Plancherel}.
Cela nous permettra de ne demander des estimations de $\hat{S}^j$
que sur la partie $\Pc$ des fonctions sph\'eriques.

\subsection{Fonction d'aire et $\hat{S}^j(\omega)$}

Nous d\'emontrons ici la proposition~\ref{prop_fcnd'aire_hat{S}N}.
Dans cette sous-section, on utilise les notations des sous-sections~\ref{subsec_rq_repN},
\ref{subsec_mes_pl}, et \ref{subsec_mesure_planchrel_nonradiale}.
En particulier, en utilisant la formule de Plancherel non radiale 
donn\'ee dans le th\'eor\`eme~\ref{formule_Plancherel},
et la base orthonorm\'ee de $L^2(\Rb^{v'})$
$\zeta_\alpha,\alpha\in E_l,l\in \Nb^{v'}$
pour calculer la norme Hilbert-Schmidt des op\'erateurs,
on a pour toute fonction $f\in\Sc(N)$ :
\begin{equation}
  \label{formule_pl_HS}
\nd{f}_{L^2(N)}^2
\,=\,
\int
\sum_l \sum_{\alpha\in E_l} \nn{k^*.\Pi_{r^*,\Lambda^*}(f).\zeta_\alpha}^2
dm'(r^*,k^*,\Lambda^*)
\quad.
\end{equation}

La proposition~\ref{prop_fcnd'aire_hat{S}N} d\'ecoule du lemme suivant :

\begin{lem}[\mathversion{bold}{$\nd{S^j(f)}$}]
\label{lem_nd_Sjf}
Soit $f\in\Sc(N)$. On a :
$$
\nd{S^j(f)}_{L^2(N)}^2
\,=\,
\int
\sum_{l\in\Nb^{v'}} 
\nn{\hat{S}^j(\phi^{r^*,\Lambda^*,l})}^2
\sum_{\alpha\in E_l}
\nn{k^*.\Pi_{r^*,\Lambda^*}(f).\zeta_\alpha}^2
dm'(r^*,k^*,\Lambda^*)\quad.
$$
\end{lem}

En effet, admettons ce lemme. 
On a donc :
\begin{eqnarray*} 
\nd{S^j(f)}_{L^2(N)}^2
&\leq&
\int
\sum_{l\in\Nb^{v'}} 
\nn{\sup_{\phi\in\Pc}\hat{S}^j(\phi)}^2
\sum_{\alpha\in E_l}
\nn{k^*.\Pi_{r^*,\Lambda^*}(f).\zeta_\alpha}^2
dm'(r^*,k^*,\Lambda^*)\\
&&=\,
 \nn{\sup_{\phi\in\Pc}\hat{S}^j(\phi)}^2
\int
\sum_{l\in\Nb^{v'}} 
\sum_{\alpha\in E_l}
\nn{k^*.\Pi_{r^*,\Lambda^*}(f).\zeta_\alpha}^2
dm'(r^*,k^*,\Lambda^*)\\
&&=\,
\nn{\sup_{\phi\in\Pc}\hat{S}^j(\phi)}^2
\int
\nd{k^*.\Pi_{r^*,\Lambda^*}(f)}_{HS}^2
dm'(r^*,k^*,\Lambda^*)\\
&&=\,
\nn{\sup_{\phi\in\Pc}\hat{S}^j(\phi)}^2
\nd{f}_{L^2(N)}^2
\quad,
\end{eqnarray*}
gr\^ace \`a la formule de Plancherel~(\ref{formule_pl_HS}).

\paragraph{D\'emontrons le lemme~\ref{lem_nd_Sjf}.}
Par Fubini, on a :
$$
\nd{S^j(f)}_{L^2(N)}^2
\,=\,
\int_{s=0}^\infty 
\nd{\partial_s^j(f*\mu_s)}_{L^2(N)}^2
s^{2j-1}ds
\quad.
$$
Appliquons la formule de Plancherel~(\ref{formule_pl_HS})
\`a $\partial_s^j(f*\mu_s)\in \Sc(N)$;
$$
\nd{\partial_s^j(f*\mu_s)}_{L^2(N)}
\,=\,
\int
\sum_{l\in\Nb^{v'}} 
\sum_{\alpha\in E_l}
\nn{k^*.\Pi_{r^*,\Lambda^*}(\partial_s^j(f*\mu_s)).\zeta_\alpha}^2
dm'(r^*,k^*,\Lambda^*)
\quad.
$$

Appliquons le lemme suivant \`a $\Pi=k^*.\Pi_{r^*,\Lambda^*}$
et $\zeta=\zeta_\alpha$ :
\begin{lem}
\label{lem_simple_Pi}
Pour une repr\'esentation $(\Hc,\Pi)\in \hat{N}$ et un vecteur $\zeta\in \Hc$,
on a :
$$
\Pi(\partial_s^j(f*\mu_s)).\zeta
\,=\,
\partial_s^j.\left(\Pi(f)\Pi(\mu_s).\zeta\right)
 \quad.  
$$
\end{lem}
Ce lemme sera  d\'emontr\'e \`a la fin de cette sous-section.

Comme  $\mu_s$ est une mesure radiale de masse finie,
l'op\'erateur $k^*.\Pi_{r^*,\Lambda^*}(\mu_s)$ 
vaut d'une part toujours
$\Pi_{r^*,\Lambda^*}(\mu_s)$
et d'autre part en restriction \`a chaque sous espace $V_l$,
 l'identit\'e \`a la constante 
 $<\mu_s,\phi^{r^*,\Lambda^*,l}>$ pr\`es 
(voir sous-section~\ref{subsec_rq_repN}).
On en d\'eduit :
\begin{eqnarray*}
k^*.\Pi_{r^*,\Lambda^*}(f)
k^*.\Pi_{r^*,\Lambda^*}(\mu_s).\zeta_\alpha
&=&
k^*.\Pi_{r^*,\Lambda^*}(f)
\left(<\mu_s,\phi^{r^*,\Lambda^*,l}>\Id.\zeta_\alpha\right)\\
&=&
<\mu_s,\phi^{r^*,\Lambda^*,l}>
k^*.\Pi_{r^*,\Lambda^*}(f).\zeta_\alpha
\quad.
\end{eqnarray*}

Rassemblons les deux arguments pr\'ec\'edents; on a :
$$
k^*.\Pi_{r^*,\Lambda^*}(\partial_s^j(f*\mu_s)).\zeta_\alpha
\,=\,
\partial_s^j.<\mu_s,\phi^{r^*,\Lambda^*,l}>
k^*.\Pi_{r^*,\Lambda^*}(f).\zeta_\alpha
\quad,
$$
et 
$$
\sum_{\alpha\in E_l}
\nn{k^*.\Pi_{r^*,\Lambda^*}(\partial_s^j(f*\mu_s)).\zeta_\alpha}^2
\,=\,
 \nn{\partial_s^j.<\mu_s,\phi^{r^*,\Lambda^*,l}>}^2
\sum_{\alpha\in E_l}
\nn{k^*.\Pi_{r^*,\Lambda^*}(f).\zeta_\alpha}^2
\quad.
$$
Revenons \`a 
\begin{eqnarray*}
\nd{S^j(f)}_{L^2(N)}
&=&
\int_{s=0}^\infty 
\int
\sum_{l\in\Nb^{v'}} 
\sum_{\alpha\in E_l}
\nn{k^*.\Pi_{r^*,\Lambda^*}(\partial_s^j(f*\mu_s)).\zeta_\alpha}^2
dm'(r^*,k^*,\Lambda^*)
s^{2j-1}ds\\
&=&
\int_{s=0}^\infty 
\int
\sum_{l\in\Nb^{v'}} \nn{\partial_s^j.<\mu_s,\phi^{r^*,\Lambda^*,l}>}^2
\sum_{\alpha\in E_l}
\nn{k^*.\Pi_{r^*,\Lambda^*}(f).\zeta_\alpha}^2
dm'(r^*,k^*,\Lambda^*)\\
&&\qquad\qquad
s^{2j-1}ds\quad.
\end{eqnarray*}
En utilisant Fubini, 
on trouve :
\begin{eqnarray*}
\nd{S^j(f)}_{L^2(N)}
&=&
\int
\sum_{l\in\Nb^{v'}} 
\int_{s=0}^\infty \nn{\partial_s^j.<\mu_s,\phi^{r^*,\Lambda^*,l}>}^2s^{2j-1}ds\\
&&\qquad\qquad
\sum_{\alpha\in E_l}
\nn{k^*.\Pi_{r^*,\Lambda^*}(f).\zeta_\alpha}^2
dm'(r^*,k^*,\Lambda^*)\quad.
\end{eqnarray*}
Gr\^ace \`a la d\'efinition de $\hat{S}^j$,
on en d\'eduit le lemme~\ref{lem_nd_Sjf}.
Pour que la preuve soit compl\`ete,
il reste \`a montrer le  lemme~\ref{lem_simple_Pi}.

\begin{proof}[du lemme~\ref{lem_simple_Pi}]
Comme $\mu_s$ est une mesure de probabilit\'e,
$\Pi(\mu_s).\zeta$ est bien d\'efinie (par dualit\'e) comme vecteur de $\Hc$.

Comme $f*\mu_s$ et $\partial_s^j(f*\mu_s)\in \Sc(N)$, on voit 
pour un vecteur $\zeta'\in\Hc$
\begin{eqnarray*}
<\Pi(\partial_s^j(f*\mu_s)).\zeta,\zeta'>
&=&
\int_N 
\partial_s^j(f*\mu_s)(n) <\Pi(n^{-1}).\zeta,\zeta'>
dn\\
&=&
\partial_s^j
\int_N 
(f*\mu_s)(n) <\Pi(n).\zeta,\zeta'>
dn\quad.
\end{eqnarray*}  
On d\'efinit donc par dualit\'e le vecteur de $\Hc$ :
$$
\Pi(\partial_s^j(f*\mu_s)).\zeta
\,=\,
\partial_s^j
\int_N 
(f*\mu_s)(n) \Pi(n).\zeta
dn\quad.
$$

On peut ais\'ement calculer 
(car $f\in\Sc(N)$ et $\mu_s$ est une mesure de probabilit\'e \`a support compact) :
\begin{eqnarray*}
  \int_N 
(f*\mu_s)(n) \Pi(n^{-1}).\zeta
dn
&=&
\int_N \int_N f(n'^{-1}n) \Pi(n^{-1}).\zeta
dn d\mu_s(n') \\
&=&
\int_N \int_N f(n_1) 
\Pi(n_1^{-1}{n'}^{-1}).\zeta 
dn d\mu_s(n') \quad,
\end{eqnarray*}
apr\`es le changement de variable $n_1=n'^{-1}n$;
gr\^ace \`a $\Pi(n_1^{-1}{n'}^{-1})=\Pi(n_1^{-1})\Pi({n'}^{-1})$,
 on a finalement :
$$
  \int_N 
(f*\mu_s)(n) \Pi(n^{-1}).\zeta
dn
\,=\,
\int_N f(n_1) 
\Pi(n_1^{-1}).
\int_{S_1}
\Pi({n'}^{-1}).\zeta 
d\mu_s(n') \;dn
\quad.  
$$
\end{proof}

Ainsi, le lemme~\ref{lem_simple_Pi} est d\'emontr\'e,
donc  le lemme~\ref{lem_nd_Sjf}, 
puis la proposition~\ref{prop_fcnd'aire_hat{S}N} le sont \'egalement.

\subsection{Estimations de $\hat{S}^j$}

Cette sous-section est consacr\'ee \`a la d\'emonstration 
de la proposition~\ref{prop_maj_hat{S}lambdaN}.
Nous admettrons plusieurs lemmes,
qui seront d\'emontr\'es dans la sous-section suivante.

Fixons $\phi=\phi^{r^*,\Lambda^*,l}\in\Pc$,
et  $h\in\Nb-\{0\}$. On note $A^*=D_2(\Lambda^*)$.

\begin{lem}
  \label{lem_maj_partial_h_mu_s_phi}
  L'expression~$\partial_s^h<\mu_s,\phi>$ 
  peut s'\'ecrire comme une combinaison lin\'eaire sur 
  les deux couples d'entiers
  $g=(h_1,h_2)$ et $j=(j_1,j_2)$ ( et l'entier $\tilde{h}$ si $v$ est impaire )
  tels que :
  \begin{itemize}
  \item  $h_1+h_2=h$ si $v$ est paire,
et  $h_1+h_2+\tilde{h}=h$ si $v$ est impaire
(et dans ce cas, on note $\bar{h}=h_1+h_2$),
  \item et $0\leq j_i\leq h_i/2$, $i=1,2$,
  \end{itemize}
  de :
  \begin{eqnarray*}
    b^{g,h}(s)
    \,:=\,
    s^{d_1+d_2}
    \int_{r=0}^1 
    \check{b}^{\tilde{h},h'_1+j_1}(r,s)\quad 
    \fbs {\frac{z-2}2} ^{(h'_2+j_2)} (s^2\sqrt{1-r^4} \nn{A^*})\\
    {(\sqrt{1-r^4} \nn{A^*})}^{h'_2+j_2}
    \;  r^{v-1}{(1-r^4)}^{\frac{z-2}2}dr  
    \quad ,
  \end{eqnarray*}
  o\`u  on a not\'e :
  \begin{itemize}
  \item les entiers $d_i=d(j_i,h_i)$, $i=1,2$, 
    (d\'ecrits dans le lemme~\ref{lem_maj_partial_h_mu_s_phi}), 
  \item $h_i=2h'_i,2h'_i+1$ pour $i=1,2$,
  \end{itemize}
  ainsi que :
  $$
  \check{b}^{\tilde{h},n}(r,s)
  \,=\,
  \int_{S_1^{(v)}} 
  {(irr^*x_v)}^{\tilde{h}}
  e^{irsr^*x_v}
  \partial_{s'}^n
  {\left[
      \overset{v'}{\underset{j=1}  \Pi}
      \flz {l_j}  {\lambda_j^*s'r^2\frac{\nn{\pr j X}^2} 2}\right]}_{s'=s^2}
  d\tilde{\sigma}_v(X)
  \quad.
  $$
\end{lem}

Jusqu'\`a la fin de cette section,
nous gardons les notations de ce lemme, 
que nous d\'emontrerons dans la sous-section~\ref{subsec_dem_lem_technique}.
La proposition~\ref{prop_maj_hat{S}lambdaN}
est donc \'equivalente \`a la proposition suivante :

\begin{prop}
  \label{prop_maj_I_h}
  Les int\'egrales :
  $$
  I^{g,j}
  \,:=\,
  \int_0^\infty \nn{b^{g,j}(s)}^2s^{2h-1}ds
  $$
  sont born\'ees ind\'ependemment de $\Lambda^*,r^*,l$, 
  et ce pour tous les param\`etres 
  $g,j$
  comme dans le lemme~\ref{lem_maj_partial_h_mu_s_phi},
  tant que $h< v+1+3/4, (z-1)/4, (z-2)/2$.
\end{prop}

Le reste de cette sous-section est consacr\'ee \`a sa d\'emonstration,
d'abord dans le cas  $\overline{h}\not=0, r^*\not=0$ ou $r^*=0$,
puis dans le cas $\overline{h}=0, r^*\not=0$.

\subsubsection{Majorations des 
  \mathversion{bold}{$I^{\tilde{h},h'_1,j_1,h_2,j_2}$}
  si  \mathversion{bold}{$\overline{h}\not=0$} 
  ou \mathversion{bold}{$r^*=0$}}

Nous souhaitons montrer la proposition~\ref{prop_maj_I_h}
lorsque $(\overline{h}\not=0,r^*\not=0)$ ou $(r^*=0)$,
si $h<(z-2)/2$.
On utilise la majoration des int\'egrales
$\check{b}^{\tilde{h},n}(r,s)$,
donn\'ee dans le lemme :

\begin{lem}
  \label{lem_maj_check_b_h}  
  Les expressions 
  $\check{b}^{\tilde{h},n}(r,s)$
  pour $0\leq n\leq \tilde{h}$,
  sont major\'ees \`a une constante pr\`es 
  de $v,\tilde{h}$ par :
  $$
  {(\nn{A^*}r^2)}^n
  s^{-\tilde{h}}
  \quad \sum_{0\leq i\leq \tilde{h}}
  {(\nn{A^*}s^2r^2)}^i
  \quad .
  $$
\end{lem}

Admettons ce dernier lemme, et montrons 
la proposition~\ref{prop_maj_I_h}
lorsque $\overline{h}\not=0, r^*\not=0$ 
ou $r^*=0$.

Commen\c cons par majorer l'expression de
$  b^{g,j}(s)$
(voir lemme~\ref{lem_maj_partial_h_mu_s_phi})
par Cauchy-Schwarz :
\begin{eqnarray*}
  \nn{b^{g,j}(s)}^2
  &\leq&
  s^{2(d_1+d_2)}\nn{A^*}^{2(h'_2+j_2)}
  \int_0^1
  \nn{\check{b}^{\tilde{h},h'_1+j_1}(r,s)r^{v-1}{(1-r^4)}^{-\frac 14}}^2dr\\
  &&\qquad
  \int_0^1
  \nn{\fbs{\frac{z-2}2}^{(h'_2+j_2)}(s^2\sqrt{1-r^4} \nn{A^*})
    {(1-r^4)}^{\frac{z-2}2+\frac{h'_2+j_2}2+\frac 14}}^2dr
  \quad.
\end{eqnarray*}
Utilisons le  lemme~\ref{lem_maj_check_b_h}.
Si $r^*\not=0$,
l'expression 
$\nn{b^{g,j}(s)}^2$
est major\'ee \`a une constante pr\`es par :
\begin{eqnarray*}
  &&  s^{2(d_1+d_2)}\nn{A^*}^{2(h'_2+j_2)} 
  \int_{0}^1
  {\left({s^{-\tilde{h}}(\nn{A^*}r^2)}^{h'_1+j_1}
      {(\nn{A^*}s^2r^2)}^i
      r^{v-1}{(1-r^4)}^{-\frac 14}\right)}^2dr\\
  &&\qquad\qquad  \int_{0}^1
  \nn{\fbs{\frac{z-2}2}^{(h'_2+j_2)}(s^2\sqrt{1-r^4} \nn{A^*})
    {(1-r^4)}^{\frac{z-2}2+\frac{h'_2+j_2}2+\frac 14}}^2dr \\
  &&\quad  =\,
  \nn{A^*}^{2(h'_2+j_2+h'_1+j_1+i)} 
  s^{2(d_1+d_2)-2\tilde{h}+4i}
  \int_{0}^1
  r^{2(v-1+2i+2(h'_1+j_1))}{(1-r^4)}^{-\frac 12}dr\\
  &&\quad\qquad\qquad  \int_{0}^1
  \nn{\fbs{\frac{z-2}2}^{(h'_2+j_2)}(s^2\sqrt{1-r^4} \nn{A^*})
    {(1-r^4)}^{\frac{z-2}2+\frac{h'_2+j_2}2+\frac 14}}^2dr
  \quad.  
\end{eqnarray*}
Si $r^*=0$, on obtient la m\^eme expression avec $\tilde{h}=i=0$.
Comme  la premi\`ere int\'egrale en~$r$ est finie, 
$I^{g,j}$
est major\'ee \`a une constante pr\`es 
qui ne d\'epend que de $v$ et $h$ 
par le maximum sur $0\leq i\leq \tilde{h}$ des int\'egrales :
\begin{eqnarray*}
  J^{(i)}
  &:=&
  \nn{A^*}^{2(h'_2+j_2+h'_1+j_1+i)} 
  \int_{0}^\infty
  s^{2(d_1+d_2)-2\tilde{h}+4i}\\
  &&\quad  \int_{0}^1
  \nn{\fbs{\frac{z-2}2}^{(h'_2+j_2)}(s^2\sqrt{1-r^4} \nn{A^*})
    {(1-r^4)}^{\frac{z-2}2+\frac{h'_2+j_2}2+\frac 14}}^2dr
  \quad s^{2h-1}ds
  \quad.  
\end{eqnarray*}
Rassemblons les exposants de $s$ 
dans chaque int\'egrale $J^{(i)}$;
on utilise d'abord $-\tilde{h}+h=\overline{h}=h_1+h_2$,
puis l'\'egalit\'e~(\ref{egalite_d(j,h)+h}) :
$$
2(d_1+d_2)-2\tilde{h}+4i   +2h-1
\,=\,
2(d_1+h_1+d_2+h_2)+4i-1
\,=\,
4(h'_1+j_1+h'_2+j_2)+4i-1
\quad.
$$
Effectuons le changement de variable 
$s'=s^2 \sqrt{1-r^4} \nn{A^*}$ :
$$
J^{(i)}
\,=\,
\int_{s'=0}^\infty 
\nn{\fbs{\frac{z-2}2}^{(h'_2+j_2)}(s')}^2
{s'}^{h'_2+j_2+h'_1+j_1+2i-1}\frac{ds'}2
\int_{0}^1 {(1-r^4)}^{z-2+\frac 12-(h'_1+j_1+2i)}dr  
\quad.
$$
Dans la derni\`ere expression de $J^{(i)}$,
l'int\'egrale contre $dr$ est finie lorsque  pour tout $j_1$ et $i$ :
$$
z-2+\frac 12-(h'_1+j_1+2i) \,>\,-1
\quad;
$$
or on a :
$$
h'_1+j_1+2i\leq \overline{h}+2\tilde{h}
\,=\,
h+\tilde{h}\leq 2h 
\quad ;
$$
donc ces int\'egrales sont finies 
lorsque $2h<z-1/2$ (qui est toujours v\'erifi\'e).
L'exposant de~$s'$ dans l'int\'egrale contre $ds'$
dans la derni\`ere expression de $J^{(i)}$,
est :
$$
h'_2+j_2+h'_1+j_1+2i-1
\,=\,
d_1+d_2+\overline{h}+2i-1
\quad;
$$
Comme nous nous sommes plac\'es dans le cas o\`u $\overline{h}\not=0$,
cet exposant est positif ou nul. 
D'apr\`es le lemme~\ref{lem_maj_int_fb},
cette int\'egrale est finie lorsque
$h'_2+j_2+h'_1+j_1+2i-1 < 2(z-2)/2$;
or on a :
$$
h'_2+j_2+h'_1+j_1+2i\leq h_1+h_2+2\tilde{h}
\,=\,
\overline{h}+2\tilde{h}
\,=\,
h+\tilde{h} 
\,\leq\, 2h
\quad;
$$
donc l'int\'egrale contre $ds'$ 
sera finie lorsque $1-q+2h<-1$ ou encore $h<(z-2)/2$.

La proposition~\ref{prop_maj_I_h}
est donc d\'emontr\'ee 
dans les cas  $(\overline{h}\not=0,r^*\not=0)$ ou $(r^*=0)$,
car dans ce dernier cas $\overline{h}=h\geq1$.

\subsubsection{Majorations des 
  \mathversion{bold}{$I^{0,0}$}}

Nous souhaitons montrer la proposition~\ref{prop_maj_I_h}
lorsque $\overline{h}=0$ et $r^*\not=0$.
Nous supposons donc ici $r^*\not=0$ et $h=\tilde{h}$, 
ou encore $g=j=(0,0)=0$.
On utilise deux majorations de
$\nn{\check{b}^{h,0}(r,s)}$
donn\'ees dans le lemme qui suit,
ainsi qu'un lemme technique.

\begin{lem}
  \label{lem_maj_checkb_h0}
  Si $\Lambda^*\not=0,r^*\not=0$,
  l'expression $\check{b}^{h,0}(r,s)$ 
  est major\'ee \`a une constante pr\`es 
  d'une part par : ${(rr^*)}^h$, 
  d'autre part par :
  $$
  \frac1{{(rr^*s^2)}^h}
  \sum_{0\leq j\leq 2h}
  {(\nn{A^*}s^2r^2)}^j
  \quad .
  $$
\end{lem}

\begin{lem}
  \label{lem_maj_J_axh}
  Fixons $j\in \Nb-\{0\}$.
  On pose pour $x>0,a>0$ :
  $$
  J_{a,x;j}^{(i)}
  \,:=\,
  \int_{x^{-\frac1j}}^\infty
  {\left(\int_{0}^1 
      \frac{{(as^2r^2)}^i}
      {x r^j s^{2j} }    
      \nn{\fb{\frac{z-2}2}{s^2\sqrt{1-r^4} a}}
      r^{v-1}{(1-r^4)}^{\frac{z-2}2}dr  \right)}^2
  s^{2j-1}ds\quad .
  $$
  Ces int\'egrales pour $i=0,\ldots 2j$
  sont major\'ees ind\'ependemment de $a>0,x>0$
  lorsque $z>3/4$, $j<(z-1)/4$, et $j<v+1+3/4$.
\end{lem}

Admettons ces deux lemmes et montrons 
la proposition~\ref{prop_maj_I_h}
lorsque $\overline{h}=0$ et $r^*\not=0$,
si  $v>3$, $h< v+1+3/4,(z-1)/4$.

L'int\'egrale $I^{0,0}$
est la somme des deux int\'egrales :
\begin{eqnarray*}
  J^\infty
  &:=&
  \int_{0}^{1/r^*}\nn{b^{0,0}(s)}^2s^{2h-1}ds
  \quad ,\\
  J_0
  &:=&
  \int_{1/r^*}^\infty\nn{b^{0,0}(s)}^2s^{2h-1}ds
  \quad .
\end{eqnarray*}
Comme $\check{b}^{h,0}(r,s)$ 
est major\'ee par ${(rr^*)}^h$ \`a une constante pr\`es
(lemme~\ref{lem_maj_checkb_h0}),
et comme les fonctions de Bessel sont born\'ees par~$1$, 
l'expression 
$b^{0,0}(s)$ 
est born\'ee \`a une constante pr\`es par 
${r^*}^h$. 
On obtient donc que l'int\'egrale $J^\infty$ est major\'ee \`a une constante pr\`es par :
$$
\int_0^{1/r^*} {r^*}^{2h} s^{2h-1}ds
\,=\,
{r^*}^{2h}\frac{{(1/r^*)}^{2h}}{2h}
\,=\,
\frac 1{2h}
\quad .
$$

D'apr\`es la seconde majoration 
de $\check{b}^{h,0}(r,s)$ 
du lemme~\ref{lem_maj_checkb_h0},
l'expression 
$b^{0,0}(s)$ 
est major\'ee \`a une constante pr\`es 
par le maximum des int\'egrales
$$
\int_0^1 
\frac{{(\nn{A^*}s^2r^2)}^i}{{(r^* r s^2)}^h}
\nn{\fb {\frac{z-2}2}{s^2\sqrt{1-r^4} \nn{A^*}}}
r^{v-1}{(1-r^4)}^{\frac{z-2}2}dr  
\; , \quad
i=0,\ldots,2h
\quad .
$$
Ainsi, l'int\'egrale 
$J_0$
est major\'ee \`a une constante pr\`es 
qui ne d\'epend que de $v,h$ 
par le maximum sur 
$i=0,\ldots, 2h$ des int\'egrales
$J_{\nn{A^*},{r^*}^h,h}^{(i)}$;
d'apr\`es le lemme~\ref{lem_maj_J_axh},
$J_0$ 
est donc major\'e ind\'ependemment de  
$r^*,A^*,l$ lorsque
$h< (z-1)/4$ et $h<v+1+3/4$.

Ceci ach\`eve la d\'emonstration 
de la proposition~\ref{prop_maj_I_h}
dans le cas $r^*\not=0$ et $\overline{h}=0$.
Comme les autres cas 
ont \'et\'e pr\'ec\'edemment d\'emontr\'es,
la proposition~\ref{prop_maj_I_h}
est prouv\'ee 
sous r\'eserve que 
les lemmes~\ref{lem_maj_partial_h_mu_s_phi},
\ref{lem_maj_check_b_h},
\ref{lem_maj_checkb_h0},
et~\ref{lem_maj_J_axh} soient d\'emontr\'es.

\subsection{D\'emonstration des lemmes techniques}
\label{subsec_dem_lem_technique}

Dans cette sous-section, 
nous d\'emontrons les lemmes
techniques utilis\'ees dans la sous-section pr\'ec\'edente.

\subsubsection{D\'emonstration 
  du lemme~\ref{lem_maj_partial_h_mu_s_phi}}

Pour d\'emontrer le lemme~\ref{lem_maj_partial_h_mu_s_phi},
nous avons besoin de l'expression de $<\mu_s,\phi>$ :
\begin{lem}
  \label{lem_expr_mu_phi}
  Rappelons que pour $n\in\Nb-\{0\}$, 
  $\tilde{\sigma}_n$ d\'esigne la mesure de masse 1
  sur la sph\`ere euclidienne $S_1^{(n)}$ de $\Rb^n$.
  On a :
  \begin{eqnarray*}
    <\mu_s,\phi>
    \,=\,
    2\int_0^1
    \int_{S_1^{(v)}}
    e^{itr^*x_v}
    \overset{v'}{\underset{j=1}  \Pi}
    \flz {l_j} {\lambda^*_jt^2\frac{\nn{\pr j X }^2} 2}
    d\tilde{\sigma}_v(X)\\
    \fb {\frac{z-2}2} {s^2\sqrt{1-r^4} \nn{A^*}} \;
    r^{v-1}{(1-r^4)}^{\frac{z-2}2}dr
    \quad .  
  \end{eqnarray*}
\end{lem}

\begin{proof}[du lemme~\ref{lem_expr_mu_phi}]
  D'apr\`es l'expression de $\phi$ donn\'ee dans le
  th\'eor\`eme~\ref{thm_fsphO},
  et  comme la mesure $\mu_s$ est radiale,
  on a :
  \begin{eqnarray*}
    && <\mu_s,\phi>
    \,=\,
    <\mu_s,\Theta^{r^*,\Lambda^*,l}>
    \,=\,
    \int_{S_1} \Theta^{r^*,\Lambda^*,l} (s.n) d\mu(n)\\
    &&\quad=\,
    2\int_0^1 \int_{S_1^{(v)}}\int_{S_1^{(z)}}
    \Theta^{r^*,\Lambda^*,l}
    (srX,s^2\sqrt{(1-r^4)}A)
    d\tilde{\sigma}_z(A)d\tilde{\sigma}_v(X)
    r^{v-1}{(1-r^4)}^\frac{z-2}2dr
    \;,
  \end{eqnarray*}
  gr\^ace \`a  l'expression de $\mu$,
  proposition~\ref{prop_expression_mu}. 
  Par d\'efinition de $\Theta^{r^*,\Lambda^*,l}$
  th\'eor\`eme~\ref{thm_fsphO}, 
  il vient :
  \begin{eqnarray*}
    \int_{S_1^{(v)}}\int_{S_1^{(z)}}
    \Theta^{r^*,\Lambda^*,l}
    (srX,s^2\sqrt{(1-r^4)}A)
    d\tilde{\sigma}_z(A)d\tilde{\sigma}_v(X)\\
    =\,
    \int_{S_1^{(v)}}
    e^{i srr^*x_v}
    \overset{v'}{\underset{j=1}  \Pi}
    \flz {l_j} {\frac{\lambda^*_j}2 \nn{\pr j {srX} }^2} 
    d\tilde{\sigma}_v(X)\\
    \int_{S_1^{(z)}}
    e^{is^2\sqrt{(1-r^4)}<A^*,A>}
    d\tilde{\sigma}_z(A)
    \quad.
  \end{eqnarray*}
  Or la derni\`ere int\'egrale contre $\tilde{\sigma}_z$ vaut
  $\fb {\frac{z-2}2} {s^2\sqrt{1-r^4} \nn{A^*}}$,
  lemme~\ref{lem_fcn_bessel_intsph}.
\end{proof}

Nous pouvons maintenant 
d\'emontrer le  lemme~\ref{lem_maj_partial_h_mu_s_phi}.

L'expression :
$$
\partial_s^h\left[
  e^{irsr^*x_v}
  \overset{v'}{\underset{j=1}  \Pi}
  \flz {l_j} {\lambda_js^2r^2\frac{\nn{\pr j X}^2} 2}
  \fb {\frac{z-2}2}{s^2\sqrt{1-r^4} \nn{A^*}}\right]
\quad,
$$
est \'egale \`a une combinaison lin\'eaire 
sur $\tilde{h}+\overline{h}=h$ 
et $h_1+h_2=\overline{h}$ 
du produit des 3 fonctions :
$$
\partial_s^{\tilde{h}}\left[e^{irsr^*x_v>}\right]
\partial_s^{h_1}\left[
  \overset{v'}{\underset{j=1}  \Pi}
  \flz {l_j} {\lambda_js^2r^2\frac{\nn{\pr j X}^2} 2}\right]
\partial_s^{h_2}\left[\fb {\frac{z-2}2}{s^2\sqrt{1-r^4} \nn{A^*}}\right]
\quad .
$$
\begin{enumerate}
\item On calcule ais\'ement :
  $$
  \partial_s^{\tilde{h}}\left[e^{irsr^*x_v>}\right]
  \,=\,
  {(irr^*x_v>)}^{\tilde{h}}e^{irsr^*x_v}
  \quad.
  $$
\item D'apr\`es le lemme~\ref{lem_derivee_fcn_s2},
  comme fonction de $s^2$ d\'eriv\'ee $h_1$ fois, 
  l'expression :
  $$
  \partial_s^{h_1}\left[
    \overset{v'}{\underset{j=1}  \Pi}
    \flz {l_j} {\lambda_js^2r^2\frac{\nn{\pr j X}^2} 2}\right]
  \quad,
  $$
  est combinaison lin\'eaire
  sur $0\leq j_1\leq h_1/2$ de :
  $$
  s^{d_1}\partial_{s'}^{h'_1+j_1}{\left[
      \overset{v'}{\underset{j=1}  \Pi}
      \flz {l_j} {\lambda_js'r^2\frac{\nn{\pr j X}^2} 2}\right]}_{s'=s^2}
  \quad .
  $$ 
\item  D'apr\`es le lemme~\ref{lem_derivee_fcn_s2},
  comme fonction de $s^2$ d\'eriv\'ee $h_2$ fois, 
  l'expression :
  $$
  \partial_s^{h_2}\left[\fb {\frac{z-2}2}{s^2\sqrt{1-r^4} \nn{A^*}}\right]
  \quad,
  $$
  est combinaison lin\'eaire sur $0\leq j_2\leq h_2/2$ de :
  \begin{eqnarray*}
    &&       s^{d_2}\partial_{s'}^{h'_2+j_2}{\left[\fb {\frac{z-2}2}{s'\sqrt{1-r^4} \nn{A^*}}\right]}_{s'=s^2}\\
    &&\quad   =\,
    s^{d_2}{(\sqrt{1-r^4} \nn{A^*})}^{h'_2+j_2}
    \fbs {\frac{z-2}2}^{(h'_2+j_2)} (s^2\sqrt{1-r^4} \nn{A^*})
    \quad .
  \end{eqnarray*}
\end{enumerate}
On conclut gr\^ace \`a l'expression de $<\mu_s,\phi>$ donn\'ee dans le lemme~\ref{lem_expr_mu_phi}.

\subsubsection{D\'emonstration 
  du lemme~\ref{lem_maj_check_b_h} dans le cas 
  \mathversion{bold}{$r^*=0$}}

Dans le cas $r^*=0$, 
on a $h=\overline{h},\tilde{h}=0$ et :
$$
\check{b}^{0,n}(r,s)
=
\int_{S_1^{(v)}} 
\partial_{s'}^n
{\left[
    \overset{v'}{\underset{j=1}  \Pi}
    \flz {l_j} {\lambda_js'r^2\frac{\nn{\pr j X}^2} 2}\right]}_{s'=s^2}
d\tilde{\sigma}_v(X)
\quad.
$$

L'expression :
\begin{equation}
  \label{expression_der_s'_n_fcnlag}
  \partial_{s'}^{n} 
  \overset{v'}{\underset{j=1}  \Pi}
  \flz {l_j}
  {\lambda_js'r^2\frac{\nn{\pr j X}^2} 2}
  \quad,    
\end{equation}
est combinaison lin\'eaire de :
\begin{equation}
  \label{expression_cl_der_s'_n_fcnlag}
  \underset{j\in J}{\Pi}
  {(\lambda_jr^2\frac{\nn{\pr j X}^2} 2)}^{e_j}
  \flsz {l_j}^{(e_j)}(\lambda_js'r^2\frac{\nn{\pr j X}^2} 2)  
  \quad,    
\end{equation}
o\`u $J$ parcourt les sous-ensembles de $\Nb \cap[1,v']$  
tel que  $e_j\in\Nb$  et $\sum_{j\in J} e_j=n$.
Comme les fonctions de Laguerre et leurs d\'eriv\'ees 
sont born\'ees,
on en d\'eduit que 
l'expression~(\ref{expression_der_s'_n_fcnlag})
est major\'ee \`a une constante pr\`es de $n,v$
par la somme de 
$$
\nn{A^*}^n \underset {j\in J}\Pi 
{(r^2\frac{\nn{\pr j X}^2} 2)}^{e_j}
\,=\,
{(r^2 \nn{A^*})}^n
\underset{j\in J} \Pi {(\frac{\nn{\pr j X}^2} 2)}^{e_j}
\quad,
$$
puis que $\nn{  \check{b}^{0,n}(r,s)}$
est major\'ee \`a une constante pr\`es de $n,v$
par  ${(r^2 \nn{A^*})}^n$.

Le lemme~\ref{lem_maj_check_b_h} 
est donc d\'emontr\'e dans le cas $r^*=0$.

\subsubsection{D\'emonstration 
  du lemme~\ref{lem_maj_check_b_h} dans le cas 
  \mathversion{bold}{$r^*\not=0, \overline{h}\not=0$}}

La d\'emonstration du lemme~\ref{lem_maj_check_b_h}
repose sur $\tilde{h}$ int\'egrations par partie.
Avant de donner une d\'emonstration g\'en\'erale, 
nous allons d'abord montrer le cas particulier :

\begin{lem}
  \label{lem_maj_check_b_h_10}  
  L'expression 
  $\check{b}^{1,0}(r,s)$
  est major\'ee \`a une constante pr\`es 
  de $v$ par :
  $$
  \frac {1+(\nn{A^*}s^2r^2)}s
  \quad .
  $$
\end{lem}

\begin{proof}[de lemme~\ref{lem_maj_check_b_h_10}]
  Choisissons pour atlas de la sph\`ere $S_1^{(v)}$ euclidienne de $\Rb^v$
  les deux calottes sph\'eriques de p\^oles 
  $X_1$ et $-X_1$ :
  $$
  C_1
  \,:=\,
  \{ X\in S_1^{(v)} \, ,\; <X,X_1>\;>\; -\frac12\}
  \quad\mbox{et}\quad
  C_2
  \,:=\,
  \{ X\in S_1^{(v)} \, , \; <X,-X_1>\;>\; \frac12\}
  \quad.
  $$
  Fixons une partition de l'unit\'e de la sph\`ere pour cet atlas :
  c'est-\`a-dire deux fonctions $C^\infty$ 
  $\psi_1,\psi_2$ d\'efinies sur la sph\`ere $S_1^{(v)}$
  telles que :
  $$
  \supp \psi_i \subset C_i
  \quad\mbox{et}\quad
  0\leq \psi_i\leq 1 \; i=1,2
  \quad\mbox{et}\quad
  \psi_1+\psi_2 \,=\,1
  \quad\mbox{sur}\;S_1^{(v)}
  \quad.
  $$ 
  Comme carte pour $C_1$,
  on consid\`ere 
  la projection st\'er\'erographique 
  de p\^ole  $X_1$ 
  et on note $C'_1$ son image sur $\Rb^{v-1}$.
  De m\^eme, pour $C_2$,
  on consid\`ere la projection st\'er\'erographique 
  de p\^ole  $-X_1$ 
  et on note $C'_2$ son image sur $\Rb^{v-1}$.
  $C'_1$ et $C'_2$ sont compacts.
  Par ces changements de cartes, 
  les points $X=\sum_i x_i X_i$ de la sph\`ere 
  sont param\`etr\'es par les coordonn\'ees $x_i, i\not=1$;
  de plus, la mesure $\tilde{\sigma}_v$ 
  s'envoie sur une mesure de densit\'e
  contre la mesure de Lebesgue $dx$ 
  sur chaque domaine $C_i, i=1,2$.
  La densit\'e $D_i,i=1,2$ est $C^\infty$.

  On a donc en d\'ecomposant sur cet atlas
  puis en effectuant une int\'egration par~partie :
  \begin{eqnarray*}
    \check{b}^{1,0}(r,s)
    &=&
    \int_{S_1^{(v)}} 
    irr^*x_v
    e^{irsr^*x_v}
    {\left[
        \overset{v'}{\underset{j=1}  \Pi}
        \flz {l_j} {\lambda_js'r^2\frac{\nn{\pr j X}^2} 2}\right]}_{s'=s^2}
    d\tilde{\sigma}_v(X)\\
    &=&
    \int_{S_1^{(v)}} 
    \frac{x_v}s \partial_{x_v}e^{isrr^*x_v}
    \overset{v'}{\underset{j=1} \Pi} \flsz {l_j} d\tilde{\sigma}_v(X)\\
    &=&
    \sum_{i=1,2}
    \int_{C_i} 
    \frac{x_v}s  \partial_{x_v} e^{isrr^*x_v}
    \overset{v'}{\underset{j=1} \Pi} \flsz {l_j} 
    \psi_i(X)d\tilde{\sigma}_v(X)\\
    &=&
    \sum_{i=1,2}
    \int_{C'_i} 
    \frac1s e^{isrr^*x_v}
    \partial_{x_v} \left[\overset{v'}{\underset{j=1} \Pi} 
      \flsz {l_j} x_v \psi_i(X) D_i\right] dx 
    \quad .
  \end{eqnarray*}
  Nous avons omis l'argument des fonctions $\flsz {m_j}$, 
  lorsqu'il vaut $\lambda_j s^2r^2\frac{\nn{\pr j {X}}}2$, 
  avec $X$ fonction des param\`etres de la cartes $C_i$.
  On a donc :
  \begin{equation}
    \label{maj_debut_checkb_10}
    \nn{\check{b}^{1,0}(r,s)}
    \,\leq\,
    \frac1s
    \sum_{i=1,2}
    \int_{C'_i} 
    \nn{\partial_{x_v} 
      \left[\overset{v'}{\underset{j=1} \Pi} \flsz {l_j}
        x_v \psi_i(X) D_i\right]} dx 
    \quad .
  \end{equation}
  Estimons la d\'eriv\'ee :
  \begin{eqnarray*}
    \partial_{x_v} 
    \left[\overset{v'}{\underset{j=1} \Pi} \flsz {l_j} 
      x_v \psi_i(X) D_i\right]
    &=&
    \overset{v'}{\underset{j=1} \Pi} \flsz {l_j}
    \partial_{x_v} \left[x_v \psi_i(X) D_i\right]\\
    &&\quad+
    x_v \psi_i(X) D_i
    \partial_{x_v} 
    \left[\overset{v'}{\underset{j=1} \Pi} \flsz {l_j}
    \right]
    \quad.
  \end{eqnarray*}
  Comme les fonctions de Laguerre et leurs d\'eriv\'ees 
  sont born\'ees, 
  on a :
  \begin{eqnarray*}
    \nn{\overset{v'}{\underset{j=1} \Pi} \flsz {l_j} } 
    &\leq&
    C \quad,\\
    \partial_{x_v} \left[\overset{v'}{\underset{j=1} \Pi} \flsz {l_j} \right]
    &=&
    \sum_{j_0} \lambda_{j_0}s^2r^2 \partial_{x_v} \left[\frac{\nn{\pr {j_0} X }^2} 2\right]
    \flsz {l_{j_0}}  '
    \Pi_{j\not =j_0} \flsz {l_j}  \quad,\\
    \nn{\partial_{x_v} \left[\overset{v'}{\underset{j=1} \Pi} \flsz {l_j} \right]}
    &\leq&
    C  s^2r^2   
    \sum_{j_0} \lambda_{j_0} \nn{\partial_{x_v} \left[\frac{\nn{\pr {j_0} X}^2} 2\right]}\quad  ,
  \end{eqnarray*}
  o\`u $C$ est une constante de $v$.
  On en d\'eduit pour $i=1,2$ :
  \begin{eqnarray*}
    \int_{C'_i} 
    \nn{\partial_{x_v} 
      \left[\overset{v'}{\underset{j=1} \Pi} \flsz {l_j}  
        x_v \psi_i(X) D_i\right]} dx 
    \,\leq\,
    \int_{C'_i} 
    C \nn{\partial_{x_v} \left[x_v \psi_i(X) D_i\right]}
    dx\\
    +
    \sum_{j_0} \lambda_{j_0}     C  s^2r^2   \int_{C'_i}  
    \nn{\partial_{x_v} 
      \left[\frac{\nn{\pr {j_0} X}^2} 2\right]}
    x_v \psi_i(X) D_i dx
    \quad  .
  \end{eqnarray*}
  Or chaque domaine $C'_i$ compact;
  la fonction densit\'e $D_i$ est $C^\infty$;
  les int\'egrales
  des normes de $x_v \psi_i(X)D_i$ et de sa d\'eriv\'ee,
  multipli\'e par 1 
  ou par la  d\'eriv\'e de $\nn{\pr {j_0} X }^2$ 
  sont finies et born\'ees par une constante not\'ee encore $C$
  qui ne d\'epend que de $v$.
  On obtient donc :
  $$
  \int_{C'_i} 
  \nn{\partial_{x_v} 
    \left[\overset{v'}{\underset{j=1} \Pi} \flsz {l_j}
      x_v \psi_i(X) D_i\right]} dx 
  \,\leq\,
  C\left( 1+  s^2r^2 \sum_{j_0} \lambda_{j_0}  \right)\quad  .
  $$
  On conclut gr\^ace \`a la majoration ci dessus
  et $\nn{A^*}\sim \sum_{j_0} \lambda_{j_0}$ 
  ainsi que~(\ref{maj_debut_checkb_10}).
\end{proof}

Nous pouvons maintenant utiliser la m\'ethode
que nous venons de d\'ecrire
pour montrer le lemme~\ref{lem_maj_check_b_h}.

On a :
\begin{eqnarray*}
  \check{b}^{\tilde{h},n}(r,s)
  &=&
  \int_{S_1^{(v)}} \ldots
  d\tilde{\sigma}_v(X)
  \,=\,
  \sum_{i=1,2} \int_{C'_i} \ldots \psi_i D_i dx\\
  &=&
  \sum_{i=1,2} \int_{C'_i} 
  {(irx_v)}^{\tilde{h}}
  {(irs)}^{-\tilde{h}}
  \partial_{x_v}^{\tilde{h}}\left[e^{irsr^*x_v}\right]\\
  &&\quad
  \partial_{s'}^{n}{\left[\overset{v'}{\underset{j=1} \Pi} 
      \flz {l_j} {\lambda_js'r^2\frac{\nn{\pr j X}^2} 2}\right]}_{s'=s^2}
  \psi_i D_i dx
  \quad .
\end{eqnarray*}
Apr\`es $\tilde{h}$ int\'egrations par partie, on a :
\begin{eqnarray*}
  \check{b}^{\tilde{h},n}(r,s)
  &=&
  {(-s)}^{-\tilde{h}}
  \sum_{i=1,2} \int_{C'_i} 
  e^{irsr^*x_v}\\
  &&\qquad
  \partial_{x_v}^{\tilde{h}} \partial_{s'}^{n}{\left[
      \overset{v'}{\underset{j=1} \Pi} \flz {l_j}
      {\lambda_js'r^2\frac{\nn{\pr j X}^2} 2} x_v^{\tilde{h}} \psi_i D_i
    \right]}_{s'=s^2}   dx
  \quad .  
\end{eqnarray*}
Or l'expression~(\ref{expression_der_s'_n_fcnlag})
est combinaison lin\'eaire de~(\ref{expression_cl_der_s'_n_fcnlag})
o\`u $J$ parcourt les sous-ensembles de $\Nb \cap[1,v']$  tel que
$e_j\in\Nb$  et $\sum_{j\in J} e_j=n$;
puis l'expression :
\begin{equation}
  \label{expr_der_x_s'_fcnlag}
  \partial_{x_v}^{\tilde{h}} \partial_{s'}^{n}{\left[
      \overset{v'}{\underset{j=1} \Pi} \flz {l_j}
      {\lambda_js'r^2\frac{\nn{\pr j X}^2} 2} x_v^{\tilde{h}} \psi_i D_i
    \right]}_{s'=s^2} 
  \quad,  
\end{equation}
est combinaison des :  
$$
\partial_{x_v}^{\tilde{h}}.
\left[
  \overset{j\in J}{\Pi}{(\lambda_jr^2\frac{\nn{\pr j X}^2} 2)}^{e_j}
  \flsz {l_j} ^{(e_j)}(\lambda_js^2r^2\frac{\nn{\pr j X}^2} 2)   
  x_v^{\tilde{h}} \psi_i D_i
\right]
\quad;
$$
donc gr\^ace aux estimations des fonctions de
Laguerre et de ses d\'eriv\'ees,
l'expression~(\ref{expr_der_x_s'_fcnlag}) 
est major\'ee \`a une constante pr\`es 
qui ne d\'epend que de $v$ et $\tilde{h}$ par :
$$
{(\nn{A^*}r^2)}^n
\quad \sum_{0\leq j\leq \tilde{h}}
{(\nn{A^*}s^2r^2)}^j
\quad .
$$

Apr\`es int\'egration sur chaque domaine $C'_i$, 
puis sommation et
multiplication par $s^{-\tilde{h}}$,
on obtient la majoration voulue de 
$\nn{\check{b}^{\tilde{h},n}(r,s)}$.

Le lemme~\ref{lem_maj_check_b_h} 
est donc d\'emontr\'e dans le cas $r^*\not=0, \overline{h}\not=0$.

\subsubsection{D\'emonstration
  du lemme~\ref{lem_maj_checkb_h0}}

On a une majoration grossi\`ere de
$$
\check{b}^{h,0}(r,s)
=
\int_X 
{(irr^*x_v)}^{h} e^{irsr^*x_v}
\overset{v'}{\underset{j=1} \Pi} 
\flz {l_j} {\lambda_js^2r^2\frac{\nn{\pr  j X}^2} 2}
d\tilde{\sigma}_v(X)
\quad ,
$$
\`a une constante pr\`es 
par ${(rr^*)}^h$ 
car les fonctions de Laguerre sont born\'ees.

Pour la seconde majoration donn\'ee dans le lemme,
nous proc\'edons comme dans la d\'emonstration du 
lemme~\ref{lem_maj_check_b_h_10}.
On~a:
\begin{displaymath}
  \begin{array}{l}
    \displaystyle{
      \check{b}^{h,0}(r,s)
      \,=\,
      \int_{S_1^{(v)}} \ldots
      d\tilde{\sigma}_v(X)
      \,=\,
      \sum_{i=1,2} \int_{C'_i} \ldots \psi_i D_i dx}\\
    \displaystyle{\quad
      \,=\,
      \sum_{i=1,2} \int_{C'_i} 
      {(irr^*x_v)}^h
      {(irr^*s)}^{-2h}
      \partial_{x_v}^{2h}.\left[e^{irsr^*x_v}\right]
      \overset{v'}{\underset{j=1} \Pi} 
      \flz {l_j} {\lambda_js^2r^2\frac{\nn{\pr j X}^2} 2}
      \psi_i D_i dx
      \quad ,}
  \end{array}  
\end{displaymath}
et donc apr\`es $2h$ int\'egrations par partie :
\begin{eqnarray*}
  \check{b}^{h,0}(r,s)
  &=&
  s^{-2h}{(irr^*)}^{-h}
  \sum_i \int_{C_i} 
  e^{irsr^*x_v}\\
  &&\quad
  \partial_{x_v}^{2h}.\left[
    \overset{v'}{\underset{j=1} \Pi} \flz {l_j} {\lambda_js^2r^2\frac{\nn{\pr j
          X}^2} 2} x_v^h \psi_i D_i\right] dx
  \quad.  
\end{eqnarray*}
En proc\'edant de la m\^eme mani\`ere 
que dans la d\'emonstration du
lemme~\ref{lem_maj_check_b_h},
on obtient que l'expression
$$
\partial_{x_v}^{2h}\left[
  \overset{v'}{\underset{j=1} \Pi} 
  \flz {l_j} {\lambda_js^2r^2\frac{\nn{\pr j X}^2} 2} x_v^h \psi_i D_i\right] 
\quad,
$$
est major\'ee \`a une constante pr\`es 
(de $v,h$) par:
$$
\sum_{0\leq j\leq 2h}
{(\nn{A^*}s^2r^2)}^j
\quad ;
$$
puis la majoration voulue de
$\nn{\check{b}^{h,0}(r,s)}$.

Ceci ach\`eve la d\'emonstration
du lemme~\ref{lem_maj_checkb_h0}.

\subsubsection{D\'emonstration
  du lemme~\ref{lem_maj_J_axh}}

Dans le cas o\`u $i=0$, 
en majorant les fonctions de Bessel par 1, 
on voit :
\begin{eqnarray*}
  J_{a,x;j}^{(0)}
  &\leq&
  \nn{x}^{-2}
  \int_{x^{-\frac1j}}^\infty
  {\left(  \int_{0}^1 
      r^{v-1-j}{(1-r^4)}^{\frac{z-2}2}dr  \right)}^2
  s^{-2j-1}ds\\
  &&=\,  
  {\left(  \int_{0}^1 
      r^{v-1-j}{(1-r^4)}^{\frac{z-2}2}dr  \right)}^2
  x^{-2} 
  \int_{x^{-\frac1j}}^\infty
  s^{-2j-1}ds
  \quad .
\end{eqnarray*}
et donc l'int\'egrale $J_{a,x;j}^{(0)}$  est major\'ee \`a une constante pr\`es par $1/(2j)$.

Dans le cas o\`u $i=1,\ldots, 2j$, 
par Cauchy-Schwarz,
l'expression 
${(J_{a,x;j}^{(i)})}^2$ 
est major\'ee par le produit de:
$$
x^{-4}
\int_{x^{-\frac1j}}^\infty
{\left(s^{-2j-\frac12}\right)}^2ds
\,=\,
\frac 1{4j} 
\quad ,
$$
et de l'int\'egrale :
\begin{eqnarray*}
  &&
  \int_{x^{-1}}^\infty
  {\left(  \int_{0}^1 
      {(as^2r^2)}^i
      \nn{\fb{\frac{z-2}2}{s^2\sqrt{1-r^4} a}}
      r^{v-1-k}{(1-r^4)}^{\frac{z-2}2}dr  \right)}^4
  s^{2\frac{-1}2}ds \\
  &&\leq\,
  \int_{0}^\infty
  {\left(  \int_{0}^1  
      {(as^2r^2)}^i
      \nn{ \fb{\frac{z-2}2}{s^2\sqrt{1-r^4} a}}
      r^{v-1-k}{(1-r^4)}^{\frac{z-2}2}dr  \right)}^4
  \frac{ds}s  \\
  &&=\,
  \int_{s'=0}^\infty
  {\left(  \int_{0}^1  
      {(s'r^2)}^i
      \nn{\fb{\frac{z-2}2}{s'\sqrt{1-r^4}}}
      r^{v-1-j}{(1-r^4)}^{\frac{z-2}2}dr  \right)}^4
  \frac{ds'}{s'}
\end{eqnarray*}
apr\`es le changement de variable $s'=as^2$.
Cette derni\`ere expression est major\'ee gr\^ace \`a l'in\'egalit\'e
de H\"older appliqu\'ee \`a l'int\'egrale contre $dr$ par :
\begin{eqnarray*}
  \int_{s'=0}^\infty
  {s'}^{4i}\quad
  \int_{0}^1  
  {\left(\nn{\fb{\frac{z-2}2}{s'\sqrt{1-r^4}}}
      {(1-r^4)}^{\frac{z-2}2+\frac 38}\right)}^4dr\\
  {\left(\int_{0}^1  
      {\left(r^{v-1-j+2i}{(1-r^4)}^{-\frac 38}\right)}^{\frac 43}
      dr \right)}^3 \quad\frac{ds'}{s'}\quad.
\end{eqnarray*}
Dans cette derni\`ere expression,
la seconde int\'egrale contre $dr$ est bien finie lorsque 
$$
\forall\,i=1,\ldots, 2j\quad:\quad
v-1-j+2i\,>\,-3/4
\quad,
$$
c'est-\`a-dire lorsque $j<v+1+3/4$;
donc les int\'egrales~${(J_{a,x;j}^{(i)})}^2$ 
sont major\'ees \`a une constante pr\`es par le maximum
sur $i=1,\ldots,2j$ des int\'egrales :
\begin{eqnarray*}
  \int_{s'=0}^\infty
  \int_{0}^1  
  {\left(\nn{\fb{\frac{z-2}2}{s'\sqrt{1-r^4}}}
      {(1-r^4)}^{\frac{z-2}2+\frac 38}\right)}^4dr
  {s'}^{4i} \frac{ds'}{s'}\\
  =\,
  \int_{\tilde{s}=0}^\infty
  \nn{\fb{\frac{z-2}2}{\tilde{s}}}^4
  {\tilde{s }}^{4i-1} d\tilde{s}
  \int_{0}^1  
  {(1-r^4)}^{4(\frac{z-2}2+\frac 38)-4i\frac 12}dr
  \quad,
\end{eqnarray*}
apr\`es Fubini et le changement de variable 
$\tilde{s}=s'\sqrt{1-r^4}$.
Dans l'expression pr\'ec\'edente,
l'int\'egrale contre $dr$ est finie tant que :
\begin{eqnarray*}
  4(\frac{z-2}2+\frac 38)-4i\frac 12>-1, i=1,\ldots, 2j
  &\Longleftrightarrow&
  4(\frac{z-2}2+\frac 38)-4\frac 12>-1\\
  \quad &\Longleftrightarrow&
  z>\frac 34
  \quad ,
\end{eqnarray*}
ce qui est le cas.
D'apr\`es le lemme~\ref{lem_maj_int_fb},
l'int\'egrale contre $d\tilde{s}$ est finie
lorsque 
$4i-1-1<4(z-2)/2$
pour $i=1,\ldots, 2j$
donc lorsque 
$j< (z-1)/4$
ce qui est le cas.

Ceci ach\`eve la d\'emonstration du lemme~\ref{lem_maj_J_axh},
et de tous les lemmes utilis\'es 
dans la sous-section pr\'ec\'edente.

\section{Multiplicateurs de Fourier sur $N_{v,2}$}
\label{sec_multiplicateur}

Dans cette section, 
on s'int\'eresse 
aux multiplicateurs
d\'efinis par passage en Fourier sph\'erique 
sur le groupe $N=N_{v,2}$
nilpotent libre \`a 2 pas.
C'est le probl\`eme d\'efini comme suit:
soit $f\in L^\infty(m)$;
on d\'efinit
la fonction $F$ la distribution radiale 
dont la transform\'ee de Fourier est $f$
(sous-section~\ref{subsec_inversion});
on d\'efinit l'op\'erateur de convolution par $F$ 
sur les fonctions simples de $N$ :
$T_f(h)\:=\, h*F$.
Le probl\`eme des multiplicateurs en Fourier
consiste \`a trouver des conditions sur la fonction $f$ 
pour que
l'op\'erateur $T_f$ s'\'etende en un op\'erateur continu 
$L^p(N)\rightarrow L^p(N)$ pour chaque $p\in ]1,\infty[$.

\paragraph{Autre probl\`eme des multiplicateurs.}
Le probl\`eme  des multiplicateurs de Fourier
a \'et\'e \'etudi\'e dans le cas euclidien
et indirectement sur les groupes de type~H,
groupes nilpotents ayant un calcul de Fourier.
En fait, l'\'etude dans ces cas 
est \'equivalente \`a un probl\`eme de multiplicateurs 
pour des op\'erateurs diff\'erentiels radiaux.
En effet,
dans le cas euclidien, 
cela se ram\`ene au probl\`eme des multiplicateurs pour le
laplacien~$-\Delta$ standard:
chercher des conditions sur la fonction $m$ 
pour que l'op\'erateur $m(-\Delta)$ s'\'etende en un op\'erateur   continu 
$L^p\rightarrow L^p$
revient \`a consid\'erer le probl\`eme des multiplicateurs en Fourier pour la fonction $m(t^2)$,
auquel H\"ormander a apport\'e une r\'eponse optimale.
Dans le cas des groupes de Heisenberg,
le probl\`eme des multiplicateurs de Fourier se ram\`ene 
au probl\`eme des multiplicateurs 
pour les deux op\'erateurs :
le sous-laplacien ``de Kohn'',
et la d\'erivation du centre.
Il en va de m\^eme plus g\'en\'eralement 
sur les groupes de~type~H
\cite{steinmullericci}.

\paragraph{Notre travail}
repose sur le th\'eor\`eme des int\'egrales singuli\`eres,
puis des estimations de normes $L^2$ \`a poids sur~$N$.
Cela conduit \`a effectuer de longs calculs techniques ``c\^ot\'e Fourier''
proches du cas des groupes de type H \cite{steinmullericci}.
Le r\'esultat n'est qu'un premier pas dans l'\'etude du probl\`eme.
Il n'inclut pas les fonctions constantes.
Nous souhaitons essentiellement pr\'esenter nos solutions \`a quelques points techniques.
En particulier, 
nous proposons une d\'ecomposition de l'espace ``c\^ot\'e Fourier''
et nous traduisons la multiplication par la norme ``c\^ot\'e groupe'',
en op\'erateur ``c\^ot\'e Fourier'' (voir les op\'erateurs $\Xi$ et $\aleph$ plus loin).

\paragraph{Organisation de cette section.}
Apres avoir pos\'e plusieurs notations,
en particulier 
celle d'une partition de l'unit\'e c\^ot\'e Fourier,
nous donnons l'\'enonc\'e du th\'eor\`eme
annonc\'e dans l'introduction, puis la d\'emarche de la preuve.
Cette derni\`ere utilise des estimations de normes $L^2$ sur le groupe,
que nous d\'emontrons dans la section suivante
gr\^ace \`a de op\'erations effectu\'ees ``c\^ot\'e Fourier''.

\subsection{Notations et r\'esultat}

Nous supposerons $v'>2$.

\subsubsection{Partition de l'unit\'e sur \mathversion{bold}{$\Pc$}}

La partition est indic\'ee par le produit 
des ensembles 
$I:= I_r\times I_\Lambda \times I_l$
\index{Notation!Ensemble de param\`etres!$I,I_r,I_\Lambda,I_l$}
o\`u :
\begin{itemize}
\item $I_l$ est l'ensemble des $v'$-uplets 
  $\zeta=(\zeta_1,\ldots,\zeta_{v'})\in \Nb^{v'}$;
\index{Notation!Param\`etres!$\zeta_i,\eta_i,\delta_{i,j},r$}
\item $I_\Lambda$ est l'ensemble :
  $$
  I_\Lambda
  \,:=\,
  \{
  (\eta,\delta)
  \quad :\quad
  \eta \in I_\Lambda^1
  \; ,\;
  \delta\in I_\Lambda^2(\eta)
  \}
  \quad,
  $$
  o\`u $I_\Lambda^1$ 
  est l'ensemble des $v'$-uplets 
  $\eta=(\eta_1,\ldots,\eta_{v'})\in \Zb^{v'}$
  v\'erifiant :
  $$
  \eta_i \,\leq\, \eta_{i+1}+1 \;,
  \quad i=1,\ldots ,v'-1
  \quad,
  $$
  et pour $\eta\in I_\Lambda^1$,
  $I_\Lambda^2(\eta)$ 
  est l'ensemble des $v'$-uplets 
  $\delta={(\delta_{i,j})}_{1\leq i<j\leq v'}\in \Zb^{v'(v'-1)/2}$
  v\'erifiant la condition :
  $$
  2^{\eta_j} -4.2^{\eta_i}
  \,\leq\,
  2^{\delta_{i,j}}
  \,\leq\,
  4.2^{\eta_j} -2^{\eta_i}
  \quad 1\leq i<j\leq v'
  \quad.
  $$
\item  $I_r$ est l'ensemble des $\theta\in\Rb$,
  qui est omis si $v=2v'$.
\end{itemize} 

Soit  $\chi$ une fonction $C^\infty$,
support\'ee dans l'intervalle $[1/1,2]$
telle que :
$$
\chi\,\geq\, 0
\qquad\mbox{et}\qquad
\forall \; y>0\quad :\quad
\sum_{j=-\infty}^{+\infty} \chi(2^{-j}y) 
\,=\, 1 
\quad.
$$
\index{Notation!Param\`etres!$\iota$}
On d\'efinit la fonction  
$\chi_{\iota}:
\Pc \rightarrow \Cb$ 
pour $\iota=(\theta, (\eta,\delta),\zeta)\in I$
par :
\index{Notation!$\chi_\iota$}
\begin{eqnarray*}
  \chi_{\iota}
  (r,\Lambda,l)
  &=&
  \chi(2^{-\theta} \nn{r}^4)
  \;
  \underset{1\leq i\leq v'}{\Pi}
  \chi(2^{-\zeta_i} (l_i+1))\;
  \chi(2^{-\eta_i} \lambda_i^2)
  \underset{i < j }{\Pi}
  \chi(2^{-\delta_{i,j}} (\lambda_j^2-\lambda_i^2))
  \quad,
\end{eqnarray*}
en convenant que 
dans le cas $v=2v'$, 
les termes en $r$ sont omis.
Les $\chi_\iota$ ont un sens 
pour des param\`etres 
$\iota\in \Zb\times\Zb^{v'}\times  \Zb^{\frac{v'(v'-1)}2} \times \Nb^{v'}$ , 
mais dans ce cas $\chi_\iota=0$ sur $\Pc$.

On obtient une partition de l'unit\'e de 
$\Pc$ :
\begin{equation}
  \label{formule_partition_unite_I}
  \forall\, \phi\in\Pc
  \quad:\quad
  \sum_{\iota\in I} \chi_{\iota}(\phi)
  \,=\, 1 \quad.  
\end{equation}
C'est une somme localement uniform\'ement finie 
c'est-\`a-dire
qu'il existe une borne $C$ 
(ne d\'ependant que de $\chi$ et $v'$) telle que
pour tout $\phi\in \Pc$,
le nombre de param\`etre $\iota$ v\'erifiant 
$\chi_\iota(\phi)\not=0$
est au plus $C$.

\subsubsection{Partition de l'unit\'e de supp \mathversion{bold}{$\chi_\iota$}}

On d\'efinit pour $h\in \Zb$ la fonction $\chi_h\in L^2(m)$ par :
\index{Notation!$\chi_h$}
$$
\chi_h(\phi)
\,=\,
\chi\left(
  2^{-h}
  (\sum_{i=1}^{v'}
  \lambda_i (2 l_i+1)  
  +\nn{r}^{2})\right)
\quad.
$$
On pose pour des param\`etres
$h\in \Zb$,
et $\iota=(\theta,\eta,\delta,\zeta)\in I$:
\index{Notation!$s_\iota$}
$$
s_\iota
\,:=\,
\sum_i 2^{\frac{\eta_i}2+\zeta_i} 
\quad \left(+2^{\frac{\theta}2}\right)
\quad.
$$
Sur le support de la fonction $\chi_\iota$, on a :
$$
2^{-2} s_\iota
\,\leq\,
\sum_i \lambda_i(2l_i+1) \quad \left( +\nn{r}^2\right)
\,\leq\,
2^4 s_\iota
\quad.
$$
On a une partition de l'unit\'e finie
sur le support de $\chi_\iota$ :
\begin{equation}
  \label{formule_partition_unite_I_h}
  \forall \phi\in \supp \chi_\iota
  \quad:\quad
  \sum_{h\in \Zb \,:\,
    2^{-3}s_{\iota} <2^h<2^5 s_{\iota}}
  \chi_h(\phi)
  \,=\,1
  \quad.
\end{equation}

\subsubsection{Op\'erateurs sur \mathversion{bold}{$\Pc$}}

Nous avons d\'ej\`a rappel\'e nos conventions de notation 
sur le transform\'ee de Fourier euclidienne.
On utilisera aussi ici la transform\'ee de Fourier sur
$\Zb^{v'}$ et $\Tb^{v'}$, 
o\`u $\Tb$ d\'esigne le tore $\Rb/[0,2\pi]$.

On d\'efinit les op\'erateurs 
$\Delta_i,\partial_{\lambda_i},\partial_{r}$
\index{Notation!Op\'erateur!$\partial_{\lambda_i},\partial_{r}$}
\index{Notation!Op\'erateur!$\Delta_i$}
sur les fonctions tests
$f:\Pc\rightarrow \Cb \in \Dc(\Pc)$ 
par 
($\Lambda=(\lambda_1,\ldots,\lambda_{v'})\in \Lc,
l=(l_1,\ldots,l_{v'})\in \Nb^{v'},
r\in \Rb^+$ ) :
\begin{eqnarray*}
  \Delta_i.f(r,\Lambda,l)
  &=&
  f(r,\Lambda,(l_1,\ldots,l_i+1,\ldots,l_v'))
  -f(r,\Lambda,l)\quad,\\
  \partial_{\lambda_i}.f(r,\Lambda,l)
  &=&
  \partial_{\lambda_i}.
  [\lambda_i \mapsto 
  f(r,(\lambda_1,\ldots,\lambda_i,\ldots,\lambda_{v'}),l)]
  \quad,\\
  \partial_{r}.f(r,\Lambda,l)
  &=&
  \partial_{r}.
  [r \mapsto 
  f(r,(\lambda_1,\ldots,\lambda_i,\ldots,\lambda_{v'}),l)]
  \quad.
\end{eqnarray*}
On \'etend ces op\'erateurs par dualit\'e 
sur l'ensemble $\Dc'(\Pc)$ des distributions de $\Pc$.

On \'etend les fonctions 
$f:\Pc\rightarrow \Cb$ 
de mani\`ere triviale 
en une fonction
\begin{equation}
  \label{formule_etendre_triv_fcn_Pc}
  \tilde{f}:
  \left\{ \begin{array}{rcl}
      \Rb\times\Rb^{v'}\times\Zb^{v'}
      &\rightarrow& 
      \Cb\\
      (r,\Lambda,l)
      &\mapsto&
      \left\{  \begin{array}{ll}
          f(r,\Lambda,l)
          &\mbox{si}\; (r,\Lambda,l)\in \Pc\quad,\\
          0
          &\mbox{sinon}\quad.
        \end{array}\right.
    \end{array}\right.  
\end{equation}
On identifie souvent une fonction sur $\Pc$, 
et son extension triviale.

On remarque :
\begin{lem}
  \label{lem_equivalence_densite_eta_supp_chi_iota}
  Soit $\iota=(\theta, (\eta,\delta),\zeta)\in I$.
  Sur le support de la fonction $\chi_\iota$,
  on a l'\'equivalence de
  la fonction densit\'e 
  de la mesure $\eta'$ sur le simplexe $\Lc$ 
  par rapport \`a la mesure de Lebesgue $d\Lambda$
  $$
  \frac{d\eta}{d\Lambda}
  \,\sim\,
  2^{2d(\eta,\delta)}
  \quad,
  $$
  o\`u on a not\'e :
  $$
  d(\eta,\delta)
  \,:=\,
  \left\{
    \begin{array}{ll}
      \nn{\delta}+\frac14\nn{\eta}
      &\mbox{dans le cas}\; v=2v'\quad,\\
      \nn{\delta}+\frac34\nn{\eta}
      &\mbox{dans le cas}\; v=2v'+1\quad.
    \end{array}\right.
  $$
  En particulier, on a:
  $$
  \forall\, f\in L^\infty(m)
  \quad:\quad
  \nd{f\chi_\iota}_{L^2(m)}
  \,\sim\,
  2^{d(\eta,\delta)}
  \nd{f\chi_\iota}
  _{L^2(\Rb\times\Rb^{v'}\times\Zb^{v'})}
  \quad,
  $$
  o\`u on a identifi\'e $f\chi_\iota\in L^2(m)$ et sa fonction 
  \'etendue de mani\`ere triviale 
  sur $\Rb\times\Rb^{v'}\times\Zb^{v'}$.
\end{lem}

\subsubsection{Op\'erateurs sur 
  \mathversion{bold}{$L^2(\Rb\times\Rb^{v'}\times\Zb^{v'})$}}

On associe \`a un param\`etre $\iota=(\theta,(\eta,\delta),\zeta)\in I$,
les quantit\'es suivantes :
\index{Notation!$L_\iota,R_\iota,D_\iota$}
\begin{eqnarray*}
L_\iota
&=&
\min_{i<j}\delta_{i,j}+4\max_{i_0}\nn{\eta_{i_0}-\min_{i<j}\delta_{i,j}}\quad,\\
R_\iota
&=&
\frac12 \max(\theta,\min_{i<j}\delta_{i,j})\quad,\\
D_\iota
&=&
8\max_{i_0}\nn{\eta_{i_0}-\min_{i<j}\delta_{i,j}} 
+\max(-\min_{i<j}\delta_{i,j}+\theta,0)
+4\max_i \zeta_i\quad,
\end{eqnarray*}
et l'op\'erateur non born\'e auto-adjoint positif 
sur l'espace de Hilbert 
$L^2(\Rb\times\Rb^{v'}\times\Zb^{v'})$
donn\'e par :
$$
T_{\iota}
\,:=\,
\Id+
  \sum_{i=1}^{v'}
  2^{\frac{L_\iota}2}\partial_{\lambda_i}^2
  +2^{\frac{R_\iota}2}\partial_{r}^2
+  2^{\frac{D_\iota}2}
  \sum_{i=1}^{v'}
  \Delta_i^2
\quad.
$$
\index{Notation!Op\'erateur!$T_\iota$}

\subsubsection{\'Enonc\'e du th\'eor\`eme}

Avec les notations pr\'ec\'edentes,
on peut \'enoncer le r\'esultat de notre \'etude :

\begin{thm}[Multiplicateurs]
  \label{thm_multiplicateurs}
  \index{Multiplicateurs}
  Soit $f\in L^\infty(m)$, 
  une fonction de $\Pc$.

  On note $F$ la distribution radiale
  dont la transform\'ee de Fourier est $f$.
  On d\'efinit alors l'op\'erateur $T_f$ 
  de convolution par $F$
  sur les fonctions simples de $N=N_{v,2}$. 

 On identifie
  $f_\iota=f\chi_\iota$ 
  \`a son extension triviale 
  sur $\Rb\times\Rb^{v'}\times\Zb^{v'}$.

  S'il existe $\epsilon>Q/2$
  tel que la somme suivante :
$$
\sum_{\iota\in I}
    s_\iota^{-\frac Q4}
2^{d(\eta,\delta)}
\nd{ T_\iota^{\frac{\epsilon}2}
  .f_\iota }\quad,
$$
est finie,
  alors l'op\'erateur $T_f$
  s'\'etend en un op\'erateur continu 
  $L^p\rightarrow L^p$
  pour tout $1<p<\infty$.
\end{thm}

Les hypoth\`eses de r\'egularit\'e 
(``plus de $Q/2$ d\'eriv\'ees''),
et de d\'ecroissance
ne sont pas v\'erifi\'ees pour les fonctions $f$ constantes.
Cependant, 
les moyens mis en oeuvre dans cette \'etude 
(en particulier l'expression des op\'erateurs $\Xi$ et $\aleph$)
permettront dans des travaux ult\'erieurs d'am\'eliorer ce premier pas
pour des fonctions assez r\'eguli\`eres mais peu d\'ecroissantes,
puis \'eventuellement d'adapter certaines m\'ethodes des groupes de type~H.

Remarquons que dans le cas d'un groupe de type~H,
la forme des valeurs propres pour les fonctions sph\'eriques 
du sous-laplacien de Kohn et des d\'erivations du centre
permet de regarder de mani\`ere \'equivalente
les probl\`emes des multiplicateurs de Fourier 
et spectraux conjoints pour ces op\'erateurs diff\'erentiels radiaux;
en particulier,
certaines propri\'et\'es connues 
des multiplicateurs spectraux sur les groupes de Lie nilpotents
\cite{christ,alexo} sont utilis\'ees dans \cite{steinmullericci}. 
De tels ph\'enom\`enes dans le cas de $N_{v,2}$ ne semblent pas aussi directs.
En effet, le centre est de dimension $v(v-1)/2$, il n'y a pas de ``directions ind\'ependantes''
et $\sum_{j=1}^{v'} \lambda_j (2 l_j+1) +r^{2}$ 
est la valeur propre du sous-laplacien $L$ pour 
la fonction sph\'erique born\'ee 
$\phi^{r,\Lambda,l}\in\Pc$
(voir (\ref{eg_valp_deltax})).
Il sera aussi moins facile que dans la cas d'un groupe de type~H
d'utiliser sur $N_{v,2}$ les r\'esultat sur les multiplicateurs de Fourier 
au probl\`eme des multiplicateurs spectraux pour le sous-laplacien
(i.e. trouver un exposant optimal pour la classe de Sobolev locale des
fonctions qui sont des mutliplicateurs spectraux).

\subsection{D\'emarche de la d\'emonstration du 
  th\'eor\`eme~\ref{thm_multiplicateurs}}

D'apr\`es le th\'eor\`eme des int\'egrales singuli\`eres \cite{CoifW}, 
pour que l'op\'erateur $T_f$ s'\'etende en un op\'erateur continu
$L^p(N)\rightarrow L^p(N)$,
il suffit de montrer :
\begin{eqnarray}
  \exists\, C_1>0
  &\; :\;&
  \forall h 
  \quad
  \nd{T_f(h)}_{L^2}
  \,\leq\, C_1 \nd{h}_{L^2}
  \label{cond_int_sing_1}\\
  \exists\, C_2,C_3>0
  &\; :\; &
  \forall y,y_0\in N \quad
  \int_{\nn{ny_0^{-1}}>C_2\nn{yy_0^{-1}}}
  \nn{F(ny^{-1})-F(ny_0^{-1})} dn 
  \,<\, C_3
  \label{cond_int_sing_2}  
\end{eqnarray}

D'apr\`es le corollaire~\ref{cor_mespl_nonrad},
la condition~(\ref{cond_int_sing_1}) 
est  d\'ej\`a v\'erifi\'ee avec $C_1=\nd{f}_{L^\infty}$.

Avant de passer \`a l'\'etude de la 
condition~(\ref{cond_int_sing_2}),
nous explicitons 
les d\'ecompositions de fonctions 
dont nous aurons besoin pour cette \'etude.

\subsubsection{D\'ecompositions des fonctions sur \mathversion{bold}{$\Pc$}}

Soit $f\in L^\infty(m)$.
On d\'efinit pour $\iota\in I$ la fonction
$f_\iota:=f \chi_\iota\in L^2(m)$.
\index{Notation!$f_\iota,f_{\iota,h}$}
D'apr\`es la d\'ecomposition~(\ref{formule_partition_unite_I}) de l'unit\'e, 
on voit que l'on a d\'ecompos\'e : 
$$
f
\,=\,
\sum_{\iota\in I}
f_\iota
\quad.
$$

Nous d\'ecomposons aussi $f_\iota$ sur $\supp\, \chi_\iota$.
On d\'efinit les fonctions 
$f_{\iota,h}\in L^2(m)$ par :
$$
f_{\iota,h}
\,:=\,
f_{\iota}\chi_h
\,=\,
f\chi_\iota\chi_h 
\; ,\quad
\iota\in I,\; 
h\in\Zb
\;:\;
2^{-3}s_{\iota} <2^h<2^5 s_{\iota}
\quad.
$$
D'apr\`es la d\'ecomposition~(\ref{formule_partition_unite_I_h}) de l'unit\'e, 
on a :
$$
f_\iota
\,=\,
\sum_{h\in \Zb \,:\,
  2^{-3}s_{\iota} <2^h<2^5 s_{\iota}}
f_{\iota,h}
\quad.
$$

\subsubsection{D\'ecompositions de fonctions radiales sur \mathversion{bold}{$N$}}

On d\'efinit
les fonctions 
$F_\iota,F_{\iota,h}\in {L^2(N)}^\natural$
\index{Notation!$F_\iota,F_{\iota,h}$}
comme les fonctions radiales 
dont les transform\'ees de Fourier sph\'eriques 
sont respectivement
$f_\iota,f_{\iota,h}$. 
D'apr\`es la proposition~\ref{cor_transf_noyau},
la fonction 
$\chi_h\in L^2(m)$ 
est la transform\'ee de Fourier du noyau de l'op\'erateur
$\chi(2^{-h}L)$ ,

On d\'ecompose :
\begin{equation}
  \label{egalite_decompositions_iota_iotah}  
  F
  \,=\,
  \sum_{\iota\in I}
  F_\iota
  \qquad\mbox{et}\qquad
  F_\iota 
  \,=\,
  \sum_{2^{-3}s_{\iota} <2^h<2^5 s_{\iota}}
  F_{\iota,h}
  \quad,
\end{equation}
la somme pour $F$ converge dans $\Sc(N)'$
(l'ensemble des distributions temp\'er\'ees),
et la somme pour $F_\iota$ est finie.
Pour $F_\iota$, on voit :
\begin{equation}
  \label{egalite_F_iota_h} 
  F_{\iota,h}
  \,=\,
  F_{\iota}*\chi(2^{-h}L).\delta_0 
  \quad,  
\end{equation}
o\`u l'on a not\'e 
$\chi(2^{-h}L).\delta_0$
le noyau de l'op\'erateur
$\chi(2^{-h}L)$.

Par homog\'en\'eit\'e du sous-laplacien 
on a :
\begin{equation}
  \label{expr_chi_dil_L}
  \chi(2^{-h}L).\delta_0 (y)
  \,=\,
  2^{h\frac Q2} \;
  \chi(L).\delta_0 \,(2^{\frac h2}y)
  \quad.
\end{equation}

En effet,
on a :
$L^{\delta_r}=
r^2 L$,
o\`u on a not\'e $\delta_r$ la dilatation 
sur des fonctions $h:N\rightarrow \Cb$ 
et sur des op\'erateurs $T$ sur un espace de fonctions 
stables par dilatation:
$$
\delta_r.h
\,:=\,
h(r.)
\qquad\mbox{et}\qquad
T^{\delta_r} .h
\,:=\,
\delta_{r^{-1}}.
\left( T.(\delta_r.h)\right)
\quad.
$$

On en d\'eduit que pour une fonction $m\in \Sc(\Rb)$,
en notant $M\in \Sc(N)$ le noyau de l'op\'erateur $m(L)$,
et $M_r\in \Sc(N)$ le noyau de l'op\'erateur $m(r^2L)$,
on a la propri\'et\'e d'homog\'en\'eit\'e du sous-laplacien :
$$
M_r
\,=\,
r^{-Q}
\delta_{\frac 1r}.M 
\quad.
$$
\index{Sous-laplacien!homog\'en\'eit\'e}

\subsubsection{\'Etude de la condition~(\ref{cond_int_sing_2})}

Pour les param\`etres :
$$
\iota\in I
\; ,\quad
h\in\Zb
\; :\;
2^{-3}s_{\iota} <2^h<2^5 s_{\iota}
\; ,\quad
y,z\in N
\; ,\quad
r>0
\quad,
$$
on note l'int\'egrale :
$$
I_{\iota,h;y,z;r}
\,:=\,
\int_{\nn{yn^{-1}}>4r}
\nn{F_{\iota,h}(yn^{-1})-F_{\iota,h}(zn^{-1})} 
dn
\quad,
$$
et son supremum sur $\nn{yz^{-1}}<r$ et $h$:
$$
S_{\iota,r}
\,:=\,
\sup 
\left\{I_{\iota,h;y,z;r}
  \; \left| \;
    \begin{array}{l}
      y,z\in N
      \quad\mbox{v\'erifiant}\quad
      \nn{yz^{-1}}<r\\
      h\in\Zb
      \quad\mbox{v\'erifiant}\quad
      2^{-3}s_{\iota} <2^h<2^5 s_{\iota}
    \end{array}\right.
\right\}
\quad.
$$
D'apr\`es les 
d\'ecompositions~(\ref{egalite_decompositions_iota_iotah})
et le fait que le nombre des entiers
$h\in\Zb$
v\'erifiant
$2^{-3}s_{\iota} <2^h<2^5 s_{\iota}$
est born\'e ind\'ependemment de $\iota\in I$,
la deuxi\`eme condition~(\ref{cond_int_sing_2}) 
sera v\'erifi\'ee, 
si la condition suivante  est satisfaite :
\begin{equation}
  \label{cond_int_sing_2'}
  \exists\, C>0
  \quad :\quad
  \forall \, r>0
  \qquad
  \sum_{\iota \in I}
  S_{\iota,r}
  \;<\; C \quad.
\end{equation} 

Pour \'etudier l'int\'egrale $I_{\iota,h;y,z;r}$,
nous avons besoin de d\'efinir
pour $\iota\in I$, pour $\epsilon'>0$ et
$g\in L^\infty(m)$ 
dont le support est inclus dans celui de $\chi_\iota$
(en particulier $g\in L^2(m)$),
\index{Notation!$N_{\iota,{\epsilon'}}(g)$}
$$
N_{\iota,{\epsilon'}}(g)
:=
\nd{{(1+s_\iota^2\nn{n}^4 )}^{\epsilon'} 
  G (n)}_{L^2(n\in N)}
\quad,
$$
en notant $G\in {L^2(N)}^\natural$ 
la fonction dont la transform\'ee
de Fourier est $g$.
\begin{prop}
  \label{prop_etude_int_I_iota,h}
  Pour $(Q+1)/2>\epsilon>Q/2 $,
  l'expression $S_{\iota,r}$ est major\'ee
  \`a  une constante 
  qui ne d\'epend que de $\epsilon,p$,
  par :
$$
\min (s_\iota^\frac12 r,1)\;
    s_\iota^{-\frac Q4}
      {(rs_\iota^\frac12)}^{-\epsilon+\frac Q2}
  N_{\iota,\frac{\epsilon}4}(f_\iota)
  \quad.    
$$
\end{prop}
Dans la section qui suit, 
nous montrerons la proposition suivante :
\begin{prop}
  \label{prop_controle_N}
  Soit $\iota(\theta,(\eta,\delta),\zeta) \in I$ 
  un param\`etre.

  Pour tout $\epsilon>0$, on a pour toute fonction
  $g\in L^\infty(m)$
  identifi\'ee  
  \`a sa fonction \'etendue de mani\`ere triviale sur 
  $\Rb\times\Rb^{v'}\times\Zb^{v'}$:
  $$
  \supp\, g
  \,\subset \,
  \supp \, \chi_\iota
  \;\Longrightarrow\;
  N_{\iota,\epsilon}(g) 
  \,\leq\,C\,
  2^{d(\eta,\delta)}
  \nd{T_\iota^{2\epsilon}
    .g }_{L^2(\Rb\times\Rb^{v'}\times\Zb^{v'})}
  \;,
  $$
  o\`u $C$ d\'esigne une constante
  qui ne d\'epend que de $v,\epsilon$,
  et en particulier pas de la fonction $g$ 
  ou du param\`etre $\iota$.
\end{prop}

Admettons ces deux propositions.
On en d\'eduit que l'expression $S_{\iota,r}$ 
est major\'ee \`a une constante pr\`es par : 
\begin{equation}
  \label{expression_r}
\min (s_\iota^\frac12 r,1)\;
    s_\iota^{-\frac Q4}
      {(rs_\iota^\frac12)}^{-\epsilon+\frac Q2}
2^{d(\eta,\delta)}
\nd{ T_\iota^{\frac{\epsilon}2}
  .f_\iota }_{L^2(\Rb\times\Rb^{v'}\times\Zb^{v'})}
\; .  
\end{equation}
\begin{itemize}
\item Si $\iota$ est tel que 
$s_\iota ^\frac12 r\leq 1$,
alors l'expression~(\ref{expression_r}) est major\'ee par :
$$
s_\iota^\frac12 r
    s_\iota^{-\frac Q4}
      {(rs_\iota^\frac12)}^{-\epsilon+\frac Q2}
2^{d(\eta,\delta)}
\nd{ T_\iota^{\frac{\epsilon}2}
  .f_\iota }
\quad;
$$
de plus, on a $r\leq s_\iota^{-\frac12}$ d'o\`u :
$$
s_\iota^\frac12 r
    s_\iota^{-\frac Q4}
      {(rs_\iota^\frac12)}^{-\epsilon+\frac Q2}
\,=\,
s_\iota^{\frac12(1-\epsilon)} 
r^{1-\epsilon+\frac Q2}
\,\leq\, 
s_\iota^{-\frac12\frac Q2} 
\quad.
$$
\item Sinon, on a :
$s_\iota^\frac12 r\geq 1$,
et l'expression~(\ref{expression_r}) est major\'ee par :
$$
    s_\iota^{-\frac Q4}
      {(rs_\iota^\frac12)}^{-\epsilon+\frac Q2}
2^{d(\eta,\delta)}
\nd{ T_\iota^{\frac{\epsilon}2}
  .f_\iota }\\
\leq 
    s_\iota^{-\frac Q4}
2^{d(\eta,\delta)}
\nd{ T_\iota^{\frac{\epsilon}2}
  .f_\iota }
\quad.
$$
\end{itemize}
La somme sur $\iota\in I$ de ce terme
est born\'ee pour une puissance
$\epsilon\in]Q/2,-(Q+1)/2$,
ind\'ependemment de $r>0$ lorsque la somme
$$
\sum_{\iota\in I}
    s_\iota^{-\frac Q4}
2^{d(\eta,\delta)}
\nd{ T_\iota^{\frac{\epsilon}2}
  .f_\iota }\quad,
$$
sera finie.
C'est  exactement l'hypoth\`ese du 
th\'eor\`eme~\ref{thm_multiplicateurs}.

Donc le th\'eor\`eme~\ref{thm_multiplicateurs}
sera d\'emontr\'e
lorsque les propositions~\ref{prop_etude_int_I_iota,h}
et~\ref{prop_controle_N}
seront prouv\'ees.
C'est ce que nous faisons 
dans la sous-section et dans la section qui suivent.

\subsection{\'Etude de l'int\'egrale $I_{\iota,h;y,z;r}$}

Cette sous-section est consacr\'ee \`a la d\'emonstration
de la proposition~\ref{prop_etude_int_I_iota,h}.
C'est une cons\'equence du lemme qui suit
lorsque $K=s_\iota ^\frac 12$.
\begin{lem}
  \label{lem_etude_int_I_iota,h}
On se donne un param\`etre $K>0$ inf\'erieur 
(\`a une constante qui ne d\'epend que de $\epsilon,v$)
\`a $2^{\frac h2}$.
  Pour $\nn{yz^{-1}}<r$ et $\epsilon>Q/2$,
  l'int\'egrale $I_{\iota,h;y,z;r}$ est major\'ee
  \`a  une constante 
  qui ne d\'epend que de $v,\epsilon$, 
  par :
$$
\{\min_{d=0,1}    2^{\frac h2 d}r^d\}\;
    K^{-\frac Q2}
      {(1+rK)}^{-\epsilon+\frac Q2}
  \nd{
    {(1+K\nn{n})}^\epsilon
    F_\iota(n) }_{L^2(n\in N)}\quad.
$$
\end{lem}

\begin{proof}[du lemme~\ref{lem_etude_int_I_iota,h}]
  Dans la preuve,
  $C$ d\'esigne une constante 
  qui ne d\'epend que des dimensions 
  du groupe $N$, et \'eventuellement de $\epsilon$
  et qui pourra \'evoluer au cours du calcul.
  Lorsque $\nn{yz^{-1}}<r$,
  on contr\^ole l'int\'egrale $I_{\iota,h;y,z;r}$
  soit  d'apr\`es l'in\'egalit\'e triangulaire,
  soit gr\^ace aux in\'egalit\'es de Taylor
  (\cite{varo}, Th\'eor\`eme IV.7.3) : 
  $$
  \mbox{soit par}\quad
  2 \int_{\nn{n'}>2r}
  \nn{F_{\iota,h}(n')}
  dn'\; ,
  \quad\mbox{soit par}\quad
  C r \max_i 
  \int_{\nn{n'}>2r}
  \nn{X_i.F_{\iota,h}(n')} 
  dn'
  \quad.
  $$
  On cherche donc \`a contr\^oler 
  pour $D=\Id$ ou $X_i, i=1,\ldots,v$
  les int\'egrales :
  $$
  \int_{\nn{n'}>2r}
  \nn{D.F_{\iota,h}(n')} dn'
  \quad;
  $$
  En utilisant l'\'ecriture~(\ref{egalite_F_iota_h}),
  on a :
  $$
  \int_{y'\in N}
  \nn{D.\chi(2^{-h}L).\delta_0 (y')} 
  \int_{\nn{n'}>2r}
  \nn{F_\iota(n'{y'}^{-1})} 
  dn'\; dy' 
  \;\leq\;
  J_{r,h,\epsilon,D}
  \nd{
    {(1+K\nn{n})}^\epsilon
    F_\iota(n) }_{L^2(n\in N)}\quad,
  $$
  par Cauchy-Schwarz et un changement de variable 
  $n=n'{y'}^{-1}$, en d\'enotant l'int\'egrale:
  $$
  J_{r,h,\epsilon,D}
  \,:=\,
  \int_{y'\in N} 
  \nn{D.\chi(2^{-h}L)(y')} 
  \sqrt{
    \int_{\nn{n'}>2r}
    {(1+K\nn{n'{y'}^{-1}})}^{-2\epsilon}
    dn'}  dy' 
  $$
  Il suffit donc d'estimer les int\'egrales 
  $J_{r,h,\epsilon,D}$.
  Or d'une part,
  d'apr\`es l'\'egalit\'e~(\ref{expr_chi_dil_L}),
  on a :
  $$
  D.\chi(2^{-h}L).\delta_0 (y')
  \,=\,
  2^{\frac h2 (Q+d(D))} \;
  D.\chi(L).\delta_0 \,(2^{\frac h2}y')
  \quad,
  $$
  o\`u $d(D)=0$ si $D=\Id$ 
  et $1$ si $D=X_i$ 
  est le degr\'e de l'op\'erateur diff\'erentiel $D$;
  d'autre part, 
  on obtient de l'in\'egalit\'e triangulaire :
  $$
  \forall\, x,y\in N
  \quad :\quad
  (1+\nn{xy})
  \,\leq \,
  C' (1+\nn{x})(1+\nn{y})
  \quad,
  $$
  o\`u 
  $C'$ 
  d\'esigne la constante de l'in\'egalit\'e triangulaire;
  on en d\'eduit :
  $$
  \int_{\nn{n'}>2r}
  {(1+K\nn{n'{y'}^{-1}})}^{-2\epsilon} dn'
  \,\leq\,
  {C'}^{2\epsilon} {(1+K\nn{y'})}^{2\epsilon} \;
  \int_{\nn{n'}>2r}
  {(1+K\nn{n'})}^{-2\epsilon} dn'
  \quad.
  $$
  On calcule facilement cette derni\`ere int\'egrale :
  gr\^ace au passage en coordonn\'ees polaires
  (voir (\ref{formule_changement_polaire})),
  puis au changement de variable $\rho'=K\rho$ :
  \begin{eqnarray*}
    &&\int_{\nn{n'}>2r}
    {(1+K\nn{n'})}^{-2\epsilon} dn'
    \,=\,  
    \int_{\rho=2r}^\infty
    {(1+K\rho)}^{-2\epsilon} \rho^{Q-1}d\rho\\
    &&\quad=\,
    K^{-Q}
    \int_{\rho=2rK}^\infty
    {(1+\rho')}^{-2\epsilon} {\rho'}^{Q-1}d\rho'
    \,\leq\,
    K^{-Q}
    \int_{\rho=2rK}^\infty
    {(1+\rho')}^{-2\epsilon+Q-1}d\rho'\\
    &&\quad =\,
    K^{-Q}
    {(2\epsilon-Q)}^{-1}
    {(1+2r K)}^{-2\epsilon+Q}
    \quad.  
  \end{eqnarray*}
  On a donc :
  $$
  \int_{\nn{n'}>2r}
  {(1+K\nn{n'{y'}^{-1}})}^{-2\epsilon} dn'
  \,\leq\,
  C
  {(1+K\nn{y'})}^{2\epsilon}
    K^{-Q}
  {(1+rK)}^{-2\epsilon+Q}
  \quad,
  $$
  puis en rassemblant ce qui pr\'ec\`ede et en utilisant l'hypoth\`ese sur $K$ :
  \begin{eqnarray*}
    J_{r,h,\epsilon,D}
    &\leq&
    C\,
    2^{\frac h2 (Q+d(D))} 
    \sqrt{
K^{-Q}
      {(1+rK)}^{-2\epsilon+Q}}\\
    &&\qquad
    \int_{y'\in N} 
    \nn{
      D.\chi(L).\delta_0 \,(2^{\frac h2}y')}
    {(1+2^{\frac h2}\nn{y'})}^{\epsilon} 
    dy' \\
    &&
    \,=\,
    C\,
    2^{\frac h2 d(D)} 
    K^{-\frac Q2}
      {(1+rK)}^{-\epsilon+\frac Q2}
    \int_N 
    \nn{
      D.\chi(L).\delta_0 \,(y'')}
    {(1+\nn{y''})}^{\epsilon} 
    dy''\quad,
  \end{eqnarray*}
  gr\^ace au changement de variable $y''=2^{\frac h2}y'$.
  Or on a
  $\chi\in\Sc(\Rb^+)$, 
  d'o\`u 
  $\chi(L).\delta_0\in\Sc(N)$;
  les derni\`eres int\'egrales contre $dy''$ sont donc finies.
  On en d\'eduit que 
  pour tous les op\'erateurs 
  $D=\Id$ ou $X_i, i=1,\ldots,v$,
  les int\'egrales
  $J_{r,h,\epsilon,D}$
  sont born\'ees \`a une constante pr\`es par :
$$
    2^{\frac h2 d(D)} 
    K^{-\frac Q2}
      {(1+rK)}^{-\epsilon+\frac Q2}
  \quad.
  $$
\end{proof}

\section{Contr\^ole de la norme $N_{\iota,\epsilon}$}

Le but de cette section est de 
d\'emontrer la proposition~\ref{prop_controle_N}.

Par la formule~(\ref{formule_plancherel_radiale})
de Plancherel, on a :
\begin{eqnarray*}
  N_{\iota,\epsilon}^2(g)
  &=&
  \int_{\Pc}
  \nn{<{(1+s_\iota^2 \nn{n}^4 )}^{\epsilon}
    G,\phi>} ^2dm(\phi)\\
  &=&
  \int_{\Pc}
  \nn{<G,
    {(1+s_\iota^2 (\nn{X}^4+\nn{A}^2))}^{\epsilon}\phi>} ^2dm(\phi)
\end{eqnarray*}
Nous cherchons donc \`a exprimer les termes 
$\nn{X}^2\phi$ et $\nn{A}^2\phi$
en fonction des op\'erateurs de d\'ecalage et de d\'erivation appliqu\'es aux param\`etres de 
$\phi=\phi^{r,\Lambda,l}(n=\exp(X+A))$.

Rappelons :
$$
\phi^{r,\Lambda,l}(\exp(X+A))
=
\int_{O(v)}
e^{i<k.D_2(\Lambda),A>}
f_{r,\Lambda,l}(k^{-1}.X)
dk
$$
o\`u la fonction $f_{r,\Lambda,l}=f_\phi$
est param\`etr\'ee par $\Pc$ 
et est d\'efinie sur $\Rb^v$ par:
$$
f_\phi (X)
\,:=\,
e^{i <r X_v^*,X>}
\overset{v'}{\underset{j=1}{\Pi}}
\flz {l_j} {\lambda_j\frac{\nn{\pr j X }^2}2}
\quad.
$$
\index{Notation!$f_\phi$}

\subsubsection{Op\'erateurs de d\'ecalage et de d\'erivation}

On utilisera les op\'erateurs de d\'ecalage 
$\Delta,\alpha,\beta,\gamma$
sur les fonctions $\Nb\rightarrow \Cb$
d\'efinis dans la sous-section~\ref{subsec_prop_fcnlag_0}.
Pour une fonction donn\'ee sur $\Nb^{v'}$, 
on d\'efinit les op\'erateurs 
$\Delta_i$, $\alpha_i$, $\beta_i$, et $\gamma_i$ 
\index{Notation!Op\'erateur!$\alpha_i,\beta_i,\gamma_i$}
\index{Notation!Op\'erateur!$\Delta_i$}
pour $i=1,\ldots, v'$
comme les op\'erateurs 
$\tau^\pm,\alpha,\beta,\gamma$ 
agissant respectivement sur la $i$\`eme variable enti\`ere de $l$.
Les propri\'et\'es pour ces op\'erateurs 
donn\'ees dans la sous-section~\ref{subsec_prop_fcnlag_0}
conduisent aux propri\'et\'es
suivantes pour les $\alpha_i,\beta_i,\gamma_i$:
\begin{eqnarray}
  \alpha_i. 
  \overset{v'}{\underset{j=1}{\Pi}}
  \flz {l_j} {y_j}
  &=& 
  y_i \flsz {l_i} ' \underset{j\not=i} \Pi\flz {l_j} {y_j}
  \quad ,  \label{propalpha} \\
  \beta_i.\overset{v'}{\underset{j=1}{\Pi}}
  \flz {l_j} {y_j}
  &=&
  y_i \flsz {l_i} \underset{j\not=i}\Pi \flz {l_j} {y_j}
  \quad ,  \label{propbeta}\\
  \gamma_i.\overset{v'}{\underset{j=1}{\Pi}}
  \flz {l_j} {y_j}
  &=& 
  \alpha_i \flsz {l_i} ' \underset{j\not=i}\Pi \flz {l_j} {y_j}
  \quad . \nonumber
\end{eqnarray}
On remarque que l'on peut d\'eduire 
en d\'erivant encore l'\'egalit\'e (\ref{propalpha}) 
suivant $y_i$ une autre \'egalit\'e:
\begin{equation}\label{propalphagamma}
  \gamma_i.\overset{v'}{\underset{j=1}{\Pi}}
  \flz {l_j} {y_j}
  \,=\,
  \alpha_i.\flsz {l_i}' \underset{j\not=i}\Pi \flz {l_j}  {y_j}
  \,=\,
  (\flsz {l_i}'+ y_i \flsz {l_i} '') 
  \underset{j\not=i}\Pi \flz {l_j} {y_j  }
  \quad .
\end{equation}
Gr\^ace \`a ces \'egalit\'es, 
on a facilement quelques propri\'et\'es 
de d\'erivation et de d\'ecalage pour la fonction $f_\phi$ :
\begin{eqnarray}
  \frac{\nn{\pr i X}^2}2 f_\phi(X)
  &=&
  \frac{\beta_i}{\lambda_i}.  f_\phi(X)
  \quad , \label{fprX}\\
  \partial_{\lambda_i}.  f_\phi(X)
  &=&
  \frac{\alpha_i}{\lambda_i}.  f_\phi(X)
  \quad , \label{fder1}\\
  \partial_{\lambda_i}^2.  f_\phi(X)
  &=&
  \left(-\frac{\alpha_i}{\lambda_i^2}+{(\frac{\alpha_i}{\lambda_i})}^2\right).f_\phi(X)
  \quad . \label{fder2}
\end{eqnarray}

On d\'efinit l'op\'erateur $\Xi$ sur les fonctions sur $\Pc$ :
$$
\Xi
\,:=\,
2 \sum_i \frac{\beta_i}{\lambda_i}    \quad 
\left(-\frac12\partial_r^2\quad
  \mbox{si}\;v=2v'+1\right)\quad.
$$

De l'\'egalit\'e (\ref{fprX}), on obtient facilement pour $n=\exp(X+A)$:
$$
\nn{X}^2  \phi (n)
\,=\,
\Xi.  \phi (n)
\quad.
$$

On d\'efinit aussi l'op\'erateur sur les fonctions de $\Pc$ :
\begin{eqnarray*}
  \aleph
  \,:=\,
\sum_m \partial_{\lambda_m}^2
    -2\frac{\alpha_m}{\lambda_m}.\partial_{\lambda_m}
    +\frac{\alpha_m^2+\alpha_m}{\lambda_m^2}
  +\sum_{i<j}
  \frac{-4}
  {\lambda_j^2-\lambda_i^2}(\lambda_i\partial_{\lambda_i}-\alpha_i -\lambda_j\partial_{\lambda_j}+\alpha_j)\\
  +4\sum_{i<j}\frac{\lambda_j^2+\lambda_i^2}{{(\lambda_j^2-\lambda_i^2)}^2}
  (-\alpha_j +\lambda_j \gamma_j\frac{\beta_i}{\lambda_i}
  -\alpha_i+\lambda_i\gamma_i \frac{\beta_j}{\lambda_j}
  -2\alpha_i\alpha_j)\\
  +\left(
    \sum_i 
    \frac2{\lambda_i^2}(
    -\lambda_i\gamma_i \partial_{r}^2
    -2ri\partial_{r}\alpha_i   
    +{r}^2\frac{\beta_i}{\lambda_i}+ir\partial_{r})
-\frac2{\lambda_i}\partial_{\lambda_i}
    \quad \mbox{si}\; v=2v'+1.\right)
\end{eqnarray*}

\begin{lem}[Expression de \mathversion{bold}{$\nn{X}^2\phi$} et \mathversion{bold}{$\nn{A}^2\phi$}]
  \label{lem_expr_TX_TA}
  On a :
  \begin{eqnarray*}
    \Xi.\phi(\exp(X+A))
    &=&
    \nn{X}^2  \phi (\exp(X+A))\\
    \aleph.\phi (\exp(X+A))
    &=&
    \nn{A}^2 \phi (\exp(X+A))
  \end{eqnarray*}
\end{lem}

Nous venons de montrer la premi\`ere \'egalit\'e.
La sous-section qui suit sera consacr\'ee 
\`a montrer la seconde.

\begin{rem}
\label{rem_derivation_TAX}
Nous remarquons 
gr\^ace \`a~(\ref{egalite_alpha_2+1}),
(\ref{egalite_alpha_1-2}),
(\ref{egalite_alpha12_casp=1}),
(\ref{egalite_betagamma}),
au lemme~\ref{lem_puissance_beta},
que chaque terme  de $\aleph$, et de $\Xi$, puis $\Xi^2$ 
se d\'eveloppe en des termes comportant au moins 
``une d\'erivation discr\`ete ou non''.
\end{rem}

\subsection{\'Etude de $\nn{A}^2\phi$}

Dans cette sous-section, 
pour all\'eger les notations, 
on sous-entendra les arguments et les param\`etres 
du terme $\phi=\phi^{r,\Lambda,l}(\exp(X+A))$.

On remarque :  
$$
\nn{A}^2   \phi
\,=\,
\int_{O(v)}
\Delta_{A^*=k.D_2(\Lambda)}\left\{ A^*\rightarrow e^{i<A^*,A>}\right\}\;
f_\phi(k^{-1}.X)\;
dk
\quad ,
$$
o\`u $\Delta$ est le laplacien sur l'ensemble $\Ac_v$ des matrices
antisym\'etriques.
D'apr\`es son expression ``en coordonn\'ees polaires''
(lemme~\ref{lem_laplacien_coord_pol}),
on a:
$$
\nn{A}^2   \phi
\,=\,
F_1+F_2+F_3\quad\left(+F_4+F_5\quad\mbox{si}\; v=2v'+1\right)
\quad ,
$$
o\`u les $F_i$ sont donn\'ees par :
\begin{eqnarray*}
  F_1
  &=&
  \int_{O(v)} \sum_m
  \partial_{\lambda_m}^2\left\{ \Lambda \rightarrow  e^{i<k.D_2(\Lambda),A>}\right\} \;
  f_\phi(k^{-1}.X)
  \;dk \quad ,\\
  F_2
  &=&
  \int_{O(v)}
  \sum_{i<j}\sum_{(m,n)\in I_{i,j}}D^2
  [k\rightarrow  e^{i<k.D_2(\Lambda),A>}](\psi^{-1}(k.E_{m,n}), \psi^{-1}(k.E_{m,n}))\\
  &&\qquad\qquad\qquad\qquad\qquad
  \qquad\qquad\qquad\qquad\qquad
  f_\phi(k^{-1}.X)
  \;dk \quad ,\\
  F_3
  &=&
  \int_{O(v)}
  \sum_{i<j}
  \frac{-4}{\lambda_j^2-\lambda_i^2}(\lambda_i\partial_{\lambda_i}-\lambda_j\partial_{\lambda_j}).
  \left\{\Lambda \rightarrow  e^{i<k.D_2(\Lambda),A>}\right\}\;
  f_\phi(k^{-1}.X)
  \;dk \quad, 
\end{eqnarray*}
et dans le cas $v=2v'+1$:
\begin{eqnarray*}
  F_4
  &=&
  \int_{O(v)}
  \sum_{i, \tilde{i}=2i,2i-1} \frac1{\lambda_i}
  D^2
  [k\rightarrow  e^{i<k.D_2(\Lambda),A>}](\vec{E}_{\tilde{i},v},\vec{E}_{\tilde{i},v})\;
  f_\phi(k^{-1}.X)
  \;dk \quad ,\\
  F_5
  &=&
  \int_{O(v)}
  \sum_i \frac2{\lambda_i}
  -\partial_{\lambda_i}\left\{\Lambda \rightarrow  e^{i<k.D_2(\Lambda),A>}\right\}\;
  f_\phi(k^{-1}.X)
  \;dk \quad.
\end{eqnarray*}

\subsubsection{Exprimons diff\'eremment \mathversion{bold}{$F_1$}}

On a :
$$
F_1
\, =\,
\sum_m
\partial_{\lambda_m}^2.\phi
-E_1-E_2
$$
o\`u les termes ``correctifs'' $E_1$ et $E_2$ sont donn\'es par :
\begin{eqnarray*}
  E_1
  &=&
  2\sum_m
  \int_{O(v)}
  \partial_{\lambda_m}.e^{i<k.D_2(\Lambda),A>}
  \partial_{\lambda_m}.f_\phi(k^{-1}.X)
  \;dk  \quad,\\
  E_2
  &=&
  \sum_m
  \int_{O(v)}
  e^{i<k.D_2(\Lambda),A>}
  \partial_{\lambda_m}^2 
  f_\phi(k^{-1}.X)
  \;dk  \quad.
\end{eqnarray*}
On peut exprimer \`a l'aide du terme $\phi$ d\'ecal\'e ou d\'eriv\'e ces termes $E_1,E_2$. D'abord pour le terme $E_1$, on voit que d'apr\`es les \'egalit\'es (\ref{fder1}) et la lin\'earit\'e des op\'erateurs $\alpha_m$, on a:
\begin{eqnarray*}
  E_1
  &=&
  2\sum_m 
  \int_{O(v)}
  \partial_{\lambda_m} e^{i<k.D_2(\Lambda),A>}
  \quad 
  \frac{\alpha_m}{\lambda_m}. f_\phi(k^{-1}.X) \
  ;dk \\
  &=&
  2\sum_m \frac{\alpha_m}{\lambda_m}.
  \int_{O(v)}
  \partial_{\lambda_m} e^{i<k.D_2(\Lambda),A>}
  \quad  
  f_\phi(k^{-1}.X) 
  \;dk \quad .
\end{eqnarray*}
Or on voit :
$$
\begin{array}{l}
  \displaystyle{
    \int_{O(v)}
    \partial_{\lambda_m} e^{i<k.D_2(\Lambda),A>}
    \;  f_\phi(k^{-1}.X) \;dk}\\
  \displaystyle{
    \quad=\,
    \partial_{\lambda_m}. 
    \int_{O(v)}
    e^{i<k.D_2(\Lambda),A>}
    f_\phi(k^{-1}.X)dk
    -\int_{O(v)}
    e^{i<k.D_2(\Lambda),A>}
    \partial_{\lambda_m}f_\phi(k^{-1}.X)dk
    \quad .}
\end{array}
$$
D'o\`u on d\'eduit
d'apr\`es les \'egalit\'es (\ref{fder1}) et la lin\'earit\'e des op\'erateurs $\alpha_m$:
\begin{eqnarray*}
  E_1
  &=&
  2\sum_m \frac{\alpha_m}{\lambda_m}.\partial_{\lambda_m}. \phi
  -\frac{\alpha_m}{\lambda_m}.\int_{O(v)}
  e^{i<k.D_2(\Lambda),A>}\partial_{\lambda_m}f_\phi(k^{-1}.X)dk\\
  &=&
  2\sum_m \frac{\alpha_m}{\lambda_m}.\partial_{\lambda_m}. \phi
  -\frac{\alpha_m}{\lambda_m}.\int_{O(v)}
  e^{i<k.D_2(\Lambda),A>}\frac{\alpha_m}{\lambda_m}f_\phi(k^{-1}.X)dk\\
  &=&
  2\sum_m \frac{\alpha_m}{\lambda_m}.\partial_{\lambda_m}
  -{(\frac{\alpha_m}{\lambda_m})}^2 \quad .\phi
  \quad .
\end{eqnarray*}
Ensuite, pour le terme $E_2$, d'apr\`es les \'egalit\'es (\ref{fder2}), on a:
$$
E_2
\,=\,
\sum_m 
-\frac{\alpha_m}{\lambda_m^2}
+{(\frac{\alpha_m}{\lambda_m})}^2 
\quad .\phi \quad .
$$
On voit donc que l'on peut exprimer le terme $F_1$ comme:
\begin{eqnarray*}
  F_1
  &=&
  \left[ 
    \sum_m \partial_{\lambda_m}^2
    -
    \left(2\sum_m \frac{\alpha_m}{\lambda_m}.
      \partial_{\lambda_m}
      -{(\frac{\alpha_m}{\lambda_m})}^2\right)
    -
    \left(\sum_m 
      -\frac{\alpha_m}{\lambda_m^2}
      +{(\frac{\alpha_m}{\lambda_m})}^2 
    \right) \right]. \phi\\
  &=&
  \left[ \sum_m \partial_{\lambda_m}^2
    -2\frac{\alpha_m}{\lambda_m}.\partial_{\lambda_m}
    +\frac{\alpha_m^2+\alpha_m}{\lambda_m^2} \right]. \phi 
  \quad .
\end{eqnarray*}

\subsubsection{Exprimons diff\'eremment le terme \mathversion{bold}{$F_3$}}

On a :
$$
F_3
\,=\,
\sum_{i<j}
\frac{-4}{\lambda_j^2-\lambda_i^2}(\lambda_i\partial_{\lambda_i}-\lambda_j\partial_{\lambda_j}).\phi-E_3
\quad .
$$
o\`u le terme ``correctif'' $E_3$ est donn\'e par:
\begin{eqnarray*}
  E_3
  &=&
  \int_{O(v)}
  e^{i<k.D_2(\Lambda),A>}
  \sum_{i<j}
  \frac{-4}{\lambda_j^2-\lambda_i^2}
  (\lambda_i\partial_{\lambda_i}-\lambda_j\partial_{\lambda_j}).
  f_\phi(k^{-1}.X)
  \;dk \\
  &=&
  \sum_{i<j}
  \frac{-4}{\lambda_j^2-\lambda_i^2}(\alpha_i-\alpha_j).\phi
  \quad ,
\end{eqnarray*}
d'apr\`es (\ref{fder1}).
On voit donc que l'on peut exprimer le terme $F_3$ comme:
\begin{eqnarray*}
  F_3
  &=&
  \left[ \sum_{i<j}
    \frac{-4}{\lambda_j^2-\lambda_i^2}(\lambda_i\partial_{\lambda_i}-\lambda_j\partial_{\lambda_j})
    -
    \sum_{i<j}
    \frac{-4}{\lambda_j^2-\lambda_i^2}(\alpha_i-\alpha_j)
  \right].\phi\\
  &=&
  \left[ \sum_{i<j}
    \frac{-4}{\lambda_j^2-\lambda_i^2}(\lambda_i\partial_{\lambda_i}-\alpha_i -\lambda_j\partial_{\lambda_j}+\alpha_j)
  \right].\phi  
  \quad .
\end{eqnarray*}

\subsubsection{Exprimons diff\'eremment \mathversion{bold}{$F_2$ }}

La somme est sur 
$(m,n)\in I_{i,j}$ et $i<j$ :
\begin{eqnarray*}
  F_2
  &=&
  \sum
  2{\frac{d}{dt}}_{t=0}^2
  \int_{O(v)}
  e^{i<ke^{tD_{k,\Lambda}\psi^{-1}(k.B_{m,n}^{i,j})}.D_2(\Lambda),A>}
  \,
  f_\phi(k^{-1}.X)
  \;dk\\
  &=&
  \sum
  2
  \int_{O(v)}
  e^{i<k.D_2(\Lambda),A>}
  \,
  {\frac{d}{dt}}_{t=0}^2  f_\phi
  ({(ke^{-t.D_{k,\Lambda}\psi^{-1}(k.B_{m,n}^{i,j})}})^{-1}.X)
  \;dk \quad ,
\end{eqnarray*}
o\`u on  a not\'e : (les $\{ E_{m,n} \}$ designent la base canonique des matrices antisym\'etriques)
$$
B_{m,n}^{i,j}
\,=\,
D_{k,\Lambda}\psi^{-1}(k.E_{m,n})
\,=\,
\frac1{\lambda_j^2-\lambda_i^2}
(\epsilon_j \lambda_j E_{m_j,n_j}
+ \epsilon_i \lambda_i E_{m_i,n_i})
\quad,
$$
et o\`u les indices et les signes sont donn\'es par le tableau :
\begin{center}
  \begin{tabular}{l | l | l | l | l}
    $(m,n)$ & $\epsilon_j$ & $(m_j,n_j)$ & 
    $\epsilon_i$ & $(m_i,n_i)$\\
    \hline
    (2i-1,2j)&-& (2i-1,2j)&+&(2i,2j-1)\\
    (2i,2j)&+&(2i,2j-1)&-&(2i-1,2j)\\
    (2i-1,2j)&+&(2i-1,2j-1)&+&(2i,2j)\\
    (2i,2j-1)&-&(2i,2j)&-&(2i-1,2j-1)
  \end{tabular}  .
\end{center}
Fixons $i<j$, et $(m,n)\in I_{i,j}$ et calculons  
$$
{\frac{d}{dt}}_{t=0}^2  f_\phi
({(ke^{-tB})}^{-1}.X) \; ,
\quad B=B_{m,n}^{i,j}\quad.
$$
Pour cela, nous allons effectuer un d\'eveloppement limit\'e \`a l'ordre 2. Commen\c cons par:
\begin{eqnarray*}
  \nn{ \pr p {{(ke^{-tB})}^{-1}.X}}^2
  &=&
  \nn{ \pr p {e^{tB}k^{-1}.X}}^2
  \,=\,
  \nn{ \pr p {(I+tB+
      \frac{t^2}2 B^2)k^{-1}.X}}^2 +o(t^2)\\
  &=&
  \nn{ \pr p {k^{-1}.X}}^2
  +
  2t<\pr p {Bk^{-1}.X}, \pr p {k^{-1}.X}>\\
  &&\;
  +
  t^2(<\pr p {B^2k^{-1}.X}, \pr p {k^{-1}.X}>
  +\nn{ \pr p {Bk^{-1}.X}}^2 )
  +o(t^2) \quad.
\end{eqnarray*}
On obtient donc le d\'eveloppement limit\'e 
(en convenant que lorsque leurs arguments ne sont pas explicit\'es, 
les fonctions $\flsz {l_p} $ 
et leurs d\'eriv\'ees sont prises en l'argument 
$(\lambda_p\frac{\nn{\pr p X }^2}2) $) : 
$$
\flz {l_p}  {\lambda_p\frac{\nn{\pr p {{(ke^{-tB})}^{-1}X} }^2}2}
\,=\,
\flsz {l_p} +ta_p\flsz {l_p}  '+\frac{t^2}2(b_p\flsz {l_p}  '+a_p^2 \flsz {l_p}  '')
+o(t^2)\quad,
$$ 
o\`u on a not\'e :
\begin{eqnarray*}
  a_p
  &=&
  a_p(B)
  \,=\,
  \lambda_p <\pr p {Bk^{-1}.X}, \pr p {k^{-1}.X}>
  \quad ,\\
  b_p
  &=&
  b_p(B)
  \,=\,
  \lambda_p(<\pr p {B^2k^{-1}.X}, \pr p {k^{-1}.X}>
  +\nn{ \pr p {Bk^{-1}.X}}^2 )
  \quad .
\end{eqnarray*}
On voit : 
\begin{itemize}
\item $a_p(B)=b_p(B)=0$ lorsque $p\not =i,j$,
\item dans le cas $v=2v'+1$,      
  $<X_v^*,{(ke^{-tB})}^{-1}.X>=<X_v^*,k^{-1}.X>$.
\end{itemize}
On en d\'eduit :
$$
\begin{array}{l}
  {\frac{d}{dt}}_{t=0}^2
  f_\phi({(ke^{-t.B_{m,n}^{i,j}})}^{-1}.X)\\
  \quad=\,
  \left( (b_j \flsz {l_j}'+a^2_j\flsz {l_j} '')\flsz {l_i} + 
    (b_i \flsz {l_i} '+a^2_i\flsz {l_i} '')\flsz {l_j} + 
    2 a_ia_j\flsz {l_i}  '\flsz {l_j}  '
  \right)\underset{p\not=i,j}{ \Pi} \flsz {l_p} \;
  \left( e^{i <rX_v^*,k^{-1}X>} \right)
  \; .
\end{array}
$$
Pour calculer
\begin{equation}
  \label{formule_sum_mn_der2}
  \sum_{(m,n)\in I_{i,j}}{\frac{d}{dt}}_{t=0}^2
  f_\phi({(ke^{-t.B_{m,n}^{i,j}})}^{-1}.X)
  \quad,  
\end{equation}
il suffit donc  
\begin{enumerate}
\item de calculer les expressions $a_i^2, a_j^2, a_ia_j$,
\item de calculer les expressions $b_i,b_j$,
\item puis de les sommer sur $(m,n)\in I_{i,j}$,
  ces expressions \'etant d\'ependante 
  de la matrice antisym\'etrique
  $B=B_{m,n}^{i,j}$. 
\end{enumerate}

\paragraph{Calculons les expressions \mathversion{bold}{$a_i$} et \mathversion{bold}{$a_j$}.} 
Comme on a:
$$
<\pr p {\vec{E}_{m,n}k^{-1}X},\pr p {k^{-1}X}>
\,=\,
\left\{
  \begin{array}{ll}
    {[k^{-1}X]}_n{[k^{-1}X]}_m
    &\mbox{si}\; p=i \quad ,\\
    -{[k^{-1}X]}_m{[k^{-1}X]}_n
    &\mbox{si}\; p=j \quad ,
  \end{array}\right.\\
$$
on en d\'eduit:
\begin{eqnarray*}
  a_i
  &=&
  \frac{\lambda_i}{\lambda_j^2-\lambda_i^2}
  (\epsilon_j \lambda_j {[k^{-1}.X]}_{n_j}{[k^{-1}.X]}_{m_j}
  +
  \epsilon_i \lambda_i {[k^{-1}.X]}_{n_i}{[k^{-1}.X]}_{m_i}
  ) \quad ,\\
  a_j
  &=&
  -\frac{\lambda_j}{\lambda_j^2-\lambda_i^2}
  (\epsilon_j \lambda_j {[k^{-1}.X]}_{n_j}{[k^{-1}.X]}_{m_j}
  +
  \epsilon_i \lambda_i {[k^{-1}.X]}_{n_i}{[k^{-1}.X]}_{m_i}
  )\quad .
\end{eqnarray*}
On en d\'eduit les expressions $a_i^2, a_j^2, a_ia_j$ : par exemple, 
\begin{eqnarray*}
  a_i^2
  & =&
  \frac{\lambda_i^2}{{(\lambda_j^2-\lambda_i^2)}^2}
  (\lambda_j^2 {[k^{-1}.X]}_{n_j}^2{[k^{-1}.X]}_{m_j}^2
  +\lambda_i^2 {[k^{-1}.X]}_{n_i}^2{[k^{-1}.X]}_{m_i}^2\\
  &&\qquad
  +2\epsilon_i\epsilon_j \lambda_i\lambda_j {[k^{-1}.X]}_{n_j}{[k^{-1}.X]}_{m_j}{[k^{-1}.X]}_{n_i}{[k^{-1}.X]}_{m_i}
  ) \quad .
\end{eqnarray*}
\paragraph{Calculons les expressions \mathversion{bold}{$b_i$} et \mathversion{bold}{$b_j$}.} 
Commen\c cons par le premier terme de $b_p, p=i,j$.
Comme on a d'une part :
$$
B^2
\,=\,
\frac1{{(\lambda_j^2-\lambda_i^2)}^2}
(\lambda_j^2 E_{m_j,n_j}^2+ \lambda_i^2 E_{m_i,n_i}^2)
\quad ,
$$
et d'autre part :
$$
<\pr p {\vec{E}^2_{m,n}k^{-1}X} , \pr p {k^{-1}X}>
\,=\,
\left\{\begin{array}{ll}
    -{[k^{-1}X]}_m^2
    &\mbox{si}\; p=i \quad ,\\
    -{[k^{-1}X]}_n^2
    &\mbox{si}\; p=j \quad ,\\
  \end{array}\right.
$$
on en d\'eduit que 
$  <\pr p {B^2k^{-1}.X}, \pr p {k^{-1}.X}>$
vaut :
\begin{eqnarray*}
  \frac{-1}{{(\lambda_j^2-\lambda_i^2)}^2}
  (\lambda_j^2  {[k^{-1}X]}_{m_j}^2 
  + \lambda_i^2 {[k^{-1}X]}_{m_i}^2)
  &\mbox{si}\; p=i \quad ,\\
  \frac{-1}{{(\lambda_j^2-\lambda_i^2)}^2}
  (\lambda_j^2  {[k^{-1}X]}_{n_j}^2 
  + \lambda_i^2 {[k^{-1}X]}_{n_i}^2)
  &\mbox{si}\; p=j \quad .  
\end{eqnarray*}
Maintenant, calculons le deuxi\`eme terme de $b_p, p=i,j$.
Comme on a d'une part :
$$
\nn{ \pr p {Bk^{-1}.X}}^2 
\,=\,
\frac1{{(\lambda_j^2-\lambda_i^2)}^2}
(\lambda_j^2
\nn{ \pr p {E_{m_j,n_j}k^{-1}.X}}^2 
+
\lambda_i^2
\nn{ \pr p {E_{m_i,n_i}k^{-1}.X}}^2 )
\quad,
$$
et d'autre part :
$$
\nn{\pr p {\vec{E}_{m,n} k^{-1}X} }^2
\,=\,
\left\{\begin{array}{ll}
    {[k^{-1}X]}_n^2
    &\mbox{si}\; p=i \quad ,\\
    {[k^{-1}X]}_m^2
    &\mbox{si}\; p=j \quad , \\
  \end{array}\right.  
$$
on en d\'eduit :
$$
\nn{\pr p {Bk^{-1}.X}}^2 
\,=\,
\frac1{{(\lambda_j^2-\lambda_i^2)}^2}
\left\{\begin{array}{ll}
    (\lambda_j^2 {[k^{-1}X]}_{n_j}^2
    +
    \lambda_i^2 {[k^{-1}X]}_{n_i}^2)
    &\mbox{si}\; p=i \quad ,\\
    (\lambda_j^2 {[k^{-1}X]}_{m_j}^2
    +
    \lambda_i^2 {[k^{-1}X]}_{m_i}^2)
    &\mbox{si}\; p=j \quad .\\
  \end{array}\right.  
$$
Rassemblons les deux termes de $b_p,p=i,j$. 
On obtient : 
\begin{eqnarray*}
  b_i
  &=&
  \frac{\lambda_i}{{(\lambda_j^2-\lambda_i^2)}^2}
  \left(\lambda_j^2 ({[k^{-1}X]}_{n_j}^2-{[k^{-1}X]}_{m_j}^2)
    + \lambda_i^2 ({[k^{-1}X]}_{n_i}^2-{[k^{-1}X]}_{m_i}^2)\right) \quad ,\\
  b_j
  &=&
  -\frac{\lambda_j}{{(\lambda_j^2-\lambda_i^2)}^2}
  \left(\lambda_j^2 ({[k^{-1}X]}_{n_j}^2-{[k^{-1}X]}_{m_j}^2) + \lambda_i^2 ({[k^{-1}X]}_{n_i}^2-{[k^{-1}X]}_{m_i}^2)\right)
  \quad .
\end{eqnarray*}

Nous allons maintenant sommer  sur $(m,n)\in I_{i,j}$ chacune des nouvelles expressions des
$b_i,b_j, a_i^2, a_j^2, a_ia_j$.

\paragraph{Sommons \mathversion{bold}{$a_i^2, a_j^2,a_ia_j$}.} 
On remarque que 
lorsque l'on va sommer sur $(m,n)\in I_{i,j}$,
les doubles produits 
$$
2\epsilon_i\epsilon_j \lambda_i\lambda_j 
{[k^{-1}.X]}_{n_j}{[k^{-1}.X]}_{m_j}
{[k^{-1}.X]}_{n_i}{[k^{-1}.X]}_{m_i}
\quad,
$$ 
vont dispara\^\i tre et l'on va obtenir :
\begin{eqnarray*}
  \sum_{(m,n)} a_i^2(B_{m,n})
  &=&
  \frac{\lambda_i^2}{{(\lambda_j^2-\lambda_i^2)}^2}
  (\lambda_j^2+\lambda_i^2)
  \nn{ \pr j {k^{-1}.X}}^2 \nn{\pr i {k^{-1}.X}}^2
  \quad,\\
  \sum_{(m,n)} a_j^2(B_{m,n})
  &=&
  \frac{\lambda_j^2}{{(\lambda_j^2-\lambda_i^2)}^2}
  (\lambda_j^2+\lambda_i^2)
  \nn{ \pr j {k^{-1}.X}}^2 \nn{\pr i {k^{-1}.X}}^2
  \quad ,\\
  \sum_{(m,n)} a_ia_j(B_{m,n})
  & =&
  -\frac{\lambda_i\lambda_j}{{(\lambda_j^2-\lambda_i^2)}^2}
  (\lambda_j^2+\lambda_i^2)
  \nn{ \pr j {k^{-1}.X}}^2 \nn{\pr i {k^{-1}.X}}^2
  \quad .
\end{eqnarray*}
\paragraph{Sommons \mathversion{bold}{$b_i,b_j$}.} 
On obtient :
\begin{eqnarray*}
  \sum_{(m,n)} b_i
  &=&
  2\lambda_i
  \frac{\lambda_j^2+\lambda_i^2}
  {{(\lambda_j^2-\lambda_i^2)}^2}
  \left(\nn{ \pr j {k^{-1}X}}^2
    -\nn{\pr i {k^{-1}X}}^2\right)
  \quad ,\\
  \sum_{(m,n)} b_j
  &=&
  -2\lambda_j\frac{\lambda_j^2+\lambda_i^2}{{(\lambda_j^2-\lambda_i^2)}^2}
  \left(\nn{ \pr j {k^{-1}X}}^2
    -\nn{\pr i {k^{-1}X}}^2\right)
  \quad .
\end{eqnarray*}

\paragraph{Autre expression de \mathversion{bold}{$F_2$}.}
On en d\'eduit donc que 
la somme~(\ref{formule_sum_mn_der2}) est le produit de 
$$
\frac{\lambda_j^2+\lambda_i^2}{{(\lambda_j^2-\lambda_i^2)}^2}
\Pi_{p\not=i,j} \flsz {l_p} 
\quad\left( e^{i <r X_v^*,k^{-1}.X>} \right)
\quad,
$$
avec:
\begin{eqnarray*}
  -2\lambda_j(\nn{ \pr j {k^{-1}X}}^2-\nn{\pr i {k^{-1}X}}^2)\flsz {l_j}  '\flsz {l_i} 
  + \lambda_j^2\nn{ \pr j {k^{-1}.X}}^2 \nn{\pr i {k^{-1}.X}}^2\flsz {l_j}  ''\flsz {l_i} \\
  +2\lambda_i(\nn{ \pr j {k^{-1}X}}^2-\nn{\pr i {k^{-1}X}}^2) \flsz {l_i}  '\flsz {l_j}
  +\lambda_i^2\nn{ \pr j {k^{-1}.X}}^2 \nn{\pr i {k^{-1}.X}}^2\flsz {l_i}  ''\flsz {l_j} \\
  -      2\lambda_i\lambda_j\nn{ \pr j {k^{-1}.X}}^2 \nn{\pr i {k^{-1}.X}}^2 \flsz {l_i} '\flsz {l_j}  '
  \quad .
\end{eqnarray*}
Gr\^ace aux \'egalit\'es (\ref{propbeta}),
(\ref{propalpha}) et (\ref{propalphagamma}),
on peut r\'e\'ecrire ce qui pr\'ec\`ede de la mani\`ere suivante :
\begin{eqnarray*}
  \frac{\lambda_j^2+\lambda_i^2}{{(\lambda_j^2-\lambda_i^2)}^2}
  (-4\alpha_j +4\lambda_j \gamma_j\frac{\beta_i}{\lambda_i}
  -4\alpha_i-4\lambda_i\gamma_i \frac{\beta_j}{\lambda_j}
  -8\alpha_i\alpha_j)f_\phi(k^{-1}.X)
  \quad .
\end{eqnarray*}
On voit donc que l'on peut exprimer le terme $F_2$ comme:
$$
F_2
\,=\,
\left[4\sum_{i<j}\frac{\lambda_j^2+\lambda_i^2}{{(\lambda_j^2-\lambda_i^2)}^2}
  (-\alpha_j +\lambda_j \gamma_j\frac{\beta_i}{\lambda_i}
  -\alpha_i+\lambda_i\gamma_i \frac{\beta_j}{\lambda_j}
  -2\alpha_i\alpha_j)\right] .\phi
\quad .
$$

Supposons maintenant $v=2v'+1$. Il reste \`a exprimer comme op\'erateur de d\'eriv\'e et de d\'ecalage sur $\phi$ les termes $F_4,F_5$.

\subsubsection{Exprimons diff\'eremment \mathversion{bold}{$F_5$}}

D'apr\`es les \'egalit\'es (\ref{fder1}), on a:
$$
F_5
\,=\,
\sum_m -\frac2{\lambda_m}\partial_{\lambda_m}.\phi
-E_5
\quad ,
$$
o\`u le terme ``correctif'' $E_5$ est donn\'e par:
$$
E_5
\,=\,
\sum_m -\frac2{\lambda_m}
\int_{O(v)}
e^{i<k.D_2(\Lambda),A>}
\quad \partial_{\lambda_m}.f_\phi(k^{-1}.X)
\;dk 
\,=\,
\sum_m -\frac2{\lambda_m}\frac{\alpha_m}{\lambda_m}\;.\phi
\quad ,
$$
gr\^ace \`a (\ref{fder1}).
On voit donc que l'on peut exprimer le terme $F_5$ comme:
$$
F_5
\,=\,
\sum_m \left[
  -\frac2{\lambda_m}(\partial_{\lambda_m}
  +\frac{\alpha_m}{\lambda_m})
\right] .\phi \quad .
$$

\subsubsection{Exprimons diff\'eremment \mathversion{bold}{$F_4$}}

Comme dans le cas du terme  $F_2$, on voit :
$$
F_4
\,=\,
\int_{O(v)}
\sum_{i, \tilde{i}=2i,2i-1}
\frac1{\lambda_i^2}
e^{i<k.D_2(\Lambda),A>}
\frac{d}{dt}^2_{t=0}
f_\phi({(ke^{-t E_{\tilde{i},v}})}^{-1}.X)
\;dk \quad ,
$$

\paragraph{D\'eveloppement limit\'e de
  \mathversion{bold}{$f_\phi({(ke^{-tE_{\tilde{i},v}})}^{-1}.X)$}.} 
On a pour $\tilde{i}=2i$ ou $2i-1$ :
\begin{eqnarray*}
  \nn{\pr p {{(ke^{-tE_{\tilde{i},v}})}^{-1}.X}}^2
  &=&
  \nn{\pr p {k^{-1}X} }^2
  +2t<\pr p {E_{\tilde{i},v}k^{-1}X}, \pr p {k^{-1}X}>\\
  &&
  +t^2(<\pr p {E^2_{\tilde{i},v}k^{-1}X},\pr p {k^{-1}X}>
  +\nn{\pr p {E_{\tilde{i},v} k^{-1}X} }^2)+\ldots
\end{eqnarray*}
Or ici, on voit :
\begin{eqnarray*}
  <\pr p {E_{\tilde{i},v}k^{-1}X},\pr p {k^{-1}X}>
  &=&
  \left\{
    \begin{array}{ll}
      {[k^{-1}X]}_v{[k^{-1}X]}_{\tilde{i}}
      &
      \mbox{si}\; p=i\quad , \\
      0
      &
      \mbox{sinon} \quad , 
    \end{array}\right.\\
  <\pr p {E^2_{\tilde{i},v}k^{-1}X},\pr p {k^{-1}X}>
  &=&
  \left\{
    \begin{array}{ll}
      -{[k^{-1}X]}_{\tilde{i}}^2
      &
      \mbox{si}\; p=i\quad , \\
      0
      &
      \mbox{sinon} \quad ,
    \end{array}\right. \\
  \nn{\pr p {E_{\tilde{i},v} k^{-1}X} }^2
  &=&
  \left\{
    \begin{array}{ll}
      {[k^{-1}X]}_v^2
      &
      \mbox{si}\; p=i\quad , \\
      0
      &
      \mbox{sinon} \quad .
    \end{array}\right. 
\end{eqnarray*}
On peut donc faire un d\'eveloppement limit\'e : 
$$
\begin{array}{l}
  \displaystyle{
    \Pi_i \flz {l_i} 
    {\lambda_i\frac{\nn{\pr p {{(ke^{-t.E_{\tilde{i},v}})}^{-1}X} }^2}2}
    \,=\,
    \left( \flsz {l_i}  
      +\frac{\lambda_i}2 2t {[k^{-1}X]}_v{[k^{-1}X]}_{\tilde{i}} \flsz {l_i} '\right.}\\
  \displaystyle{
    \quad    \left. 
      +\frac{\lambda_i}2 t^2 (-{[k^{-1}X]}_{\tilde{i}}^2+{[k^{-1}X]}_v^2) \flsz {l_i} '
      +\frac12 {(\frac{\lambda_i}2 2t {[k^{-1}X]}_v{[k^{-1}X]}_{\tilde{i}})}^2 \flsz {l_i} ''+o(t^2)\right)\Pi_{p\not =i}\flsz {l_p} \quad.}
\end{array}
$$
On peut aussi faire un d\'eveloppement limit\'e de
$e^{i <rX_v^*,{(ke^{-t.E_{\tilde{i},v}})}^{-1}X>}$. 
Commen\c cons par:
$$  
\begin{array}{l}
  <X_v^*,{(ke^{-t.E_{\tilde{i},v}})}^{-1}X>
  \,=\,
  <X_v^*,e^{t.E_{\tilde{i},v}}k^{-1}X>\\
  \quad=\,
  <X_v^*,k^{-1}X>
  +t<X_v^*,E_{\tilde{i},v}k^{-1}X>
  +\frac{t^2}2<X_v^*,E_{\tilde{i},v}^2k^{-1}X> +o(t^2)
  \quad.
\end{array}
$$
On voit :
\begin{eqnarray*}
  <X_v^*,E_{\tilde{i},v}k^{-1}X>
  &=&
  {[k^{-1}X]}_{\tilde{i}}\\
  <X_v^*,E_{\tilde{i},v}^2k^{-1}X>
  &=&
  -{[k^{-1}X]}_v \quad .
\end{eqnarray*}
On en d\'eduit le d\'eveloppement limit\'e : 
\begin{eqnarray*}
  e^{i <rX_v^*,{(ke^{-t.E_{\tilde{i},v}})}^{-1}X>}
  &=&
  e^{i <rX_v^*,k^{-1}X>}\\
  &&\;
  \left( 1
    +x(t{[k^{-1}X]}_{\tilde{i}}-\frac{t^2}2{[k^{-1}X]}_v)
    +\frac12 {(xt{[k^{-1}X]}_{\tilde{i}})}^2+o(t^2)\right)
  \quad.
\end{eqnarray*}
On peut donc obtenir le d\'eveloppement limit\'e de
$  f_\phi({(ke^{-tE_{\tilde{i},v}})}^{-1}.X)$ 
comme produit des d\'eveloppements de 
$$
\Pi_i \flz {l_i}  
{\lambda_i\frac{\nn{\pr p {{(ke^{-t.E_{\tilde{i},v}})}^{-1}X} }^2}2}
\qquad\mbox{et}\quad
e^{i <rX_v^*,{(ke^{-t.E_{\tilde{i},v}})}^{-1}X>}
\quad.
$$ 
Le terme en $\frac12 t^2$ donne une autre expression de :
$$
\begin{array}{l}
  \displaystyle{
    \frac{d}{dt}^2_{t=0} f_\phi({(ke^{-tE_{\tilde{i},v}})}^{-1}.X)}\\
  \displaystyle{\quad=\,
    \left(\lambda_i  (-{[k^{-1}X]}_{\tilde{i}}^2+{[k^{-1}X]}_v^2)\flsz {l_i} '
      +{(\lambda_i {[k^{-1}X]}_v{[k^{-1}X]}_{\tilde{i}})}^2 \flsz {l_i} ''
      +
      2r\lambda_i{[k^{-1}X]}_v{[k^{-1}X]}_{\tilde{i}}^2\flsz {l_i} '\right.}\\
  \displaystyle{\qquad    \left.
      +
      ({(r{[k^{-1}X]}_{\tilde{i}})}^2-r{[k^{-1}X]}_v)
      \flsz {l_i}  \right) 
    \Pi_{i\not =p} \flsz {l_p}  
    \; e^{i <rX_v^*,k^{-1}X>} 
    \quad.}
\end{array}
$$
\paragraph{Autre expression de \mathversion{bold}{$F_4$}.}
Il reste \`a sommer sur $\tilde{i}$ :
$$
\begin{array}{l}
  \displaystyle{
    \sum_{\tilde{i}=2i,2i-1}\frac{d}{dt}^2_{t=0}
    f_\phi({(ke^{-tE_{\tilde{i},v}})}^{-1}.X)}\\
  \displaystyle{\quad=\,
    \left(
      \lambda_i  (-\nn{\pr i {[k^{-1}X]}}^2+2{[k^{-1}X]}_v^2)\flsz {l_i} '
      +\lambda_i^2 {[k^{-1}X]}_v^2\nn{\pr i {[k^{-1}X]}}^2 \flsz {l_i}  ''\right.}\\
  \displaystyle{\qquad    \left.+ 2r\lambda_i{[k^{-1}X]}_v\nn{\pr i {[k^{-1}X]}}^2 \flsz {l_i}'
      +({r}^2\nn{\pr i {[k^{-1}X]}}^2-2r{[k^{-1}X]}_v)\flsz {l_i} 
    \right) }\\
  \displaystyle{\qquad \qquad   \Pi_{i\not =p} \flsz {l_p} \; 
    e^{i <rX_v^*,k^{-1}X>} }\\
  \displaystyle{\quad=\,
    \left(-2\alpha_i
      -2\lambda_i\gamma_i \partial_{r}^2
      -4ri\partial_{r}\alpha_i
      +({r}^22\frac{\beta_i}{\lambda_i}+2ir\partial_{r})
    \right).f_\phi(k^{-1}.X)
    \quad,}
\end{array}
$$
gr\^ace aux \'egalit\'es (\ref{propalpha}), (\ref{propbeta}) et (\ref{propalphagamma}).
Par lin\'earit\'e des op\'erateurs $\alpha_i,\beta_i,\gamma_i,\partial_{r}$, on a:
$$
F_4
\,=\,
\left[\sum_i 
  \frac2{\lambda_i^2}(
  -\alpha_i-\lambda_i\gamma_i \partial_{r}^2
  -2ri\partial_{r}\alpha_i   
  +{r}^2\frac{\beta_i}{\lambda_i}+ir\partial_{r})
\right].\;\phi
\quad.
$$

\subsubsection{R\'esumons}
On obtient donc l'expression :
\begin{eqnarray*}
  \nn{A}^2\phi
  &:=&
\sum_m \partial_{\lambda_m}^2
    -2\frac{\alpha_m}{\lambda_m}.\partial_{\lambda_m}
    +\frac{\alpha_m^2+\alpha_m}{\lambda_m^2}
  +\sum_{i<j}
  \frac{-4}
  {\lambda_j^2-\lambda_i^2}(\lambda_i\partial_{\lambda_i}-\alpha_i -\lambda_j\partial_{\lambda_j}+\alpha_j)\\
  &&\quad
  +4\sum_{i<j}\frac{\lambda_j^2+\lambda_i^2}{{(\lambda_j^2-\lambda_i^2)}^2}
  (-\alpha_j +\lambda_j \gamma_j\frac{\beta_i}{\lambda_i}
  -\alpha_i+\lambda_i\gamma_i \frac{\beta_j}{\lambda_j}
  -2\alpha_i\alpha_j)\quad,
\end{eqnarray*}
\`a laquelle il faut ajouter si $v=2v'+1$ :
\begin{eqnarray*}
    \sum_i 
    \frac2{\lambda_i^2}(
    -\alpha_i-\lambda_i\gamma_i \partial_{r}^2
    -2ri\partial_{r}\alpha_i   
    +{r}^2\frac{\beta_i}{\lambda_i}+ir\partial_{r})
    +\sum_m -\frac2{\lambda_m}(\partial_{\lambda_m}
    +\frac{\alpha_m}{\lambda_m})\\
    =
    \sum_i 
    \frac2{\lambda_i^2}(-\lambda_i\gamma_i \partial_{r}^2
    -2ri\partial_{r}\alpha_i   
    +{r}^2\frac{\beta_i}{\lambda_i}+ir\partial_{r})
    -\frac2{\lambda_i}\partial_{\lambda_i}
  \end{eqnarray*}

Ceci ach\`eve la d\'emonstration du lemme~\ref{lem_expr_TX_TA}.

\subsection{Propri\'et\'es des op\'erateurs $\Xi,\aleph$}

Nous aurons besoin de conna\^\i tre
une expression manipulable des
puissances de l'op\'erateur~$\Xi^2+\aleph$.
Nous utiliserons toujours dans ce qui suit,
la notation $\nn{a}=\sum_i a_i$
pour un $n$-upplet $a=(a_1,\ldots,a_n)\in \Zb^n$.

\begin{lem}[Expression des puissances de \mathversion{bold}{$\Xi^2+\aleph$}]
  \label{lem_expr_puiss_T}
  Pour $\epsilon \in \Nb$, 
  dans le cas $v=2v'$,
  l'op\'erateur 
  ${(\Xi^2+\aleph)}^{\epsilon} $ 
  peut se mettre sous la forme de la somme 
  des termes suivants :
  \begin{eqnarray*}
    E_P
    :=
    Q_P(\tau)
    \;
    \Pi_{i<j}
    \frac{{(\lambda_i^2+\lambda_j^2 )}^{N_{i,j}}}
    {{(\lambda_j^2-\lambda_j^2)}^{D_{i,j}}}
    \quad
    \Pi_i\;
    \lambda_i^{a_i^+-a_i^-} 
    \partial_{\lambda_i}^{b_i}
    \;
    \{ l_i^{c_i}
    \Delta_i^{d_i} \}
    \quad\left(
      r^e {(i\partial_r)}^f
    \right)\quad,
  \end{eqnarray*}
  o\`u
  $Q_P\in \Qb[X^+_1,X^-_1,\ldots,X^+_{v'},X^-_{v'}]$ 
  et  
  $\tau
  =
  (\tau_1^+,\tau_1^-
  ,\ldots,
  \tau_{v'}^+,\tau_{v'}^-)$.  
  La somme est \`a prendre 
  sur l'ensemble des param\`etres $I_\epsilon$
  qui est d\'ecrit dans ce qui suit :
  \begin{itemize}
  \item $I_\epsilon$ est l'ensemble des param\`etres 
    $P=(A,F)$
    o\`u
    $A=(a,b,c,d)$ si $v=2v'$
    et 
    $A=(a,b,c,d,e,f)$ si $v=2v'+1$,
    et $F=(N,D)$;
  \item les param\`etres 
    $a=(a^+,a^-)\in \Nb^{v'}\times\Nb^{v'}$
    et $b,c,d\in \Nb^{v'}$, 
    et dans le cas $v=2v'+1$, $e,f\in \Nb$ 
    v\'erifient :
    \begin{equation}
      \label{condition_aef_lem_exp_puiss_T}
      \begin{array}{c}
2\nn{N}+\nn{a-}+\nn{a+}\leq 4(\nn{b}+2\nn{d})-2\epsilon
\, ,\quad
e\leq f+2\nn{d}-2\epsilon
\,,\quad \nn{c}\leq 2\nn{d}
\quad,\\
        \left(\mbox{et}\quad 
          \mbox{les entiers}\quad
          e,f \; 
          \mbox{ont m\^eme parit\'e si v=2v'+1} \right)
        \quad,
      \end{array}
    \end{equation}
et :
\begin{equation}
  \label{condition_bdf_lem_exp_puiss_T}
0<\nn{b}+f + \nn{d}\leq 4\epsilon \quad.
\end{equation}
  \item les param\`etres 
    $N=(N_{i,j}),D=(D_{i,j})\in \Nb^q$
    v\'erifient :
    \begin{equation}
      \label{condition_ND_lem_exp_puiss_T}
      \nn{N}\leq \epsilon\; ,
      \qquad
      \nn{D}\leq 2\epsilon 
      \quad;  
    \end{equation}
  \item les param\`etres 
    $a,b,N,D$ et $e,f$ dans le cas $v=2v'+1$
    ``qui sont les param\`etres portant sur $\Lambda$ 
    et $r$ dans le cas $v=2v'+1$'',
    v\'erifient la relation d'homog\'en\'eit\'e :
    \begin{equation}
      \label{condition_homogeneite_lem_exp_puiss_T}    
      2\nn{N}-2\nn{D} 
      +\nn{a^+}-\nn{a^-}
      -\nn{b}
      \qquad\left(+\frac{e-f}2 \right)
      \,=\,
      -2\epsilon \quad.
    \end{equation}
  \end{itemize}
  Les polyn\^omes $Q_P,P\in I_\epsilon$ sont tels que la somme 
  $\sum_P E_P$ 
  soit sym\'etrique en chacun des groupes de termes 
  $$
  \begin{array}{c}
    \lambda_i\; ,\; i=1,\ldots, v'
    \qquad\qquad
    l_i \; ,\; i=1,\ldots, v'
    \qquad\qquad
    \lambda_i^2\pm\lambda_j^2 \; ,\; i<j
    \quad,\\
    \Delta_i \; ,\; i=1,\ldots, v'
    \qquad\qquad
    \partial_{\lambda_i}\; ,\; i=1,\ldots, v'
    \quad.
  \end{array}
  $$
  En particulier, 
  le cardinal de l'ensemble $I_\epsilon$
  est finie, et ne d\'epend que de $\epsilon$.
\end{lem}

La condition~(\ref{condition_bdf_lem_exp_puiss_T}) 
exprime le fait que dans chaque terme $E_P$,
il y a au moins un d\'erivation (discr\`ete ou non),
et qu'il ne peut y avoir en tout que $4\epsilon$.
D'apr\`es la remarque~\ref{rem_derivation_TAX},
c'est d\'ej\`a le cas pour $\epsilon=1$.
La preuve est faite par r\'ecurrence sur $\epsilon$.

\begin{proof}[du lemme~\ref{lem_expr_puiss_T}]
  Le pas $\epsilon=1$ est vrai,
  d'apr\`es
  les \'egalit\'es~(\ref{egalite_puissance_alpha_p=1}),
  (\ref{egalite_puissance_beta_p=1}),
  (\ref{egalite_puissance_gamma_p=1}),
  (\ref{egalite_alpha12_casp=1}),
  (\ref{egalite_betagamma})   
  et les lemmes~\ref{lem_puissance_beta}
  et~\ref{lem_puissance_der,alpha}
  donn\'es dans la sous-section~\ref{subsec_prop_fcnlag_0}.
  Supposons le pas $\epsilon$ vrai. 
  Montrons que le pas $\epsilon+1$ est alors vrai.
  Commen\c cons par montrer que l'op\'erateur
  $\Xi^2{(\Xi^2+\aleph)}^\epsilon$
  est de la forme voulue.
  D'apr\`es le lemme~\ref{lem_puissance_beta}
  et l'hypoth\`ese de r\'ecurrence
  en particulier la sym\'etrie des termes,
  il suffit de montrer que
  les op\'erateurs
  $$
  \lambda_{i_0}^{-1}\{l_{i_0}^{q_{i_0}}\Delta_{i_0}^{p_{i_0}+q_{i_0}}\}\;
  \lambda_{j_0}^{-1}\{l_{j_0}^{q_{j_0}}\Delta_{j_0}^{p_{j_0}+q_{j_0}}\}\;
  \left({(i\partial_{r})}^{2p*}\right)
  Q_P(\tau)
  \Pi_i\;
  \lambda_i^{a_i^+-a_i^-} 
  \partial_{\lambda_i}^{b_i}
  \; 
  \{ l_i^{c_i}
  \Delta_i^{d_i} \}
  \quad\left(
    {r}^e {(i\partial_{r})}^f
  \right)\quad,
  $$
  pour les param\`etres :
  $$
  p_{i_0}+p_{j_0}\; (+p^*)=2 
  \quad ,\quad
  p_{i_0},p_{j_0},p^*\in \Nb 
  \quad ,\quad
  0\leq q_{i_0} \leq p_{i_0}
  \quad ,\quad
  0\leq q_{j_0} \leq p_{j_0}
  \quad,
  $$
  et $a,b,c,d(,e,f)$
  v\'erifiant les conditions~(\ref{condition_aef_lem_exp_puiss_T})
et~(\ref{condition_bdf_lem_exp_puiss_T})
  peuvent s'\'ecrire comme une somme d'op\'erateurs de la forme :
  \begin{eqnarray*}
    \tilde{Q}_P(\tau)
    \;\Pi_i\;
    \lambda_i^{\tilde{a^+}_i-\tilde{a^-}_i} 
    \partial_{\lambda_i}^{\tilde{b}_i}
    \;
    \{ l_i^{\tilde{c}_i}
    \Delta_i^{\tilde{d}_i} \}
    \quad\left(
      {r}^{\tilde{e}} {(i\partial_{r})}^{\tilde{f}}
    \right)\quad.
  \end{eqnarray*}
  avec 
  $\tilde{Q}_P\in \Qb[X^+_1,X^-_1,\ldots,X^+_{v'},X^-_{v'}]$,
  $\tilde{a},\tilde{b},\tilde{c},\tilde{d}
  (,\tilde{e},\tilde{f})$  
  v\'erifiant les 
  conditions~(\ref{condition_aef_lem_exp_puiss_T}) 
  pour $\epsilon+1$. 
  Les propri\'et\'es de d\'erivation en $r$, 
  et les 
  \'egalit\'es~(\ref{egalite_ldeltaPl})-(\ref{egalite_ldelta2Pl}) 
  donn\'es dans la sous-section~\ref{subsec_prop_fcnlag_0}
  permettent d'affirmer que c'est vrai.

  Ensuite montrons que l'op\'erateur
  $$
  \sum_m {( \partial_{\lambda_m}
    -\frac{\alpha_m}{\lambda_m} )}^2 
  {(\Xi^2+\aleph)}^\epsilon
  \quad,
  $$
  est de la forme voulue.
  D'apr\`es le lemme~\ref{lem_puissance_der,alpha}
  et l'hypoth\`ese de r\'ecurrence,
  il suffit de montrer que l'op\'erateurs
  $$
  \sum_{i_0}
  \lambda_{i_0}^{-a'}
  \partial_{\lambda_{i_0}}^{b'}
  \{l_{i_0}^{d'}\Delta_{i_0}^{c'}\}
  \;
  Q_P(\tau)
  \;
  \Pi_{i<j}
  \frac{{(\lambda_i^2+\lambda_j^2 )}^{N_{i,j}}}
  {{(\lambda_j^2-\lambda_j^2)}^{D_{i,j}}}
  \;
  \Pi_i\;
  \lambda_i^{a_i^+-a_i^-} 
  \partial_{\lambda_i}^{b_i}
  \; 
  \{ l_i^{c_i}
  \Delta_i^{d_i} \} 
  \quad,
  $$
  pour les param\`etres
  $P=(a,b,c,d(,e,f); N,D)\in I_\epsilon$
  et :
  $$
  (a',b',c',d')\in\Nb^4
  \quad,\quad
  a'+b'=2
  \quad,\quad
  b'+c'\leq 2 
  \quad,\quad
  c'\leq 2d'
  \quad,
  $$
  peuvent s'\'ecrire comme une somme 
  d'op\'erateurs de la forme :
  $$
  \tilde{Q}_P(\tau)
  \;
  \Pi_{i<j}
  \frac{{(\lambda_i^2+\lambda_j^2 )}^{\tilde{N}_{i,j}}}
  {{(\lambda_j^2-\lambda_j^2)}^{\tilde{D}_{i,j}}}
  \quad
  \Pi_i\;
  \lambda_i^{\tilde{a}_i} 
  \partial_{\lambda_i}^{\tilde{b}_i}
  \;
  \{ l_i^{\tilde{c}_i}
  \Delta_i^{\tilde{d}_i} \}
  \quad\left(
    {r}^{\tilde{e}} {(i\partial_{r})}^{\tilde{f}}
  \right)\quad;
  $$
  la somme se fait sur les param\`etres
  $\tilde{P}=
  (\tilde{a},\tilde{b},\tilde{c},\tilde{d}(,\tilde{e},\tilde{f});
  \tilde{N},\tilde{D})\in I_{\epsilon+1}$;
  les polyn\^omes
  $\tilde{Q}_P\in \Qb[X^+_1,X^-_1,\ldots,X^+_{v'},X^-_{v'}]$
  sont tels que la somme est sym\'etrique.
  Les deux \'egalit\'es qui suivent,
  et  celles (\ref{egalite_ldeltaPl})-(\ref{egalite_ldelta2Pl}) 
  permettent d'affirmer que c'est vrai.
  \begin{eqnarray*}
    &&
    \left(\lambda_1^{-1}\partial_{\lambda_1}
      +
      \lambda_2^{-1}\partial_{\lambda_2}\right)
    \frac{{(\lambda_1^2+\lambda_2^2 )}^N}
    {{(\lambda_1^2-\lambda_2^2)}^D}
    \,=\,  
    2N 
    \frac{{(\lambda_1^2+\lambda_2^2 )}^{N-1}}
    {{(\lambda_1^2-\lambda_2^2)}^{D}}\quad, \\
    &&
    \left( \partial_{\lambda_1}^2+\partial_{\lambda_2}^2\right).
    \frac{{(\lambda_1^2+\lambda_2^2 )}^N}
    {{(\lambda_1^2-\lambda_2^2)}^D}
    \,=\,  
    (2N - 8ND )
    \frac{{(\lambda_1^2+\lambda_2^2 )}^{N-1}}
    {{(\lambda_1^2-\lambda_2^2)}^{D}} 
    +
    4D(D+1) 
    \frac{{(\lambda_1^2+\lambda_2^2 )}^N}
    {{(\lambda_1^2-\lambda_2^2)}^{D+1}}
    \quad.
  \end{eqnarray*}

  Les m\^emes arguments montrent que l'op\'erateur
  $$
  \sum_m \frac2{\lambda_m}(\partial_{\lambda_m}
  -\frac{\alpha_m}{\lambda_m}).
  Q_P(\tau)
  \;
  \Pi_{i<j}
  \frac{{(\lambda_i^2+\lambda_j^2 )}^{N_{i,j}}}
  {{(\lambda_j^2-\lambda_j^2)}^{D_{i,j}}}
  \;
  \Pi_i\;
  \lambda_i^{a_i} 
  \partial_{\lambda_i}^{b_i}
  \; 
  \{ l_i^{c_i}
  \Delta_i^{d_i} \} 
  \quad,
  $$
  est \'egalement de la forme voulue.

  On calcule directement :
  $$
  \left(\lambda_1\partial_{\lambda_1}
    +
    \lambda_2\partial_{\lambda_2}\right)
  \frac{{(\lambda_1^2+\lambda_2^2 )}^N}
  {{(\lambda_1^2-\lambda_2^2)}^D}
  \,=\,
  2N
  \frac{{(\lambda_1^2+\lambda_2^2 )}^{N-1}}
  {{(\lambda_1^2-\lambda_2^2)}^{D-1}}
  +
  2D
  \frac{{(\lambda_1^2+\lambda_2^2 )}^N}
  {{(\lambda_1^2-\lambda_2^2)}^D}
  \quad,
  $$
  ce qui permet d'affirmer que l'op\'erateur 
  $$
  \frac1
  {\lambda_{j_0}^2-\lambda_{i_0}^2}
  \left( \lambda_{i_0}\partial_{\lambda_{i_0}}
    -\lambda_{j_0}\partial_{\lambda_{j_0}} \right)
  .{(\Xi^2+\aleph)}^\epsilon
  \quad,
  $$
  est de la forme voulue.

  Enfin, les propri\'et\'es de d\'erivation en $r$, 
  et les 
  \'egalit\'es~(\ref{egalite_ldeltaPl})-(\ref{egalite_ldelta2Pl}) 
  permettent d'affirmer 
  que les autres op\'erateurs de $\aleph$ 
  appliqu\'es \`a ${(\Xi^2+\aleph)}^\epsilon$ 
  sont \'egalement de la forme voulue.

\end{proof}

Ce lemme technique permet d'estimer les normes 
$N_{\iota,\epsilon}$, objet de la 
proposition~\ref{prop_controle_N}.

\subsection{D\'emonstration de la proposition~\ref{prop_controle_N}}

Dans cette preuve, 
$C$ d\'esigne une constante 
qui ne d\'epend que de $v,\epsilon$ et qui peut \'evoluer au cours du calcul.
Nous convenons aussi que le signe $\sim$ entre deux expressions $a$ et $b$ strictement positives,
signifie qu'il existe une constante $C$ telle que 
$C^{-1}b\leq a\leq Cb$.

Soit $\iota\in I$, $\epsilon>0$ et
$g\in L^\infty(m)$.
On suppose que le support de $g$ est inclus dans celui de $\chi_\iota$
(en particulier $g\in L^2(m)$), 
et que $g$ est assez r\'eguli\`ere.
On note $G\in {L^2(N)}^\natural$ 
la fonction dont la transform\'ee
de Fourier est $g$.

D'apr\`es la formule~(\ref{formule_plancherel_radiale})
de Plancherel, on a :
\begin{eqnarray*}
  N_{\iota,\epsilon}(g)^2 
 &\sim&
  \nd{(1+ s_\iota^{2\epsilon} \nn{n}^{4\epsilon} ) G (n)}_{L^2(n\in N)}^2\\
  &&\,=\,
  \int_\Pc 
  \nn{<(1+s_\iota^{2\epsilon} \nn{n}^{4\epsilon} )
    G (n),\phi>}^2 dm(\phi) \quad.
\end{eqnarray*}
Maintenant utilisons les op\'erateurs $\Xi,\aleph$
que nous avons d\'efinis sur les fonctions $g:\Pc\rightarrow \Rb$ assez r\'eguli\`eres.
D'apr\`es le lemme~\ref{lem_expr_TX_TA}, on a :
$$
(1+s_\iota^{2\epsilon} \nn{n}^{4\epsilon} )\phi(n)
\,=\,
(\Id +s_\iota^{2\epsilon}{(\Xi^2+\aleph)}^{\epsilon} ).\phi(n)
\quad,
$$
et dans cette derni\`ere expression,
l'op\'erateur $(\Id +s_\iota^{2\epsilon}{(\Xi^2+\aleph)}^{\epsilon})$
agit sur les param\`etres de la fonction sph\'erique $\phi$
(nous identifions la fonction $\phi$ et ses param\`etres dans $\Pc$).

Par cons\'equent, on a :
\begin{eqnarray*}
<(1+s_\iota^{2\epsilon} \nn{n}^{4\epsilon} )
    G (n),\phi>
&=&
  \int_N    G (n) (1+s_\iota^{2\epsilon} \nn{n}^{4\epsilon} )\phi(n) dn\\
&=&
\int_N    G (n) (\Id +s_\iota^{2\epsilon}{(\Xi^2+\aleph)}^{\epsilon} ).\phi(n) dn\\
&=&
(\Id +s_\iota^{2\epsilon}{(\Xi^2+\aleph)}^{\epsilon}).
\int_N  G (n)\phi(n) dn\\
&=&
(\Id +s_\iota^{2\epsilon}{(\Xi^2+\aleph)}^{\epsilon}).g(\phi)
\quad,
\end{eqnarray*}
puis :
$$
  N_{\iota,\epsilon}(g)^2 
 \,\sim\,
  \nd{(\Id +s_\iota^{2\epsilon}{(\Xi^2+\aleph)}^{\epsilon} )
    g}_{L^2(m)}^2\quad.
$$

Supposons $\epsilon\in\Nb$.
Comme on a $s_\iota\sim \sum_i \lambda_i(2l_i+1)+r^2$,
et d'apr\`es le lemme~\ref{lem_expr_puiss_T}, 
on voit :
$$
  \nd{(\Id +s_\iota^{2\epsilon}{(\Xi^2+\aleph)}^{\epsilon} )
    g}_{L^2(m)}^2
\,\leq\,C^2\,
\left(
  \nd{g}_{L^2(m)}^2
  +\sum_{P\in J_\epsilon}
\nd{F_P. g}_{L^2(m)}^2\right)
\quad,
$$
o\`u :
  \begin{eqnarray*}
    F_P
    :=
    Q_P(\tau)
    \;
    \Pi_{i<j}
    \frac1{{(\lambda_j^2-\lambda_i^2 )}^{D_{i,j}}}
    \quad
    \Pi_i
    \lambda_i^{a_i} 
    \partial_{\lambda_i}^{b_i}
    \{ l_i^{c_i}
    \Delta_i^{d_i} \}
    \quad\left(
      r^e {(i\partial_r)}^f
    \right)\quad,
  \end{eqnarray*}
 o\`u
  $Q_P\in \Qb[X^+_1,X^-_1,\ldots,X^+_{v'},X^-_{v'}]$ 
  et  
  $\tau
  =
  (\tau_1^+,\tau_1^-
  ,\ldots,
  \tau_{v'}^+,\tau_{v'}^-)$.  
  La somme est \`a prendre 
  sur l'ensemble des param\`etres $J_\epsilon$
  qui est d\'ecrit dans ce qui suit :
  \begin{itemize}
  \item $J_\epsilon$ est l'ensemble fini des param\`etres 
    $P=(A,D)$
    o\`u
    $A=(a,b,c,d)$ si $v=2v'$
    et 
    $A=(a,b,c,d,e,f)$ si $v=2v'+1$;
  \item les param\`etres 
    $a\in \Zb^{v'}$
    et $b,c,d\in \Nb^{v'}$, 
    et dans le cas $v=2v'+1$, $e,f\in \Nb$ 
    v\'erifient :
    \begin{equation}
      \label{condition_aef_lem_exp_puiss_T1}
      \begin{array}{c}
2\nn{N}+\nn{a-}+\nn{a+}\leq 4(\nn{b}+2\nn{d})
\, ,\quad
e\leq f+2\nn{d}
\,,\quad \nn{c}\leq 2\nn{d}
        \quad,\\
        \left(\mbox{et}\quad 
          \mbox{les entiers}\quad
          e,f \; 
          \mbox{ont m\^eme parit\'e si v=2v'+1} \right)
        \quad,
      \end{array}
    \end{equation}
et :
\begin{equation}
  \label{condition_bdf_lem_exp_puiss_T1}
0< \nn{b}+f + \nn{d}\leq 4\epsilon \quad.
\end{equation}
  \item le param\`etre
    $D=(D_{i,j})\in \Nb^q$
    v\'erifie :
    \begin{equation}
      \label{condition_D_lem_exp_puiss_T1}
      \nn{D}\leq 2\epsilon 
      \quad;  
    \end{equation}
  \item les param\`etres 
    $a,b,N,D$ et $e,f$ dans le cas $v=2v'+1$
    ``qui sont les param\`etres portant sur $\Lambda$ 
    et $r$ dans le cas $v=2v'+1$'',
    v\'erifient la relation d'homog\'en\'eit\'e :
    \begin{equation}
      \label{condition_homogeneite_lem_exp_puiss_T1}    
      -2\nn{D} 
      +\nn{a}
      -\nn{b}
      \qquad\left(+\frac{e-f}2 \right)
      \,=\,
      0 \quad.
    \end{equation}
  \end{itemize}

\subsubsection{Majoration de \mathversion{bold}{$\nd{F_P. g}^2,P\in J_\epsilon$}}

Les op\'erateurs~$\tau^+$ de translations ne changent pas la somme sur $\Nb^{v'}$;
et les op\'erateurs~$\tau^-$ 
peuvent juste faire dispara\^\i tre quelques termes.
Donc on a :
\begin{eqnarray*}
  \nd{F_P. g}^2
  \,\leq\,C^2\,
  \int_{\Rb}\int_\Lc \sum_l
  \Pi_{i<j}
  \frac1{{(\lambda_j^2-\lambda_i^2)}^{2D_{i,j}}}
  \;
  \Pi_i\;
  \lambda_i^{2a_i} 
  l_i^{2c_i}
  \;
  \nn{r}^{2e} 
  \;\\
  \nn{
    \Pi_i
    \partial_{\lambda_i}^{b_i}
    \Delta_i^{d_i}
    \;
    \partial_{r}^f
    .g}^2
  \frac{d\eta'}{d\Lambda}
  d\Lambda dr
  \quad,  
\end{eqnarray*}
les param\`etres pour les termes en~$r$, 
\'etant omis dans le cas~$v=2v'$, 
$d\eta'/d\Lambda$ d\'esignant
la fonction densit\'e 
de la mesure~$\eta'$ sur le simplexe~$\Lc$ 
par rapport \`a la mesure de Lebesgue~$d\Lambda$.
Or le support de la fonction~$g$ 
est inclus dans celui de~$\chi_\iota$.
Outre l'\'equivalent de~$d\eta'/d\Lambda$
du lemme~\ref{lem_equivalence_densite_eta_supp_chi_iota},
on a sur le support de $\chi_\iota$,
les estimations suivantes :
\begin{itemize}
\item
  $\nn{r}^4\sim  2^{\theta}$, d'o\`u
  $\nn{r}^{2e} \leq C 2^{\theta \frac e2}$,
\item
  $\lambda_i^2\sim  2^{\eta_i}$
  d'o\`u
  $\Pi_i\lambda_i^{2a_i} \leq C
  \Pi_i 2^{\eta_i a_i}$,
\item
  $ l_i\sim  2^{\zeta_i}$
  d'o\`u
  $\Pi_i\; l_i ^{2c_i}\leq C^2
  \Pi_i\; 2^{2 c_i\zeta_i}$,
\item
  $\lambda_j^2-\lambda_i^2 \sim 2^{\delta_{i,j}}$
  d'o\`u
  $\Pi_{i<j}
  {(\lambda_j^2-\lambda_j^2 )}^{2D_{i,j}}
  \geq C^{-2}
  \Pi_{i<j}2^{2 \delta_{i,j}D_{i,j}}$.
\end{itemize}
On obtient la majoration :
$$
\nd{F_{P}. g}^2
\,\leq\,C^2\,
2^{2d(\eta,\delta)}
2^{\exp}
\int_\Lc \int_{\Rb}\sum_l
\nn{\Pi_i
  \partial_{\lambda_i}^{b_i}
  \Delta_i^{d_i}
  \partial_{r}^f
  .g}^2 
dr d\Lambda 
\quad,
$$
o\`u on note momentan\'ement l'exposant :
$$
\exp
:=
-2\delta.D
+a.\eta
+\frac e2 \theta
+2c.\zeta
\quad.
$$

Notons $\delta_\iota^m=\min_{i<j}\delta_{i,j}$ et $\zeta_\iota^M=\max_i\zeta_i$.
On a $\delta.D\geq \delta_\iota^m \nn{D}$.
Utilisons
$P\in J_\epsilon$.
On a 
d'apr\`es la condition~(\ref{condition_homogeneite_lem_exp_puiss_T1}) :
$$
-2\delta.D
\,\leq\,
-2\delta_\iota^m\nn{D}
\,=\,
(\nn{b}-\nn{a}+\frac{f-e}2)\delta_\iota^m
\quad,
$$
d'o\`u :
\begin{eqnarray*}
\exp &\leq&
(\nn{b}-\nn{a}+\frac{f-e}2)\delta_\iota^m
+a.\eta
+\frac e2 \theta
+2c.\zeta\\
&&=\,
\nn{b}\delta_\iota^m
+\sum_ia_i(\eta_i-\delta_\iota^m)
+\frac e2 (-\delta_\iota^m+\theta)
+\frac f2 \delta_\iota^m +2c.\zeta\quad.
\end{eqnarray*}
Or on a gr\^ace \`a la condition~(\ref{condition_aef_lem_exp_puiss_T1}) :
\begin{eqnarray*}
c.\zeta
&\leq&
\zeta_\iota^M\nn{c}
\,\leq\,
\zeta_\iota^M2\nn{d}
\quad,\\
\frac e2 (-\delta_\iota^m+\theta)
&\leq &
\frac12(f+2\nn{d})\max(-\delta_\iota^m+\theta,0)
\quad,\\
\sum_ia_i(\eta_i-\delta_\iota^m)
&\leq&
\sum \nn{a_i}\nn{\eta_i-\delta_\iota^m}
\,\leq\,
4(\nn{b}+2\nn{d})\max_i\nn{\eta_i-\delta_\iota^m}
\quad,\\
\end{eqnarray*}
On obtient donc  que l'exposant $\exp$ est major\'e par :
\begin{eqnarray*}
\nn{b}\delta_\iota^m
+4(\nn{b}+2\nn{d})\max_i\nn{\eta_i-\delta_\iota^m}
+\frac12(f+2\nn{d})\max(-\delta_\iota^m+\theta,0)
+\frac f2 \delta_\iota^m+4\zeta_\iota^M\nn{d}
\quad,\\
=\,
L_\iota\nn{b}+R_\iota f+D_\iota \nn{d}
\end{eqnarray*}
 par d\'efinition de $L_\iota,R_\iota,D_\iota$,
puis :
$$
  \nd{F_{P}. g}^2
  \,\leq\, C^2
  2^{2d(\eta,\delta)}
    2^{L_\iota\nn{b}+R_\iota f+D_\iota \nn{d}}
  \int_\Lc \int_{\Rb} \sum_l
  \nn{
    \Pi_i
      \partial_{\lambda_i}^{b_i}
      \Delta_i^{d_i}
      \partial_{r}^f
    .g}^2 
  dr d\Lambda\quad.
$$
On note encore 
$g:\Rb\times\Rb^{v'}\times\Zb^{v'}\rightarrow \Cb$
la fonction 
$g:\Pc \rightarrow \Cb$
\'etendue de mani\`ere triviale
au sens~(\ref{formule_etendre_triv_fcn_Pc}).
On a obtenu 
la majoration de la norme
pour $P\in J_\epsilon$ :
$$
  \nd{F_{P}. g}^2_{L^2(m)}
  \,\leq\, C^2
  2^{2d(\eta,\delta)}
  \nd{
    \Pi_i
    {(2^{\frac{L_\iota}2}\partial_{\lambda_i})}^{b_i}
    {(2^{\frac{D_\iota}2}\Delta_i)}^{2d_i}
    {(2^{\frac{R_\iota}2}\partial_{r})}^f
g}^2_{L^2(\Rb\times\Rb^{v'}\times\Zb^{v'})}
  \quad.  
$$

\subsubsection{Sommation sur \mathversion{bold}{$P\in J_\epsilon$}}

Gr\^ace \`a la formule de Plancherel pour $g$, on a :
$$
  \nd{(\Id +s_\iota^{2\epsilon}{(\Xi^2+\aleph)}^{\epsilon} ) g}_{L^2(m)}^2
\,\leq\,C^2\,
  2^{2d(\eta,\delta)}
\sum_{P\in J_\epsilon'}
  \nd{F_{P}. g}^2_{L^2(\Rb\times\Rb^{v'}\times\Zb^{v'})}
  \;,
$$
o\`u $J_{\epsilon}'$ est l'ensemble $J_{\epsilon}$,
auquel on a rajout\'e l'\'el\'ement $P=0$
et pour lequel $F_P'=\Id$.

Utilisons la formule de Plancherel sur $L^2(\Rb\times\Rb^{v'}\times\Zb^{v'})$;
nous noterons les variables duales  
de $l\in \Zb^{v'},\Lambda\in \Rb^{v'},r\in \Rb$ par
$\widehat{l}\in \Tb^{v'},\widehat{\Lambda}\in\Rb^{v'}$
et $\widehat{r}\in \Rb$ si $v=2v'+1$.
\begin{eqnarray*}
&&  \nd{
    \Pi_i
    {(2^{\frac{L_\iota}2}\partial_{\lambda_i})}^{b_i}
    {(2^{\frac{D_\iota}2}\Delta_i)}^{2d_i}
    {(2^{\frac{R_\iota}2}\partial_{r})}^f
g}^2_{L^2(\Rb\times\Rb^{v'}\times\Zb^{v'})}\\
&&\qquad=\,
    \int_{\Tb^{v'}} \int_{\Rb^{v'}} \int_{\Rb}
\nn{
      \Pi_i
      {(2^{\frac{L_\iota}2}
        \widehat{\lambda_i})}^{b_i}
      {(2^{\frac{D_\iota}2}
        (\widehat{l_i}-1))}^{d_i}
      \;
      {(2^{\frac{R_\iota}2}
        \widehat{r})}^f
\;\Fc.g}^2 
    d\widehat{r} d\widehat{\Lambda}d\widehat{l} 
    \quad,  
\end{eqnarray*}
avec $0\leq \nn{b}+f +\nn{d} \leq 4\epsilon$.
En sommant sur ces param\`etres, on a :
\begin{eqnarray*}
  &&\sum_{0\leq \nn{b}+f+\nn{d}\leq 4\epsilon}
\nn{
      \Pi_i
      {(2^{\frac{L_\iota}2}
        \widehat{\lambda_i})}^{b_i}
      {(2^{\frac{D_\iota}2}
        (\widehat{l_i}-1))}^{d_i}
      \;
      {(2^{\frac{R_\iota}2}
        \widehat{r})}^f
}^2 \\
  &&\qquad=\,
  {\left(1+2^{\frac {L_\iota}2}\sum_i
      \widehat{\lambda_i}^2
      +2^{\frac{R_\iota}2}
      \widehat{r}^2
      +2^{\frac{D_\iota}2}\sum_i
      \nn{\widehat{l}_i-1}^2
    \right)}^{4\epsilon}\quad.
\end{eqnarray*}
On en d\'eduit : 
$$
\sum_{0\leq \nn{b}+f+\nn{d}\leq 4\epsilon}
  \nd{F_{P}. g}^2
  \,\leq\,C^2
  \int_{\Tb^{v'}} \int_{\Rb^{v'}} \int_{\Rb}
  \widehat{T_\iota}^{4\epsilon}
  \nn{\Fc.g}^2
  d\widehat{r} d\widehat{\Lambda}d\widehat{l} 
  \quad,
$$
o\`u on d\'efinit la fonction 
$$
\widehat{T_\iota}
\; :\;
\left\{
  \begin{array}{rcl}
    \Rb\times\Rb^{v'}\times\Tb^{v'}
    &\longrightarrow&\Rb^{*+}\\
    (\widehat{r},\widehat{\Lambda},\widehat{l})
    &\longmapsto&
1+2^{\frac {L_\iota}2}\sum_i
      \widehat{\lambda_i}^2
      +2^{\frac{R_\iota}2}
      \widehat{r}^2
      +2^{\frac{D_\iota}2}\sum_i
      \nn{\widehat{l}_i-1}^2
  \end{array}\right.
\quad.
$$

Nous avons donc obtenu :
$$
  \nd{(\Id +s_\iota^{2\epsilon}{(\Xi^2+\aleph)}^{\epsilon} ) g}_{L^2(m)}^2
\leq\,C^2\,
  2^{2d(\eta,\delta)}
\nd{  \widehat{T_\iota}^{2\epsilon}\Fc.g}_{L^2(\Rb\times\Rb^{v'}\times \Tb^{v'})}
\quad.
$$

\subsubsection{Majoration de \mathversion{bold}{$N_{\iota,\epsilon}(g)$} obtenue}

Gr\^ace \`a la formule de Plancherel sur 
$\Rb\times\Rb^{v'}\times\Zb^{v'}$,
la transform\'ee de Fourier de l'op\'erateur~$T_\iota$,
est \'egale \`a la multiplication par la 
fonction~$\widehat{T_\iota}$.
Et donc ici, on a :
$$
\nd{  \widehat{T_\iota}^{2\epsilon}\Fc.g}_{L^2(\Rb\times\Rb^{v'}\times \Tb^{v'})}
=\,
\nd{ T_\iota^{2\epsilon}g}_{L^2(\Rb\times\Rb^{v'}\times \Zb^{v'})}
\quad.
$$

R\'esumons. Nous avons montr\'e jusqu'\`a pr\'esent que
pour tout $\epsilon\in\Nb$, on a :
$$
\begin{array}{c}
  \exists\, C=C(p,\epsilon)>0
  \qquad
  \forall\, \iota\in I
  \qquad 
  \forall\, g\in L^2(m)
  \qquad
  \supp\, g
  \subset 
  \supp \, \chi_\iota
  \qquad :\quad  \\
    N_{\iota,\epsilon}(g)
    \,\leq\,C\,
    2^{d(\eta,\delta)}
   \nd{ T_\iota^{2\epsilon} .g }_{L^2( dr,d\Lambda,dl)}
   \quad.
 \end{array}
 $$
Par interpolation, 
ce r\'esultat est vrai pour tout $\epsilon>0$.
Ceci ach\`eve la d\'emonstration de la 
proposition~\ref{prop_controle_N},
et donc du th\'eor\`eme~\ref{thm_multiplicateurs}.


\chapter{Appendice}

La premi\`ere section de ce chapitre est consacr\'ee
aux propri\'et\'es connues de certaines fonctions sp\'eciales.
La seconde section pr\'ecise 
ce que nous avons appel\'e le passage en polaire 
sur les matrices antisym\'etriques.

\section{Fonctions sp\'eciales}
\label{sec_fonctionspeciale}

Dans cette section, 
nous donnons les propri\'et\'es des fonctions
$\Gamma$,
$\fbs \alpha$ de Bessel,
$\fls n \alpha$ de Laguerre,
$h_k$ de Hermite Weber.

\subsection{Fonction $\Gamma$}
\index{Fonction sp\'eciale!$\Gamma$} 
Nous rappelons :

\begin{itemize}
\item l'\'equation fonctionnelle :
  \begin{equation}
    \label{eq_fonc_Gamma}
    \forall \alpha \in \Cb-\{0,-1,\ldots\},
    \quad
    \alpha\Gamma(\alpha)=\Gamma(\alpha+1)   
    \quad;
  \end{equation}
\item pour $p,q>0$,
  \begin{equation}
    \label{eg_int_Gamma}
    \int_0^1 x^{p-1}{(1-x)}^{q-1} dx
    \,=\,
    \frac{\Gamma(p)\Gamma(q)}{\Gamma(p+q)}
  \end{equation}
\item l'estimation uniforme 
  locale en $x>0$ lorsque $y\rightarrow\infty$ 
  \cite{tit} : 
  $$
  \Gamma(x+iy)
  \,  \sim \,\sqrt{2\pi} 
  e^{-\frac \pi 2 y}\nn{y}^{x-\frac12}  
  \quad ,
  $$
  dont on d\'eduit :
  \begin{equation}
    \label{cqtit}
    \forall [a,b]\subset]0,\infty[\quad
    \exists C>0    \quad
    \forall x\in [a,b]\quad
    \forall y\in \Rb\quad
    \nn{\Gamma(x+iy)}
    \,  \geq \;C^{-1}\; 
    e^{- 2y}  
    \quad ,
  \end{equation}
  et lorsque $n\rightarrow\infty$ :
  \begin{equation}
    \label{cqstir}
    {C_{n+\alpha}^n}^{-1}
    \,=\,
    \frac{\Gamma(n+1)\Gamma(\alpha+1)}{\Gamma(n+\alpha+1)}
    \, \sim  \,
    {(\frac{e}{n})}^\alpha
    \Gamma(\alpha+1)     
    \quad.
  \end{equation}
\end{itemize}

\subsection{Fonction $\fbs \alpha$}\label{app_bessel}

Pour $\alpha>0$,
on d\'efinit les \textbf{fonctions de Bessel r\'eduites}
\index{Fonction sp\'eciale!de Bessel $\fbs \alpha$}
(comme \cite{Far}, chapitre II):
$$
\fb \alpha{z}
\,=\,
\Gamma (\alpha+1){(\frac{z}2)}^{-\alpha}J_\alpha(z)
\,=\,
\sum_{\nu =0}^\infty {(-1)}^\nu 
\frac{z^{2\nu}\Gamma (\alpha+1)}
{2^{2\nu}\nu !\Gamma(\nu +\alpha +1)}
\quad .
$$
Souvent, on omettra le qualificatif ``r\'eduit''.

Gr\^ace \`a son expression en s\'erie enti\`ere 
et \`a l'estimation pour $z\rightarrow\infty$ :
$J_\alpha(z)=O(z^{-\frac{1}{2}})$
\cite{szego} \S  1.71,
la fonction $\fbs \alpha$  v\'erifie :
$$
\fbs \alpha'(z) \,=\, -\frac{\Gamma (\alpha+1)}{\Gamma (\alpha+2)} z
\fb {\alpha+1}{z}
\quad\mbox{et}\quad
\fb \alpha z \,=\, O(z^{-\frac{1}{2}-\alpha})\quad ,
$$
et donc par r\'ecurrence sur $k\in \Nb$:
\begin{equation}
  \label{der_bessel}
  \forall\, x>0 \qquad 
  \nn{\fbs \alpha^{(k)}(x)}\leq C_kx^{-\alpha -\frac{1}2} \quad .
\end{equation}

\begin{lem}[Majoration des int\'egrales pour $\fbs \alpha$]
  \label{lem_maj_int_fb}
  Les int\'egrales suivantes :
  $$
  \int_{s=0}^\infty 
  \nn{\fbs \alpha ^{(k)} (s)}^2  
  s^\beta ds
  $$
  sont finies lorsque $\beta>-1$ et
  $\beta<2\alpha$ pour tout $k\in \Nb$.

  Les int\'egrales suivantes :
  $$
  \int_{s=0}^\infty 
  \nn{\fbs \alpha ^{(k)} (s)}^4  
  s^\beta ds
  $$
  sont finies lorsque $\beta>-1$ et
  $\beta-1<4\alpha$ pour tout $k\in \Nb$.

\end{lem}

\begin{proof}[du lemme~\ref{lem_maj_int_fb}]
  La fonction $\fbs \alpha$ est enti\`ere. 
  Donc les int\'egrales en $0$ sont finies.
  En $\infty$, d'apr\`es la majoration~(\ref{der_bessel}),
  la premi\`ere int\'egrale est finie lorsque :
  $2(-\alpha-\frac12) +\beta<-1$,
  et la seconde lorsque 
  $4(-\alpha-\frac12) +\beta<-1$.
\end{proof}

Nous avons besoin des propri\'et\'es suivantes \cite{Far} :
\begin{lem}
  \label{lem_fcn_bessel_intsph}
  Pour un param\`etre $\alpha$ 
  de la forme $\alpha=(n-2)/2, n\in\Nb,n\geq 3$, 
  la fonction de Bessel r\'eduite se met sous la forme 
  $$
  \fb {\frac {n-2}2}{\nn{x}}
  \,=\,
  \int_{S_1^{(n)}} e^{i<x,y>} d\tilde{\sigma}_n(y)
  \quad,
  $$
  o\`u $\tilde{\sigma}_n$ d\'esigne la mesure 
  sur la sph\`ere unit\'e euclidienne $S_1^{(n)}$ de $\Rb^n$. 
\end{lem}

\begin{lem}
  \label{lem_fcnBessel_equadiff}  
  Pour $n>0$ et $\alpha\in \Cb$,
  le syst\`eme  :
  $$
  \left\{
    \begin{array}{l}
      4x y''+4n y'+\alpha y
      \,=\,
      0\\
      y(0)=1
    \end{array}\right.
  $$
  a pour unique solution $C^\infty$ au voisinage de 0  
  la fonction enti\`ere $y(x)=\fb {n-1}{\mu\sqrt{x}}$ o\`u $\mu^2=\alpha$.
  Elle est born\'ee si et seulement si $\alpha \geq 0$.  
\end{lem}

\subsection{Fonction $\fls n \alpha$}
\label{app_lag}

On note $L_n^{(\alpha)}$ le  polyn\^ome  
de Laguerre de degr\'e $n$ et de param\`etre $\alpha>-1$.
On peut le d\'efinir 
par les conditions d'orthogonalit\'e 
et de normalisation suivantes
\cite{thangavelu,szego} :
\begin{equation}
  \label{cond_orth_norm_poly_lag}
  \forall\; n,m\in \Nb
  \quad:\quad
  \int_{x=0}^\infty 
  e^{-x} x^\alpha L_n^{(\alpha)}(x)dx
  \,=\,
  \Gamma(\alpha+1) 
  C_{n+\alpha}^n \delta_{n,m}
  \quad.
\end{equation}
On appelle \textbf{fonction de Laguerre}
\index{Fonction sp\'eciale!de Laguerre $\fls n \alpha,\flbs n \alpha$} 
de degr\'e $n$ et de param\`etre $\alpha$ 
la fonction not\'ee $\fls n {\alpha}$ sur~$\Rb$ 
donn\'ee par :
$$
\fl{n}{\alpha}{x}=L_n^{\alpha}(x)e^{-\frac{x}{2}}
\quad .
$$    
On aura besoin de la fonction de Laguerre \textbf{normalis\'ee}
de  degr\'e  $n$  et  de param\`etre $\alpha$ 
que l'on note $\flbs n {\alpha}$ et qui est donn\'ee par :
$$
\flbs {n}{\alpha}
\,:=\,
\frac{\fls  {n}{\alpha} }{C_{n+\alpha}^n}
$$

La fonction de Laguerre normalis\'ee 
est born\'ee par sa valeur~1 en~0. 
On omettra souvent le qualificatif ``normalis\'e''. 

Rappelons \cite{szego} :
d'une part le calcul des d\'eriv\'es des polyn\^omes de Laguerre : 
${L_k^{\alpha}}'=-L_{k-1}^{{\alpha}+1}$,
dont on d\'eduit :
\begin{equation}
  \label{derpsi}
  \fls{k}{\alpha }'
  \,=\,
  \frac{-1}{2}\fls{k}{{\alpha}}-\fls{k-1}{{\alpha}+1}
  \quad,
\end{equation}
d'autre part la d\'ecroissance exponentielle 
des fonctions de Laguerre :
$$
\exists \, C>0
\qquad \forall\, k\in\Nb\,,\; x\in\Rb^+
\quad:\quad
\nn{\flb k {\alpha}{x} }
\, \leq \,
Ce^{-\gamma x}
\quad,
$$
dont on d\'eduit
par r\'ecurrence sur l'\'egalit\'e~(\ref{derpsi}), 
que les fonctions de Laguerre normalis\'ees 
$\flbs n {\alpha}$ et leurs d\'eriv\'ees 
sont born\'ees ind\'ependemment du degr\'e $k$ 
(mais pas du param\`etre $\alpha$ et du nombre de d\'eriv\'es).

\begin{lem}[\cite{Far} proposition~V.11]
  \label{lem_fcn_hypergeom}
  \index{Fonction sp\'eciale!$F(\alpha,\gamma;z)$}
  L'\'equation hyperg\'eom\'etrique confluente 
  de param\`etre $\alpha, \gamma \in \Cb$ s'\'ecrit :
  \begin{equation}
    \label{eq_hypergeom_confl}
    zy''+(\gamma-z)y'-\alpha y
    \,=\,0
    \quad.
  \end{equation}
  Pour $\gamma\not= 0,-1,-2,\ldots$,
  on appelle fonction hyperg\'eom\'etrique confluente 
  de param\`etres $\alpha,\gamma$ 
  la fonction enti\`ere not\'ee $F(\alpha,\gamma;.)$ 
  et donn\'ee par :
  $$
  F(\alpha,\gamma;z)
  \,=\,
  \sum_{k=0}^\infty \frac{{(\alpha)}_k}{{(\gamma)}_k}
  \frac{z^k}{k!}
  \quad,
  $$
  o\`u l'on note pour $\beta\in \Cb$ et $k\in \Nb$:
  ${(\beta)}_k=
  \beta(\beta+1)\ldots(\beta+k)$.
  La fonction $F(\alpha,\gamma;.)$  est solution 
  de l'\'equation~(\ref{eq_hypergeom_confl}) 
  de param\`etre $\alpha, \gamma $;
  et toute solution $C^\infty$ au voisinage de $0$ 
  lui est proportionnelle.

  Supposons $\gamma>0$.
  \begin{enumerate}
  \item Si $\alpha$ est un entier n\'egatif :
    $\alpha =-l, l\in \Nb$, 
    alors $F(\alpha,\gamma;.)$ est un polyn\^ome,
    de Laguerre  de degr\'e $l$ \`a un constante pr\`es :
    $$
    F(-l,\gamma;z)
    \,=\,
    \frac1{C_{l+\gamma-1}^l}L_l^{\gamma-1}(z)
    \quad.
    $$
  \item Si $\Re \alpha<\gamma$ 
    et si $\alpha$ n'est pas un entier n\'egatif, 
    alors on a l'estimation suivante pour $z$ grand :
    $$
    F(\alpha,\gamma;z)
    \,\sim\,
    \frac{\Gamma(\gamma)}{\Gamma(\alpha)}
    e^z z^{\alpha-\gamma}
    \quad .
    $$
  \end{enumerate} 
\end{lem}

D'apr\`es~\cite{Mark},
lorsque $\alpha,\alpha+\beta>-1$.   
on a les estimations suivantes :
\begin{equation}\label{estimation_Mark}
  \int_{x=0}^\infty
  \nn{\fl{n}{\alpha+\beta}{x}}^2
  x^{\alpha}
  dx
  \sim
  \left|\begin{array}{ll}
      n^{\alpha}&\mbox{si}\; \beta<\frac12\\
      n^{\alpha}\ln n&\mbox{si}\; \beta=\frac12\\
      n^{\alpha+2\beta-1}&\mbox{si}\; \beta>\frac12
    \end{array}\right. \quad .
\end{equation}
On en d\'eduit le lemme suivant :
\begin{lem}
  \label{lem_cq_Mark}
  Pour $j,\alpha\in\Nb-\{0\}$,
  les int\'egrales :
  $$
  \int_{x=0}^\infty
  \nn{\flbs {l}{\alpha}^{(m)}(x)}^2
  x^{2j-1}
  dx
  \quad\mbox{o\`u}\quad 0\leq m \leq j
  \quad,
  $$  
  sont born\'ees ind\'ependemment de $l$ tant que 
  $j\leq \alpha$.
\end{lem}

\begin{proof}[du lemme~\ref{lem_cq_Mark}]
  Gr\^ace \`a une r\'ecurrence sur~(\ref{derpsi}),  
  $\fls l {v'-1}^{(m)}$
  est  une combinaison  lin\'eaire de  $\fls{l-n}{v'-1+n}$ 
  o\`u  $ 0 \leq n \leq m$;
  ainsi il suffit de majorer pour $ 0\leq n\leq m$
  les int\'egrales :
  $$
  I(\alpha,j,l,m,n)
  \,=\,
  \int_{x=0}^\infty
  \nn{\frac{\fls{l-n}{\alpha+n}(x)}  
    {C^l_{l+\alpha}}}^2 
  x^{2j-1} dx
  \quad.
  $$
  D'apr\`es l'estimation~(\ref{estimation_Mark}) 
  et gr\^ace \`a la formule (\ref{cqstir}), 
  on a les \'equivalents pour $l$ grand :
  \begin{itemize}
  \item si $\alpha +n -(2j-1) >\frac 12$,
    $$
    I(\alpha,j,l,m,n)
    \,\sim\,
    {(l-n)}^{2j-1}
    {\left(    {(\frac{e}{l})}^{\alpha}
        \Gamma(\alpha+1)     \right)}^2
    \,\sim \,
    l^{2j-1-2\alpha}
    e^{2\alpha}{\Gamma(\alpha+1)}^2
    \quad;
    $$
  \item si $\alpha +n -(2j-1) >\frac 12$
    \begin{eqnarray*}
      I(\alpha,j,l,m,n)
      &\sim&
      {(l-n)}^{2j-1+2(\alpha +n -(2j-1))-1}
      {\left(    {(\frac{e}{l})}^{\alpha}
          \Gamma(\alpha+1)     \right)}^2\\
      &\sim &
      l^{-2j-1 +2n}
      e^{2\alpha}{\Gamma(\alpha+1)}^2
      \quad.
    \end{eqnarray*}
  \end{itemize}
  Donc les int\'egrales
  $I(\alpha,j,l,m,n), 0\leq n\leq m\leq j$
  sont born\'ees ind\'ependemment de $l$
  tant que $2j-1-2\alpha \leq 0$.
\end{proof}

\subsection{Propri\'et\'es des fonctions $\flsz n=\fls n 0$}
\label{subsec_prop_fcnlag_0}

On adopte la notation : 
$$
\flsz l
\,:=\,
\flbs l 0
\,=\,
\fls l 0
\quad.
$$

D'apr\`es les conditions~(\ref{cond_orth_norm_poly_lag})
d'orthogonalit\'e et de normalisation 
des polyn\^omes de Laguerre, 
les fonctions : 
$$
\Pi_{j=1}^{p'} \flsz {l_j}
\; : \;
x=(x_1,\ldots, x_{p'})\in \Rb^{p'}
\,\mapsto\, \Pi_{j=1}^{p'} \flz {l_j} {x_j}
\; ,\quad l=(l_1,\ldots, l_{p'})\in\Nb^{p'}
\quad,
$$ 
forment une base orthonormale de l'espace $L^2(\Rb^{p'})$.

\subsubsection{Op\'erateurs de D\'ecalage}

\paragraph{Op\'erateurs \mathversion{bold}{$\tau^+,\tau^-, \Delta$}.}
\index{Notation!Op\'erateur!$\Delta,\tau^\pm$}
Pour une fonction $R$ d\'efinie sur $\Nb$, on d\'efinit
les fonctions de d\'ecalage $\tau^+,\tau^-$ par:
$$
\tau^+ R(l)=R(l+1)
\quad\mbox{et}\quad
\tau^- R(l)=
\left\{
  \begin{array}{ll}
    R(l-1) &\mbox{si}\; l\geq 1\\
    0&\mbox{si}\; l=0
  \end{array}
\right.
$$

On d\'efinit l'op\'erateur de diff\'erence  sur les fonctions de $\Nb$ :
$$
\Delta
\,=\,
\tau^+ -\Id
\quad.
$$

L'op\'erateur $\Delta$ commute avec $\tau^+$ et $\tau^-$.

Nous montrons ici que 
gr\^ace aux propri\'et\'es des polyn\^omes de Laguerre, 
certains op\'erateurs de d\'ecalage 
sur la fonction de Laguerre $\flsz l$ 
sont \'egaux \`a des op\'erateurs 
de d\'erivation ou de multiplication par la variable.

\paragraph{Op\'erateur \mathversion{bold}{$\beta$}.}
\index{Notation!Op\'erateur!$\beta$}
D'apr\`es \cite{szego} page 101, on a :
$$
lL_l (x)=(-x+2l-1)L_{l-1}(x)
-(l-1)L_{l-2}(x)
\quad,
$$
donc ici :
$$
x \flz l x
\,=\,
-(l+1)\flz {l+1}  x +(2l+1)\flz l  x  
-l\flz {l-1}  x
\quad.
$$
On pose alors pour une fonction $R: \Nb\rightarrow\Cb$ :
$$
\beta.R 
\,:=\,  
-(l+1)\tau^+.R +(2l+1)R
-l\tau^-.R
\quad,
$$
et on a :
$$
\beta. \flz l x 
\,=\,
x \flz l x
\quad.
$$

\paragraph{Op\'erateur \mathversion{bold}{$\alpha$}.}
\index{Notation!Op\'erateur!$\alpha$}
Encore d'apr\`es \cite{szego} page 102, on a :
$$
xL_l'(x)= l\left(L_l(x) -L_{l-1}(x)\right)    
\quad,
$$
donc ici :
\begin{eqnarray*}
  x\flsz l  '(x)
  &=&
  xe^{-\frac x 2}\left( L_l'(x) -\frac 12 L_l(x)\right)\\
  &=&
  e^{-\frac x 2}\left( l(L_l(x)-L_{l-1}(x)) -\frac 12 xL_l(x)\right)\\
  &=&
  l(\flz l x - \flz {l-1}  x ) -\frac 12 \beta.\flz l  x
  \quad.
\end{eqnarray*}
On pose  alors pour une fonction $R: \Nb\rightarrow\Cb$ :
$$
\alpha.R
\,:=\,  
l( R-\tau^- .R)-\frac12  \beta.R 
\,:=\,  
-\frac12 R
-\frac {l} 2 \tau^-.R
+\frac {l+1}2\tau^+.R
\quad,
$$
et on a :
$$
\alpha. \flz l x 
\,=\,
x \flsz l '(x)
\quad.
$$

\paragraph{Op\'erateur \mathversion{bold}{$\gamma$}.}
Toujours d'apr\`es \cite{szego} page 102,
on a :
$$
L_l^{(\alpha)} 
\,=\,
L_l^{(\alpha+1)}-
L_{l-1}^{(\alpha+1)}
\quad\mbox{et}\quad
{L_l^{(\alpha)}}'
\,=\,
-L_{l-1}^{(\alpha+1)}
\quad,
$$
et donc :
$$
L_l'-L_{l-1}'
\,=\,
-L_{l-1}
\quad.
$$
En utilisant l'expression de $\alpha$ et cette derni\`ere \'egalit\'e, on a :
\begin{eqnarray*}
  \alpha.\flsz l  '
  &=&
  -\frac12\flsz l '
  -\frac l2 \flsz {l-1} '
  +\frac {l+1}2\flsz {l+1} ' \\   
  &=&
  e^{-\frac x2}\left(
    -\frac12 L_l'
    -\frac l2 L_{l-1}'
    +\frac {l+1}2L_{l+1}'\quad
    -\frac12( -\frac12L_l
    -\frac l2 L_{l-1}
    +\frac {l+1}2 L_{l+1})\right) \\  
  &=&
  e^{-\frac x2}\left(
    -\frac l2 L_{l-1}
    -\frac {l+1}2L_l\quad
    -\frac12( -\frac12L_l
    -\frac l2 L_{l-1}
    +\frac {l+1}2 L_{l+1})\right)\\
  &=&
  e^{-\frac x2}\left(
    -\frac{2l+1}4L_l-\frac l4 L_{l-1}-\frac{l+1}4 L_{l+1}
  \right)\quad.
\end{eqnarray*}
On en d\'eduit :
$$
\alpha.\flsz l  '
\,=\,
-\frac{2l+1}4\flsz l 
-\frac l4 \flsz {l-1} 
-\frac{l+1}4 \flsz {l+1} 
\quad.
$$
On d\'efinit l'op\'erateur sur les fonctions de $\Nb$ :
\index{Notation!Op\'erateur!$\gamma$}
$$
\gamma
\,:=\,
-\frac{2l+1}4I-\frac l4 \tau^- -\frac{l+1}4 \tau^+
\quad,
$$
et on a :
$$
\gamma.\flsz l =\alpha.\flsz l  '
\quad.
$$

\subsubsection{Lien entre les op\'erateurs \mathversion{bold}{$\Delta$} et \mathversion{bold}{$\alpha, \beta,\gamma$}}

On v\'erifie directement :
\begin{eqnarray}
  4(\alpha^2+\alpha)
  &=&
  -{\tau^-}^2\left( (l+2)(l+1)\Delta \right)
  -\tau^- \left( (l+1)l \Delta \right)
  +(l+1)(l+2)(\Delta^2+2\Delta)\quad.
  \label{egalite_alpha_2+1}\\
  2\alpha
  &=&
  (\Id+\tau^-)(l\Delta)+\tau^++\tau^-\quad,
  \label{egalite_puissance_alpha_p=1}\\
  \beta
  &=&
  -\tau^-(l\Delta^2)-(2\Id-\tau^-)\Delta\quad,   
  \label{egalite_puissance_beta_p=1}\\
  4\gamma  
  &=&
  -\tau^-(l\Delta^2)-(2\Id-\tau^-)\Delta +2(2l+1)\Id\quad.
  \label{egalite_puissance_gamma_p=1}
\end{eqnarray}

Sur les fonctions de deux variables $l_1,l_2$ enti\`eres, 
on en d\'eduit en particulier :
\begin{eqnarray}
\alpha_1-\alpha_2
&=&
(\tau_1^-+I)(l_1+1)\Delta_1
-(\tau_2^-+I)(l_2+1)\Delta_2 
\label{egalite_alpha_1-2}\\
  2(\alpha_1+\alpha_2+2\alpha_1\alpha_2 )
  &=&
  (\Id+\tau^-_1)(l_1+1)\Delta_1
  +(\Id+\tau^-_2)(l_2+1)\Delta_2 \nonumber\\
  &&\quad
  +(\Id+\tau^-_1)(\Id+\tau^-_2)\;
  (l_1+1)(l_2+1)\; \Delta_1\Delta_2 
  \quad,  \label{egalite_alpha12_casp=1}\\
  4\gamma_1\beta_2
  &=&
  4\beta_1\beta_2
  -2\tau_2^-(2l_1+1)l_2\Delta_2^2-2(2\Id-\tau_2^-)(2l_1+1)\Delta_2
  \quad. \label{egalite_betagamma}   
\end{eqnarray}

On cherche \`a \'ecrire la puissance de certains op\'erateurs 
en fonction de $\Delta$, la multiplication par $l$,
\`a des d\'ecalages pr\`es.
Pr\'ecisons la forme de ces d\'ecalages :
c'est une combinaison lin\'eaire finie \`a
coefficients rationnels 
en les op\'erateurs ${\tau^+}^{k^+}$ et ${\tau^-}^{k^-}$ 
o\`u $k^+,k^-\in \Nb$.
On peut donc l'\'ecrire sous la forme :
$P(\tau^+,\tau^-):=P(\tau^\pm)$,
o\`u $P\in\Qb[X,Y]$ 
est un polyn\^ome 
de deux variables \`a coefficients rationnels.

Comme
on calcule directement :
$[{\tau^+}^p, l] =p{\tau^+}^p$
et
$[{\tau^-}^p, l] =-p{\tau^-}^p$,
on a :
\begin{equation}
  \label{prop_com_poly_tau}
  \forall \, P\in \Qb[X,Y] \, , \quad
  \exists!\, Q_p\in \Qb[X,Y]
  \quad : \quad
  [P(\tau^\pm),l]
  \,=\,
  Q_P(\tau^\pm)
  \quad.  
\end{equation}

On remarque que les  termes suivants 
(pour $P\in \Qb[X,Y]$):
$$
w:=l\Delta P(\tau^\pm) l^q
\quad,\qquad
x:=l\Delta^2 P(\tau^\pm) l^q
\quad,
$$
peuvent se mettre sous la forme d'une somme de terme du type 
$R(\tau^\pm) l^r \Delta^s$
o\`u $R\in \Qb[X,Y]$ et $r,s\in \Nb$.
En effet, d'une part, 
d'apr\`es la propri\'et\'e~(\ref{prop_com_poly_tau})
et la commutativit\'e des op\'erateurs $\Delta, \tau^\pm$, 
il existe un polyn\^ome 
$Q\in \Qb[X,Y]$ tel que
$$
w\,=\,
\left(P(\tau^\pm)l+Q(\tau^\pm)\right)
\Delta l^q  
\quad\mbox{et}\quad
x\,=\,
\left(P(\tau^\pm)l+Q(\tau^\pm)\right)
\Delta^2 l^q  
\quad;
$$
d'autre part,
on connait le commutateur :
$[\Delta,l^q]=\tau^+ (l^q- {(l-1)}^q)$;
calculons l'autre commutateur $[\Delta^2, l^q]$.
en commen\c cant  par :
$$
\Delta^2 l^q
\,=\,
\Delta l^q\Delta - \Delta [l^q,\Delta]
\,=\,
(l^q \Delta - [l^q,\Delta])\Delta 
- \Delta [l^q,\Delta]
\quad;
$$
on d\'eduit de l'expression de $[\Delta, l^q]$ :
\begin{eqnarray*}
  [\Delta^2, l^q]
  &=&
  - [l^q,\Delta]\Delta 
  - \Delta [l^q,\Delta]\\
  &=&
  -\tau^+ \left(l^q- {(l-1)}^q\right) \Delta
  -\Delta \tau^+ \left(l^q- {(l-1)}^q\right)
  \quad ;\\
  \Delta  \left(l^q- {(l-1)}^q\right)
  &=&
  l^q \Delta -  \tau^+ \left(l^q- {(l-1)}^q\right)\\
  &&\quad
  - {(l-1)}^q\Delta
  +\tau^+ \left({(l-1)}^q- {(l-2)}^q\right)\\
  &=&
  \left( l^q- {(l-1)}^q\right) \Delta 
  +\tau^+ \left(-l^q+2 {(l-1)}^q -{(l-2)}^q\right)
  \quad,
\end{eqnarray*}
puis gr\^ace \`a la commutativit\'e des op\'erateurs
$\Delta$ et $\tau^+$ :
$$
[\Delta^2, l^q]
\,=\,
-\tau^+ \left\{
  2\left(l^q- {(l-1)}^q\right) \Delta
  +\tau^+ \left(-l^q+2 {(l-1)}^q -{(l-2)}^q\right)\right\}
\quad;
$$
on en d\'eduit :
\begin{eqnarray}
  w
  &=&
  l\Delta P(\tau^\pm) l^q\nonumber\\
  &=&
  \left(P(\tau^\pm)l+Q(\tau^\pm)\right)
  (l^q\Delta+ \tau^+ (l^q- {(l-1)}^q))  \quad,
  \label{egalite_ldeltaPl}\\
  x
  &=&
  l\Delta^2 P(\tau^\pm) l^q\nonumber\\
  &=&
  \left(P(\tau^\pm)l+Q(\tau^\pm)\right)
  \left( l^q \Delta^2\right.\nonumber\\
  &&\quad
  \left. + \tau^+\left\{
      2\left(l^q- {(l-1)}^q\right) \Delta
      +\tau^+ \left(-l^q+2 {(l-1)}^q -{(l-2)}^q\right)\right\}
  \right)
  \quad.
  \label{egalite_ldelta2Pl}
\end{eqnarray}

\begin{lem}[Puissance de \mathversion{bold}{$\beta$}]
  \label{lem_puissance_beta}
  Pour $p\in \Nb, p\geq 1$, l'op\'erateur 
  $\beta^p$
  est la somme sur $q=0,\ldots, p$ des op\'erateurs :
  $$
  P_q(\tau^\pm)  l^q \Delta^{p+q}
  \quad\mbox{o\`u}\quad
  P_q\in\Qb[X,Y]\quad.
  $$
\end{lem}

\begin{proof}[du lemme~\ref{lem_puissance_beta}]
  Le cas $p=1$ est \'evident 
  d'apr\`es l'\'egalit\'e~(\ref{egalite_puissance_beta_p=1}).
  Montrons ce lemme par r\'ecurrence. 
  On fixe $p$, et on suppose la propri\'et\'e du lemme vraie au rang~$p$.
  Montrons maintenant que la propri\'et\'e du lemme est alors vraie au rang~$p+1$.
  D'apr\`es l'expression~(\ref{egalite_puissance_beta_p=1}) de $\beta$, 
  et l'hypoth\`ese de r\'ecurrence, on a : 
  \begin{eqnarray*}
    \beta^{p+1}
    &=&
    \left(-\tau^-(l\Delta^2)-(2\Id-\tau^-)\Delta   \right)
    \left( P_q(\tau^-)  l^q \Delta^{p+q}\right)\\
    &=&
    -\tau^-l P_q(\tau^\pm) \Delta^2 l^q \Delta^{p+q}
    -(2\Id-\tau^-)
    \left( P_q(\tau^\pm) \Delta l^q \Delta^{p+q}\right)
    \quad,
  \end{eqnarray*}
  car les op\'erateurs $\tau^-$ et $\Delta$ commutent.
  Le premier terme du membre de gauche est de la forme voulue gr\^ace au 
  calcul~(\ref{egalite_ldelta2Pl});
  le deuxi\`eme l'est \'egalement, car on a d\'ej\`a directement calcul\'e :
  $[\Delta,l^q]=\tau^+ (l^q- {(l-1)}^q)$.
\end{proof}

On aura \'egalement besoin de l'expression des puissances de 
l'op\'erateur~$\partial_\lambda-\alpha/\lambda$
en fonction de $\Delta$ :

\begin{lem}[Puissance de \mathversion{bold}{$\partial_\lambda-\alpha/\lambda$}]
  \label{lem_puissance_der,alpha}
  Pour une fonction 
  $$
  R: 
  \left\{\begin{array}{rcl}
      \Rb\times \Nb
      &\rightarrow& \Cb \\
      \lambda,l
      &\mapsto&
      R(\lambda,l)
    \end{array}\right.
  \quad,
  $$
  l'expression 
  $$
  {\left(\partial_\lambda-\frac{\alpha}{\lambda}\right)}^p.R\, (\lambda,l)
  $$
  peut s'\'ecrire sous la forme d'une somme sur 
  $$
  (a,b,c,d)\in\Nb^4
  \quad\mbox{tels que}\quad
  a+b=p \; ,\qquad
  b+c\leq p \; ,\qquad
  d\leq c
  \quad,
  $$ 
  de termes de la forme :
  $$
  P_{a,b,c,d}(\tau^\pm)
  \lambda^{-a} \partial_\lambda^b (l^d\Delta^c).R\, (\lambda,l)
  \quad,\quad\mbox{o\`u}\quad
  P_{a,b,c,d}\in \Qb[X,Y]
  \quad.
  $$
\end{lem}

\begin{proof}[du lemme~\ref{lem_puissance_der,alpha}]
  Le cas $p=1$ est \'evident 
  d'apr\`es l'\'egalit\'e~(\ref{egalite_puissance_alpha_p=1}).
  Montrons ce lemme par r\'ecurrence. 
  On fixe $p$, et on suppose la propri\'et\'e du lemme vraie au rang~$p$.
  Montrons maintenant que la propri\'et\'e du lemme est alors vraie au rang~$p+1$.
  D'apr\`es l'expression~(\ref{egalite_puissance_alpha_p=1}) de $\alpha$, 
  et l'hypoth\`ese de r\'ecurrence, l'expression  
  $$
  {\left(\partial_\lambda-\frac{\alpha}{\lambda}\right)}^{p+1}.R\,
  (\lambda,l)
  $$
  peut s'\'ecrire sous  la forme d'une somme sur 
  le quadruplet $(a,b,c,d)\in\Nb^4$ tel que $a+b=p$ et
  $b+c\leq p, d\leq c$ de :
  $$
  \left(\partial_\lambda-
    \frac{(\Id+\tau^-)(l\Delta)+\tau^++\tau^-}{2\lambda}\right)
  P_{a,b,c,d}(\tau^\pm)
  \lambda^{-a} \partial_\lambda^b (l^d\Delta^c).R\, (\lambda,l)
  \quad,
  $$
  donc gr\^ace \`a la commutativit\'e des op\'erateurs 
  en $\lambda$ et $l$,
  sous la forme :
  \begin{eqnarray*}
    P_{a,b,c,d}(\tau^\pm)
    \partial_\lambda\lambda^{-a} \partial_\lambda^b (l^d\Delta^c).R\, (\lambda,l)
    +
    \frac12
    (\Id+\tau^-)(l\Delta)
    P_{a,b,c,d}(\tau^\pm)
    \lambda^{-a-1} \partial_\lambda^b (l^d\Delta^c).R\, (\lambda,l)\\
    +
    \frac{\tau^++\tau^-}2
    P_{a,b,c,d}(\tau^\pm)
    \lambda^{-a-1} \partial_\lambda^b (l^d\Delta^c).R\, (\lambda,l)
    \quad,
  \end{eqnarray*}
  Le  dernier terme de la somme pr\'ec\'edente se met sous la forme
  voulue. De m\^eme pour le deuxi\`eme gr\^ace au 
  calcul~(\ref{egalite_ldeltaPl}) et \`a la commutativit\'e des op\'erateurs 
  en $\lambda$ et $l$, ainsi que pour le premier, car on voit :
  $\partial_\lambda\lambda^{-a} \partial_\lambda^b
  =-a\lambda^{-a-1} \partial_\lambda^b
  +\lambda^{-a} \partial_\lambda^{b+1}$.
\end{proof}

\subsection{Fonction de Hermite-Weber}
\label{subsec_fcn_hermiite}

Les \textbf{fonctions de Hermite-Weber}
\index{Fonction sp\'eciale!de Hermite Weber $h_k,h_\alpha$} 
$h_k,k\in\Nb$ sur $\Rb$ sont donn\'ees par :
$$
h_k(x)
\,=\,
{(2^k k!\sqrt{\pi})}^{-\frac k2} e^{-\frac{x^2}2} H_k(x)
\qquad\mbox{o\`u}\qquad 
H_k(s)={(-1)}^k e^{s^2}{(d/ds)}^k e^{-s^2}
\quad,$$ 
$H_k$ est le polyn\^ome de Hermite de degr\'e $k$.

Rappelons (voir section 5.6 de \cite{szego}), que les fonctions de Hermite-Weber $h_k,k\in \Nb$ sur $\Rb$ forment une base orthonormale de $L^2(\Rb)$ et que chaque fonction $h_k$ v\'erifie l'\'equation diff\'erentielle: $y''+(2k+1-x^2)y=0$.

On d\'efinit les \textbf{fonctions de  Hermite-Weber} 
$h_\alpha,\alpha\in\Nb^n$ sur $\Rb^n$ par :
$$
h_\alpha
\,=\,
\Pi_{i=1}^n h_{\alpha_i}
\quad.
$$

\section{Matrices antisym\'etriques}
\label{sec_app_matrice_antisym}

Nous pr\'ecisons ici les coordonn\'ees polaires 
sur l'ensemble $\Ac_v$ des matrices antisym\'etriques de taille $v$. 
\index{Notation!Espace!$\Ac_v$}
On aura \`a distinguer les cas $v=2v'$ et $v=2v'+1$.
On note $O(v)$ le groupe des matrices orthogonales,
et $SO(v)$  le groupe des matrices orthogonales de d\'eterminant 1.

\subsection{R\'eduction}
\label{subsec_reduction}

Le groupe $O(v)$ agit par conjugaison sur $\Ac_v$:
$$
\forall k\in O(v)\; , 
A\in \Ac_v\; :
\quad 
k.A
\,=\,
kAk^{-1}
\quad.
$$
Pour en d\'ecrire les orbites,
on d\'efinit le simplexe $\Lc$,
et son adh\'erence $\bar{\Lc}$ :
  \index{Notation!Ensemble de param\`etres!$\Lc,\bar{\Lc}$}
\begin{eqnarray*}
  \Lc
  &:=&\{\Lambda=(\lambda_1,\ldots,\lambda_{v'})\in\Rb^{v'}
  \; :\;
  \lambda_1\,>\,\ldots\,>\,\lambda_{v'}\,>\,0\, \} \quad,\\
  \bar{\Lc}
  &:=&
  \{\Lambda=(\lambda_1,\ldots,\lambda_{v'})\in\Rb^{v'}
  \; :\;
  \lambda_1\,\geq\,\ldots\,\geq \,\lambda_{v'}\,\geq \,0\, \}
  \quad.
\end{eqnarray*}
\`A un \'el\'ement $\Lambda\in \Rb^{v'}$, on associe 
la matrice antisym\'etrique $D_2(\Lambda)$ de taille $v$ :
\index{Notation!Matrice antisym\'etrique!$D_2(\Lambda)$}
$$
D_2(\Lambda)
\,=\,
\left[
  \begin{array}{cccc}
    \lambda_1 J &&&\\
    0 & \ddots & 0&\\
    &&\lambda_{v'} J&\\
    &&&(0)
  \end{array}\right]
\qquad \mbox{o\`u}\quad
J\,=\,
\left[  \begin{array}{cc}
    0&1\\
    -1&0
  \end{array}\right]
\quad,
$$
\index{Notation!Matrice antisym\'etrique!$J$}
c'est\`a dire :
$$
\mbox{si v=2v'} \, :
\quad 
\left[
  \begin{array}{cccc}
    \lambda_1 J &&&\\
    0 & \ddots & 0&\\
    &&\lambda_{v'} J&
  \end{array}\right]
\qquad
\mbox{et si v=2v'+1} \, :
\quad 
\left[
  \begin{array}{cccc}
    \lambda_1 J &&&\\
    0 & \ddots & 0&\\
    &&\lambda_{v'} J&\\
    &&&0
  \end{array}\right]
\quad.
$$

\begin{prop}[\mathversion{bold}{$\Ac_v/O(v)$}]
  \label{prop_Ac_v/O(v)}
  Toute matrice antisym\'etrique $A\in\Ac_v$
  est orthogonalement semblable \`a une matrice diagonalis\'ee 
  par bloc 2-2.
  En effet, le polyn\^ome caract\'eristique de $A$ est de la forme : 
  $$
  P(x)\,=\,
  \left\{  \begin{array}{ll}
      \Pi_{i=1}^{v'}(x^2-{\lambda_i}^2)
      &\mbox{si}\;v=2v'\\
      x\Pi_{i=1}^{v'}(x^2-{\lambda_i}^2)
      &\mbox{si}\;v=2v'+1
    \end{array}\right. \quad,
  $$ 
  o\`u 
  $\Lambda=(\lambda_1,\ldots,\lambda_{v'})\in \Rb^{v'}$;
  il existe une matrice orthogonale $k\in O(v)$ telle que :
  $A=k^{-1}D_2(\Lambda)k$.
  On peut choisir $\Lambda\in\bar{\Lc}$.
\end{prop}

La d\'emonstration se fait de mani\`ere \'el\'ementaire
par r\'ecurence sur la taille $v$ de la matrice,
et gr\^ace aux propri\'et\'es des endomorphismes normaux.

On d\'efinit plus g\'en\'eralement pour $\Lambda\in \Rb^{v'}$, 
la matrice antisym\'etrique $D_2^{\epsilon}(\Lambda)$ de taille $v$ :
\index{Notation!Matrice antisym\'etrique!$D_2^\epsilon(\Lambda)$}
$$
D_2^{\epsilon}(\Lambda)
\,=\,
\left[
  \begin{array}{cccccc}
    \lambda_1 J &&&&\\
    0 & \ddots & 0&&\\
    &&\lambda_{v'-1} J&&\\
    &&&\epsilon\lambda_{v'} J&\\
    &&&&(0)
  \end{array}\right]
\qquad \mbox{o\`u}\quad
\epsilon\in\{1,-1\}
\quad.
$$
\'Evidemment, on a $D_2^1(\Lambda)=D_2(\Lambda)$,
et $D_2(\lambda_1,\ldots,\lambda_{v'})=D_2(\lambda_1,\ldots,-\lambda_{v'})$.

\begin{prop}[\mathversion{bold}{$\Ac_v/SO(v)$}]
  \label{prop_Ac_v/SO(v)}
  Pour toute matrice antisym\'etrique $A\in\Ac_v$
  il existe $k\in SO(v), \,\Lambda\in\bar{\Lc}$ et $\epsilon=\pm 1$ tels que :
  $A=k^{-1}D^\epsilon_2(\Lambda)k$.
\end{prop}

\subsection{Isomorphisme $Sp(n)\cap O(n) \sim U_n$}
\label{subsec_SpO}

D\'efinissons pour $n\in\Nb-\{0\}$,
les matrices antisym\'etriques de taille $2n\times 2n$ :
$$
J_n
\,:=\,
J^{+1}_n
\,:=\,
\left[  \begin{array}{ccc}
    J&\ldots&0\\
    &\ddots&\\
    0&\ldots& J
  \end{array}\right]
\quad\mbox{et}\quad
J^{-1}_n
\,:=\,
\left[  \begin{array}{ccc}
    J&\ldots&0\\
    &\ddots&\\
    0&\ldots&-J
  \end{array}\right]
\quad.
$$
\index{Notation!Matrice antisym\'etrique!$J_n,J_n^\pm$}
et la complexification:
\index{Notation!Isomorphisme!$\psi_c^n,\psi_c^{n,\pm}$}
\begin{eqnarray*}
  \psi_c^{(n,+1)}
  =
  \psi_c^{(n)}
  & :&  x_1,y_1,\ldots,x_n,y_n
  \;\rightarrow \;
  x_1+i y_1,
  \ldots,
  x_n+iy_n \quad,\\
  \psi_c^{(n,-1)}
  & :&  x_1,y_1,\ldots,x_n,y_n
  \;\rightarrow \;
  x_1+i y_1,
  \ldots,
  x_{n-1}+iy_{n-1},
  y_n+ix_n \quad.
\end{eqnarray*}

\begin{prop}[Matrices orthogonales commutant avec 
  \mathversion{bold}{$J_n^\epsilon$}]
  \label{prop_matrice_ortho_commutant_J}
  Soit $n\in \Nb-\{0\}$.
  Les matrices orthogonales de taille $2n$ 
  qui commutent avec $J_n^\epsilon$
  ont pour d\'eterminant 1. 
  On a un
  isomorphisme $\psi_1^{(n,\epsilon)}$ entre 
  \begin{itemize}
  \item le groupe
    des matrices $2n-2n$ orthogonales qui commutent avec 
    $J_n^\epsilon$,
  \item le groupe $U_n$ des matrices unitaires de taille $n$. 
  \end{itemize}
  Il v\'erifie :
  $$
  \forall\,k,X
  \quad:\quad 
  \psi_c^{(n,\epsilon)} (k.X) =\psi_1^{(n,\epsilon)}(k).\psi_c^{(n,\epsilon)}(X)
  \quad.
  $$
  On notera $\psi_1^{(n,+1)}=\psi_1^{(n)}$.
\end{prop}
\index{Notation!Isomorphisme!$\psi_1^n,\psi_1^{n,\pm}$}
\begin{proof}[rapide de la proposition~\ref{prop_matrice_ortho_commutant_J}]
  Soit $k\in O(2n)$ qui commute avec $J_n^\epsilon$. 
  Une base de l'espace propre 
  associ\'ee \`a la valeur propre $-1$ pour $k$
  peut s'\'ecrire sous la forme :
  $e_1,J_n^\epsilon.e_1,e_2,J_n^\epsilon.e_2\ldots$ 
  donc la dimension de ce sous-espace propre  est paire et $\det k=1$.

  \'Ecrivons la matrice $k$ par bloc 2-2:
  $$
  k=
  \left[\begin{array}{ccc}
      k'_{1,1}&\ldots&k'_{1,n}\\
      &\ddots&\\
      k'_{n,1}&&k'_{n,n}
    \end{array}\right]
  \quad\mbox{o\`u}\quad
  k'_{l,c}=\left[
    \begin{array}{cc}
      k_{2l-1,2c-1}&k_{2l-1,2c}\\
      k_{2l,2c-1}&k_{2l,2c}
    \end{array}\right]\quad;
  $$
  On d\'efinit les coefficients d'une matrice complexe $u$ 
  $u_{l,c}=k_{2l-1,2c-1}+i k_{2l,2c-1}, 
  1\leq l,c\leq n$, si $\epsilon=1$,
  et si $\epsilon=-1$ :
  $$
  u_{l,c}
  \,=\,
  \left\{
    \begin{array}{ll}
      k_{2l-1,2c-1}+i k_{2l,2c-1},
      &\quad\mbox{si}\quad 
      1\leq i,j < n,\\
      k_{2l-1,2n}-ik_{2l-1,2n-1}
      &\quad\mbox{si}\quad 
      (i,j)=(i,n)\not=(n,n),\\
      k_{2n,2c-1}-ik_{2n,2c}
      &\quad\mbox{si}\quad 
      (i,j)=(n,j)\not=(n,n),\\
      k_{2n,2n}-ik_{2n,2n-1}
      &\quad\mbox{si}\quad 
      (i,j)=(n,n).
    \end{array}\right.
  $$
  
  Comme la matrice $k$ est orthogonale, 
  la matrice $u$ est unitaire.
\end{proof}

\subsection{Passage en coordonn\'ees polaires}

En fait, nous allons nous int\'eress\'es \`a l'ensemble 
$\Dc_v$ des matrices antisym\'etriques 
auxquelles on associe comme dans la proposition~\ref{prop_Ac_v/O(v)}
$\Lambda\in\Lc\subset\overline{\Lc}$.
On obtient le passage en coordonn\'ees polaires :
\begin{lem}[Passage en coordonn\'ees polaires]\label{lem_pass_coord_pol}
  \index{Coordonn\'ee polaire!sur les matrices antisym\'etriques}

  Il existe $\eta$ une mesure sur le simplexe $\Lc\subset \Rb^{v'}$
  donn\'ee par :
  \index{Notation!Mesure!$\eta$}
  $$
  d\eta(\Lambda)=
  \left\{\begin{array}{ll}
      c \Pi_{j<k} {(\lambda_j^2-\lambda_k^2)}^2 
      d\Lambda
      &\mbox{si}\; v=2v'\\
      c \Pi_i \lambda_i^2
      \Pi_{j<k} {(\lambda_j^2-\lambda_k^2)}^2 
      d\Lambda
      &\mbox{si}\; v=2v'+1
    \end{array}\right.
  \qquad (\mbox{o\`u}\quad 
  d\Lambda=d\lambda_1\ldots d\lambda_{v'})
  \quad ;
  $$
  la constante $c$ est telle que l'on ait le passage en coordonn\'ees
  polaires sur l'ensemble des matrices antisym\'etriques $\Ac_v$:
  \begin{equation}
    \label{chgtpolma}
    \int_{\Ac_v} g(A) dA \,=\,
    \int_{O(v)} \int_{\Lc}
    g(k.D_2(\Lambda))
    d\eta (\Lambda) dk
    \quad.
  \end{equation}
\end{lem}
Ce lemme est d\'ej\`a bien connu :
dans \cite{stric} (page 398), on  
renvoit \`a \cite{helgason} page 382. 
Nous le montrons ici ``\`a la main''. 

Nous avons besoin aussi 
de r\'esultats sur le Laplacien sur $\Ac_v$. 
Rappelons que 
lorsque la base canonique de l'espace vectoriel $\Ac_v$ 
est $E_{i,j}, i<j$ 
o\`u $E_{i,j}$ 
d\'esigne la matrice antisym\'etrique 
dont toutes les entr\'ees sont nulles 
sauf celle de la $i$\`eme ligne 
et $j$\`eme colonne qui vaut 1 
et celle de la $j$\`eme ligne 
et $i$\`eme colonne qui vaut -1, 
le laplacien sur l'espace vectoriel $\Ac_v$ 
est donn\'e par:
$$
\Delta f
\,=\,
\sum_{i<j} D^2f(E_{i,j},E_{i,j})\quad .
$$
Nous aurons en fait besoin 
de l'expression du Laplacien 
en coordonn\'ees polaires 
sur le sous ensemble $\Dc_v$. 
Pour cela, d\'efinissons l'application 
$$
\psi\,:\;
\left\{
  \begin{array}{rcl}
    O(v)\times \Lc
    &\longrightarrow& 
    \Ac_v\\
    k,\Lambda
    &\longmapsto&
    k.D_2(\Lambda)
  \end{array}\right.
\quad,
$$
dont l'image est $\Dc_v$,
le sous groupe $K_r$ de $O(v)$ comme l'ensemble des matrice de la forme :
$$
\left[ \begin{array}{cccc}
    r_{\theta_1}&&&\\
    0&\ddots&0&\\
    &&r_{\theta_{v'}}&\\
    &&&(1)
  \end{array}\right]
\quad\mbox{o\`u}\quad
\begin{array}{c}
  \mbox{la matrice}\;
  r_\theta\; \mbox{d\'esigne une rotation du plan :}\\
  r_\theta =
  \left[
    \begin{array}{cc}
      \cos \theta&-\sin \theta\\
      \sin \theta&\cos\theta
    \end{array}
  \right]
\end{array}
\quad .
$$
On voit :
$$
K_r=\{ k\in O (v)\quad \forall \Lambda\in \Lc\quad 
\psi(k,\Lambda)=D_2(\Lambda)\} 
\quad.
$$ 
On note 
\begin{itemize}
\item l'espace homog\`ene $\tilde{K}=O(v)/K_r$ 
  comme vari\'et\'e quotient de $O(v)$ par le sous groupe $K_r$
  (l'action consid\'er\'ee est celle \`a droite),
\item $p: O(v)\rightarrow\tilde{K}$
  la submersion canonique,
\item $d\tilde{k}$ la mesure sur la vari\'et\'e homog\`ene induite par
  les deux mesure de Haar normalis\'ee de masse 1.
\end{itemize}
On peut donc d\'efinir l'application que l'on note encore 
$\psi$ : 
$$
\psi\,:\;
\left\{
  \begin{array}{rcl}
    \tilde{K}\times \Lc
    &\longrightarrow& 
    \Ac_v\\
    p(k),\Lambda
    &\longmapsto&
    k.D_2(\Lambda)
  \end{array}\right.
\quad,
$$
qui devient une bijection sur $\Dc_v$,
et aussi un diff\'eomorphisme de vari\'et\'es. 

Pr\'ecisons les espaces tangents de ces vari\'et\'es.
\'Evidemment, on identifie $\Tb \Lc\sim \Rb^{v'}$ et $\Tb\Ac_v\sim \Rb^{\frac{v(v-1)}2}$ en tout point. 
On munit l'espace $O(v)$ des cartes locales au point $k\in O(v)$:
$$
{\left(a_{i,j}\right)}_{1\leq i<j\leq v}
\longmapsto
k\exp \left(\sum_{i<j}a_{i,j} E_{i,j}\right)
\quad,
$$
o\`u $\exp$ d\'esigne l'application matricielle exponentielle, 
et $E_{i,j}, 1\leq i<j\leq v$ la base canonique des matrices antisym\'etriques donn\'ee plus haut.
L'espace tangent au point $k$  peut donc s'identifier \`a
$$
\Tb_k O(v)
\,\sim\,
\{ k\vec{A}\; ,\quad \vec{A} 
\;\mbox{matrice antisym\'etrique de taille p}\,
\}
\quad.
$$
Par passage au quotient, 
lorsque l'on a choisi $k\in O(v)$ tel que $p(k)=\tilde{k}$ pour chaque $\tilde{k}$,
la vari\'et\'es $\tilde{K}$ est munie  des cartes locales au point $\tilde{k}$ :
$$
{\left(a_{i,j}\right)}_{1\leq i<j\leq v, (i,j)\not= (2k-1,2k)}
\longmapsto
p\left(k\exp\left( \sum a_{i,j} E_{i,j}\right)\right)
\quad,
$$
et l'espace tangent au point $\tilde{k}$ peut donc s'identifier \`a
$$
\Tb_{\tilde{k}} \tilde{K}
\,\sim\,
\left\{ k\vec{A}\; ,\quad\vec{A} \;
  \begin{array}{c}
    \mbox{matrice antisym\'etrique de taille p,}\\
    \mbox{dont les blocs 2-2 sur la diagonale sont nuls}
  \end{array}
  \,\right\}
\quad,
$$
dont nous fixons la base canonique form\'ee par les vecteurs 
$kE_{i,j}$, pour
$1\leq i<j\leq v$ et $(i,j)\not= (2k-1,2k)$.
Lorsqu'il n'y aura pas de confusion possible, on notera de la m\^eme mani\`ere $k\in O(v)$ et son image $p(k)$, et de m\^eme pour les vecteurs des espaces tangents $\Tb O(v)$ et $\Tb \tilde{K}$. 
On identifie 
$\Tb_{\tilde{k},\Lambda}(\tilde{K}\times \Lc)\sim \Tb_{\tilde{k}} \tilde{K}\oplus \Tb_\Lambda \Lc$
et sa base canonique est donc form\'ee par les vecteurs suivants:
\begin{itemize}
\item les vecteurs $\vec{E}_{i,j}=kE_{i,j}, i<j,(i,j)\not = (2l-1, 2l)$,
\item les vecteurs $\vec{E}_{2l-1,2l}, l=1,\ldots, v'$, le vecteur colonne de $\Tb \Lc\sim \Rb^{v'}$ dont toutes les entr\'ees sont nulles sauf la $l$\`eme qui vaut 1;
  on peut les identifier au champs de vecteurs $\partial_{\lambda_i}$ d\'eriv\'ee en les variables de $\Lambda$. 
\end{itemize}

Dans la suite, on identifie 
la diff\'erentielle  $D_{k,\Lambda}\psi$
\`a l'application lin\'eaire sur 
$\Tb_{\tilde{k}} \tilde{K}\oplus \Tb_\Lambda \Lc$.
Elle est donn\'ee par :
$$
D_{k,\Lambda}\psi(ke^{t\vec{A}},\vec{\Lambda})
\,=\,
{\frac{d}{dt}}_{t=0}
{\frac{d}{du}}_{u=0}
\psi (ke^{t\vec{A}},\Lambda+u\vec{\Lambda})
\quad.
$$

\begin{lem}[Laplacien en coordonn\'ees polaires]\label{lem_laplacien_coord_pol}
  \index{Coordonn\'ee polaire!sur les matrices antisym\'etriques!et laplacien}

  Soit $f$ une fonction sur $\Ac_v, v\geq 1$.
  On pose $g=f\circ\psi$ : c'est une fonction sur $\tilde{K}\times \Lc$. 
  On a  pour $A=\psi(k,\Lambda)$ :
  \begin{eqnarray*}
    \Delta f(A)
    &=&
    \sum_l 
    \partial^2_{\lambda_l}g
    +\sum_{i<j}\sum_{(m,n)\in I_{i,j}}D^2_{k,\Lambda}g(D\psi^{-1}(k.E_{m,n}),D\psi^{-1}(k.E_{m,n}))\\
    &&\quad
    +\sum_{i<j}
    -\frac4{\lambda_j^2-\lambda_i^2}(\lambda_i\partial_{\lambda_i}-\lambda_j\partial_{\lambda_j}).g \\
    &&
    \left\{+\sum_i\sum_{\tilde{i}=2i,2i-1} \frac1{\lambda_i^2}
      D^2.g(\vec{E}_{\tilde{i},v},\vec{E}_{\tilde{i},v})
      -\frac{2}{\lambda_i}\partial_{\lambda_i}.g
      \;\mbox{si}\, v=2v'+1\right\}
    \quad ;
  \end{eqnarray*}
  On a omis l'argument $(k,\Lambda)$ de $g$ et de ses d\'eriv\'ees.
  Les sommes sur un seul indice se font sur l'ensemble $\{ 1, \ldots, v'\}$. 
  Les sommes indic\'ees par $i<j$ se font sur l'ensemble $\{1\leq i<j<v'\}$, qui est vide lorsque $v'=1$. 
  Pour $1\leq i<j \leq v', v'>1$, on a not\'e $I_{i,j}=\{(2i-1,2j-1),(2i,2j),(2i-1,2j),(2i,2j-1)\}$. 

  On a le calcul explicite des 
  $D\psi^{-1}(k.E_{m,n})=D_{k,\Lambda}\psi^{-1}(k.E_{m,n}), (m,n)\in I_{i,j}$ :
  $$
  \begin{array}{@{D\psi^{-1}(}l@{)\,=\,}l}
    k.E_{2i-1,2j-1}
    &
    \frac1{\lambda_j^2-\lambda_i^2}
    (-\lambda_j\vec{E}_{2i-1,2j}+\lambda_i\vec{E}_{2i,2j-1})\\
    k.E_{2i,2j}
    &
    \frac1{\lambda_j^2-\lambda_i^2}
    (-\lambda_i\vec{E}_{2i-1,2j}+\lambda_j\vec{E}_{2i,2j-1})\\
    k.E_{2i-1,2j}
    &
    \frac1{\lambda_j^2-\lambda_i^2}
    (\lambda_j\vec{E}_{2i-1,2j-1}+\lambda_i\vec{E}_{2i,2j})\\
    k.E_{2i,2j-1}
    &
    \frac1{\lambda_j^2-\lambda_i^2}
    (-\lambda_i\vec{E}_{2i-1,2j-1}-\lambda_j\vec{E}_{2i,2j})
    \quad ,
  \end{array}
  $$
  o\`u on a not\'e $\vec{E}_{i,j}=kE_{i,j}, i<j,(i,j)\not = (2l-1, 2l)$.
\end{lem}

\subsection{D\'emonstrations}

Nous d\'emontrons ici les deux lemmes de la sous section pr\'ec\'edente.

\subsubsection{D\'emonstration de lemme~\ref{lem_pass_coord_pol}}

Comme $\Dc_v$ est un ouvert dense de~$\Ac_v$,
on a :
$$
\int_{\Ac_v} g(A) dA
\,=\,
\int_{\Dc_v} g(A) dA
\,=\,
\int_\Lc \int_{\tilde{K}}
g(k.D_2(\Lambda))
\nn{\det D_{k,\Lambda}\psi} \quad
d\Lambda dk \quad ,
$$
o\`u  $\det D_{k,\Lambda}\psi$ le d\'eterminant de l'application 
lin\'eaire $ D_{k,\Lambda}\psi$ vue dans les bases canoniques des espaces tangents.

On note $I_{i,j}$ est l'ensemble ordonn\'e :
$$
I_{i,j}
\,=\,
\{(2i-1,2j-1),(2i,2j),(2i-1,2j),(2i,2j-1)\}
\quad.
$$

Nous allons expliciter $D_{k,\Lambda}\psi$, son d\'eterminant et ainsi que son inverse $D\psi^{-1}$ dont nous aurons besoin dans la d\'emonstration de lemme~\ref{lem_laplacien_coord_pol}.

On a pour $k\vec{A}\in \Tb_k O(v)$ :
$$
D_{(k,\Lambda)}\psi(k\vec{A})
=
k.[\vec{A},D_2(\Lambda)] 
\quad\mbox{(le crochet est celui des matrices)}
\quad .
$$
On a pour $\vec{\Lambda}\in \Tb \Lc$
$$
D_{(k,\Lambda)}\psi(\vec{\Lambda})
=
k .D_2(\vec{\Lambda})
\quad .
$$

L'application lin\'eaire $D_{(k,\Lambda)}\psi$ envoie 
\begin{itemize}
\item $\vec{E}_{2l-1,2l}$ sur $k.E_{2l-1,2l}$, $l=1,\ldots,v'$. Donc on a:
  $$
  D\psi^{-1} k.E_{2l-1,2l}=\vec{E}_{2l-1,2l}=\partial_{\lambda_l}
  \quad .
  $$
\item pour tout $i<j$ (lorsque $v'>1$), l'espace vectoriel 
  $\vec{V}_{i,j}$ engendr\'e par les vecteurs $\vec{E}_{l,m},(l,m)\in I_{i,j}$ 
  sur l'espace vectoriel 
  $k.V_{i,j}$ engendr\'e par les vecteurs $k.E_{l,m},(l,m)\in I_{i,j}$. 
  De plus l'application $D_{(k,\Lambda)}\psi$ restreinte au d\'epart \`a $\vec{V}_{i,j}$ et \`a l'arriv\'ee \`a $k.V_{i,j}$
  se repr\'esente dans la base de d\'epart $\vec{E}_{l,m},(l,m)\in I_{i,j}$ et d'arriv\'ee
  $k.E_{l,m},(l,m)\in I_{i,j}$ par la matrice:
  $$
  \left[\begin{array}{cccc}
      0&0&-\lambda_j&-\lambda_i\\
      0&0&\lambda_i&\lambda_j\\
      \lambda_j&-\lambda_i&0&0\\
      \lambda_i&-\lambda_j&0&0\\
    \end{array}\right]
  \quad,
  $$
  dont le d\'eterminant est ${(\lambda_i^2-\lambda_j^2)}^2$, et l'inverse est:
  $$
  \frac1{\lambda_j^2-\lambda_i^2}
  \left[\begin{array}{cccc}
      0&0&\lambda_j&-\lambda_i\\
      0&0&\lambda_i&-\lambda_j\\
      -\lambda_j&-\lambda_i&0&0\\
      \lambda_i&\lambda_j&0&0\\
    \end{array}\right]
  \quad .
  $$
  On en d\'eduit le calcul explicite des $D\psi^{-1}(k.E_{m,n}), (m,n)\in I_{i,j}$ donn\'e dans l'\'enonc\'e du lemme~\ref{lem_laplacien_coord_pol}.
\item $\vec{E}_{2i,v}$ sur $-\lambda_i k.E_{2i-1,v}$ et
  $\vec{E}_{2i-1,v}$ sur $\lambda_i k.E_{2i,v}$ (lorsque $v=2v'+1$) pour $i=1,\ldots,v'$. Et donc en restriction \`a $\vec{E}_{2i,v},\vec{E}_{2i-1,v}$ et \`a son image, le d\'eterminant de $D\psi$ est $\lambda_i^2$ et son inverse est donn\'ee par :
  $$
  D\psi^{-1}(k.E_{2i,v})=\frac 1{\lambda_i}\vec{E}_{2i-1,v}
  \quad\mbox{et}\quad
  D\psi^{-1}(k.E_{2i-1,v})=
  -\frac 1{\lambda_i}\vec{E}_{2i,v}
  \quad .
  $$
\end{itemize}
On en d\'eduit :
$$
\nn{\det D_{k,\Lambda}\psi}
\,=\,
\left\{\begin{array}{ll}
    \Pi_{j<k} {(\lambda_j^2-\lambda_k^2)}^2 
    &\mbox{si}\; v=2v'\\
    \Pi_i \lambda_i^2
    \Pi_{j<k} {(\lambda_j^2-\lambda_k^2)}^2 
    &\mbox{si}\; v=2v'+1
  \end{array}\right.
\quad,
$$

\subsubsection{D\'emonstration du lemme~\ref{lem_laplacien_coord_pol}}

Comme le laplacien est \'egale \`a la trace de la matrice hessienne dans la base canonique $\{E_{m,n}\}$, et dans toute autre base orthonorm\'ee $\{k.E_{m,n}\}$ pour tout $k\in O(v)$ et en particulier pour $k$ de la d\'ecomposition en polaire de $A$, on a:
$$
\Delta f (A=k.D_2(\Lambda))
=
\sum_{m<n} D^2f(E_{m,n},E_{m,n})
=
\sum_{m<n} D^2f(k.E_{m,n},k.E_{m,n})
\quad .
$$

Comme $Df\circ D\psi=Dg$, on a:
$$
D^2f(D\psi,D\psi)+DfD^2\psi
= D^2g
\quad ,
$$
et donc sur $\Dc_v$, on a:
\begin{equation}\label{eq_der2}
  D^2f=\left( D^2g -Dg\circ D\psi^{-1}\circ D^2\psi \right)(D\psi^{-1},D\psi^{-1})
  \quad .
\end{equation}

Pour obtenir la formule, il nous suffit d'expliciter l'op\'erateur encore $D^2\psi$ ainsi que les expressions $D\psi^{-1}(A)$ et $D\psi^{-1}\circ D^2\psi(D\psi^{-1}A,D\psi^{-1}A)$ pour des matrices $A=k.E_{m,n}$.

Concernant $D^2\psi$, on voit que:
\begin{eqnarray}
  D^2\psi(k\vec{A}_1,k\vec{A}_2)
  &=&
  k.[\vec{A}_2,[\vec{A}_1,D_2(\Lambda)]]\quad ,
  \label{der2aa} \\
  D^2\psi(\vec{\Lambda}_1,\vec{\Lambda}_2)
  &=&0\label{der2ll} \quad .  
\end{eqnarray}

Explicitons $D\psi^{-1}\circ D^2\psi (D\psi^{-1}(A),D\psi^{-1}(A))$:
\begin{itemize}
\item pour $A=k.E_{2l-1,2l}$, d'apr\`es (\ref{der2ll}):
  $$
  D^2\psi (D\psi^{-1}(k.E_{2l-1,2l}),D\psi^{-1}(k.E_{2l-1,2l}))
  =
  D^2\psi(\vec{E}_{2l-1,2l},\vec{E}_{2l-1,2l})=0
  \quad ,
  $$
  et donc
  $D\psi^{-1}\circ D^2\psi (D\psi^{-1}(A),D\psi^{-1}(A))=0$.
\item pour $A\in k.V_{i,j}$ (lorsque $v'>1$): 
  on pose 
  $k\vec{A}=D\psi^{-1}(A)\in  \vec{V}_{i,j}$ 
  calcul\'e pr\'ec\'edemment. 
  On a donc d'une part :
  $$
  D\psi(k\vec{A})=A=k.[\vec{A},D_2(\Lambda)] \quad ,
  $$
  et d'autre part, d'apr\`es (\ref{der2aa}) :
  \begin{eqnarray*}
    D^2\psi (D\psi^{-1}(A),D\psi^{-1}(A))
    &=&
    D^2\psi (\vec{A},\vec{A})
    =
    k.[\vec{A},[\vec{A},D_2(\Lambda)]]\\
    &=&
    k.[\vec{A},k^{-1}.A]
    =
    [k.\vec{A},A] \quad .
  \end{eqnarray*}

  Dans les 4 cas $A=k.E_{m,n}, (m,n)\in I_{i,j}$,
  on a
  $k.\vec{A}=\vec{E}_{m,n}$ par choix de notation donc :
  $$
  {[} k. \vec{A} , A {]}
  =
  \frac1{\lambda_j^2-\lambda_i^2}
  (-\lambda_jk.E_{2j-1,2j}+\lambda_ik.E_{2i-1,2i})
  \quad ,
  $$
  d'apr\`es ce qui pr\'ec\`ede,
  puis :
  $$
  D\psi^{-1}\circ D^2\psi
  (D\psi^{-1}(A),D\psi^{-1}(A))
  =
  \frac1{\lambda_j^2-\lambda_i^2}
  (-\lambda_j\vec{E}_{2j-1,2j}+\lambda_i\vec{E}_{2i-1,2i})
  \quad .
  $$
\item pour $A=k.E_{l,v}, l=2i,2i-1,1\leq i\leq v'$ (lorsque $v=2v'+1$) dans les deux cas :
  $$
  D^2\psi (D\psi^{-1}(A),D\psi^{-1}(A))
  =[k.\vec{A},A]=
  -\lambda_i^{-1}k.E_{2i-1,2i}
  \quad ,
  $$
  puis
  $$
  D\psi^{-1}\circ D^2\psi
  (D\psi^{-1}(A),D\psi^{-1}(A))  
  =
  -\lambda_i^{-1}\vec{E}_{2i-1,2i}
  \quad .
  $$
\end{itemize}

Et donc, gr\^ace \`a l'\'egalit\'e~(\ref{eq_der2}), 
on obtient les expressions de $D^2f(kE_{m,n})$, puis l'on somme. 
On obtient l'expression de lemme~\ref{lem_laplacien_coord_pol} voulue.


\addcontentsline{toc}{chapter}{\numberline{}Index}

\printindex

\bibliographystyle{alpha}

\bibliography{bibthese}

\end{document}